\documentclass [leqno, spanish, 10pt]{amsart}

\usepackage[T1]{fontenc}
\usepackage[utf8]{inputenc}

\usepackage{geometry}

\usepackage[headings, myheadings]{fullpage}

\usepackage{lmodern}
\usepackage{microtype}

\usepackage[all]{xy}
\usepackage{amssymb, amsmath, amsfonts, amsthm, amsbsy, amscd, amsxtra, stmaryrd}
\usepackage[bbgreekl]{mathbbol} 
\usepackage[vcentermath]{youngtab} 
\usepackage{pgf,interval}
\usepackage{ytableau}
\usepackage[mathscr]{euscript}
\usepackage{upgreek}
\usepackage{mathtools}
\usepackage{tikz}

\usetikzlibrary{arrows}
\usepackage{subcaption}

\usepackage{url}
\usepackage{hyperref}
\hypersetup{
	colorlinks=true, linktocpage=true, 
}

\usepackage{etoolbox}
\apptocmd{\sloppy}{\hbadness 10000\relax}{}{}

\usepackage{thmtools} 
\newtheorem{theorem}{Theorem}[section]
\newtheorem{corollary}[theorem]{Corollary}
\newtheorem{lemma}[theorem]{Lemma}
\newtheorem{definition}[theorem]{Definition}

\theoremstyle{definition}

\declaretheoremstyle[
  headfont=\normalfont\bfseries,
  sharenumber = theorem,
  bodyfont=\normalfont,
    qed = {\hbox{$\triangleleft$}}
]{examplestyle2}

\declaretheorem[
  style=examplestyle2,
  title=Example,
  refname={example,examples},
  Refname={Example,Examples}
]{example}

\declaretheorem[
  style=examplestyle2,
  title=Remark,
  refname={remark, remarks},
  Refname={Remark, Remarks}
]{remark}
\numberwithin{equation}{section}

\makeatletter
\def\l@subsection{\@tocline{2}{0pt}{1pc}{4.6em}{}}
\renewcommand{\tocsubsection}[3]{%
  \indentlabel{\@ifnotempty{#2}{\hspace*{2.3em}\makebox[2.3em][l]{%
    \ignorespaces#1 #2.\hfill}}}#3}
\makeatother

\newcommand{\rat}{\mathbb{Q}}
\newcommand{\cc}{\omega}
\newcommand{\one}{\mathbb{1}}
\newcommand{\monster}{\mathbb{M}}
\newcommand{\bw}{\mathscr{W}}

\newcommand{\cmlt}{\bar{\circ}}
\newcommand{\eperp}{\lb e \ra^{\perp}}

\newcommand{\cdmlt}{\odot}
\newcommand{\herm}{\mathbb{Herm}}

\newcommand{\mat}{\mathbb{M}}
\newcommand{\sect}{\mathscr{K}}
\newcommand{\isect}{\mathscr{T}}
\newcommand{\msect}{\mathscr{M}}
\newcommand{\fl}{\si}

\newcommand{\szero}{\mathbb{N}}

\newcommand{\midem}{\mathbb{M}}
\newcommand{\idem}{\mathbb{I}}
\newcommand{\spec}{\Pi}
\newcommand{\specp}{\spec^{\perp}}

\newcommand{\talg}{\mathbb{S}}

\newcommand{\sphere}{\mathbb{S}}

\newcommand{\ones}{\mathbb{1}}

\newcommand{\quat}{\mathbb{H}}
\newcommand{\cayley}{\mathbb{O}}

\newcommand{\ideal}{\mathbb{I}}

\newcommand{\fie}{\mathbb{k}}
\newcommand{\fiel}{\mathbb{l}}

\newcommand{\so}{\mathfrak{so}}

\newcommand{\chr}{\operatorname{char}}
\newcommand{\diag}{\Delta}
\newcommand{\zalg}{\mathbb{Z}}
\newcommand{\trip}{\mathfrak{T}}
\newcommand{\tmlt}{\odot}
\newcommand{\nahm}{\mathfrak{Nahm}}
\newcommand{\mtens}{\mu}
\newcommand{\alg}{\mathbb{A}}
\newcommand{\halg}{\hat{\mathbb{A}}}
\newcommand{\balg}{\mathbb{B}}

\newcommand{\cl}{\mathfrak{c}}
\newcommand{\hbalg}{\hat{\balg}}
\newcommand{\mprod}{\circ}

\newcommand{\h}{\mathfrak{h}}
\newcommand{\kl}{\mathfrak{k}}

\newcommand{\om}{\omega}

\newcommand{\mlt}{\circ}
\newcommand{\hmlt}{\hat{\circ}}

\newcommand{\ricc}{\bar{\rho}}
\newcommand{\lalg}{\bar{\alg}}
\newcommand{\lmlt}{\bar{\mlt}}
\newcommand{\nmlt}{\tilde{\mlt}}
\newcommand{\bmlt}{\lmlt_{\be}}

\newcommand{\ka}{\kappa}

\renewcommand{\part}{\vdash}

\newcommand{\Id}{\operatorname{Id}}

\newcommand{\dum}{\,\cdot\,\,}
\newcommand{\Ga}{\Gamma}

\newcommand{\cycle}{\mathsf{Cycle}}

\def\cn#1{c_{#1}}
\def\bn#1{b_{#1}}

\newcommand{\ric}{\operatorname{\rho}}

\newcommand{\id}{\operatorname{Id}}

\newcommand{\scal}{\si}

\renewcommand{\j}{\mathsf{i}}

\newcommand{\djrd}{\odot}

\renewcommand{\S}{\mathscr{S}}

\newcommand{\la}{\lambda}
\newcommand{\ep}{\epsilon}

\newcommand{\fiet}{\fie^{\times}}

\newcommand{\ext}{\bigwedge}

\newcommand{\perm}{\operatorname{perm}}

\newcommand{\eno}{\operatorname{End}}

\newcommand{\si}{\sigma}

\newcommand{\sign}{\operatorname{sgn}}

\newcommand{\integer}{\mathbb{Z}}

\newcommand{\re}{\operatorname{Re}}
\newcommand{\im}{\operatorname{im}}

\newcommand{\B}{\mathcal{B}}

\newcommand{\lb}{\langle}
\newcommand{\ra}{\rangle}

\newcommand{\ste}{\mathbb{V}}

\newcommand{\al}{\alpha}
\newcommand{\be}{\beta}
\newcommand{\ga}{\gamma}
\newcommand{\spn}{\operatorname{Span}}

\newcommand{\su}{\mathfrak{su}}

\newcommand{\proj}{\mathbb{P}}

\newcommand{\Aut}{\operatorname{Aut}}

\newcommand{\g}{\mathfrak{g}}

\newcommand{\ad}{\operatorname{ad}}
\newcommand{\Ad}{\operatorname{Ad}}

\newcommand{\tensor}{\otimes}

\newcommand{\rea}{\mathbb R}
\newcommand{\com}{\mathbb C}

\newcommand{\tr}{\operatorname{tr}}

\DeclareMathOperator{\aut}{\mathfrak{aut}}

\newcommand{\der}{\operatorname{Der}}

\newcommand{\ealg}{\mathbb{E}}

\let\oldtocsection=\tocsection
\let\oldtocsubsection=\tocsubsection
\renewcommand{\tocsection}[2]{\hspace{0em}\oldtocsection{#1}{#2}}
\renewcommand{\tocsubsection}[2]{\hspace{1em}\oldtocsubsection{#1}{#2}}

\begin{document}
\title[Commutative algebras with nondegenerate invariant trace form]{Commutative algebras with nondegenerate invariant trace form and trace-free multiplication endomorphisms}

\author{Daniel J.~F. Fox}
\address{Departamento de Matemática Aplicada a la Ingeniería Industrial\\ Escuela Técnica Superior de Ingeniería y Diseño Industrial\\ Universidad Politécnica de Madrid\\Ronda de Valencia 3\\ 28012 Madrid España}
\email{daniel.fox@upm.es} 



\begin{abstract}
A commutative algebra is exact if its multiplication endomorphisms are trace-free and is Killing metrized if its Killing type trace-form is nondegenerate and invariant. A Killing metrized exact commutative algebra is necessarily neither unital nor associative. Such algebras can be viewed as commutative analogues of semisimple Lie algebras or, alternatively, as nonassociative generalizations of étale (associative) algebras. Some basic examples are described and there are introduced quantitative measures of nonassociativity, formally analogous to curvatures of connections, that serve to facilitate the organization and characterization of these algebras.
\end{abstract}

\maketitle

\setcounter{tocdepth}{1}  

\tableofcontents

\section{Introduction}
A commutative algebra $(\alg, \mlt)$ is \emph{exact} if its multiplication endomorphism $L_{\mlt}(x) \in \eno(\alg)$ defined by $L_{\mlt}(x)y = x\mlt y$ are trace-free, and is \emph{Killing metrized} if its \emph{Killing form} $\tau_{\mlt}(x, y) = \tr L_{\mlt}(x)L_{\mlt}(y)$ is nondegenerate and invariant. Such an algebra has no unit and is not associative. This article provides evidence that the class of Killing metrized exact commutative algebras is a tractable commutative analogue of the class of semisimple Lie algebras, and introduces a quantitative measure of sectional nonassociativity, formally analogous to sectional curvature, that generalizes the usual Norton inequality and can be used to organize the study of such algebras and, in some particular cases, to characterize and classify them.

Throughout the article the base field is denoted $\fie$ and $\fiet$ is the multiplicative group comprising its units. Where not specified the characteristic of $\fie$ is assumed to be zero, although this is stated explicitly where necessary. For some constructions that depend on positivity it is assumed that $\fie = \rea$.

An \emph{algebra} means a $\fie$-vector space $\alg$ equipped with a $\fie$-bilinear map $\mlt:\alg \times \alg \to \alg$ that is called the \emph{multiplication}. It need not be either unital or associative. In this paper all algebras are finite-dimensional as $\fie$-vector spaces, and this assumption is not repeated.

Let $(\alg, \mlt)$ be a commutative algebra. The image $L_{\mlt}(x)$ of the linear map $L_{\mlt}:\alg \to \eno_{\fie}(\alg)$ defined by $L_{\mlt}(x)y = x\mlt y$ is called the \emph{multiplication endomorphism} determined by $x \in \alg$.

A bilinear form $h \in \tensor^{2}\alg^{\ast}$ is \emph{invariant} if for all $x, y, z \in \alg$,
\begin{align}
&h(x\mlt y, z) = h(x, y\mlt z).
\end{align}
The algebra $(\alg, \mlt)$ is \emph{metrized} if there is a nondegenerate invariant symmetric bilinear form $h \in S^{2}\alg^{\ast}$ (a \emph{metric}). For background on metrized commutative algebras see \cite{Bordemann, Nadirashvili-Tkachev-Vladuts}. 

Provided $\chr \fie$ is relatively prime to $6$, the \emph{cubic polynomial} $P \in S^{3}\alg^{\ast}$ of the metrized commutative algebra $(\alg, \mlt, h)$ is defined by defined by $6P(x) = h(x\mlt x, x)$. The multiplication of $(\alg, \mlt, h)$ is determined completely from its associated cubic polynomial via polarization by
\begin{align}\label{hpolar}
\begin{aligned}
h(x, y\mlt z) =& P(x + y + z) - P(x + y) - P(y + z) - P(x + z) + P(x) + P(y) + P(z).
\end{aligned}
\end{align}
Metrized commutative algebras can be studied by applying techniques from analysis and invariant theory to the cubic polynomial \cite{Fox-cubicpoly, Nadirashvili-Tkachev-Vladuts}. 

That a commutative algebra be metrized means that its multiplication endomorphisms $L_{\mlt}(x)$ are $h$-self-adjoint. This has strong consequences for the spectral properties of multiplication endomorphisms $L_{\mlt}(x)$. If $\fie = \rea$ and $h$ has definite signature, it means that the $L_{\mlt}(x)$ are semisimple with real eigenvalues. 

Despite its good features, the class of metrized commutative algebras is too general to admit a good structure theory. In order to have some hope of reasonable classification and characterization some further conditions need to be imposed. One possibility is to link the invariant metric to the algebra structure in some a priori fashion. The conditions considered here are that certain particular trace-forms, the Ricci and Killing forms, be invariant and nondegenerate.

For a Lie algebra the Killing form is automatically invariant. For a Jordan algebra over a field of characteristic not equal to $2$, the trace-form $\tr L_{\mlt}(x \mlt y)$ is automatically invariant \cite{Schafer}, but the Killing type trace-form generally is not invariant. For a commutative algebra the invariance of some symmetric bilinear form is in general not automatic and the choice of the particular trace-form required to be invariant or nondegenerate matters and conditions requiring such invariance or nondegeneracy should be regarded as part of the algebraic structure. To explain this more precisely, some definitions are needed.

An \emph{automorphism} of an algebra $(\alg, \mlt)$ is an invertible endomorphism $\phi \in \eno(\alg)$ such that $\phi(x \mlt y) = \phi(x)\mlt \phi(y)$. The automorphisms of $(\alg, \mlt)$ constitute a closed subgroup $\Aut(\alg, \mlt)$ of $GL(\alg)$. An automorphism of a metrized commutative algebra $(\alg, \mlt, h)$ is an algebra automorphism that preserves $h$. The group $\Aut(\alg, \mlt, h)$ of automorphisms of $(\alg, \mlt, h)$ equals the intersection $\Aut(\alg, \mlt) \cap O(h)$, where $O(h)$ is the orthogonal group of $h$. 

By a trace-form is meant a multilinear form constructed from powers of multiplication endomorphisms and the multiplication itself. The importance of trace-forms is that they are automorphism invariants. For a general discussion of trace-forms from the point of view of invariant theory see \cite{Popov}. 

Any linear trace-form is a multiple of $\tr L_{\mlt}(x)$, so the only interesting general condition involving linear trace-forms is that $\tr L_{\mlt}(x) = 0$ for all $x \in \alg$. A commutative algebra $(\alg, \mlt)$ is \emph{exact} if $\tr L_{\mlt}(x) = 0$ for all $x \in \alg$. Exactness of a commutative algebra is analogous to the unimodularity of a Lie algebra. Here it serves as a helpful simplifying assumption, in the same sense that restricting to unimodular Lie algebras excludes the wilder solvable Lie algebras. 

The most general bilinear trace-form is a linear combination of the form
\begin{align}
\al \tr L_{\mlt}(x)L_{\mlt}(y) + \be \tr L_{\mlt}(x) \tr L_{\mlt}(y) + \ga \tr L_{\mlt}(x \mlt y).
\end{align}
The invariance or nondegeneracy of a particular trace-form are additional structural conditions not satisfied by a generic commutative algebra.
In the presence of the exactness condition the only possibility is the Killing type trace-form 
\begin{align}
\tau_{\mlt}(x, y) = \tr L_{\mlt}(x)L_{\mlt}(y). 
\end{align}
A commutative algebra is \emph{Killing invariant} if its Killing form is invariant, and is \emph{Killing metrized} if its Killing form is invariant and nondegenerate.
For a Killing metrized commutative algebra, an automorphism necessarily preserves the Killing form, so in this case $\Aut(\alg, \mlt) = \Aut(\alg, \mlt, \tau_{\mlt})$.

The invariance and nondegeneracy of the Killing form of a Killing metrized commutative algebra can be further strengthened. Over a general base field $\fie$, $\tau_{\mlt}$ can be supposed to be moreover anisotropic, and when $\fie = \rea$ (or, more generally, $\fie$ is real closed) it can be supposed to be positive definite. Stronger results are obtained in these settings.

One motivation for considering the class of Killing metrized commutative algebras is that it can be regarded as a commutative analogue of the class of semisimple Lie algebras. The semisimplicity corresponds to the nondegeneracy of the metric. On the other hand, the invariance of the metric is an independent condition that, while automatic for Lie algebras, must be imposed in the setting of commutative algebras. This suggests that relaxing the nondegeneracy condition could be interesting, and this is the reason for delineating the more general class of Killing invariant commutative algebras.

The relation with Lie algebras is in fact stronger than mere analogy as the tensor product of simple Lie algebras is a Killing metrized commutative algebra; see Lemma \ref{mcatensorlemma}. 
The class of Killing metrized commutative algebras is essentially the smallest that contains the algebra tensor products of semisimple Lie algebras (a precise formulation is not given here). This alone gives good motivation for studying metrized commutative algebras. Further motivation is given by the identification of interesting special examples such as those described here in Sections \ref{simplicialsection}, \ref{jordansection}, and \ref{triplesection} and Example \ref{voaexample}. 

The \emph{associator} $[x, y, z]$ of $(\alg, \mlt)$ is the trilinear map $\tensor^{3}\alg^{\ast} \to \alg$ defined by
\begin{align}
[x, y, z] = (x\mlt y)\mlt z - x\mlt(y\mlt z) = \left(L_{\mlt}(x \mlt y) - L_{\mlt}(x)L_{\mlt}(y)\right)z.
\end{align}
An analogy, not developed in full detail here, treating the structure tensor of the multiplication as parallel to an affine connection, leads to regarding the associator as an analogue of the curvature tensor of a connection. 
Following this analogy, the \emph{Ricci form} of a commutative algebra $(\alg, \mlt)$ is defined to be the bilinear trace-form 
\begin{align}\label{ricciform}
\ric_{\mlt}(x, y) = \tr L_{\mlt}(x\mlt y) - \tr L_{\mlt}(x)\tr L_{\mlt}(y) =\tr L_{\mlt}(x\mlt y) - \tau_{\mlt}(x, y).
\end{align}
The expression \eqref{ricciform} formally mimics the part of the expression for the Ricci curvature that depends algebraically (and not differentially) on the connection coefficients with respect to a background frame (its Christoffel symbols). Its vanishing is a weak quasi-associativity condition. From this point of view, the invariance of a symmetric bilinear form is the algebraic analogue of a symmetric covariant tensor being parallel, so that the invariance of the Ricci-form is formally analogous to the Einstein equations for a Riemannian metric (which imply that its Ricci form is parallel). For an exact commutative algebra the Ricci form is the negative of the Killing form, $\ric_{\mlt} = -\tau_{\mlt}$, and this line of thinking provides an independent motivation for studying Killing metrized algebras. The most natural generalization of the Killing metrized condition to the context of not necessarily exact commutative algebras is Ricci metrizability. A commutative algebra is \emph{Ricci invariant} if its Ricci form is invariant, and \emph{Ricci metrized} if its Ricci form is invariant and nondegenerate.

An associative Killing metrized commutative algebra is necessarily unital (it follows from Lemma \ref{ricciflatunitlemma} that a Killing metrized commutative algebra with vanishing Ricci form is necessarily unital). A unital, associative $\fie$-algebra $(\alg, \mlt)$ is \emph{split} if it is isomorphic to a product of $n$ copies of $\fie$. A unital, associative $\fie$-algebra $(\alg, \mlt)$ is \emph{étale} if $\alg\tensor_{\fie}\fiel$ is split for some field extension $\fiel$ of $\fie$. An étale algebra is necessarily commutative. An $n$-dimensional associative commutative $\fie$-algebra is Killing metrized if and only if it is an étale $\fie$-algebra of rank $n$ \cite[Proposition $1$, A.V.$47$]{Bourbaki-algebre4}. By the Artin-Wedderburn theorem, the algebra is Killing metrizable if and only if $(\alg, \mlt)$ is a direct sum of finite separable field extensions of $\fie$ whose ranks sum to $n$. In particular, over an algebraically closed field $\fie$, an étale $\fie$-algebra is necessarily split. The operative assumptions in this setting can be relaxed in various senses. The first is to consider not algebraically closed fields; this amounts to considering the general setting of étale algebras, where the complications are codified in Galois theory and its extensions. The second is to relax the associativity (and the concomitant unitality condition - unitality and associativity are linked in general \cite{Poonen}) while maintaining the commutativity and the invariance and nondegeneracy of the trace-form. This provides still  another motivation for considering Killing metrized commutative algebras.

The most basic examples of Killing metrized exact commutative algebras are the \emph{simplicial} algebras $\ealg^{n}(\fie)$, a family of commutative nonassociative algebras, one for each dimension $n \geq 2$, first studied in \cite{Dong-Griess} and \cite{Harada} as models of an algebraic structure with automorphisms equal to the symmetric group. The simplicial algebras are the only exact algebras in a one-parameter family of Ricci metrized algebras $\talg^{n}_{\al}(\fie)$ described in Section \ref{simplicialsection}. The family $\talg^{n}_{\al}(\fie)$ plays an important role in the classification of cyclic elements in semisimple Lie algebras \cite{Elashvili-Jibladze-Kac-semisimple}.

The analogy between the associator of a metrized commutative algebra and the curvature tensor of a metric leads to a notion called here \emph{sectional nonassociativity} directly modeled on the usual notion of sectional curvature. It facilitates quantifications of nonassociativity that in particular generalize the Norton inequality that plays an important role in the study of Griess algebras \cite[Chapter $8$]{Ivanov}. The Norton inequality is the same as \emph{nonnegative sectional nonassociativity}. For metrized commutative algebras, the corresponding upper bound, nonpositive sectional nonassociativity, has consequences reminiscent of conclusions stemming from nonpositive sectional curvature, for example finiteness of the automorphism group (see Lemma \ref{finiteautolemma}). However, the notion of sectional nonassociativity facilitates quantifications more precise than positivity or negativity, namely sharper bounds, or constancy. In particular, it is shown that constant sectional nonassociativity is related to notions of vanishing projective or conformal nonassociativity. For example, Theorem \ref{confassclassificationtheorem} gives a new characterization of the simplicial algebras in terms of sectional associativity; precisely, for $\fie$ algebraically closed or real, an projectively associative exact commutative algebra with nondegenerate Killing form is isomorphic to a simplicial algebra.

Following the analogy with the curvature of a metric connection, a metrized commutative algebra $(\alg, \mlt, h)$ is \emph{Killing Einstein} if there is $0 \neq \ka \in \fie$ such that $\tau_{\mlt} = \ka h$. The element $\ka$ is called the \emph{Einstein constant}. In this case $(\alg, \mlt)$ is Killing metrized, there is no need to mention $h$ in the definition. The inclusion of $h$ in the definition reflects an analogy with the Einstein equations for a Riemannian metric. These equations state that the Ricci curvature of the Levi-Civita connection of a Riemannian metric is a multiple of the metric. At least in the case of nonzero scalar curvature, they could equally well be formulated as the Ricci curvature of an affine connection is nondegenerate and parallel. The Einstein constant $\ka$ is the analogue of the scalar curvature. The Killing Einstein condition is considered here only for exact commutative algebras. The more appropriate notion for not necessarily exact metrized commutative algebras is that such an algebra is \emph{Einstein} if its Ricci-form is a nonzero multiple of $h$. Note that the Ricci form of a unital commutative algebra is necessarily degenerate for the unit $e$ satisfies $\ric_{\mlt}(e, \dum) = 0$.

This perspective is most appropriate when some natural invariant metric $h$ is fixed in advance and some work is needed to show that $\tau_{\mlt}$ is a multiple of $h$. For example, such a situation occurs in Examples \ref{voaexample} and \ref{voanortonexample}, where it is explained that results of Matsuo and Miyamoto yield that the deunitalizations of the Griess algebras of certain vertex operator algebras are Killing metrized exact commutative algebras with nonnegative sectional nonassociativity. In this context there is a natural invariant metric $h$ induced directly from the vertex operator algebra structure, and it is interesting to calculate the Einstein constant $\ka$ such that $\tau_{\mlt} = \ka h$. This is vaguely analogous to the much simpler example of a simple matrix Lie algebra, for which the Killing form is a constant multiple of the Frobenius trace form on the matrix algebra.

A variety of examples other than the $\talg^{n}_{\al}(\fie)$ are described. These include tensor products of simplicial algebras, the deunitalizations of simple Euclidean Jordan algebras (trace-free Hermitian matrices over real Hurwitz algebras, equipped with the trace-free Jordan product), the tensor products of simple Lie algebras, and the Griess algebras of OZ vertex operator algebras. 
Sections \ref{tensorsection}, \ref{jordansection}, and \ref{triplesection} and Example \ref{voaexample} describe these applications of the general theory to specific examples. They are largely independent and can be read separately.

Section \ref{tensorsection} describes partial results about the structure of the tensor products of simplicial algebras with themselves. Although $\ealg^{n}(\fie)$ is simple for all $n \geq 2$, the tensor product $\ealg^{m}(\fie)\tensor_{\fie}\ealg^{n}(\fie)$ is not simple if $m = 2$ or $m \geq 6$, although it is if $m = 3$. A peculiar and particular role is played by the tensor product with $\ealg^{3}(\fie)$ for the following reason. By \eqref{hpolar} the cubic form determined by polarizing the cubic polynomial of a metrized commutative algebra $(\alg, \mlt, h)$ can be regarded as the cubic polynomial of a metrized commutative algebra structure on the threefold sum $\alg \oplus \alg \oplus \alg$. The triple algebra construction is a special case of a more general construction described in Section \ref{triplesection} and associating to a not necessarily commutative algebra a metrized exact commutative algebra, that when applied to semisimple Lie algebras recovers the Nahm algebras studied by M. Kinyon and A. Sagle \cite{Kinyon-Sagle}. Lemma \ref{tripletensorlemma} shows that the triple algebra of a metrized commutative algebra is isomorphic to its tensor product with $\ealg^{3}(\fie)$ while the Nahm algebra of semisimple Lie algebra is isomorphic to its tensor product with $\so(3)$. This formal parallelism between $\ealg^{3}(\fie)$ and $\so(3)$ is partly explained in \cite{Ivanov}, where Ivanov has considered algebras satisfying conditions related to those here. (More generally, the point of view taken here has a lot in common with that advocated in \cite{Ivanov}, where the Killing trace-form and Norton inequality explicitly play important roles.)

The author believes that the simplicial algebras should play a role in the general theory somewhat akin to that played by Cartan subalgebras in Lie theory. In particular there should be a notion of rank for Killing metrized exact commutative algebras, related to the dimension of the maximal subalgebra isomorphic to a simplicial algebra. Substance is given to such a speculation by the results of Section \ref{jordansection}, where, among other things, it is shown that the diagonal subalgebra of a deunitalized simple Euclidean Jordan algebra is a simplicial algebra. That the deunitalizations of simple Euclidean Jordan algebras are Killing metrized exact algebras follows from results in \cite{Nadirashvili-Tkachev-Vladuts}. Here they are described in order to illustrate the notion of sectional nonassociativity. Lemma \ref{hermsectlemma} shows the simple Euclidean Jordan algebras have nonnegative sectional nonassociativity and deduces a sharp upper bound on their sectional nonassociativities from the Chern-do Carmo-Kobayashi-Böttcher-Wenzel inequality (see the appendix). The corresponding bounds on the sectional nonassociativities of their deunitalizations are given in Corollary \ref{jordandeunitalizationsectcorollary}.

To limit the paper's size, two of the most interesting applications are omitted, except for some remarks in Section \ref{conclusionsection}:
\begin{itemize}
\item The algebras of metric Weyl and Kähler Weyl curvature tensors studied in \cite{Fox-curvtensor} are Killing metrized and exact, so fit into the context considered here.
\item The author has shown \cite{Fox-frames} that the Norton algebra of the association scheme associated with the strongly regular graph \cite{Cameron-Goethals-Seidel, Cameron-Goethals-Seidel-stronglyregular} determined by a balanced two-distance tight unit norm frame in a Euclidean vector space \cite{Barg-Galzyrin-Okoudjou-Yu, Waldron-frames} is Killing metrized and exact (for a balanced equiangular tight frame this claim follows from \cite{Fox-cubicpoly}). The rich combinatorial structure gives access to explicit information about its idempotents and sectional nonassociativity. 
\end{itemize}

In recent years there has been a lot of activity related to axial algebras \cite{Hall-Rehren-Shpectorov-primitive} and certain algebras associated with $3$-transposition groups and Fischer systems. See \cite{Hall-transpositionalgebras} for a recent survey of these topics. Although such algebras often admit an invariant metric (for example see \cite[Theorem $4.1$]{Hall-Segev-Shpectorov-primitiveaxial}), in general they are \emph{not} Killing metrized. However some of them are, and this condition may identify particularly interesting algebras in these contexts. Examining in detail how axial algebras relate to the algebras considered here seems an interesting project.  

Most of the ideas presented here, in particular the notion of sectional nonassociativity, extend to a broader class of not necessarily commutative algebras, that could be called symmetric Frobenius (not necessarily unital or associative) algebras, that includes also semisimple Lie algebras. A systematic treatment in such generality is planned for the future.

In general metrized commutative algebras, bounds on sectional nonassociativity have strong implications for the spectra of the multiplication endomorphisms of idempotents. These will be studied as part of a joint project with V.~G. Tkachev.

This paper complements \cite{Fox-cubicpoly}, which studies the cubic polynomials of Euclidean Killing metrized exact commutative algebras without explicitly mentioning the algebraic aspect. A preliminary version of some of the results reported here is contained in \cite[Section $7.2$]{Fox-ahs} and reported in \cite{Fox-crm}.

\section{Basic structural results}
This section records basic structural results about metrized and Killing metrized commutative algebras that help to illustrate the importance of the metrized and Killing metrized assumptions. In this section the base field $\fie$ is arbitrary except where indicated otherwise.

An algebra is \emph{simple} if its multiplication is nontrivial and it has no nontrivial proper ideal. If $(\alg, \mlt)$ is simple then, since the multiplication is nontrivial, there must hold $\alg \mlt \alg = \alg$. Note that by this definition a simple algebra has nonzero dimension. 

\begin{lemma}\label{simpleideallemma}
If $\ideal$ is a simple ideal in a metrized commutative algebra $(\alg, \mlt, h)$, then either $\ideal$ is $h$-isotropic or $\alg = \ideal \oplus \ideal^{\perp}$ is a direct sum of $h$-nondegenerate ideals, where $\ideal^{\perp}= \{x \in \alg: h(x, a) = 0 \,\,\text{for all}\,\, a \in \ideal\}$ is the subspace $h$-orthogonal to $\ideal$.
\end{lemma}

\begin{proof}
The invariance and nondegeneracy of $h$ implies that a subspace $\ideal \subset \alg$ is an ideal if and only if $\ideal^{\perp}$ is an ideal. In this case $\ideal \cap \ideal^{\perp}$ is an ideal, so if $\ideal$ is simple, then either $\ideal \cap \ideal^{\perp} = \{0\}$, in which case $\ideal$ is $h$-nondegenerate, or $\ideal = \ideal^{\perp}$, in which case $\ideal$ is $h$-isotropic.
\end{proof}

\begin{lemma}\label{simpledecomposelemma}
If a metrized commutative algebra $(\alg, \mlt, h)$ admits a decomposition $\alg = \alg_{1}\oplus \alg_{2}$ as a direct sum of 
ideals, then $\alg_{2} \subset (\alg_{1} \mlt \alg_{1})^{\perp}$. If, moreover, $\alg_{1}$ is simple, then $\alg_{1}^{\perp} = \alg_{2}$ and the restriction of $h$ to $\alg_{1}$ is nondegenerate. Consequently, if metrized commutative algebra $(\alg, \mlt, h)$ is expressible as a direct sum $\alg = \oplus_{i = 1}^{k}\alg_{i}$ of simple ideals $\alg_{i}$, then the summands are nondegenerate and pairwise orthogonal with respect to $h$.
\end{lemma}

\begin{proof}
Any $a_{1} \in \alg_{1} \mlt \alg_{1}$ can be written in the form $\sum_{i = 1}^{k}r_{i}\mlt t_{i}$ for some $r_{i}, t_{i} \in \alg_{1}$. If $a_{2} \in \alg_{2}$, then $h(a_{1}, a_{2}) = \sum_{i = 1}^{k}h(r_{i}\mlt t_{i}, a_{2}) = \sum_{i = 1}^{k}h(r_{i}, t_{i}\mlt a_{2}) = 0$, the last equality because $t_{i} \mlt a_{2} \in \alg_{1}\cap \alg_{2} = \{0\}$. This shows $\alg_{2} \subset (\alg_{1}\mlt\alg_{1})^{\perp}$. If $\alg_{1}$ is simple, then it equals $\alg_{1} \mlt \alg_{1}$, so $\alg_{2} \subset \alg_{1}^{\perp}$. In this case, by the nondegeneracy of $h$, $\dim \alg_{1}^{\perp} = \dim \alg_{2}$, so $\alg_{1}^{\perp} = \alg_{2}$. Since this means $\alg_{1} \cap \alg_{1}^{\perp} = \alg_{1} \cap \alg_{2} = \{0\}$, the restriction of $h$ to $\alg_{1}$ is nondegenerate.
\end{proof}

An algebra is \emph{semisimple} if it is a direct sum of simple ideals. 
When considering semisimplicity, some care is necessary, as some equivalences valid for associative algebras fail in its absence, e.g. the characterization of semisimplicity in terms of the solvable radical is not valid for general algebras. In this regard see \cite{Albert-radical}. 

The multiplication defining an algebra $(\alg, \mlt)$ is \emph{faithful} if $L_{\mlt}:\alg \to \eno(\alg)$ is injective. 

\begin{lemma}\label{faithfulmultiplicationlemma}
The multiplication of a metrized semisimple commutative algebra $(\alg, \mlt, h)$ is faithful.
\end{lemma}
\begin{proof}
If $L_{\mlt}(z) = 0$, then, by the invariance of $h$, $0 = h(L_{\mlt}(z)x, y) = h(z, x\mlt y)$ for all $x, y \in \alg$, so that $z \in (\alg \mlt \alg)^{\perp}$. This shows $\ker L \subset (\alg \mlt \alg)^{\perp}$. In particular, if $\alg \mlt \alg = \alg$, then $L_{\mlt} = (\alg \mlt \alg)^{\perp} = \alg^{\perp} = \{0\}$, so $L_{\mlt}$ is injective. 
If $\alg = \oplus_{i =1 }^{k}\alg_{i}$ is a direct sum of simple ideals, then $\alg_{i}\mlt \alg_{j} = \{0\}$ if $i \neq j$, and, since $\alg_{i}$ is simple, $\alg_{i}\mlt \alg_{i} = \alg_{i}$, so $\alg \mlt \alg = \sum_{i}\alg_{i}\mlt \alg_{i} = \sum_{i}\alg_{i} = \alg$, and $L_{\mlt}$ is injective.
\end{proof}

\begin{lemma}\label{ricciflatunitlemma}
Suppose the Killing metrized commutative algebra $(\alg, \mlt)$ has vanishing Ricci form.
\begin{enumerate}
\item The unique $\ell \in \alg$ such that $\tau_{\mlt}(\ell, x) = \tr L_{\mlt}(x)$ for all $x \in \alg$ is a unit.
\item A square-zero element $z \in \alg$ is $\tau_{\mlt}$-isotropic.
\end{enumerate}
\end{lemma}
\begin{proof}
For $x, y \in \alg$, $\tau_{\mlt}(\ell \mlt x, y) = \tau_{\mlt}(\ell, x \mlt y) = \tr L_{\mlt}(x\mlt y) = \tr L_{\mlt}(x)L_{\mlt}(y) = \tau_{\mlt}(x, y)$, the first equality by the invariance of $\tau_{\mlt}$ and the third equality by the vanishing of the Ricci form. Since this holds for all $y \in \alg$, by the nondegeneracy of $\tau_{\mlt}$ there holds $x\mlt \ell  = x$ for all $x \in \alg$. If $z \in \alg$ satisfies $z \mlt z = 0$, then $\tau_{\mlt}(z, z) = \tr L_{\mlt}(z)^{2} = \tr L_{\mlt}(z \mlt z) = 0$.
\end{proof}

By Lemma \ref{ricciflatunitlemma}, an associative Killing metrized commutative algebra is necessarily unital. 

\begin{lemma}\label{dieudonnelemma}
Let $(\alg, \mlt)$ be a commutative $\fie$-algebra.
If there are $p, q, r \in \fie$ such that the symmetric bilinear form $\tau_{p, q, r}(x, y) = p\tr L_{\mlt}(x\mlt y) + q\tr L_{\mlt}(x)L_{\mlt}(y) + r \tr L_{\mlt}(x)\tr L_{\mlt}(y)$ is nondegenerate and invariant, then $(\alg, \mlt)$ is semisimple, expressible, in a unique way up to reordering the factors, as a direct sum $\alg = \oplus_{i}\alg_{i}$ of simple ideals $\alg_{i} \subset \alg$ pairwise orthogonal with respect to $\tau_{p, q, r}$.
\end{lemma}

\begin{proof}
Since $\tau_{\mlt}$ is by assumption invariant, by \cite[Theorem $2.6$]{Schafer-book} or \cite[Theorem III.$5.3$]{Jacobson} (the argument is due to Dieudonne \cite{Dieudonne-killing}, where it is given for Lie algebras), the expressibility of $\alg$ as a direct sum of simple ideals follows provided $\balg\mlt \balg \neq \{0\}$ for every ideal $\balg \neq \{0\}$. Suppose $\balg \subset \alg$ is an ideal such that $\balg\mlt \balg = \{0\}$. For $b \in \balg$ and any $A \in \eno(\alg)$ having the form $A = \id_{\alg} + \la L_{\mlt}(a)$ for $\la \in \fie$ and $a \in \alg$, the composition $L_{\mlt}(b)A$ maps $\alg$ into $\balg$ and $\balg$ into $\{0\}$, so that $(L_{\mlt}(b)A)^{2} = 0$. Hence $0 = \tr L_{\mlt}(b)A = \tr L_{\mlt}(b) + \la \tr L_{\mlt}(b)L_{\mlt}(a)$. Since $\la$ is arbitrary, this implies $\tr L_{\mlt}(b) =0$ and $\tr L_{\mlt}(a)L_{\mlt}(b) = 0$ for all $a \in \alg$. Consequently $\tau_{p, q, r}(a, b) = p\tr L_{\mlt}(a\mlt b ) + q\tr L_{\mlt}(a)L_{\mlt}(b) + r \tr L_{\mlt}(a) \tr L_{\mlt}(b) = 0$ for all $a \in \alg$, which by nondegeneracy of $\tau_{p, q, r}$ implies $b = 0$, and hence $\balg = \{0\}$. This suffices to show that $(\alg, \mlt)$ is semisimple. The pairwise orthogonality of the summands follows from Lemma \ref{simpledecomposelemma}.
\end{proof}

Lemma \ref{dieudonnelemma} implies that a Killing (Ricc) metrized exact commutative algebra is a direct sum of simple ideals orthogonal with respect to its Killing (Ricci) form. Consequently, to classify Killing (Ricci) metrized exact commutative algebras, it suffices to classify the simple ones.

\begin{lemma}\label{flatalgebralemma}
An $n$-dimensional associative Killing metrized commutative $\fie$-algebra is a split étale algebra if either $\fie$ is algebraically closed or $\fie = \rea$ and its Killing form is positive definite.
\end{lemma}

\begin{proof}
In either case, by Lemma \ref{ricciflatunitlemma}, $(\alg, \mlt)$ is unital.
By Lemma \ref{dieudonnelemma} it is a direct sum of $\tau_{\mlt}$-orthogonal simple ideals. By the Artin-Wedderburn theorem each of these simple ideals is isomorphic to a matrix algebra over a division ring over $\fie$, and by the commutativity of $\mlt$, each of these must be a commutative division ring over $\fie$, so a field extension of $\fie$. If $\fie$ is algebraically closed the conclusion follows, while if $\fie = \rea$ and $\tau_{\mlt}$ is positive definite, then each of the simple ideals is isomorphic to $\rea$ or $\com$, and since the restriction of $\tau_{\mlt}$ to each ideal is an invariant metric with definite signature, each must be isomorphic to $\rea$.
\end{proof}

\begin{example}\label{cubicexample}
The conclusion of Lemma \ref{dieudonnelemma} is false for general metrized commutative algebras as the following simple example illustrates. Let $\chr \fie = 0$ and for $\ep \in \fie$ consider the commutative, associative, unital algebra $\alg = \fie[t]/(t^{2} - \ep)$. 
For $\la \in \alg^{\ast}$ defined by $\la(x_{0} + x_{1}t) = x_{1}$, the metric $h(x, y) = \la(xy) = x_{0}y_{1} + x_{1}y_{0}$ is nondegenerate and invariant. The Killing form $\tau_{\mlt}(x, y) = x_{0}y_{0} + \ep x_{1}y_{1}$ is invariant and is nondegenerate provided $\ep \neq 0$. When $\ep = 0$, $\alg$ is not semisimple, for the ideal generated by $t$ is isotropic and squares to zero and the Killing form $\tau_{\mlt}(x, y) = 2x_{1}y_{1}$ is degenerate. When $\ep \neq 0$, the metric $h$ is insensitive to $\ep$, while the metric $\tau_{\mlt}$ is sensitive to $\ep$. The properties of the metrized algebra $(\alg, \tau)$ depend on $\ep$ if $\fie$ is not algebraically closed. For example, in the case $\fie = \rea$, when $\ep < 0$, then $\alg$ is isomorphic to the complex numbers, viewed as a $\rea$-algebra, and the Killing form has split signature, while when $\ep > 0$, $\alg$ is isomorphic to the paracomplex numbers, $\com$ viewed as a real vector space with the multiplication $u  \mlt v = \bar{u}\bar{v}$, and the Killing form is positive definite.

This example illustrates two different aspects of the conditions imposed on the Killing form. The nondegeneracy forces semisimplicity of the algebra. When the base field is not algebraically closed the structure of a metrized algebra depends on more refined properties of the metric as a bilinear form, for example, in the real case, whether it is or is not positive definite. For example, $\com$ is metrizable as a $\rea$-algebra, for $\lb u, v \ra = uv + \bar{u}\bar{v}$ is a split signature invariant $\rea$-valued metric, but $\com$ admits no definite signature invariant metric $h$, for an invariant metric $h$ must satisfy $-h(1, 1) = h(1, \j\mlt \j) = h(\j, \j)$. This shows the necessity of the assumption of definite signature in Lemma \ref{flatalgebralemma} for its conclusion in the case $\fie = \rea$.

With anisotropy of $\tau_{\mlt}$ in place of positive definiteness, the conclusion of Lemma \ref{flatalgebralemma} is false over more general fields. A simple example of an étale algebra that is not split is the degree $3$ cyclic Galois extension of $\rat$ obtained by adjoining a root of $r^{3} + r^{2} - 2r - 1$, which has positive definite Killing form and is not split over $\rat$ (if it were, it would admit three orthogonal idempotents, but in a field any idempotent equals $1$). 
\end{example}

The tensor product $\alg_{1}\tensor_{\fie}\alg_{2}$ of commutative or anticommutative $\fie$-algebras $(\alg_{1}, \mlt_{1})$ and $(\alg_{2}, \mlt_{2})$ is a commutative algebra with the product $\mlt$ defined on decomposable elements by
\begin{align}
(a_{1}\tensor b_{1})\mlt (a_{2}\tensor b_{2}) = (a_{1}\mlt_{1} a_{2})\tensor (b_{1}\mlt_{2} b_{2}),
\end{align}
and extending linearly. Alternatively, $L_{\mlt}(a_{1}\tensor a_{2}) = L_{\mlt_{1}}(a_{1})\tensor L_{\mlt_{2}}(a_{2})$. Since decomposable elements span $\alg_{1}\tensor_{\fie}\alg_{2}$ and $\tr L_{\mlt}(a_{1} \tensor a_{2}) = \tr L_{\mlt_{1}}(a_{1})\tensor L_{\mlt_{2}}(a_{2})$, $(\alg_{1}\tensor_{\fie}\alg_{2}, \mlt)$ is exact if either of $(\alg_{1}, \mlt_{1})$ or $(\alg_{2}, \mlt_{2})$ is exact. 
If $h_{1}$ and $h_{2}$ are invariant metrics on $(\alg_{1}, \mlt_{1})$ and $(\alg_{2}, \mlt_{2})$, then $h_{1}\tensor h_{2}$ defined by $(h_{1}\tensor h_{2})(a_{1}\tensor b_{1}, a_{2}\tensor b_{2}) = h_{1}(a_{1}, a_{2})h_{2}(b_{1}, b_{2})$ is invariant on $(\alg_{1}\tensor_{\fie}\alg_{2}, \mlt)$, for
\begin{align}
\begin{split}
(h_{1}\tensor h_{2})&((a_{1}\tensor b_{1})\mlt (a_{2}\tensor b_{2}), a_{3}\tensor b_{3}) =h_{1}(a_{1}\mlt_{1} a_{2}, a_{3})h_{2}(b_{1}\mlt_{1} b_{2}, b_{3}) \\
&= h_{1}(a_{1}, a_{2}\mlt_{1} a_{3})h_{2}(b_{1}, b_{2}\mlt_{2} b_{3}) = (h_{1}\tensor h_{2})(a_{1}\tensor b_{1}, (a_{2}\tensor b_{2})\mlt(a_{3}\tensor b_{3})). 
\end{split}
\end{align}
Define the tensor product of metrized commutative algebras $(\alg_{1}, \mlt_{1}, h_{1})$ and $(\alg_{2}, \mlt_{2}, h_{2})$ to be $(\alg_{1}\tensor_{\fie}\alg_{2}, \mlt, h_{1}\tensor h_{2})$. 
\begin{lemma}\label{mcatensorlemma}
\noindent
\begin{enumerate}
\item The tensor product of metrized commutative $\fie$-algebras $(\alg_{1}, \mlt_{1}, h_{1})$ and $(\alg_{2}, \mlt_{2}, h_{2})$ at least one of which is exact is a metrized exact commutative $\fie$-algebra.
\item The tensor product of Killing metrized commutative $\fie$-algebras $(\alg_{1}, \mlt_{1})$ and $(\alg_{2}, \mlt_{2})$ at least one of which is exact is a Killing metrized exact commutative $\fie$-algebra.
\item The tensor product $(\g\tensor_{\fie}\h,\mlt)$ of semisimple Lie algebras $\g$ and $\h$ over a field $\fie$ of characteristic zero is a Killing Einstein exact commutative algebra.
\end{enumerate}
\end{lemma}
\begin{proof}
If $\{a_{1}, \dots, a_{\dim\alg_{1}}\}$ and $\{b_{1}, \dots, b_{m_{2}}\}$ are bases of $\alg_{1}$ and $\alg_{2}$, then $\{a_{i}\tensor b_{j}: 1 \leq i \leq m_{1}, 1 \leq j \leq m_{2}\}$ is a basis of $\alg_{1}\tensor_{\fie}\alg_{2}$, and the Gram matrix of $h_{1}\tensor h_{2}$ with respect to this basis is the tensor product of the Gram matrices of $h_{1}$ and $h_{2}$, which are invertible because $h_{1}$ and $h_{2}$ are nondegenerate, so the tensor product $(\alg_{1}\tensor_{\fie}\alg_{2}, \mlt, h_{1}\tensor h_{2})$ is a metrized commutative algebra. If $(\alg_{1}, \mlt_{1})$ and $(\alg_{2}, \mlt_{2})$ are either both commutative or both anticommutative, the Killing form of $(\alg_{1}\tensor_{\fie}\alg_{2}, \mlt)$ equals $\tau_{\mlt_{1}}\tensor \tau_{\mlt_{2}}$, for
\begin{align}\label{tauaa}
\begin{split}
\tau_{\mlt}(a_{1}\tensor b_{1}, a_{2}\tensor b_{2}) &= \tr L_{\mlt}(a_{1}\tensor b_{1})L_{\mlt}(a_{2}\tensor b_{2})\\
&= \tr \left(L_{\mlt_{1}}(a_{1})L_{\mlt_{1}}(a_{2}) \tensor L_{\mlt_{2}}(b_{1})L_{\mlt_{2}}(b_{2}) \right) =  \tau_{\mlt_{1}}(a_{1}, a_{2})\tau_{\mlt_{2}}(b_{1}, b_{2}).
\end{split}
\end{align}
By \eqref{tauaa}, if $\tau_{\mlt_{1}}$ and $\tau_{\mlt_{2}}$ are invariant or nondegenerate, then $\tau_{\mlt}$ has the same property. This shows that the tensor product of Killing metrized algebras qua commutative algebras equals their tensor product qua metrized algebras. Similarly, if $\alg_{1} = \g$ and $\alg_{2} = \h$ are semisimple Lie algebras over a field of characteristic zero, then their Killing forms are nondegenerate and invariant, so their tensor product is a Killing metrized commutative algebra that is exact because a semisimple Lie algebra is unimodular. 
\end{proof}
(Essentially the same arguments show the analogous statements are true for Ricci metrized commutative algebras.) 

\begin{remark}
The special case of the tensor product of Lie algebras $\so(3)\tensor \g$ is isomorphic to what was called the \emph{Nahm algebra} of $\g$ in \cite{Kinyon-Sagle}. See Section \ref{triplesection} for further discussion of this example.
\end{remark}

\begin{example}
For semisimple Lie algebras $\g$ and $\h$, $\Aut(\g\tensor_{\fie}\h)$ contains $\Aut(\g)\times \Aut(\h)$ acting on each factor separately. This shows that the automorphism group of a Killing metrized exact commutative algebra can be large.
\end{example}

In any class of algebras general structural properties are reflected in the behavior of idempotents and square-zero elements. 
An element $e \in \alg$ is \emph{idempotent} if $e \mlt e = e$ and \emph{square-zero} if $e \mlt e = 0$. 
Let $\idem(\alg, \mlt) = \{0 \neq e \in \alg: e \mlt e = e\}$ and $\szero(\alg, \mlt) = \{0 \neq z \in \alg: z \mlt z  = 0\}$.

\begin{lemma}\label{idealidempotentlemma}
If the metrized commutative algebra $(\alg, \mlt, h)$ is expressible as a direct sum $\oplus_{i = 1}^{k}\alg_{i}$ of $h$-orthogonal ideals $\alg_{i} \subset \alg$ and $\pi_{i}:\alg \to \alg_{i}$ is the $h$-orthogonal projection onto the $i$th component, then $e \in \alg$ is idempotent if and only if its projections $e_{i} = \pi_{i}(e)$ are idempotents. 
\end{lemma}

\begin{proof}
Because $e_{i} \mlt e_{j} =0$ if $i \neq j$ and $e_{i}\mlt e_{i} \in \alg_{i}$, $e\mlt e = \sum_{i = 1}^{k}e_{i}\mlt e_{i}$ and so $\pi_{i}(e\mlt e) = e_{i}\mlt e_{i}$, from which the claim follows. 
\end{proof}

For $\la \in \fie$, let $\alg^{(\la)}(e) = \{x \in \alg: L_{\mlt}(e)x = \la x\}$ and define the \emph{spectrum} of $e\in \alg$ by $\spec(e) = \{\la \in \fie: \dim \alg^{(\la)}(e) > 0\}$.
For a metrized commutative algebra $(\alg, \mlt, h)$, the endomorphism $L_{\mlt}(e)$ is $h$-self-adjoint for any $e \in \alg$. If $e \in \idem(\alg, \mlt)$, then $L_{\mlt}(e)$ preserves $\eperp =  \{y \in \alg: h(e, y) =0\}$, for if $h(e, y) = 0$ then $h(L_{\mlt}(e)y, e) = h(y, e \mlt e) = h(y, e) = 0$. For an $h$-anisotropic idempotent $e$, define the \emph{orthogonal spectrum}
\begin{align}
\specp(e) = \{\la \in \fie: \text{there is}\,\, x \in \lb e \ra^{\perp}\,\, \text{such that}\,\, L_{\mlt}(e)x = \la x\},
\end{align} 
so that $\spec(e) = \specp(e) \cup \{1\}$.

Idempotents $e$ and $f$ in a commutative algebra are \emph{orthogonal} if $e \mlt f = 0$. In a metrized commutative algebra $(\alg, \mlt, h)$, orthogonal idempotents $e$ and $f$ are $h$-orthogonal, for $h(e, f) = h(e\mlt e, f) = h(e, e\mlt f) = 0$, by the invariance of $h$. When $\chr \fie \neq 2$, if an idempotent $e$ can be written as a sum of nonzero idempotents $u$ and $v$, then necessarily $u$ and $v$ are orthogonal, for in this case $u + v = e = e\mlt e = u + v + 2u \mlt v$. 

A nonzero idempotent idempotent is \emph{primitive} if it cannot be written as a sum of two nonzero idempotents.
Following \cite{Hall-Segev-Shpectorov}, a nonzero idempotent $e$ in a metrized commutative algebra is \emph{absolutely primitive} if $\dim \alg^{(1)}(e) = 1$; that is $1$ has multiplicity $1$ in $\spec(e)$.

\begin{lemma}\label{absolutelyprimitivelemma}
Suppose $\chr \fie \neq 2$. In a metrized commutative $\fie$-algebra $(\alg, \mlt, h)$, an absolutely primitive anisotropic idempotent is primitive.
\end{lemma}
\begin{proof}
Suppose the absolutely primitive anisotropic idempotent $e$ equals $u + v$ where $u$ and $v$ are idempotents, necessarily orthogonal. Write $u = \al e + x$ and $v = \be e + y$ where $\al, \be \in \fie$ and $x$ and $y$ are $h$-orthogonal to $e$. Because $h(e, e) \neq 0$, $x$ and $y$ are not multiples of $e$. That $e = u + v = (\al + \be)e + x + y$ forces $\be = 1-\al$ and $y = -x$. That $u$ and $v$ be idempotents implies the equations $\al e + x = \al^{2} + 2\al e \mlt x + x \mlt x$ and $(1-\al)e - x = (1-\al)^{2}e + 2(\al -1 )e \mlt x + x \mlt x$. The difference of these two equations yields $(2\al - 1)e + 2x = (2\al - 1)e + 2 e\mlt x$, so $L_{\mlt}(e)x = x$, that is $x \in \alg^{(1)}(e)$. Since $h(x, e) = 0$ and $h(e, e) \neq 0$, this forces $x = 0$. It follows that $u$ and $v$ are multiples of $e$, and since they are orthogonal, one of them must be $0$.
\end{proof}

A metrized commutative $\rea$-algebra is \emph{Euclidean} if the metric is positive definite. Such an algebra is called a Euclidean metrized commutative algebra; the omitted qualifier \emph{real} is implied by the use of \emph{Euclidean}. The Euclidean assumption has strong implications for the properties of idempotents.

An idempotent $e$ in a Euclidean metrized commutative algebra, $(\alg, \mlt, h)$ is \emph{minimal} if there is no $f \in \idem(\alg, \mlt)$ such that $|f|_{h} < |e|_{h}$. By definition two minimal idempotents have the same $h$-norm. Given a metrized commutative algebra $(\alg, \mlt, h)$, let $\midem(\alg, \mlt, h) \subset \idem(\alg, \mlt)$ be the set (in principal possibly empty) of minimal idempotents in $(\alg, \mlt, h)$.
The notion of minimal idempotent depends only on the positive homothety class of $h$ and not on $h$ itself. Minimal idempotents play an important role in the structure theory for Euclidean metrized commutative algebras. Similar considerations are important in the study of Griess algebras, going back at least to \cite{Meyer-Neutsch}. V.~G. Tkachev and Y. Krasnov and collaborators have studied idempotents extensively in metrized algebras \cite{Krasnov-Tkachev, Krasnov-Tkachev-idempotentgeometry, Nadirashvili-Tkachev-Vladuts, Tkachev-universality, Tkachev-correction, Tkachev-extremal, Tkachev-summary}. Lemma \ref{criticalpointlemma} summarizes some basic facts that are needed here.

\begin{lemma}[{\cite[Proposition $2.5$]{Tkachev-jordan}}, {\cite[Lemma $2.3$]{Tkachev-correction}}, {\cite[Proposition $2.1$]{Tkachev-extremal}}]\label{criticalpointlemma}
\noindent
\begin{enumerate}
\item Let $(\alg, \mlt, h)$ be a metrized commutative $\rea$-algebra with cubic polynomial $P$. 
An anisotropic $e \in \alg$ is a critical point of the restriction of $P$ to $\sphere_{h(e, e)}$ if and only if there holds one of the following conditions.
\begin{enumerate}
\item $P(e) \neq 0$ and $\tfrac{h(e, e)}{6 P(e)}e$ is idempotent in $\alg$.
\item $P(e) = 0$ and $e \mlt e = 0$.
\end{enumerate}
\item Let $(\alg, \mlt, h)$ be a nontrivial Euclidean metrized commutative algebra.
\begin{enumerate}
\item\label{minimalidempotentlemma}
$\midem(\alg, \mlt, h)$ is nonempty and any two minimal idempotents have the same $h$-norm.
\item\label{localpmaxspeclemma}
$e \in \midem(\alg, \mlt, h)$ is primitive and absolutely primitive and $\specp(e) \subset (-\infty, 1/2]$.
\end{enumerate}
\end{enumerate}
\end{lemma}

\section{Simplicial algebras}\label{simplicialsection}
This section treats a one-parameter family of algebras $\talg^{n}_{\al}(\fie)$ that play a basic role.  

\begin{definition}
Let $\fie$ be a field of characteristic zero.
For $\al \in \fie$, let $(\talg^{n}_{\al}(\fie), \mlt)$ be the $n$-dimensional commutative algebra generated by $\{e_{i}: 1\leq i \leq n\}$ subject to the relations
\begin{align}\label{talgrelations}
&e_{i} \mlt e_{i} = e_{i}, &&
e_{i} \mlt e_{j} = 
\al(e_{i} + e_{j}) ,\quad i \neq j \in \{1, \dots, n\}.
\end{align}
\end{definition}

For $\al = -1/(n-1)$ these algebras were studied by Griess and Harada as algebras with automorphism group $S_{n+1}$. The full family plays a role in the work of Elashvili-Jibladze-Kac \cite{Elashvili-Jibladze-Kac-semisimple} classifying cyclic elements in semisimple Lie algebras. Here these algebras are basic examples in relation to the notions of projective and conformal associativity and sectional nonassociativity introduced later.

\begin{lemma}\label{talgmodellemma}
Let $\fie$ be a field of characteristic $0$ and let $\al \in \fie$. 
\begin{enumerate}
\item\label{hatexactclaim} $\talg_{\al}^{n}(\fie)$ is exact if and only if $\al = -1/(n-1)$.
\item The Killing form $\tau_{\mlt}$ is invariant if and only if $\al \in \{-1/(n-1), 0, 1/2\}$.
\item $\talg^{n}_{\al}(\fie)$ is Ricci invariant.
If $n > 2$, the rank and inertial indices of $\ric_{\mlt}$ are as in Table \ref{talgtable1}, while if $n =2$, the rank and signature of $\ric_{\mlt}$ are as in Table \ref{talgtable2}, where in both tables the inertial indices are defined only if $\fie = \rea$ (or, more generally, $\fie$ is a Euclidean field). In particular, $\talg^{n}_{\al}(\fie)$ is Ricci metrized if $\al \notin\{-1/(n-2), 0, 1/2\}$. In particular:
\begin{enumerate}
\item If $\al = -1/(n-2)$, then $e = \sum_{i = 1}^{n}e_{i}$ spans the radical of $\ric_{\mlt}$.
\item If $\al = 1/2$, the radical of $\ric_{\mlt}$ is $\ker \tr L_{\mlt}$.
\item If $\al \notin \{0, 1/2\}$, then the restriction of $\ric_{\mlt}$ to $\ker \tr L_{\mlt}$ is nondegenerate.
\end{enumerate}
\begin{table}[!ht]
\parbox{.45\linewidth}{
\begin{center}
\begin{tabular}{|c|c|c|}
\hline
range of $\al$ & rank & inertia $(+, -)$\\
\hline
$(1/2, \infty)$ & $n$ & $(1, n-1)$\\
\hline
$1/2$ & $1$ & $(1, 0)$\\
\hline
$(0, 1/2)$ & $n$ & $(n, 0)$\\
\hline
$0$ & $0$ & $(0, 0)$\\
\hline
$(-1/(n-2), 0)$ & $n$ & $(0, n)$\\
\hline
$-1/(n-2)$ & $n-1$ & $(0, n-1)$\\
\hline
$(-\infty, -1/(n-2))$ & $n$ & $(1, n-1)$\\
\hline
\end{tabular}
\caption{$n > 2$}\label{talgtable1}
\end{center}
}\parbox{.45\linewidth}{
\begin{center}
\begin{tabular}{|c|c|c|}
\hline
range of $\al$ & rank & inertia $(+, -)$\\
\hline
$(1/2, \infty)$ & $2$ & $(1, 1)$\\
\hline
$1/2$ & $1$ & $(1, 0)$\\
\hline
$(0, 1/2)$ & $2$ & $(2, 0)$\\
\hline
$0$ & $0$ & $(0, 0)$\\
\hline
$(-\infty, 0)$ & $2$ & $(1, 1)$\\
\hline
\end{tabular}
\caption{$n =2$}\label{talgtable2}
\end{center}
}
\end{table}
\item For all $x, y, z \in \talg^{n}_{\al}(\fie)$, the associator of $\talg^{n}_{\al}(\fie)$ satisfies
\begin{align}\label{tconfnon0}
\begin{aligned}
[x, y, z] =  \tfrac{1}{n-1}\left(\ric_{\mlt}(x, y)z - \ric_{\mlt}(y, z)x   \right).
\end{aligned}
\end{align}
\item $\talg^{n}_{\al}(\fie)$ is simple if $\al \notin \{-1/(n-2), 0, 1/2\}$.
\item\label{idealclaim} Any ideal in $\talg^{n}_{-1/(n-2)}(\fie)$ is contained in $\spn\{\sum_{i = 1}^{n}e_{i}\}$, which is an ideal in $\talg^{n}_{-1/(n-2)}(\fie)$.
\item Any ideal in $\talg_{1/2}^{n}(\fie)$ is contained in the subspace $\ker \tr L_{\mlt}$, which is an ideal in $\talg_{1/2}^{n}(\fie)$.
\end{enumerate}
\end{lemma}

\begin{proof}
The matrix of $L_{\mlt}(e_{i})$ with respect to the basis $\{e_{1}, \dots, e_{n}\}$ is $e_{ii} + \al\sum_{j \neq i}(e_{jj} + e_{ji})$, where $e_{jk}$ is the matrix with $1$ in row $j$ and column $k$ and $0$ in every other entry.
It follows that $\tr L_{\mlt}(e_{i}) = 1 + (n-1)\al$, so the algebra is exact if and only if $\al = -1/(n-1)$. It follows also that 
\begin{align}\label{lgaga}
\begin{split}
L_{\mlt}(e_{i})L_{\mlt}(e_{i}) &= e_{ii} + \al(1 + \al)\sum_{j \neq i} e_{ij} + \al^{2}\sum_{j \neq i}e_{jj},\\
L_{\mlt}(e_{i})L_{\mlt}(e_{j}) & = \al(1 + \al)e_{ii} + \al e_{jj} + \al^{2} e_{ji} + \al e_{ij} + \al^{2}\sum_{k \neq i, k \neq j}\left(2e_{ik} +e_{kk}+e_{jk} \right).
\end{split}
\end{align}
Tracing \eqref{lgaga} shows that the Gram matrix of $\tau_{\mlt}$ with respect to $\{e_{1}, \dots, e_{n}\}$ has components
\begin{align}\label{tauperm}
\begin{aligned}
&\tau_{\mlt}(e_{i}, e_{i}) = 1 + (n-1)\al^{2},& 
&\tau_{\mlt}(e_{i}, e_{j}) = \al(2 + (n-1)\al),& &i \neq j.
\end{aligned}
\end{align}
There hold
\begin{align}\label{ricperm}
\begin{aligned}
&\ric_{\mlt}(e_{i}, e_{i}) = \tr L_{\mlt}(e_{i}) - \tau_{\mlt}(e_{i}, e_{i}) =  (n-1)\al(1 -\al), \\
&\ric_{\mlt}(e_{i}, e_{j}) =\al \tr L_{\mlt}(e_{i} +e_{j}) - \tau_{\mlt}(e_{i}, e_{j}) =   (n-1)\al^{2},& &i \neq j.
\end{aligned}
\end{align}
By \eqref{ricperm}, the Gram matrix of $\ric_{\mlt}$ with respect to the basis $\{e_{1}, \dots, e_{n}\}$ is $(n-1)\al(1-2\al)I + (n-1)\al^{2}\ones$, where $\ones$ is the all $1$s matrix, and its eigenvalues are $(n-1)\al(1-2\al)$ with multiplicity $n-1$ and $(n-1)\al(1-2\al) + n(n-1)\al^{2} = (n-1)\al(1 + (n-2)\al)$ with multiplicity $1$. The entries of Tables \ref{talgtable1} and \ref{talgtable2} follow. For $e = \sum_{i = 1}^{n}e_{i}$ there holds $\ric_{\mlt}(e, x) = (n-1)\al(1 + (n-2)\al)\ell(x)$, where $\ell(x) = \sum_{i = 1}^{n}x_{i}$ for $x = \sum_{i = 1}^{n}x_{i}e_{i}$. If $\al = -1/(n-2)$, $e$ generates the radical of $\ric_{\mlt}$. If $\al \neq -1/(n-1)$, then $\ker \tr L_{\mlt}$ is spanned by vectors of the form $e_{i} - e_{j}$, and from $\ric_{\mlt}(e_{i} - e_{j}, e_{i} - e_{j}) = 2(n-1)\al(1-2\al)$ it follows that $\ric_{\mlt}$ is nondegenerate on $\ker \tr L_{\mlt}$ provided $\al \notin \{0, 1/2\}$, while the radical of $\ric_{\mlt}$ is $\ker \tr L_{\mlt}$ if $\al = 1/2$.

By \eqref{tauperm}, for distinct $i, j, k \in \{1, \dots, n\}$,
\begin{align}
\begin{split}
&\tau_{\mlt}(e_{i}, e_{j}\mlt e_{k})  = 2\al^{2}(2 + (n-1)\al),\\
&\tau_{\mlt}(e_{i}, e_{i} \mlt e_{j})  - \tau_{\mlt}(e_{i}\mlt e_{i}, e_{j})  = \al \tau_{\mlt}(e_{i}, e_{i}) + (\al - 1)\tau_{\mlt}(e_{i}, e_{j}) = \al(2\al - 1)((n-1)\al - 1)
\end{split}
\end{align}
Since to check the invariance of $\tau_{\mlt}$ it suffices to check it on the basis $\{e_{i}: 1 \leq i \leq n\}$, this suffices to show that $\tau_{\mlt}$ is invariant if and only if $\al \in \{ -1/(n-1), 0, 1/2\}$.

From \eqref{ricperm} there follow
\begin{align}
\begin{split}
&\ric_{\mlt}(e_{i}, e_{j}\mlt e_{k})  = 2(n-1)\al^{3},\\
&\ric_{\mlt}(e_{i}, e_{i} \mlt e_{j})  - \tau_{\mlt}(e_{i}\mlt e_{i}, e_{j})  = \al \ric_{\mlt}(e_{i}, e_{i}) + (\al - 1)\ric_{\mlt}(e_{i}, e_{j}) = 0
\end{split}
\end{align}
for all distinct $i, j, k \in \{1, \dots, n\}$. Since to check the invariance of $\ric_{\mlt}$ it suffices to check it on the basis $\{e_{i}: 1\leq i \leq n\}$, this suffices to show that $\tau_{\mlt}$ is invariant if and only if $\al \in \{ -1/(n-1), 0, 1/2\}$.

Because both sides \eqref{confnon0} are trilinear forms, to check equality it suffices to check equality on basis vectors. For pairwise distinct $i$, $j$, and $k$, the nonzero associators of the basis vectors $e_{i}$ are 
\begin{align}\label{permassoc}
\begin{split}
[e_{i}, e_{i}, e_{j}] = - [e_{j}, e_{i}, e_{i}] &= -\al^{2}e_{i} + \al(1-\al)e_{j} = \tfrac{1}{n-1}\left(\ric_{\mlt}(e_{i}, e_{i})e_{j} - \ric_{\mlt}(e_{i}, e_{j})e_{i}\right), \\
[e_{i}, e_{j}, e_{k}] &= \al^{2}(e_{k} - e_{i}) = \tfrac{1}{n-1}\left(\ric_{\mlt}(e_{i}, e_{j})e_{k} - \ric_{\mlt}(e_{j}, e_{k})e_{i}\right),
\end{split}
\end{align}
and this shows \eqref{tconfnon0}. 

Suppose $\al \neq 0$. Let $\ideal$ be an ideal of $\talg^{n}_{\al}(\fie)$. If $e_{i} \in \ideal$ for some $1 \leq i \leq n$, then $e_{j} = \al^{-1}e_{i}\mlt e_{j} - e_{i} \in \ideal$ for all $j \neq i$, so $\ideal = \talg^{n}_{\al}(\fie)$. Let $a \in \ideal$. Write $a = \sum_{i = 1}^{n}a_{i}e_{i}$. Then 
\begin{align}
a\mlt e_{i} - a = ((1-\al)a_{i} + \al \sum_{j \neq i}a_{j})e_{i} =  ((1-2\al)a_{i} + \al \sum_{j = 1}^{n}a_{j})e_{i},
\end{align}
so $e_{i} \in \ideal$ provided $(1-\al)a_{i} + \al \sum_{j \neq i}a_{j} \neq 0$. Suppose $(1-2\al)a_{i} + \al \sum_{j = 1}^{n}a_{j} = 0$ for all $1 \leq i \leq n$. Equivalently the $n \times 1$ vector $\bar{a}$ with components $a_{i}$ is in the null space of the $n \times n$ matrix $(1 - 2\al)I + \al \ones$, where $\ones$ is the all $1$s matrix. Since this matrix is invertible if and only if $\al \notin \{-1/(n-2), 0, 1/2\}$, in this case, $a = 0$. If $\al = 1/2$, then $\bar{a}$ is in the codimension one subspace $\{x = \sum_{i= 1}^{n}x_{i}e_{i}: \sum_{i = 1}^{n} x_{i} =0\} = \ker \tr L_{\mlt}$; this shows that if $\al = 1/2$ any ideal is contained in $\ker \tr L_{\mlt}$. If $\al = -1/(n-2)$, then $\bar{a}$ must be in the one-dimensional subspace spanned by $n \times 1$ all $1$s vector; this shows that if $\al = -1/(n-2)$ any nontrivial ideal equals this subspace. 
\end{proof}

\begin{remark}
The algebras $\talg^{n}_{\al}(\rea)$ arise naturally in \cite[section $5$]{Elashvili-Jibladze-Kac-semisimple} on a subspace of a semisimple Lie algebra associated with the grading determined by a semisimple cyclic element.
\end{remark}

\begin{definition}
Let $\chr \fie = 0$. For $n \geq 2$, the $n$-dimensional \emph{simplicial} algebra $\ealg^{n}(\fie)$ is the quotient of $\talg^{n+1}_{1, -1/(n-1)}(\fie)$ by its maximal ideal, equipped with the induced multiplication $\mlt$. 
\end{definition}

\begin{remark}
The definition of $\ealg^{n}(\fie)$ makes sense over a field of positive characteristic $p$ provided that $p$ does not divide $n-1$.
\end{remark}

\begin{lemma}\label{twomodelslemma}
The algebras $\talg^{n}_{-1/(n-1)}(\fie)$ and $\ealg^{n}(\fie)$ are isomorphic.
\end{lemma}
\begin{proof}
Let $e_{0}, \dots, e_{n}$ be the idempotents generating $\talg^{n+1}_{-1/(n-1)}$ and satisfying the relations $e_{i}\mlt e_{j}= -\tfrac{1}{n-1}(e_{i} + e_{j})$ for $0 \leq i \neq j \leq n$. Define a linear map $\Psi:\talg^{n+1}_{-1/(n-1)} \to \talg^{n}_{-1/(n-1)}$ by $\Psi\left(\sum_{i = 0}^{n}x_{i}e_{i}\right) =\sum_{i = 1}^{n}(x_{i}- x_{0})\ga_{i}$, where $\ga_{1}, \dots, \ga_{n}$ are the idempotents generating $\talg^{n}_{-1/(n-1)}$, subject to the relations \eqref{talgrelations}, so that $\Psi(e_{i}) = \ga_{i}$ for $1 \leq i \leq n$ and $\Psi(e_{0}) = -\sum_{i = 1}^{n}\ga_{i}$. It is straightforward to check that $\ga_{0} = -\sum_{i = 1}^{n}\ga_{i}$ is an idempotent in $\talg^{n}_{-1/(n-1)}$ satisfying $\ga_{0}\mlt \ga_{i} = -\tfrac{1}{n-1}(\ga_{0} + \ga_{i})$ for $1 \leq i \leq n$, and it follows that $\Psi$ is a surjective algebra homomorphism with kernel generated by $e = \sum_{i = 0}^{n}e_{i}$. 
\end{proof}

The proof of Lemma \ref{twomodelslemma} shows that the images $\ga_{0}, \dots, \ga_{n}$ in $\ealg^{n}(\fie)$ of the generators $e_{0}, \dots, e_{n}$ of $\talg^{n+1}_{-1/(n-1)}(\fie)$ satisfy the relations
\begin{align}
\label{ealgrelations}
& \sum_{i = 0}^{n} \ga_{i} = 0,& &\ga_{i} \mlt \ga_{i} = \ga_{i},&&\ga_{i} \mlt \ga_{j} = -\tfrac{1}{n-1}(\ga_{i} + \ga_{j}),
\end{align}
and any $n$ of the $\ga_{i}$ constitute a basis of $\ealg^{n}(\fie)$. 

\begin{corollary}\label{ealgmodelcorollary}
The algebra $\ealg^{n}(\fie)$ is exact, Killing metrized, simple, and its associator satisfies
\begin{align}\label{confnon0}
\begin{aligned}
[x, y, z] =  \tfrac{1}{n-1}\left(\tau_{\mlt}(y, z)x  - \tau_{\mlt}(x, y)z \right) ,
\end{aligned}
\end{align}
for all $x, y, z \in \ealg^{n}(\fie)$. If $\fie = \rea$, the Killing form $\tau_{\mlt}$ is positive definite. 
\end{corollary}

\begin{proof}
This follows from Lemma \ref{talgmodellemma}.
When $\al = -1/(n-1)$, $\tau_{\mlt} = -\ric_{\mlt}$, so \eqref{tconfnon0} becomes \eqref{confnon0}.
\end{proof}

\begin{remark}
A concrete model for $\ealg^{n}(\fie)$ can be given as follows. Let $\mat(n+1, \fie)$ be the Frobenius algebra of $(n+1)\times(n+1)$ matrices over $\fie$ and let $\mat_{0}(n+1, \fie)$ be the subspace of trace-free elements. Let $\si:\fie \to \fie$ be an involution (an antiautomorphism of order $2$). Let $A \to \si(A)^{t}$ be the involution on $\mat(n+1, \fie)$ given by the transpose composed with $\si$. Let $\herm(n+1, \fie)  = \{A \in \mat(n+1, \fie): A^{t} = \si(A)\}$ and let $\herm_{0}(n+1, \fie) = \herm(n+1, \fie) \cap \mat_{0}(n+1, \fie)$. The metric $\tr \si(A)^{t}B = \tr AB$ on $\herm_{0}(n+1, \fie)$ is invariant with respect to the commutative multiplication $\djrd$ on $\herm_{0}(n+1, \fie)$ defined by
\begin{align}
A \djrd B = \tfrac{1}{2}(AB + BA) - \tfrac{1}{n+1}(\tr AB)I,
\end{align}
where $I$ is the identity matrix.
It is shown in Section \ref{jordansection} that the commutative algebra $(\herm_{0}(n+1, \fie), \djrd)$ is exact and Killing metrized.
For $i \in \{0,\dots, n\}$ the diagonal matrices
\begin{align}
m(i) = \tfrac{n+1}{n-1}\left(e_{ii} - \tfrac{1}{n+1}I\right) \in \herm_{0}(n+1, \fie),
\end{align}
where $e_{ij} \in \mat(n+1, \fie)$ is the matrix with the $ij$ component equal to $1$ and all other components equal to $0$, satisfy
\begin{align}
 &m(i)\djrd m(i) = m(i),& &m(i)\djrd m(j) = -\tfrac{1}{n-1}(m(i) + m(j)),&& i \neq j,
\end{align}
so serve as a model for $\ealg^{n}(\fie)$. See Lemma \ref{herm0cartanlemma} for a more general statement. 
\end{remark}

\begin{remark}
On $\ealg^{n}(\com)$, $\tau_{\mlt}$ is not anisotropic because $(1 + \sqrt{n^{2} - 1}\j)\ga_{i} + n\ga_{j} \in \ealg^{n}(\com)$ is isotropic for any $i \neq j$.
\end{remark}

\begin{remark}
The \emph{para-Hurwitz} algebra is the two-dimensional $\rea$-algebra of complex numbers $\com$ equipped with the multiplication $u \cmlt v = \bar{u}\bar{v}$ (see \cite[section VIII.$34$]{Knus-Merkurjev-Rost-Tignol}). It is isomorphic to $\ealg^{2}(\rea)$ via the $\rea$-linear map $p + \j q \to (\tfrac{q}{\sqrt{3}} - p)\ga_{1} - (\tfrac{q}{\sqrt{3}} + p)\ga_{2}$, that sends $1$ to $-\ga_{1} - \ga_{2} = \ga_{0}$.
\end{remark}

\begin{remark}
The $n$-dimensional simplicial algebra arises as the Norton algebra of the complete graph on $n$ vertices as in \cite{Cameron-Goethals-Seidel, Cameron-Goethals-Seidel-stronglyregular}. See \cite{Fox-frames}. 
\end{remark}

\begin{remark}\label{griessharadaremark}
R.~L. Griess, in unpublished work in 1977 (see the second paragraph of \cite{Smith} and \cite[appendix]{Dong-Griess}) and K. Harada, in \cite{Harada}, showed that the $n$-dimensional algebra $\ealg^{n}(\fie)$ with presentation \eqref{ealgrelations} arises by equipping the irreducible quotient of the standard permutation module for the symmetric group $S_{n+1}$ with the multiplication obtained by projecting the standard multiplication in the permutation module. Griess and Harada also showed that the automorphism group of $\ealg^{n}(\fie)$ equals $S_{n+1}$. Here this is proved as Corollary \ref{griessharadacorollary}. 

Theorem \ref{confassclassificationtheorem} characterizes the simplicial algebras as the unique algebras among a wide class of algebras that satisfy certain particular quantitative nonassociativity conditions. 
\end{remark}

\begin{lemma}\label{talgidempotentlemma}
Let $\fie$ be a field of characteristic zero and let $\al \in \fie \setminus \{0, 1/2\}$. 
Let $\{e_{1}, \dots, e_{n}\}$ be the idempotents generating $\talg^{n}_{\al}(\fie)$ and satisfying \eqref{talgrelations}. 
For $I = \{i_{1}, \dots, i_{p}\} \subset \{1, \dots, n\}$, define $e_{I}= e_{i_{1}}+ \dots + e_{i_{p}} \in \ealg^{n}(\fie)$ and write $|I| = p$. 
\begin{enumerate}
\item\label{tidem1} Every nonzero idempotent in $\talg^{n}_{\al}(\fie)$ has the form $\si_{I} = \tfrac{1}{1 - 2\al + 2\al|I|}e_{I}$ for some $I \subset \{1, \dots, n\}$ such that $|I|  \neq 1 - 1/(2\al)$.
\item\label{tidem2} There is a nonzero element in $\szero(\talg^{n}_{\al}(\fie))$ if and only if $1 - \tfrac{1}{2\al} \in \{1, \dots, n\}$, in which case every such element is a scalar multiple of an element of the form $e_{I}$ with $|I| = 1 - 1/(2\al)$.
\item\label{tidemcardinality} The cardinality of $\proj(\idem(\talg^{n}_{\al}(\fie)))\cup \proj(\szero(\talg^{n}_{\al}(\fie)))$ is $2^{n} - 1$.
\item\label{tidemnorm} There holds $\ric_{\mlt}(\si_{I}, \si_{I}) = \tfrac{(n-1)|I|\al(1- 2\al + \al|I|)}{(1 - 2\al + 2\al|I|)^{2}} = \tfrac{(n-1)|I|(1/\al - 2 + |I|)}{(1/\al - 2 + 2|I|)^{2}}$.
\item\label{tidem5} If $I \cap J = \emptyset$, $I \cap K = \emptyset$, $J \cap K = \emptyset$ and $|J| = |K|$, then $\ric_{\mlt}(e_{I}, e_{J} - e_{K}) = 0$ and $e_{I}\mlt (e_{J} - e_{K}) = -\al |I|(e_{J} - e_{K})$.
\item\label{tidem6} If $I = J \cup K$ with $J \cap K = \emptyset$, then $\ric_{\mlt}(e_{I}, |K|e_{J}-|J|e_{K}) = 0$ and $e_{I} \mlt (|K|e_{J} - |J|e_{K}) = (1 -2\al + \al  |I|)(e_{J} - e_{K})$.
\item\label{talgspec} If $|I| \neq 1 -1/(2\al)$, then $\specp(\si_{I})$ contains $\tfrac{\al |I|}{1 - 2\al + 2\al|I|}$ with multiplicity $n- |I|$ and $\tfrac{1 - 2\al + \al|I|}{1 - 2\al + 2\al|I|}$ with multiplicity $|I| - 1$.
\item\label{talgzspec} if $|I| = 1 -1/(2\al)$, then $\specp(e_{I})$ contains $\al - 1/2$ with multiplicity $ n- 1 + 1/(2\al)$ and $1/2 - \al$ with multiplicity $-1/(2\al)$.
\end{enumerate}
\end{lemma}

\begin{proof}
Any $x \in \talg^{n}_{\al}(\fie)$ can be written $x = \sum_{i = 1}^{n}x_{i}e_{i}$ for some $x_{i} \in \fie$. That $x \mlt x = \la x$ yields the equations $\la x_{i} = x_{i}(x_{i} +2\al\sum_{j \neq i}x_{j})$ for $1 \leq i \leq n$. For each $i$, either $x_{i} = 0$ or $x_{i} = \la - 2\al \sum_{j \neq i}a_{j}$. Let $I = \{i:x_{i} \neq 0\}$. If $I = \{i\}$, then $a_{i} = \la + -2\al \sum_{j \neq i}x_{j} = 1$, so $x = \la e_{i}$. Suppose $|I| \geq 2$ and choose distinct $i, j \in I$. Then $x_{i} + 2\al x_{j} = \la  -2\al \sum_{k \in I \setminus\{i, j\}}x_{k} = x_{j} + 2\al x_{i}$. Because $\al \neq 1/2$, this yields $x_{i} = x_{j}$. Since $i, j \in I$ are arbitrary, $x_{i} = \la - 2\al\sum_{j \neq i}x_{j} = \la - 2\al(|I| - 1)x_{i}$, so $\la = (1 + 2\al(|I| - 1))x_{i}$. When $\la = 1$ there results $x_{i} = \tfrac{1}{1 - 2\al + 2\al|I|}$, and $e = \si_{I}$. When $\la = 0$, this can be only if $1 - \tfrac{1}{2\al} =  |I|$. Because $\{e_{1}, \dots, e_{n}\}$ is a basis, $e_{I}$ and $e_{J}$ are linearly independent if $I \neq J$. It follows that the cardinality of $\idem(\ealg^{n})\cup \szero(\ealg^{n})$ is $\sum_{k =1}^{n}\binom{n}{k} = 2^{n} - 1$. 
Claim \eqref{tidemnorm} follows from
\begin{align}\label{ricperm2}
\begin{split}
\ric_{\mlt}(e_{I}, e_{J}) & = (n-1)\al\left((1-2\al)|I\cap J| + \al |I||J|\right),\\
\ric_{\mlt}(e_{I}, e_{I}) &= (n-1)\al|I|\left(1 - 2\al + \al |I|\right).
\end{split}
\end{align}
It is straightforward to check that
\begin{align}\label{eIeJ}
\begin{split}
e_{I} \mlt e_{J}& =  \al |J|e_{I\setminus I \cap J} + \al|I|e_{J\setminus I \cap J} + (1-2\al +\al |I| +\al |J|)e_{I\cap J}, \\
e_{I}\mlt e_{I} &= (1 - 2\al + 2\al |I|)e_{I}.
\end{split}
\end{align}
If $I \cap J = \emptyset$, $I \cap K = \emptyset$, $J \cap K = \emptyset$ and $|J| = |K|$, then $\ric_{\mlt}(e_{I}, e_{J} - e_{K}) = 0$ follows from \eqref{ricperm}. From \eqref{eIeJ} there follows $e_{I} \mlt e_{J} = \al|J|e_{I} +\al|I| e_{J}$ and similarly with $K$ in place of $J$. From this there follows $e_{I}\mlt(e_{J} - e_{K}) = \al|I|(e_{J} - e_{K}) + \al(|J| - |K|)e_{I}$. This shows \eqref{tidem5}. The proof of \eqref{tidem6} is similar. 

By \eqref{eIeJ}, if $i \in I$ then $e_{I}\mlt e_{i} = (1 - 2\al + \al |I|)e_{i} + \al e_{I}$, so if $i\neq j$ and $i, j \in I$, then $e_{I}\mlt(e_{i} - e_{j}) =  (1 - 2\al + \al |I|)(e_{i} - e_{j})$. 
By \eqref{eIeJ}, if $i \notin I$ then $e_{I}\mlt e_{i} = \al|I|e_{i} + \al e_{I}$, so if $i\neq j$ and $i, j \notin I$, then $e_{I}\mlt(e_{i} - e_{j}) = \al |I|(e_{i} - e_{j})$.
Together these observations imply \eqref{talgspec} and \eqref{talgzspec}. 
\end{proof}

\begin{remark}
From remarks in \cite[section $5$]{Elashvili-Jibladze-Kac-semisimple} it seems those authors were aware of most of the details of Lemma \ref{talgidempotentlemma}. 
\end{remark}

\begin{remark}
From claims \eqref{talgspec} and \eqref{talgzspec} of Lemma \ref{talgidempotentlemma} it follows that $\specp(\si_{I})$ and $\specp(\si_{J})$ have nontrivial intersection if and only if either $|I| = |J|$ or $|I| + |J| - 2 = -1/\al$, in which case $\tfrac{1 - 2\al + \al|I|}{1 - 2\al + 2\al|I|} = \tfrac{\al |J|}{1 - 2\al + 2\al|J|}$ is contained in $\specp(\si_{I})\cap \specp(\si_{J})$. In particular this occurs if $|I| + |J| = n+1$ and $\al = -1/(n-1)$.
\end{remark}

Specializing Lemma \ref{talgidempotentlemma} with $\al = -1/(n-1)$ yields Lemma \ref{haradalemma}, although it should be noted that the statement admits certain refinement in this case, and it is made explicitly for this reason.
The claims \eqref{idem1} and \eqref{idem2} of Lemma \ref{haradalemma}, about the forms of idempotents and nilpotents of $\ealg^{n}(\fie)$, are Lemma $2$ and Corollary $3$ of Harada's \cite{Harada}. 

\begin{lemma}\label{haradalemma}
Let $\fie$ be a field of characteristic zero.
Let $\{\ga_{0}, \dots, \ga_{n}\}$ be the spanning set of $\ealg^{n}(\fie)$ satisfying \eqref{ealgrelations}. 
For $I = \{i_{1}, \dots, i_{p}\} \subset \{0, \dots, n\}$, write $|I| = p$ and define $\ga_{I}= \ga_{i_{1}}+ \dots + \ga_{i_{p}} \in \ealg^{n}(\fie)$. Let $I^{c}$ be the subset of $\{0, \dots, n\}$ complementary to $I$, so that $|I^{c}| = n+1 - |I|$, $n+1 - 2|I^{c}| = -(n+1 - 2|I|)$, and $\ga_{I} + \ga_{I^{c}} = 0$.
\begin{enumerate}
\item\label{idem1} Every nonzero idempotent in $\ealg^{n}(\fie)$ has the form $\si_{I} = \tfrac{n-1}{n +1 - 2|I|}\ga_{I} = \tfrac{n-1}{n +1 - 2|I^{c}|}\ga_{I^{c}} = \si_{I^{c}}$, where $1 \leq |I| \leq n$, $2|I| \neq n+1$. There holds $\tau_{\mlt}(\si_{I}, \si_{I}) = \tfrac{|I|(n-1)(n - |I| + 1)}{(n- 2|I| + 1)^{2}}$.
\item\label{idem2} There is a nonzero $x \in \szero(\ealg^{n}(\fie))$ if and only if $n$ is odd, in which case every such element is a scalar multiple of an element of the form $\ga_{I}$ with $2|I| = n+1$.
\item\label{idemcardinality} The cardinality of $\proj(\idem(\ealg^{n}(\fie)))\cup \proj(\szero(\ealg^{n}(\fie)))$ is $2^{n} - 1$.
\item\label{ealgspec} If $2|I| \neq n+1$, the spectrum $\spec(\si_{I})$ contains $-\tfrac{|I|}{n+1 - 2|I|}$ with multiplicity $n- |I|$ and $\tfrac{n+1 - |I|}{n+1 - 2|I|}$ with multiplicity $|I| - 1$.
\item \label{ealgzspec} if $2|I| = n+1$, then $\spec(\ga_{I})$ contains $-(n+1)/(2(n-1))$ with multiplicity $(n-1)/2$ and $(n+1)/(2(n-1))$ with multiplicity $(n-1)/2$.
\item\label{idem4} The set of idempotents $\{\ga_{0}, \dots, \ga_{n}\}$ is an equal norm equiangular set in $\ealg^{n}(\fie)$ satisfying
\begin{align}\label{idemtight}
\sum_{i = 0}^{n}\tau_{\mlt}(x, \ga_{i})\ga_{i} = \tfrac{n+1}{n-1}x
\end{align}
for all $x \in \ealg^{n}(\fie)$.
\item\label{idemabs} The minimal idempotents of $\ealg^{n}(\rea)$ are the $(n+1)$ elements $\ga_{0}, \dots, \ga_{n}$, any $n$ of which span $\ealg^{n}(\rea)$, and all of which are absolutely primitive and primitive. For $0 \leq i \leq n$, $\spec(\ga_{i})$ contains $-1/(n-1)$ with multiplicity $n-1$.
\end{enumerate}
\end{lemma}
\begin{proof}
Claims \eqref{idem1}-\eqref{talgzspec} follow from Lemma \ref{talgidempotentlemma}. As noted, that every idempotent and square-zero element is as in \eqref{idem1} and \eqref{idem2} follows from Lemma $2$ and Corollary $3$ of Harada's \cite{Harada}. 
It follows from \eqref{tauperm} that the elements of $\{\ga_{0}, \dots, \ga_{n}\}$ constitute an equal norm equiangular set. The identity \eqref{idemtight} follows straightforwardly from \eqref{tauperm}.
Claim \eqref{idemabs} follows from the other claims. 
\end{proof}

\begin{corollary}[R. Griess {\cite[appendix]{Dong-Griess}}, K. Harada \cite{Harada}]\label{griessharadacorollary}
The automorphism group of $\ealg^{n}(\fie)$ is the symmetric group $S_{n+1}$.
\end{corollary}
\begin{proof}
Since $\Aut(\ealg^{n})$ permutes the idempotents of $\ealg^{n}(\fie)$ and preserves $\tau_{\mlt}$, it permutes the set of absolutely primitive idempotents, $\{\ga_{0}, \dots, \ga_{n}\}$, so acts as a subgroup of $S_{n+1}$. Because an element of $\Aut(\ealg^{n}(\fie))$ which fixes $n$ of the idempotents $\{\ga_{0}, \dots, \ga_{n}\}$ must act on $\ealg^{n}(\fie)$ as the identity, the group $\Aut(\ealg^{n}(\fie))$ is equal to $S_{n+1}$. 
\end{proof}
\begin{remark}
From the remarks in \cite[appendix]{Dong-Griess}, it seems likely that the proof given here of Corollary \ref{griessharadacorollary} is essentially the unpublished one referenced in that same appendix.
\end{remark}

Lemma \ref{rootsystemlemma} gives a more explicit description of the action of $\Aut(\ealg^{n}(\fie))$. 

\begin{lemma}\label{rootsystemlemma}
Let $\{\ga_{0}, \dots, \ga_{n}\}$ be the spanning set in $\ealg^{n}(\fie)$ as in \ref{ealgrelations}.
\begin{enumerate}
\item $\Phi = \{r_{ij} = \ga_{i} - \ga_{j}: 0 \leq i\neq j \leq n\}$ is a root system of type $A_{n}$.
\item For $0 \leq i \leq n-1$, the reflection $\fl_{ij}$ through the hyperplane $\tau_{\mlt}$-orthogonal to the element $r_{ij}$ is an automorphism of $\ealg^{n}(\fie)$ that acts on $\{\ga_{0}, \dots, \ga_{n}\}$ as the transposition interchanging $\ga_{i}$ and $\ga_{j}$ and fixing $\ga_{k}$ for $k \notin \{i, j\}$.The automorphism group $\Aut(\ealg^{n}(\fie))$ is identified with $S_{n+1}$ acting irreducibly as the group generated by these reflections. 
\end{enumerate}
\end{lemma}

\begin{proof}
As $n\ga_{i} = \sum_{0\leq k \leq n, k \neq i}r_{ik}$, the $r_{ij}$ span $\ealg^{n}(\fie)$. The only scalar multiple of $r_{ij}$ in $\Phi$ is $r_{ji} = -r_{ij}$. Using $i,j,k,l$ to denote pairwise distinct indices, it follows from \eqref{tauperm} that $\tau_{\mlt}(r_{ij}, r_{ij}) = 2(n+1)/(n-1)$, $\tau_{\mlt}(r_{ij}, r_{ik}) = (n+1)/(n-1)$, and $\tau_{\mlt}(r_{ij}, r_{kl}) = 0$. Hence the pairing $\lb x, y\ra = 2\tau_{\mlt}(x, y)/\tau_{\mlt}(x, x)$ is integral on $\Phi$. The reflection $\fl_{ij}(x) = x - \lb r_{ij}, x\ra r_{ij}$ through the hyperplane $\tau_{\mlt}$-orthogonal to $r_{ij}$ acts on $\{\ga_{0}, \dots, \ga_{n}\}$ as the transposition interchanging $\ga_{i}$ and $\ga_{j}$ and fixing $\ga_{k}$ for $k \notin \{i, j\}$, and hence preserves $\Phi$. This shows that $\Phi$ is a root system. Write $r_{i} = r_{i\, i+1}$. Then $\Delta = \{r_{i}:0 \leq i \leq n-1\}$ is a set of simple roots and the reflections $\fl_{i}$ through the hyperplanes $\tau_{\mlt}$-orthogonal to the $r_{i} \in \Delta$ suffice to generate an action of $S_{n+1}$. As $\lb r_{i}, r_{j}\ra$ is $-1$ if $|i - j| = 1$ and $0$ if $|i - j| > 1$, $\Phi$ has type $A_{n}$. It is straightforward to check that, for the reflection $\fl$ associated to any element of $\Phi$, $\fl(\ga_{i})\mlt \fl(\ga_{j}) = \fl(\ga_{i}\mlt \ga_{j})$ for all $0 \leq i, j \leq n$. Since the $\ga_{i}$ span $\ealg^{n}(\fie)$ and the $\fl_{ij}$ generate $S_{n+1}$, it follows that $S_{n+1}$ acts on $\ealg^{n}(\fie)$ by automorphisms permuting $\{\ga_{0}, \dots, \ga_{n}\}$. Since $\Aut(\ealg^{n}(\fie))$ acts faithfully on this set, this shows $\Aut(\ealg^{n}(\fie)) = S_{n+1}$.
\end{proof}

\begin{remark}
When $\al \notin \{0, 1/2, -1/(n-1)\}$, the argument proving Lemma \eqref{rootsystemlemma} with $\ric_{\mlt}$ in place of $\tau_{\mlt}$ shows that $\Phi = \{r_{ij} = e_{i} - e_{j}: 1 \leq i\neq j \leq n\} \subset \ker \tr L_{\mlt}$ is a root system of type $A_{n-1}$ in the codimension one subspace $\ker \tr L_{\mlt} \subset \talg_{\al}^{n}$ and that the reflection $\fl_{ij}$ through the hyperplane $\ric_{\mlt}$-orthogonal to the element $r_{ij}$ is an automorphism of $\talg^{n}_{\al}$ that fixes $\ker \tr L_{\mlt}$ and acts on $\{e_{1}, \dots, e_{n}\}$ as the transposition interchanging $e_{i}$ and $e_{j}$ and fixing $e_{k}$ for $k \notin \{i, j\}$, so the Weyl group $S_{n}$ of this root system acts on $\talg^{n}_{\al}$ by automorphisms fixing the subspaces $\ker \tr L_{\mlt}$ and $\ideal = \spn\{e_{0} = -\sum_{i =1}^{n}e_{i}\}$. 
That $S_{n}$ acts by automorphisms can be seen explicitly as follows. For $x = \sum_{k = 1}^{n}x_{k}e_{k}$ let $\ell(x) = \sum_{k = 1}^{n}x_{k}$, so that $\tr L_{\mlt} = (1 + (n-1)\al)\ell$. There hold $r_{ij} \mlt x = \al \ell(x)r_{ij} + (1-2\al)(x_{i}e_{i} - x_{j}e_{j})$ and $\ric_{\mlt}(r_{ij}, x) = (n-1)\al(1-2\al)(x_{i}- x_{j})$, so that $\lb r_{ij}, x_{i}\ra = x_{i} - x_{j}$, which all together yield $\fl_{ij}(x) = x - (x_{i} - x_{j})r_{ij}$ and can be used to check that 
\begin{align}
\begin{split}
\fl_{ij}(x\mlt y) &= x\mlt y - \left(\al(x_{i} - x_{j})\ell(y) + \al(y_{i} - y_{j})\ell(x) + (1-2\al)(x_{i}y_{i} - x_{j}y_{j})\right)r_{ij}\\
& = \fl_{ij}(x)\mlt \fl_{ij}(y).
\end{split}
\end{align}
In the special case $\al = -1/(n-2)$, $S_{n}$ fixes the unique maximal ideal $\ideal$, so, by Lemma \ref{twomodelslemma}, descends to an action of $S_{n}$ on $\ealg^{n-1}(\fie)$ by automorphisms, which recovers that described in Lemma \ref{rootsystemlemma}. 
When $\al = -1/(n-1)$, the reflections $\fl_{ij}$ still generate an action of $S_{n}$, but it is not the full automorphism group of $\talg^{n}_{-1/(n-1)}(\fie)$, as for all $\al$, the $S_{n}$ action fixes the vector $e_{0}= -\sum_{i =1}^{n}e_{i} \in \talg^{n}_{\al}$, which, as $e_{0}\mlt e_{0} = -(1 + 2(n-1)\al)e_{0}$, is for this value of $\al$ (and no other) an idempotent (for other values of $\al$ it generates a one-dimensional invariant subalgebra, but is not an idempotent).
\end{remark}

\section{Tensor products of simplicial algebras}\label{tensorsection}
This section describes some results about the structure of $\ealg^{2}(\fie)\tensor \ealg^{n}(\fie)$. It would be interesting to describe $\ealg^{m}(\fie)\tensor \ealg^{n}(\fie)$ in general. 
The structure of $\ealg^{2}(\fie)\tensor \ealg^{n}(\fie)$ depends more strongly on $n$ than might be suspected. Lemma \ref{2tensor2lemma} shows that $\ealg^{2}(\fie)\tensor \ealg^{2}(\fie)$ is isomorphic to $\ealg^{2}(\fie)\oplus \ealg^{2}(\fie)$ and Lemma \ref{2tensor6lemma} shows that $\ealg^{2}(\fie)\tensor \ealg^{n}(\fie)$ is not simple if $n \geq 6$, while Theorem \ref{ealgtripletheorem} shows that, for any $n\geq 2$, $\ealg^{3}(\fie)\tensor \ealg^{n}(\fie)$ is simple and is not isomorphic to $\ealg^{3n}(\fie)$.

Let $\{e_{i}:0 \leq i \leq 2\}$ be the minimal idempotents in $\ealg^{2}(\fie)$ and let $\{e_{\al}: 0 \leq \al \leq n\}$ be the minimal idempotents in $\ealg^{n}(\fie)$. Then $\{e_{i\al} = e_{i}\tensor e_{\al}: 0\leq i \leq 2, 0 \leq \al \leq n\}$ are idempotents in $\ealg^{2}(\fie)\tensor \ealg^{n}(\fie)$ that satisfy $\sum_{i = 0}^{2}e_{i\al} = 0$ and $\sum_{\al = 0}^{n}e_{i\al} = 0$. For pairwise distinct $\al, \be, \ga \in \{0, \dots,n\}$ and $0 \leq i \leq 2$ define
\begin{align}
\begin{aligned}	
&a_{\al\be\ga} = \tfrac{n-1}{n+1}(e_{0\al} + e_{1\be} + e_{2\ga}),&\\
&b_{i\al\be\ga} = \tfrac{n-1}{n-5}(e_{i\al} + e_{i\be} + e_{i\ga}), &&\text{if}&& n \neq 5,& &z_{i\al\be\ga} = \tfrac{4}{3}(e_{i\al} + e_{i\be} + e_{i\ga}), &&\text{if}&& n = 5.
\end{aligned}
\end{align}
Note that while $b_{i\al\be\ga}$ and $z_{i\al\be\ga}$ do not depend on the ordering of $\al$, $\be$, and $\ga$, $a_{\al\be\ga}$ does.
The elements $a_{\al\be\ga}$ can be visualized as follows. Fill a $3 \times (n+1)$ grid with the indices $i\al$ taken from $\{0, 1, 2\}\times \{0, \dots, n\}$ in lexicographic order. An element $a_{\al\be\ga}$ corresponds to each triple of boxes in distinct rows and columns. The elements $e_{i\al}$ can be recovered by the formulas
\begin{align}\label{earecovery}
\begin{aligned}
&\begin{aligned}
&\tfrac{n+1}{n-1}(a_{\al\be\ga} + a_{\al\ga\be}) + \tfrac{n-5}{n-1}b_{0\al\be\ga} = 3e_{0\al},&\\
&\tfrac{n+1}{n-1}(a_{\al\be\ga} + a_{\ga\be\al}) + \tfrac{n-5}{n-1}b_{1\al\be\ga} = 3e_{1\be},&\\
&\tfrac{n+1}{n-1}(a_{\al\be\ga} + a_{\be\al\ga}) + \tfrac{n-5}{n-1}b_{2\al\be\ga} = 3e_{2\ga},&
\end{aligned}&& \text{if}&& n \neq 5,\\
&\begin{aligned}
&\tfrac{1}{2}(a_{\al\be\ga} + a_{\al\ga\be}) + \tfrac{1}{4}z_{0\al\be\ga} = e_{0\al},&\\
&\tfrac{1}{2}(a_{\al\be\ga} + a_{\ga\be\al}) + \tfrac{1}{4}z_{1\al\be\ga} = e_{1\be},&\\
&\tfrac{1}{2}(a_{\al\be\ga} + a_{\be\al\ga}) + \tfrac{1}{4}z_{2\al\be\ga} = e_{2\ga},&
\end{aligned}&& \text{if}&& n= 5.
\end{aligned}
\end{align}
Straighttforward calculations show the relations
\begin{align}
\label{aaa} &a_{\al\be\ga} + a_{\be\ga\al} + a_{\ga\al\be} = 0,& & b_{0\al\be\ga} + b_{1\al\be\ga} + b_{2\al\be\ga} = 0,&& z_{0\al\be\ga} + z_{1\al\be\ga} + z_{2\al\be\ga} = 0,&\\
\label{nnaa}&a_{\al\be\ga}\mlt a_{\al\be\ga} = a_{\al\be\ga}, & &a_{\al\be\ga}\mlt a_{\be\ga\al} = a_{\ga\al\be}, &&a_{\al\be\ga}\mlt a_{\ga\al\be} = a_{\be\ga\al},&\\
\label{nnaa2}&a_{\al\be\ga}\mlt a_{\al\ga\be} = \tfrac{n-5}{n+1}b_{0\al\be\ga}, &&a_{\al\be\ga}\mlt a_{\ga\be\al} = \tfrac{n-5}{n+1}b_{1\al\be\ga}, &&a_{\al\be\ga}\mlt a_{\be\al\ga} = \tfrac{n-5}{n+1}b_{2\al\be\ga}, &\\
\label{nnaa25}&a_{\al\be\ga}\mlt a_{\al\ga\be} = z_{0\al\be\ga}, &&a_{\al\be\ga}\mlt a_{\ga\be\al} =z_{1\al\be\ga}, &&a_{\al\be\ga}\mlt a_{\be\al\ga} = z_{2\al\be\ga}, &\\
\label{nnbb}&b_{i\al\be\ga}\mlt b_{i\al\be\ga} = b_{i\al\be\ga},& &b_{i\al\be\ga}\mlt b_{i+1\,\al\be\ga} = b_{i+2\,\al\be\ga},&&b_{i\al\be\ga}\mlt b_{i+2\,\al\be\ga} = b_{i+1\,\al\be\ga},\\
\label{nnzz}&z_{i\al\be\ga}\mlt z_{i\al\be\ga} = 0,& &z_{i\al\be\ga}\mlt z_{i+1\,\al\be\ga} = -z_{i+2\,\al\be\ga},&&z_{i\al\be\ga}\mlt z_{i+2\,\al\be\ga} = -z_{i+1\,\al\be\ga},
\end{align}
where Latin indices are taken modulo $2$, \eqref{nnaa2} becomes \eqref{nnaa25} when $n = 5$ provided it is understood that $\tfrac{n-5}{n+1}b_{i\al\be\ga} = \tfrac{n-1}{n+1}(e_{i\al} + e_{i\be} + e_{i\ga})$, and expressions involve $z_{i\al\be\ga}$ are defined only when $n = 5$.

\begin{lemma}
If $\chr \fie = 0$, $\ealg^{2}(\fie)\tensor \ealg^{5}(\fie)$ is not isomorphic to $\ealg^{10}(\fie)$.
\end{lemma}
\begin{proof}
Equation \eqref{nnzz} shows that $\ealg^{2}(\fie)\tensor \ealg^{5}(\fie)$ contains a nontrivial square-zero element, while Lemma \ref{haradalemma} shows that $\ealg^{10}(\fie)$ contains no nontrivial square-zero element.
\end{proof}

Comparing the relations \eqref{n2ba} and \eqref{nnaa2} with \eqref{ealgrelations} shows that for each choice of $0 \leq \al < \be < \ga \leq n$ the set $\{a_{\al\be\ga}, a_{\be\ga\al}, a_{\ga\al\be}\}$ generates a subalgebra of $\ealg^{2}(\fie)\tensor \ealg^{n}(\fie)$ isomorphic to $\ealg^{2}(\fie)$.

\begin{lemma}\label{2tensor2lemma}
If $\chr \fie = 0$, $\ealg^{2}(\fie)\tensor \ealg^{2}(\fie)$ is isomorphic to $\ealg^{2}(\fie)\oplus \ealg^{2}(\fie)$. 
\end{lemma}

\begin{proof}
It suffices to find two orthogonal ideals in $\ealg^{2}(\fie)\tensor \ealg^{2}(\fie)$ each isomorphic to $\ealg^{2}$. When $n = 2$, 
\begin{align}
\label{n2ba}&b_{i\al\be\ga} = 0,&\\
\label{n2aa}& a_{\al\be\ga} + a_{\al\ga\be} = e_{0\al},&& a_{\al\be\ga} + a_{\ga\be\al} = e_{1\be},&& a_{\al\be\ga} + a_{\be\al\ga} = e_{2\ga}.&
\end{align}
By \eqref{aaa}, \eqref{n2ba}, \eqref{nnaa}, and \eqref{nnaa2}, each of the sets $\{a_{012}, a_{120}, a_{201}\}$ and $\{a_{021}, a_{210}, a_{102}\}$ consists of three idempotents summing to $0$ and satisfying the relations \eqref{ealgrelations}, so generating a subalgebra of $\ealg^{2}\tensor \ealg^{2}$ isomorphic to $\ealg^{2}$.
By \eqref{n2aa} the nine sums obtained by adding to an element of $\{a_{012}, a_{120}, a_{201}\}$ an element of $\{a_{021}, a_{210}, a_{102}\}$ are exactly the nine idempotents $\{e_{i\al}: 0 \leq i, \al \leq 2\}$. These relations imply that the product of an element of $\{a_{012}, a_{120}, a_{201}\}$ an element of $\{a_{021}, a_{210}, a_{102}\}$ is $0$, so these generate orthogonal ideals in $\ealg^{2}\tensor \ealg^{2}$ each isomorphic to $\ealg^{2}$.
\end{proof}

\begin{remark}
Over $\rea$, Lemma \ref{2tensor2lemma} also follows from \cite[Examples $1.7$ and $3.6$]{Fox-cubicpoly}.
\end{remark}

From the relations \eqref{earecovery} and \eqref{aaa}-\eqref{nnzz} it follows that, if $n > 2$ and $n \neq 5$,
\begin{align}\label{ea2nspan}
\begin{aligned}
\spn&\{a_{\al\be\ga}, a_{\be\ga\al}, a_{\ga\la\be}, a_{\al\ga\be}, a_{\ga\be\al}, a_{\be\al\ga}, b_{0\al\be\ga}, b_{1\al\be\ga}, b_{2\al\be\ga}\}\\
&= \spn\{a_{\al\be\ga}, a_{\be\ga\al}, a_{\al\ga\be}, , a_{\be\al\ga}, b_{0\al\be\ga}, b_{1\al\be\ga}, \}\\
&= \spn\{e_{0\al}, e_{0\be}, e_{0\ga}, e_{1\al}, e_{1\be}, e_{1\ga}\}\\
&= \spn\{e_{0\al}, e_{0\be}, e_{0\ga}, e_{1\al}, e_{1\be}, e_{1\ga}, e_{2\al}, e_{2\be}, e_{2\ga}\},
\end{aligned}
\end{align}
is a $6$-dimensional subalgebra of $\ealg^{2}\tensor \ealg^{n}$. When $n = 5$, \eqref{ea2nspan} is true with $z_{i\al\be\ga}$ in place of $b_{i\al\be\ga}$.

If $\al, \be, \ga, \delta, \si, \mu$ are distinct indices (this requires $n \geq 6$), then
\begin{align}
\label{ddaa}&a_{\al\be\ga} \mlt a_{\delta\si\mu} = 0,&\\
\label{ddab}&b_{0\al\be\ga}\mlt a_{\delta \si \mu} = -\tfrac{3}{n+1}a_{\delta \mu \si},&
&b_{1\al\be\ga}\mlt a_{\delta \si \mu} = -\tfrac{3}{n+1}a_{\mu \si \delta},&
&b_{2\al\be\ga}\mlt a_{\delta \si \mu} = -\tfrac{3}{n+1}a_{\si \delta \mu}.&
\end{align}

\begin{lemma}\label{2tensor6lemma}
For $\chr \fie \neq 0$, $\ealg^{2}(\fie)\tensor \ealg^{n}(\fie)$ is not simple if $n \geq 6$.
\end{lemma}
\begin{proof}
By \eqref{ea2nspan}, for any $2 \leq k \leq n-3$ the $6$-dimensional subspaces
\begin{align}
\begin{aligned}
\alg_{0} &= \spn \{e_{i\al}:0 \leq i \leq 2, 0 \leq \al,\be, \ga \leq 1\} =  \spn\{a_{i\al\be\ga}: 0 \leq i \leq 2, 0 \leq \al,\be, \ga \leq 2\}, \\
\alg_{k} &= \spn \{e_{i\al}:0 \leq i \leq 2, k \leq \al,\be, \ga \leq k+1\} \\
&=\spn \{a_{i\al\be\ga}: 0 \leq i \leq 2, k \leq \al,\be, \ga \leq k+2\},
\end{aligned}
\end{align}
are transverse, and, by \eqref{nnaa}, \eqref{nnaa2}, \eqref{nnbb}, \eqref{ddaa}, and \eqref{ddab}, $\alg_{0}\mlt \alg_{k} = \{0\}$. Hence $\alg_{0}$ and $\sum_{k = 2}^{n-3}\alg_{k}$ are complementary and $\alg_{0} \mlt \sum_{k = 2}^{n-3}\alg_{k} = \{0\}$, so $\alg_{0}$ is a nontrivial ideal.
\end{proof}

\section{Unitalization and deunitalization}\label{unitalizationsection} 
The classical Thomas construction associates with a projective connection on a manifold $M$ a Ricci flat affine connection on a one-dimensional fiber bundle over $M$ such that the lifted connection is flat if and only if the original projection connection is projectively flat. This is an extension to a curved differential geometric setting of the usual processes of homogenization and projectivization. The formally analogous notions for commutative algebras are the unitalization and deunitalization constructions described in this section. These are used in Section \ref{characterizationsection} to show that certain quasi-associativity conditions modeled on notions of projective and conformal flatness for connections on vector bundles can be used to characterize the simplicial algebras over $\rea$ or an algebraically closed field of characteristic zero. 

For a symmetric bilinear form $c \in S^{2}\alg^{\ast}$, the \emph{$c$-unitalization} of the commutative algebra $(\alg, \mlt)$ is $\hat{\alg} = \alg \oplus \fie$ equipped with the commutative multiplication $\hmlt$ defined by
\begin{align}\label{hmltdefined}
(x, \al) \hmlt (y, \be) = (x\mlt y + \al y + \be x, \al \be + c(x, y)),
\end{align}
and the bilinear form $\hat{c} \in S^{2}\halg^{\ast}$ defined by $\hat{c}((x, \la), (y, \mu)) = c(x, y) + \la \mu$. By definition $(0, 1) \in \hat{\alg}$ is a unit, so $(\halg, \hmlt)$ is a unital commutative algebra. 
That $\hat{c}$ is $\hmlt$-invariant if and only if $c$ is $\mlt$-invariant follows from the observation that
\begin{align}\label{hatcinvariant}
\hat{c}((x, \al) \hmlt (y, \be), (z, \ga)) = c(x\mlt y, z) + \al c(y, z) + \be c(x, z) + \ga c(x, y) + \al \be \ga 
\end{align}
is completely symmetric if and only if $c$ is $\mlt$-invariant.  

The \emph{unitalization} of a metrized commutative algebra $(\alg, \mlt, h)$ is its $h$-unitalization, $(\halg, \hmlt, \hat{h})$, which is a unital metrized commutative algebra. 

\begin{remark}
The usual definition of the unitalization of an algebra is what is called here the $c$-unitalization with $c$ identically zero. The unitalization as defined here is the unitalization within a category whose objects are metrized commutative algebras.
\end{remark}

Lemma \ref{unitalizationisomorphismlemma} shows that two metrized commutative algebras are isomorphic if and only if their unitalizations are isometrically isomorphic.

A \emph{homomorphism} of metrized commutative algebras is an algebra homomorphism $\Psi:(\alg, \mlt, h) \to (\balg, \cmlt, g)$ that is metric-preserving in the sense that $\Psi^{\ast}(g) = h$. By the nondegeneracy of $h$ and $g$, $\Psi$ is injective. An isomorphism of metrized commutative algebras is a bijective homomorphism of such algebras whose inverse is also a homomorphism of metrized commutative algebras; equivalently it is an isometric algebra isomorphism. The group $\Aut(\alg, \mlt, h)$ of automorphisms of $(\alg, \mlt, h)$ is $\Aut(\alg, \mlt, h) = \Aut(\alg, \mlt) \cap O(h)$.

In the context of nonunital algebras, a homomorphism of algebras that happen to be unital is not required to preserve units. 
\begin{lemma}\label{unitpreservationelemma}
Let $\Phi:(\alg, \mlt) \to (\balg, \cmlt)$ be a homomorphism of metrized commutative algebras and suppose $e \in \alg$ and $f \in \balg$ are units. Then $\Phi(e) = f$ if $f \in \im \Phi$.
\end{lemma}

\begin{proof}
If $f = \Phi(u)$, then $\Phi(x) = f \cmlt \Phi(x) = \Phi(u) \cmlt \Phi(x) = \Phi(u \mlt x)$ for all $x \in \alg$. Because $\Phi$ is injective, this implies $u$ is a unit in $\alg$, so must equal $e$, because a unit is unique if it exists.
\end{proof}

\begin{lemma}\label{unitalizationisomorphismlemma}
Let $(\halg, \hmlt, \hat{h})$ and $(\hbalg, \widehat{\cmlt}, \hat{g})$ be the unitalizations of the metrized commutative algebras $(\alg, \mlt, h)$ and $(\balg, \cmlt, g)$. 
\begin{enumerate}
\item\label{unitlift1} The \emph{lift} $\hat{\Psi}:(\halg, \hmlt, \hat{h}) \to (\hbalg, \widehat{\cmlt}, \hat{g})$, defined by $\hat{\Psi}(x, \al) = (\Psi(x), \al)$, is a homomorphism of metrized commutative algebras if and only if $\Psi:(\alg, \mlt, h) \to (\balg, \cmlt, g)$ is a homomorphism of metrized commutative algebras.
\item\label{unitlift2} A homomorphism $\Phi:(\halg, \hmlt, \hat{h}) \to (\hbalg, \widehat{\cmlt}, \hat{g})$ of metrized commutative algebras is the lift $\hat{\Psi}$ of some homomorphism $\Psi:(\alg, \mlt, h) \to (\balg, \cmlt, g)$ of metrized commutative algebras if there holds any of the following conditions:
\begin{enumerate}
\item $g$ is anisotropic.
\item $\Phi$ is surjective.
\item $\Phi(0, 1) = (0, 1)$.
\end{enumerate}
\item\label{unitlift3} The map $\Psi \to \hat{\Psi}$ is a group isomorphism between $\Aut(\alg, \mlt, h)$ and $\Aut(\halg, \hmlt, \hat{h})$. 
\end{enumerate}
\end{lemma}

\begin{proof}
From the definitions there follow
\begin{align}
\begin{split}
\hat{\Psi}(x, \al)\hmlt \hat{\Psi}(y, \be) &- \hat{\Psi}((x, \al)\hmlt (y, \be)) = (\Psi(x)\mlt \Psi(y) - \Psi(x \mlt y), \Psi^{\ast}(g)(x, y) - h(x, y)),\\
\hat{\Psi}^{\ast}(g)((x, \al), (y, \be)) &- \hat{h}((x, \al), (y, \be)) = \Psi^{\ast}(g)(x, y) - h(x, y),
\end{split}
\end{align}
which imply \eqref{unitlift1}. 
Given a linear map $\Phi:\halg \to \hbalg$, there are a linear map $\Psi:\alg \to \balg$, $u \in \balg$, $\mu \in \alg^{\ast}$, and $c \in \fie$ such that $\Phi(x, r) = (\Psi(x) + r u, \mu(x) + cr)$. That $\Phi^{\ast}(\hat{g}) = \hat{h}$ yields the equations
\begin{align}\label{ul1}
&\Psi^{\ast}(g)(x, y) - h(x,y) =- \mu(x) \mu(y) , & & c\mu(x) =- g(\Psi(x), u), & &1 - c^{2} = g(u, u),
\end{align}
for all $x, y \in \alg$,
while that $\Phi$ be an algebra homomorphism yields the equations
\begin{align}\label{ul2}
\begin{aligned}
& c - c^{2} = g(u, u),&  &\Psi(x)\cmlt \Psi(y)  - \Psi(x\mlt y) + \mu(x)\Psi(y) + \mu(y)\Psi(x) = h(x, y)u ,& \\
&u \cmlt u = (1-2c)u , & &\mu(x)\mu(y) - \mu(x \mlt y) + \Psi^{\ast}(g)(x, y) = ch(x, y),&  \\
&(1-c)\mu(x) = g(\Psi(x), u),&  &\Psi(x)\cmlt u + \mu(x)u = (1-c)\Psi(x) .
\end{aligned}
\end{align}
From \eqref{ul1} and \eqref{ul2} there results $1 - c^{2} = g(u, u) = c - c^{2}$, so that $c = 1$ and $g(u, u) = 0$. In \eqref{ul1} and \eqref{ul2} this yields $-\mu(x) = g(\Psi(x), u) = 0$ for all $x \in \alg$, so that $\mu = 0$ and $u$ is $g$-orthogonal to $\Psi(\alg)$. With these observations, \eqref{ul1} becomes $\Psi^{\ast}(g) = h$ and \eqref{ul2} becomes
\begin{align}\label{ul3}
\begin{aligned}
&\Psi(x)\cmlt \Psi(y)  - \Psi(x\mlt y)  = h(x, y)u ,& 
&u \cmlt u = -u , &    &\Psi(x)\cmlt u  =0 .
\end{aligned}
\end{align}
If $g$ is anisotropic, that $g(u, u) =0$ implies $u = 0$. Since $\Phi(x, 0) = (\Psi(x), 0)$, if $\Phi$ is surjective, then $\Psi$ is surjective and that $u$ is $g$-orthgonal to $\Psi(\alg) = \balg$ implies $u = 0$. If $(0, 1) = \Phi(0, 1) = (u, 1)$, then $u = 0$. In any of these cases, \eqref{ul3} implies $\Psi$ is a homomorphism of metrized commutative algebras. 
Because the lift of a composition of homomorphisms is the composition of their lifts, claim \eqref{unitlift3} follows from \eqref{unitlift1} and \eqref{unitlift2}. 
\end{proof}

Let $(\balg, \star, g)$ be a metrized commutative algebra and let $\alg \subset \balg$ be a $g$-nondegenerate subspace. The $g$-orthogonal projection $\pi$ of $\balg$ on $\alg$ means the projection onto $\alg$ along the transverse subspace $\alg^{\perp} = \{x \in \balg: g(x, y) =0\,\,\text{for all}\,\, y \in \alg\}$. The multiplication $\mlt$ defined on $\balg$ by $x \mlt y = \pi(\pi(x)\star \pi(y))$ satisfies
\begin{align}
\begin{split}
g(x &\mlt y, z) = g(\pi(\pi(x)\star \pi(y)), z) = g(\pi(\pi(x)\star \pi(y)),\pi(z)) = g(\pi(x)\star \pi(y), \pi(z)) \\
&= g(\pi(x), \pi(y)\star\pi(z)) = g(\pi(x), \pi(\pi(y)\star\pi(z))) = g(x, \pi(\pi(y)\star\pi(z))) = g(x, y\mlt z),
\end{split}
\end{align}
so $(\balg, \mlt, g)$ is also a metrized commutative algebra. As $\balg \mlt \balg \subset \alg$, $\alg$ is a two-sided ideal of $(\balg, \mlt, g)$. Modifying slightly terminology from \cite{Griess-monster}, the algebra $(\alg, \mlt, g)$ is called the \emph{retraction} of $(\balg, \star, g)$ onto the $g$-nondegenerate subspace $\alg$.

The \emph{deunitalization} $(\alg, \mlt, h)$ of a metrized unital commutative algebra $(\balg,\star, g)$ having a nonistropic unit $e$ is the retraction of $(\balg, \star, h =  g(e, e)^{-1}g)$ along $e$ (note that the induced metric is $h$, not $g$). The terminology \emph{deunitalization} is justified by Lemma \ref{deunitlemma} that shows that unitalization and deunitalization are inverse operations in a precise sense.
The factor of $g(e, e)^{-1}$ in the definition of the metric $h$ of a deunitalization is necessary for the first claim of Lemma \ref{deunitlemma} to hold.

\begin{lemma}\label{deunitlemma}
A unital metrized commutative algebra $(\balg, \star, g)$ with anisotropic unit is isomorphic to the unitalization $(\halg, \hmlt, \hat{h})$ of its deunitalization $(\alg, \mlt, h)$ via the map $\Pi:\balg \to \halg$ defined by $\Pi(x + \al e) = (x, \al) \in \halg$.
A metrized commutative algebra $(\alg, \mlt, h)$ is isomorphic to the deunitalization of its unitalization $(\halg, \hmlt, \hat{h})$ via the inclusion $i:\alg \to \halg$ defined by $i(x) = (x, 0)$.
\end{lemma}
\begin{proof}
From the definition of $\star$ it follows that  
\begin{align}\label{xstary}
\begin{split}
(x + \al e) \star(y + \be e) & = x\mlt y +  \al y + \be x  + \left(h(x, y) +\al\be\right)e,\\
\end{split}
\end{align}
for $x,y \in \alg$ and $\al, \be \in \fie$. Comparing \eqref{xstary} with \eqref{hmltdefined} yields the first claim.
By definition the deunitalization of $(\halg, \hmlt, \hat{h})$ is the vector space $i(\alg)$ equipped with the multiplication $\star$ defined by $i(x)\star i(y) = i(x) \hmlt i(y) - \hat{h}(i(x), i(y))(0, 1) = (x, 0) \hmlt (y, 0) - (0, h(x, y)) = (x\mlt y, 0) = i(x\mlt y)$, and the metric obtained by restricting $\hat{h}$. This shows second claim.
\end{proof}

\begin{lemma}\label{griesseinsteinlemma}
Let $(\balg, \star, g)$ be a metrized unital commutative algebra with anisotropic unit $e$ and suppose there are constants $A$ and $B$ such that 
\begin{align}\label{prech}
\tau_{(\balg, \star)}(x, y) = A g(x, e) g(y, e) + B g(x, y),
\end{align}
for $x, y \in \balg$. Then $\dim \balg = Ag(e, e)^{2} + Bg(e, e)$,
and the deunitalization $(\alg, \mlt, h)$ is exact and satisfies $\tau_{(\alg, \mlt)} = \ka h$ with $\ka = Bg(e, e) - 2$.
\end{lemma}
\begin{proof}
By \eqref{unitalizationtrace} combined with Lemma \ref{unitalizationisomorphismlemma}, if $x \in \alg$, then $\tr L_{(\B, \star)}(x) = \tr L_{\mlt}(x)$. Since $g(x, e) = 0$ if $x \in \alg$, taking $x \in \alg$ and $y = e$ in \eqref{prech} gives $ \tr L_{\mlt}(x) = \tr L_{(B, \star)}(x) = (Ag(e, e) + B)g(x, e) = 0$.
This shows that the deunitalization $(\alg, \mlt, h)$ is exact. Taking $x = e = y$ in \eqref{prech} gives $\dim \balg = \tr L_{(B, \star)}(e) = Ag(e, e)^{2} + Bg(e, e)$.
 Combined with Lemma \ref{unitalizationisomorphismlemma}, \eqref{tauunitalization} yields that $\tau_{(\balg, \star)}(x, y)= \tau_{(\alg, \mlt)}(x, y)+  2h(x, y)$,
for $x, y \in \alg$, and with \eqref{prech} this yields $\tau_{(\alg, \mlt)}(x, y) = (Bg(e, e) -2)h(x, y)$.
\end{proof}

Lemma \ref{griesseinsteinlemma} can be used to check that the deunitalizations of certain algebras are Killing metrized.

\begin{example}\label{monsterexample}
In \cite{Griess-friendlygiant}, R.~L. Griess constructed a $196883$-dimensional metrized commutative algebra $(\alg, \star, h)$, called here the \emph{$196883$-dimensional (nonunital) Griess algebra}, and constructed a finite simple group $\monster$ of monster type acting on $(\alg, \star, h)$ as automorphisms. In Griess's construction of $(\alg, \star, h)$ it is defined over $\rat$; here it is considered over $\rea$. For the construction of the monster finite simple group and the associated Griess algebra consult \cite{Frenkel-Lepowsky-Meurman} or \cite{Ivanov}. It was shown by J. Tits in \cite{Tits-friendly} that $\monster$ is the automorphism group of $(\alg, \star, h)$. In \cite{Thompson} and \cite{Griess-Meierfrankenfeld-Segev} it was established that the group $\monster$ is unique up to isomorphism, justifying calling it \emph{the} monster finite simple group. 
It is known that $\monster$ has a unique complex simple representation of dimension $196883$, from which it follows that $h$ is uniquely determined up to homothety. The final equation on page $498$ of \cite{Tits-friendly} establishes that $\tau_{\mlt}$ is a nonzero multiple of $h$ (called there a ``curious relation''; this also appears as equation $(2)$ of section $5$ of \cite{Tits-monstre}, although the normalizations in \cite{Tits-friendly} and \cite{Tits-monstre} are different). 

In studying $\monster$ it is common to work with some $196844$-dimensional unital metrized commutative algebra the deunitalization of which is the $196883$-dimensional Griess algebra $(\alg, \star, h)$, and different references, e.g. \cite{Griess-friendlygiant, Griess-monster, Conway-monster, Frenkel-Lepowsky-Meurman}, use slightly different extensions (any such algebra is commonly also called a \emph{Griess algebra}). 
Here the $196884$-dimensional extension $(\balg, \mlt, g)$ is taken as in \cite{Conway-monster}, except that $g$ is scaled so that the unit $e$ has $g$-norm $1$ (in \cite{Conway-monster}, the unit has length $3/2$). This algebra (so scaled) is called here the \emph{Conway-Griess algebra}. By the formula at the end of section $13$ of \cite{Conway-monster}, there holds \eqref{prech} with $A = 183024$ and $B = 13860$. Substituted into \eqref{prech} of Lemma \ref{griesseinsteinlemma} this shows that the Killing form of the Griess algebra satisfies $\tau_{\mlt} = (4\cdot 13^{2}\cdot 41)h = 13858 h$, which agrees with the cited result of \cite{Tits-friendly} up to normalizations.

As is explained in Example \ref{monsternortonexample}, the Griess algebra satisfies Norton's inequality, so has nonnegative sectional nonassociativity in the sense of Section \ref{nasection}.
\end{example}

\begin{example}\label{voaexample}
Calculations of Matsuo \cite{Matsuo} can be used to show the deunitalizations of Griess algebras of certain vertex operator algebras are Killing metrized exact commutative algebras and to calculate their Einstein constants with respect to a reference metric induced from the vertex operator algebra structure. This entails no new results about such algebras and amounts to translating known results into the language and context used here, so should be understood as an extended example.

A \emph{vertex operator algebra} (VOA) over a field $\fie$ is an $\fie$-vector space $\ste$ equipped with a linear map $Y: \ste \to (\eno\ste)[[z, z^{-1}]]$, called the \emph{state-field correspondence} and written
\begin{align}
Y(a, z) = \sum_{n \in \integer}a_{(n)}z^{-n-1} \in (\eno\ste)[[z, z^{-1}]],
\end{align}
taking values in endomorphism-valued formal power series, and two distinguished vectors, the \emph{vacuum vector} $\one \in \ste$ (also called the \emph{identity element}) and the \emph{conformal vector} $\cc \in \ste$ (also called the \emph{Virasoro element}), that together satisfy the axioms:
\begin{enumerate}
\item (locality) For $a, b \in \ste$, there is $N > 0$ such that $a_{(n)}b =0$ if $n \geq N$. 
\item $Y(\one, z) = \id_{\ste}$.
\item For $a \in \ste$, $Y(a, z)\one \in \ste[[z]]$ and $\lim_{n \to 0}Y(a, z) =a$.
\item (Jacobi/Borcherds identity)  For $a, b \in \ste$ and $l, m, n \in \integer$,
\begin{align}\label{voajacobi}
\begin{split}
\sum_{i =0}^{\infty}& \binom{m}{i}(a_{(l+i)}b)_{(m+n-i)} 
= \sum_{i =0}^{\infty}(-1)^{i}\binom{l}{i}\left(a_{(m+l-i)}b_{(n+i)} - (-1)^{l}b_{(n+l-i)}a_{(m+i)}\right).
\end{split}
\end{align}
\item The endomorphisms $L_{m} = \cc_{(m+1)}$ satisfy
\begin{align}
[L_{m}, L_{n}] = (m-n)L_{m+n} + \tfrac{m^{3} - m}{12}c\delta_{m+ n, 0}
\end{align}
for some constant $c$ called the \emph{central charge} of the VOA (so the conformal vector $\cc$ generates a representation of the Virasoro Lie algebra).
\item The operator $L_{0}$ is semisimple and there is a direct sum decomposition $\ste = \oplus_{n = 0}^{\infty}\ste^{n}$ where $\ste^{n}$ is the eigenspace of $L_{0}$ with eigenvalue $n$ (called the \emph{conformal weight} of a corresponding eigenvector), which is moreover finite-dimensional for all $n \geq 0$.
\end{enumerate}
These axioms force $\one \in \ste^{0}$ and $\ste^{i}_{(n)}\ste^{j} \subset \ste^{i+j-n-1}$.
(The first four axioms define a \emph{vertex algebra}; a vertex algebra satisfying also the last two axioms is a \emph{vertex operator algebra}.)
There are many equivalent definitions of vertex (operator) algebras. Here the axioms are stated as in \cite{Zhu-modularinvariance}; see also \cite{Beilinson-Drinfeld-chiral, DeSole-Kac, Frenkel-Ben-zvi, Frenkel-Lepowsky-Meurman, Gebert, Kac-vertexbook, Lepowsky-Li, Matsuo-Nagatomo}.

Here the phrase \emph{vertex operator algebra} without further qualification means a vertex operator algebra over $\com$. As for Lie algebras, it is different to speak of a real form of a complex vertex operator algebra and a real vertex operator algebra. A precise formulation of what is a real form of a VOA is given in \cite[section $2$]{Mason-fivepieces}.

A VOA is OZ (short for \emph{one-zero}) if $\ste^{0}$ is generated by $\one$ and $\ste^{1}= \{0\}$. 
For an OZ VOA, the multiplication $\star$ defined by $a \star b = a_{(1)}b$ for $a,b \in \balg$ makes $\ste^{2}$ into a commutative algebra and the symmetric bilinear form $g$ defined by $g(a, b)\one = a_{(3)}b$ for $a, b \in \ste^{2}$ is invariant with respect to $\star$. This is explained in detail in \cite[section $3.5$]{Gebert}. 
By definition, $L(\cc) = 2\Id_{\alg}$ and $g(\cc, \cc) = c/2$. Hence $e = \cc/2$ is a unit in $\alg$ such that $g(e, e) = c/8$. The triple $(\ste^{2}, \star, g)$ is called the \emph{Griess algebra} of the OZ VOA. 
Note that the Griess algebra of a OZ complex VOA is, by definition, an algebra over $\com$. If the VOA were a real VOA, then its Griess algebra would be a $\rea$-algebra.

\begin{remark}
The terminology differs from that used in example \ref{monsterexample} in two ways. In example \ref{monsterexample}, the terminology \emph{nonunital Griess algebra} refers to a particular algebra, that associated to the Monster Lie group, whereas here, in the context of VOAs, \emph{Griess algebra} refers to one of a general class of algebras. More importantly, in example \ref{monsterexample} the nonunital Griess algebra is nonunital, while the Griess algebra of an OZ VOA is by definition unital. The relation between the two uses is however straightforward. There is constructed in \cite{Frenkel-Lepowsky-Meurman} an OZ VOA, the \emph{moonshine VOA}, whose Griess algebra is a unitalization of the Griess algebra constructed by Griess. Its Griess algebra is the Conway-Griess algebra of example \ref{monsterexample}.
\end{remark}

In \cite[section $1.3$]{Matsuo} an OZ VOA is defined to be of \emph{class $\S^{n}$} for an integer $n \geq 1$ if the action of its automorphism group on the quotient $\oplus_{k =0}^{n}\ste^{k}/\oplus_{k =0}^{n}\ste^{k}_{\cc}$ has no fixed point. Here $\ste_{\cc}$ is the submodule of $\ste$ generated by the vacuum vector $\one$ and the action $L_{n}$ generated by $\cc$, and $\ste_{\cc}^{k} = \ste_{\cc}\cap \ste^{k}$. By definition, an OZ VOA of class $\S^{n}$ is of class $\S^{m}$ for all $1 \leq m < n$. 

\begin{theorem}[A. Matuso, {\cite[Theorem $2.1$]{Matsuo}}]\label{matsuotheorem}
Let $(\ste^{2}, \star, g)$ be the Griess algebra of an OZ VOA $\ste$ with vacuum vector $\one$, conformal vector $\cc$, and central charge $c$ and suppose that $g$ is nondegenerate. Let $n = \dim \ste^{2}$ and $e = \cc/2$.
\begin{enumerate}
\item If $\ste$ is of class $\S^{2}$ then for all $a \in \ste^{2}$,
\begin{align}
\tr L_{\star}(a) = \tfrac{4n}{c}g(a, \cc) = \tfrac{8n}{c}g(a, e).
\end{align}
\item If $\ste$ is of class $\S^{4}$ then for all $a, b \in \ste^{2}$,
\begin{align}
\begin{split}
\tau_{\star}(a, b) &= \tr L_{\star}(a)L_{\star}(b) = \tfrac{-2(5c^{2} - 88n + 2cn)}{c(5c + 22)}g(a, b) + \tfrac{4(5c + 22n)}{c(5c + 22)}g(a, \cc)g(b, \cc) \\
&= \tfrac{-2(5c^{2} - 88n + 2cn)}{c(5c + 22)}g(a, b) + \tfrac{16(5c + 22n)}{c(5c + 22)}g(a, e)g(b, e)\\
& = \tfrac{-2(5c^{2} - 88n + 2cn)}{c(5c + 22)}g(a, b) + \tfrac{c(5c + 22n)}{8n^{2}(5c + 22)}\tr L_{\star}(a) \tr L_{\star}(b).
\end{split}
\end{align}
\end{enumerate}
\end{theorem}

Note that the assumed nondegeneracy of $g$ implies that the central charge is nonzero.

\begin{corollary}
Let $(\ste^{2}, \star, h)$ be the Griess algebra of an OZ VOA $\ste$ of class $\S^{4}$, with vacuum vector $\one$, conformal vector $\cc$, and central charge $c$ and suppose that $h$ is nondegenerate. Let $n = \dim \ste^{2}$ and $e = \cc/2$. Then the deunitalization $(\alg, \mlt,h)$ of $(\ste^{2}, \star, g)$ is a (complex) Killing metrized exact commutative algebra with $\tau_{\mlt} = \ka h$ for $\ka = \tfrac{-5c^{2} + 88(n-2) - 2c(n+20)}{4(5c + 22)}$.
\end{corollary}

\begin{proof}
By Theorem \ref{matsuotheorem}, the Griess algebra $(\balg, \mlt, g)$ satisfies the conditions of Lemma \ref{griesseinsteinlemma} with $A = \tfrac{16(5c + 22n)}{c(5c + 22)}$ and $B =  \tfrac{-2(5c^{2} - 88n + 2cn)}{c(5c + 22)}$. Hence 
\begin{align}
n = \dim \ste^{2}= \tfrac{c(5c + 22n)}{4(5c + 22)} - \tfrac{(5c^{2} - 88n + 2cn)}{4(5c + 22)},
\end{align}
and the deunitalization $(\alg, \mlt,h)$ of $(\ste^{2}, \star, g)$ is a Killing metrized exact commutative algebra with Einstein constant 
\begin{align}
\ka = Bg(e, e) - 2= 3B - 2=  \tfrac{-5c^{2} + 88n - 2cn}{4(5c + 22)} - 2 =  \tfrac{-5c^{2} + 88(n-2) - 2c(n+20)}{4(5c + 22)}.
\end{align}
\end{proof}

For the moonshine VOA, $c = 24$ and $n = 196883$, so $A = 20336$, $B = 4620$, and $\ka = 13858$. As the moonshine VOA is defined over $\rea$ and the invariant form $g$ is positive definite, the deunitalization of the Griess algebra of the moonshine VOA is a Euclidean Killing metrized exact commutative $\rea$-algebra.

As is explained in Example \ref{voanortonexample}, a theorem of Miyamoto shows that the Griess algebras of OZ VOAs satisfy Norton's inequality, so have nonnegative sectional nonassociativity in the sense of Section \ref{nasection}.
\end{example}

\section{Characterization of simplicial algebras and classification when \texorpdfstring{$n \leq 4$}{n}}\label{characterizationsection}
This section applies the unitalization construction to obtain a characterization of simplicial algebras. It concludes with the classification of Euclidean Killing metrized exact commutative algebras of dimension at most $4$.

Let $(\halg, \hmlt, \hat{c})$ be the $c$-unitalization of the $n$-dimensional commutative algebra $(\alg, \mlt)$.
By \eqref{hmltdefined} the matrices of $L_{\mlt}(x)$ and $L_{\hmlt}(x, \al)$ with respect to appropriate bases of $\alg$ and $\hat{\alg}$ are related by
\begin{align}\label{extensionmatrix}
&L_{\hmlt}((x, \al)) = \begin{pmatrix} L_{\mlt}(x) + \al \Id_{\alg}   & x\\ c(x, \dum) & \al \end{pmatrix}.
\end{align}
From \eqref{extensionmatrix} there follow
\begin{align}\label{unitalizationtrace}
\tr L_{\hmlt}((x, \al)) &= \tr L_{\mlt}(x) + (1 + n)\al,\\
\label{tauunitalization}
\tau_{\hmlt}((x, \al), (y, \be)) &= \tau_{\mlt}(x, y) + 2c(x, y) + (1 + n)\al \be  + \be \tr L_{\mlt}(x) + \al \tr L_{\mlt}(y),\\
\label{trlxy}
\tr L_{\hmlt}((x, \al) \hmlt(y,\be)) & = \tr L_{\mlt}(x \mlt y) + \al \tr L_{\mlt}(y) + \be \tr L_{\mlt}(x) + (1 + n)\hat{c}((x, \al),(y, \be)),\\
\label{ricunitalization}
\ric_{\hmlt}((x, \al), (y, \be)) &= \ric_{\mlt}(x, y) + (n - 1)c(x, y).
\end{align}
By \eqref{ricunitalization} the $c$-unitalization of a commutative algebra $(\alg, \mlt)$ corresponding to $c = -\tfrac{1}{n - 1}\ric_{\mlt}$ is distinguished by the requirement that its Ricci form $\ric_{\hmlt}$ vanish.
\begin{definition}\label{intrinsicliftdefinition}
\noindent
\begin{enumerate}
\item The \emph{intrinsic unitalization} of an $n$-dimensional commutative algebra $(\alg, \mlt)$ with Ricci form $\ric_{\mlt}$ is the $-\tfrac{1}{n-1}\ric_{\mlt}$-unitalization of $(\alg, \mlt)$.
\item A commutative algebra $(\alg, \mlt)$ is \emph{projectively associative} if there is $c \in S^{2}\alg^{\ast}$ such that $[x, y, z] = c(y, z)x - c(x, y)z$ for all $x, y, z \in \alg$.
\end{enumerate}
\end{definition}

\begin{lemma}\label{thomaslemma}
Suppose $\chr \fie = 0$. For an $n$-dimensional commutative $\fie$-algebra $(\alg, \mlt)$ with Ricci form $\ric_{\mlt}$ the following are equivalent.
\begin{enumerate}
\item\label{thomas1} There holds $[x, y, z] = -\tfrac{1}{n -1}\ric_{\mlt}(y, z)x + \tfrac{1}{n-1}\ric_{\mlt}(x, y)z$ for all $x, y, z \in \alg$.
\item\label{thomas2} There is $c \in S^{2}\alg^{\ast}$ such that $[x, y, z] = c(y, z)x - c(x, y)z$ for all $x, y, z \in \alg$.
\item\label{thomas3} There is $c \in S^{2}\alg^{\ast}$ such that the $c$-unitalization of $(\alg, \mlt)$ is associative.
\item\label{thomas4} The intrinsic unitalization of $(\alg, \mlt)$ is associative.
\end{enumerate}
When there hold these conditions, the Ricci form $\ric_{\mlt}$ is invariant and for all $x, y, z \in \alg$ there holds
\begin{align}\label{ejkidentity}
[L_{\mlt}(x), L_{\mlt}(y)]L_{\mlt}(z) + [L_{\mlt}(y), L_{\mlt}(z)]L_{\mlt}(x) + [L_{\mlt}(z), L_{\mlt}(x)]L_{\mlt}(y) =0.
\end{align}
\end{lemma}
\begin{proof}
Claim \eqref{thomas1} implies claim \eqref{thomas2} trivially and claim \eqref{thomas4} implies \eqref{thomas3} trivially.
By \eqref{hmltdefined}, the associator of the $c$-unitalization of a commutative algebra $(\alg, \mlt)$ is given by
\begin{align}\label{unitalassoc}
[(x, \al), (y, \be), (z, \ga)] = ([x, y, z] + c(x, y)z - c(y, z)x, 0),
\end{align}
The equivalence of \eqref{thomas2} and \eqref{thomas3} is immediate from \eqref{unitalassoc}. 
If there holds \eqref{thomas2}, then $L_{\mlt}(x \mlt y) - L_{\mlt}(x)L_{\mlt}(y) = c(y, \dum)\tensor x - c(x, y)\Id_{\alg}$ for all $x, y \in \alg$, and tracing this yields $\ric_{\mlt} = (1 - n)c$, so there holds \eqref{thomas1}. The same argument shows \eqref{thomas1} implies \eqref{thomas4}. The identity \eqref{ejkidentity} follows from \eqref{thomas1} straightforwardly. There remains to show the invariance of $\ric_{\mlt}$. Rewrite \eqref{thomas1} as 
\begin{align}\label{tho1}
L_{\mlt}(x \mlt y) = L_{\mlt}(x)\circ L_{\mlt}(y) - \tfrac{1}{n - 1}x\tensor \ric_{\mlt}(y, \dum) + \tfrac{1}{n - 1}\ric_{\mlt}(x, y)\id_{\alg}.
\end{align}
From \eqref{tho1} there follows
\begin{align}\label{tho2}
\begin{split}
\ric_{\mlt}&(x \mlt y, z)   = \tr L_{\mlt}((x\mlt y)\mlt z)) - \tr L_{\mlt}(x\mlt y)L_{\mlt}(z)\\
&= \tr L_{\mlt}((x\mlt y)\mlt z)) - \tr L_{\mlt}(x)L_{\mlt}(y)L_{\mlt}(z) + \tfrac{1}{n - 1}\left(\ric_{\mlt}(y, z\mlt x) - \ric_{\mlt}(x, y)\tr L_{\mlt}(z)\right).
\end{split}
\end{align}
Combining \eqref{tho2} and \eqref{thomas1} yields
\begin{align}\label{tho3}
\begin{split}
\ric_{\mlt}&(x \mlt y, z)  - \ric_{\mlt}(x, y\mlt z) \\
& = \tr L_{\mlt}([x, y, z])  + \tfrac{1}{n - 1}\left( \ric_{\mlt}(y, z\mlt x) - \ric_{\mlt}(x, y)\tr L_{\mlt}(z) -\ric_{\mlt}(z, x\mlt y) + \ric_{\mlt}(y, z)\tr L_{\mlt}(x)\right)\\
& = \tfrac{1}{n - 1}\left(\ric_{\mlt}(z\mlt x, y) - \ric_{\mlt}(z, x\mlt y)\right).
\end{split}
\end{align}
Interchanging $x$ and $y$ in \eqref{tho3} yields
\begin{align}\label{tho4}
\begin{split}
\ric_{\mlt}(y \mlt x, z)  - \ric_{\mlt}(y, x\mlt z) & = \tfrac{1}{n - 1}\left(\ric_{\mlt}(z\mlt y, x) -\ric_{\mlt}(z, y\mlt x) \right).
\end{split}
\end{align}
Substituting \eqref{tho4} into \eqref{tho3} yields
\begin{align}\label{tho5}
\begin{split}
\ric_{\mlt}(x \mlt y, z)  - \ric_{\mlt}(x, y\mlt z) &= \tfrac{1}{(n - 1)^{2}}\left( \ric_{\mlt}(x\mlt y, z) - \ric_{\mlt}(x, y\mlt z)\right) .
\end{split}
\end{align}
Since $n > 1$, \eqref{tho5} implies $\ric_{\mlt}(x \mlt y, z)  - \ric_{\mlt}(x, y\mlt z)  = 0$.
\end{proof}

\begin{corollary}\label{thomascorollary}
Suppose $\chr \fie = 0$. A commutative $\fie$- algebra $(\alg, \mlt)$ is projectively associative if and only if its intrinsic unitalization is associative.
\end{corollary}

\begin{example}\label{ealgunitalizationexample}
By Lemma \ref{talgmodellemma}, the algebra $\talg^{n}_{\al}(\fie)$ is projectively associative for all $\al \in \fie$. Moreover, if $\al \neq 1/2$, its intrinsic unitalization is split over $\fie$, meaning it is isomorphic to a direct sum of $n+1$ copies of $\fie$. To see this for $\ealg^{n}(\fie)$, define $\hat{\ga}_{i} = (\tfrac{n-1}{n+1}\ga_{i}, \tfrac{1}{n+1}) \in \widehat{\ealg}^{n}(\fie)$. From \eqref{ealgrelations} there follow $\hat{\ga}_{i}\hmlt \hat{\ga}_{i} = \hat{\ga}_{i}$ and $\hat{\ga}_{i} \hmlt \hat{\ga}_{j} = 0$ for $0 \leq i < j \leq n$. Hence $\hat{\ga}_{0}, \dots, \hat{\ga}_{n}$ are $n+1$ orthogonal idempotents in $\hat{\ealg}^{n}(\fie)$. For $\talg^{n}_{\al}(\fie)$ with $\al \neq 1/2$, define $\hat{e}_{i} = \left(\tfrac{1}{1-2\al}e_{i}, \tfrac{\al}{2\al -1}\right)$ and $\hat{e}_{0} = \left(\tfrac{1}{2\al -1}\sum_{j = 1}^{n}e_{j}, \tfrac{1 +(n-2)\al}{1-2\al}\right)$. From \eqref{talgrelations} it follows that $\hat{e}_{i}\hmlt \hat{e}_{i} = \hat{e}_{i}$ and $\hat{e}_{i} \hmlt \hat{e}_{j} = 0$ for $0 \leq i < j \leq n$. Hence $\hat{e}_{0}, \dots, \hat{e}_{n}$ are $n+1$ orthogonal idempotents in $\hat{\talg}^{n}_{\al}(\fie)$.
\end{example}

\begin{remark}
That the identity \eqref{ejkidentity} holds for the algebra $\talg^{n}_{\al}(\fie)$ was observed in \cite{Elashvili-Jibladze-Kac-semisimple}. By Lemma \ref{thomaslemma} this is a consequence of its projective associativity.  
\end{remark}

It is convenient to conflate the associator of $\mlt$ with the tensor $\mu$ defined by $\mu(x, y, z) = [y, z, x]$. 
That is, $[x, y, z]^{l} = x^{i}y^{j}z^{k}\mu_{kij}\,^{l}$ for the tensor $\mtens_{ijk}\,^{l} = \mtens_{jk}\,^{p}\mtens_{pi}\,^{l} - \mtens_{ki}\,^{p}\mtens_{jp}\,^{l}$, where $\mtens_{ij}\,^{k}$ is the structure tensor of the multiplication $\mlt$. Attention has to be paid to the ordering and position of the indices of $\mu_{ijk}\,^{l}$. The algebra $(\alg, \mlt)$ is associative if and only if $\mtens_{ijk}\,^{l} = 0$. The Ricci form $\rho_{ij}$ is obtained as the trace $\ric_{ij} = \mu_{pij}\,^{p}$ and is symmetric. 
For a metrized commutative algebra $(\alg, \mlt, h)$, define $\mu_{ijkl} = \mu_{ijk}\,^{p}h_{pl} \in \tensor^{4}\alg^{\ast}$, so that 
\begin{align}\label{hmudefined}
\mu(a, b, c, d) = h([b, c, a], d) = h(b\mlt c, a\mlt d) - h(c\mlt a, b \mlt d)
\end{align}
for $a, b, c, d \in \alg$. From the invariance of $h$ with respect to $\mlt$ it follows that $\mu_{ijkl} = - \mu_{jikl} = -\mu_{ijlk} = \mu_{klij}$ and $\mu_{[ijk]l} =0$, so that $\mu_{ijkl}$ has the symmetries of a metric curvature tensor. 

Define the \emph{conformal nonassociativity tensor} $\om_{ijkl}$ of a metrized commutative algebra $(\alg, \mlt, h)$ of dimension $n \geq 3$ be the completely trace-free part of $\mu_{ijkl}$, 
\begin{align}\label{confnon}
\om_{ijkl} = \mu_{ijkl} + \tfrac{2}{n-2}\left(\ric_{k[i}h_{j]l} - \ric_{l[i}h_{j]k}\right) + \tfrac{2}{(n-1)(n-2)}\scal h_{l[i}h_{j]k},
\end{align}
in which $\scal = h^{ij}\ric_{ij}$ is the metric trace of the Ricci form. Equivalently,
\begin{align}\label{confnon2}
\begin{split}
\om(x, y, z, w) & =  h(y\mlt z, x \mlt w) - h(z \mlt x, w\mlt y) \\
&+ \tfrac{1}{n-2}\left(\ric(x, z)h(y, w) - \ric(y, z)h(x, w) - \ric(x, w)h(y, z) + \ric(y, w)h(x, z)\right) \\
&+ \tfrac{\scal}{(n-1)(n-2)}\left(h(x, w)h(y, z) - h(y, w)h(x, z)\right).
\end{split}
\end{align}
If $\om_{ijkl}$ vanishes then $(\alg, \mlt)$ is \emph{conformally associative}. Since the $O(n)$-module of trace-free tensors with metric curvature symmetries is trivial if $n \leq 3$, in three dimensions this is automatic. The convention is adopted that a two-dimensional metrized commutative algebra is regarded as conformally associative.

\begin{lemma}\label{prebeltramilemma}
If $\chr \fie = 0$, a conformally associative metrized commutative $\fie$-algebra $(\alg, \mlt, h)$ of dimension $n \geq 3$ is projectively associative if and only if there is $c \in \fie$ such that $\ric_{\mlt} = (n-1)c h$.
\end{lemma}
\begin{proof}
If $(\alg, \mlt, h)$ is projectively and conformally assocative, then there is $c_{ij}$ such that $c_{ki}h_{jl} - c_{jk}h_{il} = \mu_{ijkl} = -\tfrac{2}{n-1}(\ric_{k[i}h_{j]l} - \ric_{l[i}h_{j]k}) - \tfrac{2}{(n-1)(n-2)}\si h_{l[i}h_{j]k}$. Tracing this in $il$ yields $c_{kj} - nc_{jk} = \ric_{jk}$, while tracing it in $jk$ yields $c_{li} - c_{p}\,^{p}h_{il} = \ric_{li}$, from which it follows that $\ric_{ij} = \tfrac{1}{n}\si h_{ij}$. If $(\alg, \mlt, h)$ is conformally associative and $\ric_{ij} = (n-1)c h_{ij}$ for some $c \in \fie$, then $\si = n(n-1)c$, and \eqref{confnon} yields $\mu_{ijkl} = 2ch_{l[i}h_{j]k}$, so $(\alg, \mlt, h)$ is projectively associative. 
\end{proof}

\begin{lemma}\label{killingunitalizationlemma}
If $\chr \fie = 0$, the unitalization $(\halg, \hmlt, \hat{h})$ of a metrized commutative $\fie$-algebra $(\alg, \mlt, h)$ of dimension $n > 1$ is Killing Einstein if and only if $(\alg, \mlt, h)$ is exact and Killing Einstein with Einstein constant $\ka = n-1$. In this case $(\halg, \hmlt, \hat{h})$ has Killing Einstein constant $n+1$ and vanishing Ricci form.
\end{lemma}

\begin{proof}
By \eqref{tauunitalization},
\begin{align}\label{hunittauforms}
\begin{split}
\tau_{\hmlt}((x, \al), (y, \be))&- (2 + \ka)\hat{h}((x, \al), (y, \be)) \\& = \tau_{\mlt}(x, y) - \ka h(x, y) + (n-1-\ka)\al \be +  \al \tr L_{\mlt}(y) + \be \tr L_{\mlt}(x) ,
\end{split}
\end{align}
If $(\alg, \mlt, h)$ is exact and Killing Einstein with Einstein constant $\ka = n-1$, then the right-hand side of \eqref{hunittauforms} vanishes, so $(\halg, \hmlt, \hat{h})$ is Killing Einstein with Einstein constant $\hat{\ka} = 2 + \ka = n+1$. If $(\halg, \hmlt, \hat{h})$ is Killing Einstein with Einstein constant $\hat{\ka}$, then the left-hand side of \eqref{hunittauforms} vanishes for $\ka = \hat{\ka} - 2$. Taking $x = 0 = y$ in \eqref{hunittauforms} shows $\ka = n - 1$, while taking $y = 0$ and $\al = 0$ in \eqref{hunittauforms} shows $\tr L_{\mlt}(x) = 0$ for all $x \in \alg$. It follows that $(\alg, \mlt, h)$ is exact and Killing Einstein with Einstein constant $\ka = n-1$. In this case, $\ric_{\mlt} = -\tau_{\mlt} = (1-n)h$, so, by \eqref{tauunitalization} and \ref{trlxy}, $\ric_{\hmlt} =0$.  
\end{proof}

\begin{theorem}\label{confassclassificationtheorem}
Let $\fie$ be a field of characteristic zero.
\begin{enumerate}
\item If $\fie$ is algebraically closed, an $n$-dimensional commutative $\fie$-algebra $(\alg, \mlt)$ is isomorphic to $\ealg^{n}(\fie)$ if and only if it is exact, has nondegenerate Killing form, and is projectively associative.
\item An $n$-dimensional commutative $\rea$-algebra $(\alg, \mlt)$ is isomorphic to $\ealg^{n}(\rea)$ if and only if it is exact, has positive-definite Killing form $\tau_{\mlt}$, and is projectively associative.
\end{enumerate}
\end{theorem}

\begin{proof}
That $\ealg^{n}(\fie)$ has the stated properties follows from Corollary \ref{ealgmodelcorollary}, so it suffices to show the converse. 
In either case, that $(\alg, \mlt)$ be projectively associative implies it is Ricci invariant, by Lemma \ref{thomaslemma}, so that it be exact with nondegenerate Killing form implies that it is Killing metrized. 
Let $(\hat{\alg}, \hmlt)$ be the intrinsic unitalization of $(\alg, \mlt)$. By Lemma \ref{thomaslemma}, $(\hat{\alg}, \hmlt)$ is associative if and only if $(\alg, \mlt)$ is projectively associative, in which case, by the exactness of $(\alg, \mlt)$,
\begin{align}\label{confalg}
[x, y, z] = - \tfrac{1}{n-1}\tau_{\mlt}(x, y)z + \tfrac{1}{n-1}\tau_{\mlt}(y, z)x,
\end{align}
for $x, y, z \in \alg$. 
In this case, $(\hat{\alg}, \hmlt)$ is commutative, associative, and unital and, by \eqref{tauunitalization} and \eqref{hatcinvariant},
\begin{align}\label{fat}
\tau_{\hmlt}((x, \al), (y, \be)) = (1 + n)\left(\tfrac{1}{n - 1}\tau_{\mlt}(x, y) + \al\be \right),
\end{align}
is invariant with respect to $\hmlt$. 
Hence if an $n$-dimensional algebra $(\alg, \mlt)$ is exact, has nondegenerate Killing form $\tau_{\mlt}$, and is projectively associative, then its intrinsic unitalization is an $(n+1)$-dimensional commutative, associative, unital algebra with nondegenerate invariant Killing form. If $\fie = \rea$ and $\tau_{\mlt}$ is positive definite, then by \eqref{fat}, $\tau_{\hmlt}$ is positive definite. Consequently, by Lemma \ref{flatalgebralemma}, if either $\fie = \rea$ and $\tau_{\mlt}$ is positive definite or $\fie$ is algebraically closed, $(\halg, \hmlt)$ is a split étale $\fie$-algebra of rank $n+1$. Since, as shown in Example \ref{ealgunitalizationexample}, the unitalization of $\ealg^{n}(\fie)$ is also isomorphic to the direct sum of $n+1$ copies of $\fie$, it follows from Lemma \ref{unitalizationisomorphismlemma} that $(\alg, \mlt)$ is isomorphic to $\ealg^{n}(\fie)$.
\end{proof}

\begin{example}\label{3dexample}
For $c \in \fiet$, define a commutative multiplication $\mlt$ on the $\fie$-vector space $\ste$ generated by $\{e_{1}, e_{2}, e_{3}\}$ by
\begin{align}\label{n3alg}
\begin{split}
&e_{1}\mlt e_{1} = e_{2}\mlt e_{2} = e_{3}\mlt e_{3} = 0, \qquad e_{1}\mlt e_{2} = ce_{3}, \qquad e_{2}\mlt e_{3} = ce_{1}, \qquad e_{3}\mlt e_{1} = ce_{2},
\end{split}
\end{align}
and extending linearly. It is straightforward to check that $(\ste, \mlt)$ is a projectively associative Killing metrized exact commutative algebra such that $\{e_{1}, e_{2}, e_{3}\}$ is a $\tau_{\mlt}$-orthogonal basis satisfying $\tau_{\mlt}(e_{i}, e_{i}) = 2c^{2}$ for $1 \leq i \leq 3$. By Theorem \ref{confassclassificationtheorem}, if $\fie = \rea$ or $\fie$ is algebraically closed, then $(\ste, \mlt)$ is isomorphic to $\ealg^{3}(\fie)$. In fact, $(\ste, \mlt)$ is isomorphic to $\ealg^{3}(\fie)$ for any field $\fie$ of characteristic zero, as can be seen explicitly by observing that the elements $\ga_{0}, \ga_{1}, \ga_{2}, \ga_{3} \in \ste$ defined by $2c\ga_{0} = e_{1} + e_{2} + r_{3}$, $2c\ga_{1} = e_{1} - e_{2} - e_{3}$, $2c\ga_{2} = -e_{1} + e_{2} - e_{3}$, and $2c\ga_{3} = -e_{1}-e_{2} + e_{3}$ (so $c^{-1}e_{1} = -\ga_{2} - \ga_{3}$, $c^{-1}e_{2} = -\ga_{1} - \ga_{3}$, and $c^{-1}e_{3} = -\ga_{1} - \ga_{2}$) satisfy the relations \eqref{ealgrelations}.
\end{example}

\begin{corollary}\label{n23corollary}
A Killing metrized exact commutative $\fie$-algebra $(\alg, \mlt)$ of dimension $n \in \{2, 3\}$ is isomorphic to $\ealg^{n}(\fie)$ if either $\fie$ is algebraically closed or $\fie = \rea$ and $\tau_{\mlt}$ is positive definite.
\end{corollary}
\begin{proof}
Since $n \leq 3$, $(\alg, \mlt, \tau_{\mlt})$ is conformally associative, so projectively associative by Lemma \ref{prebeltramilemma}, and the conclusion follows from Theorem \ref{confassclassificationtheorem}.
\end{proof} 

\begin{example}
By Example \ref{cubicexample}, the conclusion of Corollary \ref{n23corollary} is false over $\rat$. Were the deunitalization of the degree $3$ cyclic Galois extension of $\rat$ obtained by adjoining a root of $r^{3} + r^{2} - 2r - 1$ isomorphic to $\ealg^{2}(\rat)$, then its unitalization would be isomorphic to a direct sum of three copies of $\rat$, which it is not, so, by Lemma \ref{unitalizationisomorphismlemma}, it cannot be isomorphic to $\ealg^{2}(\rat)$.
\end{example}

\begin{lemma}\label{no1dideallemma}
A Killing metrized exact commutative algebra $(\alg, \mlt)$ has no $\tau_{\mlt}$-anisotropic ideal of dimension $1$ or codimension $1$.
\end{lemma}

\begin{proof}
A $\tau_{\mlt}$-anisotropic subspace $\ideal \subset \alg$ is an ideal if and only if its $\tau_{\mlt}$-orthogonal complement $\ideal^{\perp}$ is an ideal. In this case $\alg = \ideal \oplus \ideal^{\perp}$ and $\ideal \mlt \ideal^{\perp} = \{0\}$.
Suppose $\ideal \subset \alg$ is a one-dimensional $\tau_{\mlt}$-anisotropic ideal spanned by $z \in \alg$. Because $L_{\mlt}(z) \ideal^{\perp} = \{0\}$ and $\tr L_{\mlt}(z) = 0$, $z \mlt z = 0$ and so $L_{\mlt}(z) = 0$. Consequently, $\tau_{\mlt}(z, x) = \tr L_{\mlt}(z)L_{\mlt}(x) = 0$ for all $x \in \alg$, contradicting the nondegeneracy of $\tau_{\mlt}$.
\end{proof}

\begin{corollary}
A $3$-dimensional Killing metrized exact commutative algebra with anisotropic Killing form is simple.
\end{corollary}

\begin{lemma}\label{4dnotsimplelemma}
If a $4$-dimensional Euclidean Killing metrized exact commutative algebra $(\alg, \mlt)$ is not simple, then it is isomorphic to $\ealg^{2}(\rea) \oplus \ealg^{2}(\rea)$.
\end{lemma}

\begin{proof}
If $(\alg, \mlt)$ has a two-dimensional ideal, then by Corollary \ref{n23corollary} this ideal and its orthogonal complement must be isomorphic to $\ealg^{2}(\rea)$, so it suffices to show that $(\alg, \mlt)$ has no odd-dimensional ideal. Since the orthcomplement of an ideal is an ideal, it suffices to show that $(\alg, \mlt)$ has no one-dimensional ideal. Suppose that $z \in \alg$ generates a one-dimensional ideal. Then $z \mlt z = r z$ for some $r \in \rea$. Because $L_{\mlt}(z)$ maps the orthocomplement of $z$ into the space of $z$ and is exact, $r = \tr L_{\mlt}(z) = 0$. Then $L_{\mlt}(z)^{2} = 0$, so $\tau_{\mlt}(z, z) =0$, contradicting the anisotropy of $\tau_{\mlt}$. 
\end{proof}

\begin{theorem}\label{4dclassificationtheorem}
A Euclidean Killing metrized exact commutative algebra of dimension $n \leq 4$ is isomorphic to exactly one of $\ealg^{2}(\rea)$, $\ealg^{3}(\rea)$, $\ealg^{4}(\rea)$, or $\ealg^{2}(\rea)\oplus \ealg^{2}(\rea)$.
\end{theorem}
\begin{proof}
By Corollary \ref{n23corollary} and Lemma \ref{4dnotsimplelemma} it suffices to show that a simple $4$-dimensional Euclidean Killing metrized exact commutative algebra is isomorphic to $\ealg^{4}(\rea)$. This follows from \cite[Theorem $1.8$]{Fox-cubicpoly}, which gives normal forms for the cubic polynomial of a $4$-dimensional Euclidean Killing metrized exact commutative algebra.
\end{proof}
\begin{remark}
It is desirable to find a purely algebraic proof that a simple $4$-dimensional Euclidean Killing metrized exact commutative algebra is isomorphic to $\ealg^{4}(\rea)$.
\end{remark}

A consequence of Theorem \ref{4dclassificationtheorem} is that for a Euclidean Killing metrized exact commutative algebra to have a positive-dimensional automorphism group it is necessary that its dimension be at least $5$. The examples of Section \ref{jordansection} show that this does occur. 

It is desirable to have a direct proof that positive-dimensional automorphism group implies the dimension of the algebra is at least $5$. Corollary \ref{5dimcorollary} is a result in this direction. It shows that the automorphism group of a Euclidean Killing metrized exact commutative algebra is finite when $n \leq 4$, provided that the algebra is spanned by its minimal idempotents. 

A \emph{derivation} of $(\alg, \mlt)$ is an endomorphism $D \in \eno(\alg)$ satisfying $[D, L_{\mlt}(x)] = L_{\mlt}(Dx)$ for all $x \in \alg$. The derivations of $(\alg, \mlt)$ constitute a Lie subalgebra $\der(\alg, \mlt) \subset \eno(\alg)$ with the usual commutator bracket. A derivation of $(\alg, \mlt)$ need not be $h$-anti-self-adjoint, although a derivation of a Killing metrized commutative algebra is necessarily $\tau_{\mlt}$-anti-self-adjoint.

\begin{lemma}\label{derivationlemma}
Let $(\alg, \mlt, h)$ be a metrized commutative algebra over a field of characteristic not equal to $2$. If $e \in \idem(\alg, \mlt)$ and $D \in \der(\alg, \mlt)$, then $De \in \alg^{(1/2)}(e) \subset \eperp$. In particular, if $1/2 \notin \spec(e)$, then $De = 0$ and $D$ preserves the eigenspaces of $L_{\mlt}(e)$.
\end{lemma}
\begin{proof}
Because $De = D(e\mlt e) = 2e\mlt De$, $De \in \alg^{(1/2)}(e) \subset \eperp$. If $1/2 \notin \spec(e)$, then $De = 0$, so $[D, L_{\mlt}(e)] = L_{\mlt}(De) = 0$, and $D$ preserves the eigenspaces of $L_{\mlt}(e)$.
\end{proof}

\begin{lemma}\label{derivationsizelemma}
Let $\der(\alg, \mlt)$ be the Lie algebra of algebra derivations of an $n$-dimensional Euclidean metrized exact commutative algebra $(\alg, \mlt, h)$.  If there is $e \in \midem(\alg, \mlt, h)$ such that $k = \dim\{D(e): D \in \der(\alg, \mlt)\} > 0$, then $n \geq 2k + 3$.
\end{lemma}

\begin{proof}
By hypothesis there are $e \in \midem(\alg, \mlt, h)$ and $D_{1}, \dots, D_{k} \in \der(\alg, \mlt)$ such that $D_{1}(e), \dots, D_{k}(e)$ are linearly independent. 
By Lemma \ref{derivationlemma}, $D_{1}(e), \dots, D_{k}(e) \in \alg^{(1/2)}(e)$, so $1/2$ has multiplicity at least $k$ as an eigenvalue of $L_{\mlt}(e)$. Let $\la_{1}, \dots,\la_{n-1 -k}$ be the remaining (not necessarily distinct) elements of $\spec(e)$. Because $(\alg, \mlt)$ is exact, $\sum_{i = 1}^{n-1-k}\la_{i} = -(k+2)/2$. Since $e$ is minimal, by \eqref{localpmaxspeclemma} of Lemma \ref{criticalpointlemma}, $\la_{i} \leq 1/2$ for $1 \leq i \leq n-1 -k$. Hence, by the Schwarz inequality,
\begin{align}
\begin{split}
(k+2)^{2}= 4\left(\sum_{i = 1}^{n-k-1}\la_{i}\right)^{2} \leq 4(n-k-1)\sum_{i = 1}^{n-k-1}\la_{i}^{2} \leq (n-k-1)^{2}.
\end{split}
\end{align}
Simplifying this shows $2k + 3 \leq n$.
\end{proof}

\begin{corollary}
\label{5dimcorollary}
\leavevmode
\begin{enumerate}
\item If a Euclidean metrized exact commutative algebra $(\alg, \mlt, h)$ admits $D \in \der(\alg, \mlt)$, then either $D(\midem(\alg, \mlt, h)) = \{0\}$ or $\dim \alg \geq 5$. 
\item If a Euclidean Killing metrized exact commutative algebra $(\alg, \mlt)$ is spanned by its minimal idempotents and admits a nontrivial $D \in \der(\alg, \mlt)$, then $\dim \alg \geq 5$.
\end{enumerate}
\end{corollary}
\begin{proof}
The first claim is immediate from Lemma \ref{derivationsizelemma}. Suppose $(\alg, \mlt)$ is a Euclidean Killing metrized exact commutative algebra spanned by $\midem(\alg, \mlt, \tau_{\mlt})$ and $D \in \der(\alg, \mlt)$. If $D(\midem(\alg, \mlt, h)) = \{0\}$, then, for $e \in \midem(\alg, \mlt, \tau_{\mlt})$, $0 = \tau_{\mlt}(De, x) =-\tau_{\mlt}(e, Dx)$ for all $x \in \alg$, so $D(\alg) \subset \left(\spn \midem(\alg, \mlt, \tau_{\mlt})\right)^{\perp} = \alg^{\perp} =\{0\}$, showing $D$ is trivial. Consequently, if $D$ is nontrivial, then, by the first claim, $\dim \alg \geq 5$.  
\end{proof}

\section{Quantifying nonassociativity: sectional nonassociativity}\label{nasection}
By the Beltrami theorem  (see \cite[section $4$]{Chern-Griffiths}), the Levi-Civita connection of a Riemannian metric is projectively flat if and only if the metric has constant sectional curvature. 
Here there is introduced a notion of sectional nonassociativity, modeled on the sectional curvature, such that the formally analogous statement for Killing metrized exact commutative algebras holds.
This notion of sectional nonassociativity facilitates quantifications of nonassociativity that are explored in Section \ref{nortoninequalitysection}.

\begin{definition}\label{sectdefinition}
Let $(\alg, \mlt, h)$ be a metrized commutative algebra of dimension at least $2$. Define the \emph{sectional nonassociativity} $\sect(x, y)$ of the $h$-nondegenerate subspace $\spn\{x, y\} \subset \alg$ to be
\begin{align}\label{sectnadefined}
\begin{split}
\sect(x, y) & =  \sect_{\mlt, h}(x, y) = \frac{-\mu(x, y, x, y)}{|x|_{h}^{2}|y|_{h}^{2} - h(x, y)^{2}}= \frac{-h([y, x, x], y)}{|x|_{h}^{2}|y|_{h}^{2} - h(x, y)^{2}}= \frac{h([x, x, y], y)}{|x|_{h}^{2}|y|_{h}^{2} - h(x, y)^{2}}\\
&= \frac{h(x\mlt x , y\mlt y) - h(x\mlt y, y \mlt x)}{|x|_{h}^{2}|y|_{h}^{2} - h(x, y)^{2}},
\end{split}
\end{align}
where $|x|^{2}_{h} = h(x, x)$, $|x\wedge y|^{2}_{h} = 2(|x|_{h}^{2}|y|_{h}^{2} - h(x, y)^{2})$, and $\mu$ is as in \eqref{hmudefined}.
\end{definition}

\begin{lemma}\label{sectnalemma}
Let $(\alg, \mlt, h)$ be a metrized commutative algebra of dimension at least $2$.
The quantity $\sect(x, y)$ defined by \eqref{sectnadefined} for linearly independent $x, y \in \alg$ spanning an $h$-nondegenerate subspace depends only on the subspace spanned by $x$ and $y$.
\end{lemma}

\begin{proof}
If $\bar{x} = ax + by$ and $\bar{y} = cx + dy$ span the subspace spanned by $x$ and $y$ then 
\begin{align}\label{xychange}
\begin{split}
&h(\bar{x}\mlt \bar{x}, \bar{y}\mlt \bar{y}) - h(\bar{x}\mlt \bar{y}, \bar{x}\mlt \bar{y}) = (ad - bc)^{2}\left(h(x\mlt x , y\mlt y) - h(x\mlt y, y \mlt x)\right),\\
&h(\bar{x}, \bar{y})^{2} - h(\bar{x}, \bar{x})h(\bar{y}, \bar{y}) = (ad - bc)^{2}\left(h(x, y)^{2} - h(x, x)h(y, y)\right),
\end{split}
\end{align}
from which it is apparent that $\sect(\bar{x}, \bar{y}) = \sect(x, y)$, so depends only on the span of $x$ and $y$.
\end{proof}

The dependence of $\sect =  \sect_{\mlt, h}$ on $\mlt$ and $h$ is indicated with subscripts when helpful. For $r, s \in \fiet$, $\sect_{s\mlt, rh} = s^{2}r^{-1}\sect_{\mlt, h}$, where $s\mlt$ means the multiplcation $sx\mlt y$. When $\mlt$ is scaled by $s$, $\tau_{\mlt}$ rescales by $s^{2}$, and it follows that $\sect_{\mlt, \tau_{\mlt}}$ is insensitive to such rescalings. This motivates Definition \ref{isectdefinition}.

\begin{definition}\label{isectdefinition}
The \emph{intrinsic sectional nonassociativity} of a two-dimensional $\tau_{\mlt}$-nondegenerate subspace of a Killing metrized commutative algebra $(\alg, \mlt)$ spanned by $x, y \in \alg$ is
\begin{align}\label{isectdefined}
\isect_{\mlt}(x, y) = \sect_{\mlt, \tau_{\mlt}}(x, y) = \frac{\tau_{\mlt}(x\mlt x , y\mlt y) - \tau_{\mlt}(x\mlt y, y \mlt x)}{\tau_{\mlt}(x, x)\tau_{\mlt}(y, y) - \tau_{\mlt}(x, y)^{2}}.
\end{align}
\end{definition}
The intrinsic sectional nonassociativity is an isomorpism invariant of Killing metrized commutative algebras.

A Euclidean metrized commutative algebra $(\alg,\mlt, h)$ has \emph{constant sectional nonassociativity $\ka$} if there is a constant $\ka \in \rea$ such that $\sect(x, y) = \ka$ for all subspaces $\spn\{x, y\} \subset \alg$. 

\begin{example}
The simplicial algebra $(\ealg^{n}(\rea), \mlt, \tau_{\mlt})$ has constant sectional nonassociativity $-1/(n-1)$.
\end{example}

\begin{lemma}\label{constantsectlemma}
A Euclidean metrized commutative algebra $(\alg, \mlt, h)$ of dimension at least $2$ has constant sectional nonassociativity $\ka$ if and only if
\begin{align}\label{csn}
[x, y, z] = \ka\left(h(x, y)z - h(y, z)x\right)
\end{align}
for all $x, y, z \in \alg$. In this case, $\ric_{\mlt} = \ka(\dim \alg - 1)h$. 
\end{lemma}

\begin{proof}
If there holds \eqref{csn}, then $\sect(x, y) = \ka$ by \eqref{sectnadefined}. Suppose a tensor $a_{ijkl}$ has the symmetries $a_{ijkl} = -a_{jikl} = - a_{ijlk}$ and $a_{[ijk]l} = 0$. If there holds $x^{i}y^{j}x^{k}y^{l}a_{ijkl} = 0$ for all $x, y \in \alg$, then 
\begin{align}\label{asyms}
0 = a_{ijkl} + a_{kjil} + a_{ilkj} + a_{klij} = 2a_{ijkl} - 2a_{jkil},
\end{align}
so that $a_{ijkl} = a_{jkil} = a_{kijl}$. Consequently $3a_{ijkl} = 3a_{ijkl} - 3a_{[ijk]l} = 2a_{ijkl} - a_{jkil} - a_{kijl} = 0$. That $(\alg, \mlt, h)$ has constant sectional nonassociativity $\ka$ means that $x^{i}y^{j}x^{k}y^{l}(2\ka h_{k[i}h_{j]l} - \mu_{ijkl}) =0$ for all $x, y \in \alg$. Taking $a_{ijkl} = \mu_{ijkl}$ in \eqref{asyms} yields $\mu_{ijkl} = 2\ka h_{k[i}h_{j]l}$, which implies \eqref{csn}. If there holds $\eqref{csn}$, then 
\begin{align}\label{constidentity}
L_{\mlt}(x\mlt y) - L_{\mlt}(x)L_{\mlt}(y) = \ka(h(x, y)\Id_{\alg} - x \tensor h(y, \dum))
\end{align}
and tracing \eqref{constidentity} shows $\ric_{\mlt} = \ka(\dim \alg - 1)h$. 
\end{proof}

More generallly, a metrized commutative algebra $(\alg, \mlt, h)$ over $\fie$, with $h$ possibly isotropic is defined to have constant sectional nonassociativity if there holds \eqref{csn}. This implies that that the sectional nonassociativity of any $h$-nondegenerate subspace is equal to $\ka$. 

\begin{corollary}\label{constantsectunitalizationcorollary}
A metrized commutative $\rea$-algebra $(\alg, \mlt, h)$ of dimension at least $2$ has constant sectional nonassociativity $\ka$ if and only if it is projectively associative with $\ric_{\mlt} = \ka(\dim \alg - 1)h$.
\end{corollary}
\begin{proof}
This is immediate from Lemmas \ref{thomaslemma} and \ref{constantsectlemma}.
\end{proof}

Lemma \ref{beltramilemma} is the analogue of the theorem of Beltrami. It yields that a metrized commutative algebra has zero sectional nonassociativity if and only if it is associative.
 	
\begin{lemma}\label{beltramilemma}
For an $n$-dimensional Euclidean metrized commutative algebra $(\alg, \mlt, h)$ the following are equivalent:
\begin{enumerate}
\item\label{beltrami1} It has constant sectional nonassociativity $c$.
\item\label{beltrami2} It is conformally associative and there is $c \in \rea$ such that $\ric_{\mlt} = (n-1)c h$.
\item\label{beltrami3} It is projectively and conformally associative.
\end{enumerate}
In particular:
\begin{enumerate}
\setcounter{enumi}{3}
\item\label{beltrami4} Such an algebra has vanishing sectional nonassociativity if and only if it is associative.
\item\label{beltrami5} Such an algebra is Killing metrized with Killing Einstein constant $\ka = -(n-1)c$ if and only if it is exact. In this case, if $\ka \neq 0$, it is isomorphic to $\ealg^{n}(\rea)$.
\end{enumerate}
\end{lemma}
\begin{proof}
Lemma \ref{prebeltramilemma} shows the equivalence of \eqref{beltrami2} and \eqref{beltrami3} and its proof shows that these conditions imply \eqref{beltrami1}.
By \eqref{csn} of Lemma \ref{constantsectlemma}, if $(\alg, \mlt, h)$ has constant sectional nonassociativity, then $\mu_{ijkl} = 2ch_{l[i}h_{j]k}$. This implies $[x, y, z] = -ch(y, z)x + ch(x, y)z$, so there holds $[x, y, z]=  -ch(y, z)x + ch(x, y)z$ and $\mlt$ is projectively associative. Also, in this case, $\ric_{ij} = \ricc_{ji} = 2\mu_{pij}\,^{p} = (n-1)ch_{ij}$ and $\si = n(n-1)c$. In \eqref{confnon} there results $\om_{ijkl} = 0$. This shows \eqref{beltrami1} implies \eqref{beltrami2} and \eqref{beltrami3}.  
Claim \eqref{beltrami4} follows from \eqref{csn} of Lemma \ref{constantsectlemma}. 

If an algebra satisfies \eqref{beltrami1}-\eqref{beltrami3} with $\ka \neq 0$, then it is Ricci metrized by Lemma \ref{thomaslemma}. If it is moreover exact, this implies it is Killing metrized. If it is also Killing metrized, then the invariance of both the Ricci and Killing forms implies $0 = \tr L_{\mlt}([x, y, z])$ for all $x, y, z \in \alg$. By \eqref{csn} this implies $h(x, y)\tr L_{\mlt}(x) = h(y, z)\tr L_{\mlt}(x)$ for all $x, y, z \in \alg$. Choosing $y = x$ and $z$ such that $h(x, z) = 0$, there results $|x|^{2}_{h}\tr L_{\mlt}(z) = 0$. This shows $(\alg, \mlt)$ is exact and proves the first part of claim \eqref{beltrami5}. That in this case, if $\ka \neq 0$, the algebra is isomorphic to $\ealg^{n}(\rea)$ follows from Theorem \ref{confassclassificationtheorem}.
\end{proof}

Lemma \ref{deunitalizationsectlemma} shows that unitalization increases sectional nonassociativity.

\begin{lemma}\label{deunitalizationsectlemma}
For the deunitalization $(\alg, \mlt, h)$ of a unital metrized commutative algebra $(\balg, \star, g)$ with nonistropic unit $e$, for all $x, y \in \alg$ with $h$-nondegenerate span there holds
\begin{align}\label{secasecb}
g(e, e)\sect_{(\balg, \star, g)}(x, y)  = \sect_{(\alg, \mlt, h)}(x, y) + 1.
\end{align}
\end{lemma}

\begin{proof}
This follows from a computation using the definition of the deunitalization and \eqref{sectnadefined}.
\end{proof}

\section{Bounds on sectional nonassociativity: generalizations of the Norton inequality}\label{nortoninequalitysection}
In Section \ref{characterizationsection} it was shown that certain quasi-associativity conditions modeled on notions of projective and conformal flatness for connections on vector bundles facilitate characterization of the simplicial algebras. This suggests further developing the analogy between the structure tensor of a metrized commutative algebra and a Levi-Civita connection with the idea of mimicking notions more refined than flatness, such as one-sided curvature bounds and curvature pinching.

Over $\fie = \rea$, a Euclidean metrized commutative algebra $(\alg, \mlt, h)$ has positive, nonnegative, zero, etc. sectional nonassociativity if the given qualifier is valid for every two-dimensional subspace of $\alg$. 
Lemma \ref{sectdirectsumlemma} shows that taking direct sums of Euclidean metrized commutative algebras preserves conditions such as nonpositive and nonnegative sectional nonassociativity.
\begin{lemma}\label{sectdirectsumlemma}
Let $(\alg, \mlt, h)$ be a Euclidean metrized commutative algebra and suppose that $\alg = \alg_{1} \oplus \alg_{2}$ is a decomposition into $h$-orthogonal ideals. 
\begin{enumerate}
\item\label{dssectmin} If $(\alg_{1}, \mlt, h)$ and $(\alg_{2}, \mlt, h)$ have sectional nonassociativities bounded below by constants $m_{1}$ and $m_{2}$, then $(\alg, \mlt, h)$ has sectional nonassociativity bounded below by $\min\{m_{1}, m_{2}, 0\}$.
\item\label{dssectmax} If $(\alg_{1}, \mlt, h)$ and $(\alg_{2}, \mlt, h)$ have sectional nonassociativities bounded above by constants $M_{1}$ and $M_{2}$, then $(\alg, \mlt, h)$ has sectional nonassociativity bounded above by $\max\{M_{1}, M_{2}, 0\}$.
\end{enumerate}
In particular, if $(\alg_{1}, \mlt, h)$ and $(\alg_{2}, \mlt, h)$ have nonpositive (resp. nonnegative) sectional nonassociativity, then $(\alg, \mlt, h)$ has nonpositive (resp. nonnegative) sectional nonassociativity.
\end{lemma}

\begin{proof}
For $x, y \in \alg$, write $|x|^{2} = |x|^{2}_{h}$ and $\lb x, y \ra = h(x, y)$. Observe that, for any $x_{1}, x_{2}, y_{1}, y_{2} \in \alg$,
\begin{align}\label{splitsect2}
\begin{split}
|x_{1}|^{2}&|y_{2}|^{2} + |x_{2}|^{2}|y_{1}|^{2} - 2\lb x_{1}, y_{1}\ra\lb x_{2}, y_{2}\ra \\
&= \tfrac{1}{2}\left||x_{1}|y_{2} - |y_{1}|x_{2}\right|^{2} +  \tfrac{1}{2}\left||x_{2}|y_{1} - |y_{2}|x_{1}\right|^{2} \\
&\quad + \left(|x_{1}||y_{1}| - \lb x_{1}, y_{1}\ra\right)\lb x_{2}, y_{2}\ra + \left(|x_{2}||y_{2}| - \lb x_{2}, y_{2}\ra\right)\lb x_{1}, y_{1}\ra \\
& \geq  \tfrac{1}{2}\left||x_{1}|y_{2} - |y_{1}|x_{2}\right|^{2} +  \tfrac{1}{2}\left||x_{2}|y_{1} - |y_{2}|x_{1}\right|^{2} \geq 0 .
\end{split}
\end{align}
Let $x_{i} \in \alg_{i}$ be the $h$-orthogonal projection of $x \in \alg$. Note that $\lb x_{1}\wedge y_{1}, x_{2} \wedge y_{2}\ra = |x_{1}|^{2}|y_{2}|^{2} + |x_{2}|^{2}|y_{1}|^{2} - 2\lb x_{1}, y_{1}\ra\lb x_{2}, y_{2}\ra$.
Since $x_{1}\mlt y_{2} = 0$ and $h(x_{1}, y_{2}) = 0$, for linearly independent $x, y \in \alg$,
\begin{align}\label{splitsect1}
\begin{aligned}
\sect_{\alg, h}(x, y) &= \frac{\sect_{\alg_{1}, h}(x_{1}, y_{1})|x_{1}\wedge y_{1}|^{2} + \sect_{\alg_{2}, h}(x_{2}, y_{2})|x_{2}\wedge y_{2}|^{2}}{|x_{1}\wedge y_{1}|^{2} + |x_{2}\wedge y_{2}|^{2} +2\lb x_{1}\wedge y_{1}, x_{2} \wedge y_{2}\ra }\\
& = \sect_{\alg_{1}, h}(x_{1}, y_{1})r_{1} + \sect_{\alg_{2}, h}(x_{2}, y_{2})r_{2},\\
r_{i} &= \frac{|x_{i}\wedge y_{i}|^{2} }{|x_{1}\wedge y_{1}|^{2} + |x_{2}\wedge y_{2}|^{2} +2\lb x_{1}\wedge y_{1}, x_{2} \wedge y_{2}\ra } \in (0, 1] ;  i = 1,2.
\end{aligned}
\end{align}
Hence, by \eqref{splitsect2},
\begin{align}
m_{1}r_{1} + m_{2}r_{2} \leq \sect_{\alg, h}(x, y)= \sect_{\alg_{1}, h}(x_{1}, y_{1})r_{1} + \sect_{\alg_{2}, h}(x_{2}, y_{2})r_{2} \leq M_{1}r_{1} + M_{2}r_{2}
\end{align}
Using the inequalities $\min\{a, b, c\} \leq ar_{1} + br_{2} + c(1 -r_{1} - r_{2}) \leq \max\{a, b, c\}$ for $r_{1}, r_{2} \in [0, 1]$, with $a = m_{1}$ and $b = m_{2}$ or $a = M_{1}$ and $b = M_{2}$, yields \eqref{dssectmin} and \eqref{dssectmax}.
\end{proof}

\begin{example}\label{lietensorexample}
It is more difficult to relate the sectional nonassociativity of a tensor product to the sectional nonassociativities of its factors, but in some cases something can be said.

For a compact semisimple real Lie algebra $\g$, by the Cauchy-Schwarz inequality, the number
\begin{align}
\bw(\g) = \sup_{x, y \in \g: x \wedge y \neq 0}\frac{-B_{\g}([x, y], [x, y])}{B_{\g}(x, x)B_{\g}(y, y) - B_{g}(x, y)^{2}}
\end{align}
is positive and finite. Its value is a basic invariant of $(\g, [\dum, \dum])$. In some cases it can be estimated. This estimation requires the Böttcher-Wenzel inequality stated in Theorem \ref{bwtheorem}. 

\begin{lemma}\label{ghboundlemma}
For compact semisimple real Lie algebras $\g$ and $\h$, the sectional nonassociativity of the subspace of $(\g \tensor \h, \mlt, \tau_{\mlt})$ spanned by decomposable elements $a_{1}\tensor b_{1}, a_{2}\tensor b_{2} \in \g \tensor \h$ satisfies
\begin{align}\label{sectghbounds}
0 \geq \sect(a_{1}\tensor b_{1}, a_{2}\tensor b_{2}) \geq - \bw(\g)\bw(\h).
\end{align}
\end{lemma}

\begin{proof}
The upper bound in \eqref{sectghbounds} follow from 
\begin{align}\label{sectgh1}
\begin{split}
\sect(a_{1}\tensor b_{1}, a_{2}\tensor b_{2}) & = -\tfrac{B_{\g}([a_{1}, a_{2}], [a_{1}, a_{2}])B_{\h}([b_{1}, b_{2}], [b_{1}, b_{2}])}{B_{\g}(a_{1}, a_{1})
B_{\g}(a_{2}, a_{2})B_{\h}(b_{1}, b_{1})
B_{\h}(b_{2}, b_{2}) - B_{\g}(a_{1}, a_{2})^{2}B_{\h}(b_{1}, b_{2})^{2}}\leq 0.
\end{split}
\end{align}
Because $\g$ and $\h$ are compact, $-B_{\g}([a_{1}, a_{2}], [a_{1}, a_{2}]) \geq 0$ and $-B_{\h}([b_{1}, b_{2}], [b_{1}, b_{2}])\geq 0$, so, by definition of $\bw(\g)$ and $\bw(\h)$,
\begin{align}
\begin{split}
-&\tfrac{\sect(a_{1}\tensor b_{1}, a_{2}\tensor b_{2}) }{\bw(\g)\bw(\h)} = \tfrac{1}{\bw(\g)\bw(\h)}\tfrac{\left(-B_{\g}([a_{1}, a_{2}], [a_{1}, a_{2}])\right)\left(-B_{\h}([b_{1}, b_{2}], [b_{1}, b_{2}])\right)}{B_{\g}(a_{1}, a_{1})
B_{\g}(a_{2}, a_{2})B_{\h}(b_{1}, b_{1})B_{\h}(b_{2}, b_{2}) - B_{\g}(a_{1}, a_{2})^{2}B_{\h}(b_{1}, b_{2})^{2}}\\
& \leq \tfrac{\left(B_{\g}(a_{1}, a_{1})B_{\g}(a_{2}, a_{2}) - B_{\g}(a_{1}, a_{2})^{2} \right)\left(B_{\h}(b_{1}, b_{1})B_{\h}(b_{2}, b_{2}) - B_{\h}(b_{1}, b_{2})^{2} \right)}{B_{\g}(a_{1}, a_{1})B_{\g}(a_{2}, a_{2})B_{\h}(b_{1}, b_{1})B_{\h}(b_{2}, b_{2}) - B_{\g}(a_{1}, a_{2})^{2}B_{\h}(b_{1}, b_{2})^{2}}\\
& =1 - \tfrac{B_{\g}(a_{1}, a_{2})^{2}\left(B_{\h}(b_{1}, b_{1})B_{\h}(b_{2}, b_{2}) - B_{\h}(b_{1}, b_{2})^{2} \right) + B_{\h}(b_{1}, b_{2})^{2}\left(B_{\g}(a_{1}, a_{1})B_{\g}(a_{2}, a_{2}) - B_{\g}(a_{1}, a_{2})^{2} \right)}{B_{\g}(a_{1}, a_{1})B_{\g}(a_{2}, a_{2})B_{\h}(b_{1}, b_{1})B_{\h}(b_{2}, b_{2}) - B_{\g}(a_{1}, a_{2})^{2}B_{\h}(b_{1}, b_{2})^{2}}\leq 1,
\end{split}
\end{align}
which shows the lower bound in \eqref{sectghbounds}.
\end{proof} 
For example, for $\su(n)$ and $\so(n)$, $B_{\su(n)}(X, Y) =2nf(X, Y)$ and $B_{\so(n)}(X, Y) = (n-2)f(X, Y)$ where $f(X, Y) = \tr \bar{X}^{t}Y$, so it follows from the Böttcher-Wenzel inequality, Theorem \ref{bwtheorem}, that
\begin{align}
&\bw(\so(n)) \leq \tfrac{2}{n-2},& &\bw(\su(n)) \leq \tfrac{1}{n}.
\end{align}
For the special case of $\so(n)$, viewed as skew-symmetric matrices, equipped with the Frobenius norm $f$, the essentially equivalent quantity $ \sup_{x, y \in \g: x \wedge y \neq 0}\tfrac{|[X, Y]|_{f}}{|X|_{f}|Y|_{f}}$ was shown to equal $\sqrt{2}$ if $n \geq 4$ in \cite[Theorem $6$]{Bloch-Iserles}. By Lemma \ref{ghboundlemma} thee estimates yield sharp numerical bounds on the sectional nonassociativites of subspaces spanned by decomposable elements in algebras such as $\so(m)\tensor \so(n)$, $\so(m) \tensor \su(n)$, and $\su(m)\tensor \su(n)$.

Note that Lemma \ref{ghboundlemma} does not imply that $\g \tensor \h$ has nonpositive sectional nonassociativity. In fact, Theorem \ref{compactnahmtheorem} shows that, for any compact simple real Lie algebra $\g$, $\so(3)\tensor\g$ has sectional nonassociativities of both signs.
\end{example}

From \eqref{sectnadefined} it is apparent that a metrized commutative algebra $(\alg, \mlt, h)$ has nonnegative sectional nonassociativity if and only if for all $x, y \in \alg$ there holds
\begin{align}\label{norton}
0 \leq h([x, x, y], y) = h(x\mlt x, y \mlt y) - h(x \mlt y, x \mlt y).
\end{align}
The inequality \eqref{norton} is known as \emph{Norton's inequality}. Similarly, nonpositive sectional nonassociativity is equivalent to the \emph{reverse Norton's inequality}, which is \eqref{norton} with the inequality reversed.
From the point of view taken here, Norton's inequality is analogous to nonpositive sectional curvature. The notion of sectional nonassociativity facilitates quantifications of nonassociativity more refined than Norton's inequality.

For a different perspective on the importance of the Norton inequality, see the discussion of Majorana involutions and axial vectors in \cite[Chapter $8.6$]{Ivanov}.

\begin{example}\label{monsternortonexample}
A fundamental fact about the $196884$-dimensional Conway-Griess algebra described in Example \ref{monsterexample}, established in \cite[section $16$]{Conway-monster}, is that it satisfies Norton's inequality \eqref{norton}. Hence it has nonnegative sectional nonassociativity, and by \eqref{secasecb} it follows that the $196883$-dimensional nonunital Griess algebra $(\alg, \star, h)$ has sectional nonassociativity greater than or equal to $-1$.
\end{example}

\begin{example}\label{voanortonexample}
By \cite[Theorem $6.3$]{Miyamoto-griessalgebras}, if an OZ VOA defined over $\rea$ has a positive definite bilinear form, then the Norton inequality holds on its Griess algebra. In the language used here this means that such a Griess algebra has nonnegative sectional nonassociativity. In particular this means that the (unital) Griess algebra of the moonshine VOA has nonnegative sectional nonassociativity.
\end{example}

\begin{lemma}\label{complexidempotentlemma}
Let $(\alg, \mlt, h)$ be a Euclidean metrized commutative algebra.
\begin{enumerate}
\item\label{complexzeroclaim} If $a + \j b \in \szero(\alg\tensor_{\rea}\com, \mlt)$, then
\begin{align}\label{complexszerosect}
\sect(a, b) = \tfrac{|a \mlt a|^{2}_{h}}{|a|_{h}^{2}|b|^{2}_{h} - h(a, b)^{2}} \geq 0,
\end{align}
with equality if and only if $a$ and $b$ span a trivial subalgebra of $(\alg, \mlt)$.
\item If $a + \j b \in \idem(\alg\tensor_{\rea}\com, \mlt)$, then
\begin{align}\label{complexidemsect}
\sect(a, b) = \tfrac{|b \mlt b|^{2}_{h} + |a \mlt b|^{2}_{h}}{|a|_{h}^{2}|b|^{2}_{h} - h(a, b)^{2}} \geq 0,
\end{align}
with equality if and only if $b = 0$ and $a\in \idem(\alg, \mlt)$.
\end{enumerate}
\end{lemma}
\begin{proof}
That $a + \j b \in \szero(\alg\tensor_{\rea}\com, \mlt)$ is equivalent to the equations $a\mlt a = b \mlt b$ and $a \mlt b  = 0$, and in \eqref{sectnadefined} these yield \eqref{complexszerosect}. If equality holds in \eqref{complexszerosect}, then $b\mlt b = a \mlt a = 0$, so $a$ and $b$ generate a trivial subalgebra. 
Suppose $a + \j b \in \idem(\alg\tensor_{\rea}\com, \mlt)$, so that $a\mlt a - b \mlt b = a$ and $2 a \mlt b = b$. Then 
\begin{align}
\begin{split}
h(a\mlt a, b\mlt b) & = h(b \mlt b + a, b\mlt b) = |b \mlt b|^{2}_{h} + h(a\mlt b, b) = |b\mlt b|^{2}_{h} + 2|a \mlt b|^{2}_{h},
\end{split}
\end{align}
so 
\begin{align}
\begin{split}
\sect(a, b) & =  \tfrac{h(a\mlt a , b\mlt b) -  |a \mlt b|^{2}_{h}}{|a|_{h}^{2}|b|^{2}_{h} - h(a, b)^{2}} = \tfrac{|b \mlt b|^{2}_{h} + |a \mlt b|^{2}_{h}}{|a|_{h}^{2}|b|^{2}_{h} - h(a, b)^{2}} \geq 0.
\end{split}
\end{align}
Equality holds if and only if $a \mlt b = 0$ and $b \mlt b = 0$, which imply that $b = 0$ and $a \mlt a = a$.
\end{proof}

\begin{corollary}\label{nocscorollary}
In a Euclidean metrized commutative algebra $(\alg, \mlt, h)$ with nonpositive sectional nonassociativity, any element of $\idem(\alg\tensor_{\rea}\com, \mlt)$ or $\szero(\alg\tensor_{\rea}\com, \mlt)$ is a multiple of an element of $\idem(\alg, \mlt)$ or $\szero(\alg, \mlt)$.
\end{corollary}

\begin{lemma}
If in a Euclidean metrized commutative algebra $(\alg, \mlt, h)$ there is a $C^{1}$ immersion $x:(-a, a) \to \alg$ such that $x(t)\in \idem(\alg, \mlt)$ for all $t \in (-a, a)$ and $\dot{x}(0) \neq 0$, then $1/2 \in \spec(x(0))$ and 
\begin{align}\label{smoothpathsect}
\sect(x(0), \dot{x}(0)) = \tfrac{1}{4|x(0)|^{2}_{h}} > 0.
\end{align}
\end{lemma}
\begin{proof}
Differentiating $x\mlt x = x$ shows that $\dot{x}(0)$ is an eigenvector of $L_{\mlt}(x(0))$ with eigenvalue $1/2$. As $h(x, \dot{x}) = 2h(x, x\mlt \dot{x}) = 2h(x\mlt x, \dot{x}) = 2h(x, \dot{x})$, there holds $h(x, \dot{x}) = 0$, so $x(0)$ and $\dot{x}(0)$ are linearly independent. The identity \eqref{smoothpathsect} follows from \eqref{sectnadefined}.
\end{proof}

\begin{lemma}\label{finiteautolemma}
Let  $(\alg, \mlt, h)$ be a Euclidean metrized commutative algebra $(\alg, \mlt, h)$.
\begin{enumerate}
\item If $(\alg, \mlt, h)$ has negative sectional nonassociativity, then $\Aut(\alg, \mlt, h)$ is finite.
\item If $(\alg, \mlt, h)$ has nonpositive sectional nonassociativity, either $\Aut(\alg, \mlt, h)$ is finite or $(\alg, \mlt, h)$ contains a trivial two-dimensional subalgebra (the possibilities are not exclusive).
\end{enumerate}
\end{lemma}
\begin{proof}
For $k \geq 3$, let $\theta$ be a multilinear $k$-form on a vector space $\ste$ over an algebraically closed field of characteristic zero. By a theorem of H. Suzuki \cite[Theorem B]{Suzuki-automorphismgroups}, either the group of linear automorphisms of $\theta$ is finite or there is a nonzero vector $v\in \ste$ such that $\theta(v, \dots, v, w) = 0$ for all $w\in \ste$. Regard $\mu(x, y, z) = h(x\mlt y, z)$ as a cubic form on the complexification $\ste = \alg \tensor_{\rea} \com$. An orthogonal automorphism of $(\alg, \mlt, h)$ preserves $\mu$, and extends to an automorphism of the complexified algebra. By Suzuki's theorem if $\Aut(\alg, \mlt, h)$ is not finite there exists $a + \j b \in  \alg \tensor_{\rea} \com$ such that $h((a + \j b)\mlt(a + \j b), w) = 0$ for all $w \in \alg \tensor_{\rea} \com$. By the nondegeneracy of $h$, $a + \j b \in \szero(\alg\tensor_{\rea}\com, \mlt)$.  The conclusion follows from \eqref{complexzeroclaim} of Lemma \ref{complexidempotentlemma}. 
\end{proof}

\begin{example}
Let $\g$ and $\h$ be Lie algebras over a field $\fie$ of characteristic zero. Let $\cl \subset \g$ be an abelian subalgebra and let $\kl \subset \h$ be any subalgebra. Then $\cl \tensor_{\fie} \kl \subset \g \tensor_{\fie} \h$ is a trivial subalgebra. In particular, if $\g$ and $\h$ are semisimple Lie algebras over a field of characteristic zero and $\cl$ is a Cartan subalgebra, then $\cl \tensor_{\fie}\h$ is a large trivial subalgebra of the Killing metrized exact commutative algebra $\g \tensor_{\fie}\h$.
\end{example}

This section concludes with some results showing that nonnegative sectional nonassociativity prevents the existence of square-zero elements. Example \ref{largeszeroexample} shows that some condition is necessary for such a conclusion.

\begin{example}\label{largeszeroexample}
For a Killing metrized exact commutative algebra, $\szero(\alg, \mlt)$ can be large.
Let $\g$ and $\h$ be Lie algebras over a field $\fie$ of characteristic zero. In $\g \tensor_{\fie}\h$, every decomposable element is square-zero. More generally, for any $a \in \g$, $\fie\lb a \ra \tensor_{\fie}\h$ is a trivial subalgebra. 
\end{example}

\begin{lemma}\label{rchassoclemma}
Let $(\alg,\mlt, h)$ be a Euclidean metrized commutative algebra. Let $0 \neq z \in \szero(\alg, \mlt)$.
 For $0 \neq y \in \alg$ not in the span of $z$, $\sect(z, y) \leq 0$, with equality if and only if $z \mlt y = 0$. In particular, $\szero(\alg, \mlt) = \{0\}$ if there holds any of the following conditions:
\begin{enumerate}
\item $(\alg, \mlt, h)$ has positive sectional nonassociativity.
\item\label{nsz2} $(\alg, \mlt, h)$ has nonnegative sectional nonassociativity and is semisimple.
\item\label{nsz3} $(\alg, \mlt, h)$ has nonnegative sectional nonassociativity and positive definite Killing form.
\end{enumerate}
\end{lemma}

\begin{proof}
Let $0 \neq z \in \szero(\alg, \mlt)$. For all $y \in \alg$ not in $\spn\{z\}$, there holds $\sect(z, y)(|z|^{2}_{h}|y|^{2}_{h} - h(z, y)^{2}) = - |L_{\mlt}(z)y|^{2}_{h} \leq 0$, with equality if and only if $z \mlt y = 0$. In particular, $(\alg, \mlt,h)$ cannot have positive sectional nonassociativity, and if $(\alg, \mlt,h)$ has nonnegative sectional nonassociativity, then $L_{\mlt}(z)y = 0$ for all $y \in \alg$, so $L_{\mlt}(z) =0$. 
By Lemma \ref{faithfulmultiplicationlemma}, if $L_{\mlt}(z) = 0$, then $(\alg, \mlt)$ is not semisimple. Similarly, $\tau_{\mlt}$ is not positive definite, for $L_{\mlt}(z) =0$ implies $\tau_{\mlt}(z, z) = \tr L_{\mlt}(z)^{2} = 0$.
\end{proof}

\begin{corollary}
A Euclidean Killing metrized commutative algebra $(\alg, \mlt)$ with nonnegative sectional nonassociativity contains no nontrivial square-zero elements.
\end{corollary}

\begin{proof}
By Lemma \ref{dieudonnelemma}, $(\alg, \mlt)$ is semisimple, so the claim follows from \eqref{nsz2} or \eqref{nsz3} of Lemma \ref{rchassoclemma}. 
\end{proof}

\section{Deunitalizations of simple Euclidean Jordan algebras}\label{jordansection}
An $\fie$-algebra $(\alg, \mlt)$ is \emph{Jordan} if it is commutative and satisfies $[x\mlt x, y, x]  = 0$ for all $x, y \in \alg$. 
For the background on Jordan algebras used here and in what follows see \cite{Ash-Mumford-Rapoport-Tai, Faraut-Koranyi, Jacobson-jordan, Schafer, Springer-Veldkamp}. 
 
A Jordan algebra is semisimple if and only if the invariant symmetric bilinear form $ \tr L_{\mlt}(x \mlt y)$ is nondegenerate \cite[section IV.$1$]{Schafer}, in which case the Jordan algebra is necessarily unital. 

\begin{lemma}\label{notjordanlemma}
Let $\chr \fie = 0$ and let $(\alg, \mlt, h)$ be a metrized commutative $\fie$-algebra with unitalization $(\halg, \hmlt, \hat{h})$.
\begin{enumerate}
\item\label{unitaljordanclaim} The algebra $(\halg, \hmlt)$ is Jordan if and only if for all $x, y \in \alg$ there holds
\begin{align}\label{nj}
[x\mlt x, y, x] = h(x, y)x\mlt x - h(x\mlt x, y)x.
\end{align}
\item\label{unitalpaclaim} The algebra $(\halg, \hmlt)$ is power associative if and only if for all $x \in \alg$ there holds
\begin{align}\label{npa}
&[x\mlt x, x, x] = h(x, x)x\mlt x - h(x\mlt x, x)x.
\end{align}
\item\label{unitalizationssclaim} If $(\alg, \mlt, h)$ satisfies \eqref{nj}, its unitalization is a semisimple Jordan algebra.
\end{enumerate}
\end{lemma}

\begin{proof}
By \eqref{unitalassoc}, 
\begin{align}\label{unitaljordan}
\begin{split}
[(x, \al)\hmlt(x,\al), (y, \be), (x, \al)] &= ([x\mlt x, y, x] + 2\al[x, y, x] + h(x\mlt x, y)x - h(x, y)x\mlt x, 0)\\
&= ([x\mlt x, y, x] + h(x\mlt x, y)x - h(x, y)x\mlt x, 0),
\end{split}
\end{align}
the second equality because $[x, y, x] = 0$ in any commutative algebra. Claim \eqref{unitaljordanclaim} follows. 

An algebra $(\alg, \mlt)$ over a field of characteristic zero is power associative if and only if $[x, x, x] = 0$ (which is automatic for a commutative algebra) and $[x\mlt x, x, x] = 0$ for all $x \in \alg$. Since $\chr \fie = 0$ is assumed, taking $z = y = x$ in \eqref{unitalassoc} and $y = x$ in \eqref{unitaljordan} shows that $(\halg, \mlt)$ is power associative if and only if there holds \eqref{npa}. 

If $(\alg, \mlt)$ is exact, it follows from \eqref{unitalizationtrace} that $\tr L_{\hmlt}((x, r) \hmlt(y,s)) = (n+1)\hat{h}((x, r),(y, s))$, so that, in this case, $\tr L_{\hmlt}(\dum \hmlt \dum) = (n+1)\hat{h}$ is nondegenerate on $(\halg, \hmlt)$, so $(\halg, \hmlt)$ is semisimple. 
\end{proof}

\begin{remark}
Claim \eqref{unitaljordanclaim} of Lemma \ref{notjordanlemma} is valid over any field, but claims \eqref{unitalpaclaim} and \eqref{unitalizationssclaim} are not. Since a commutative algebra over a field of characteristic not dividing $30$ is power associative if and only if $[x \mlt x, x, x] = 0$ \cite[Chapter $5$]{Schafer}, claim \eqref{unitalpaclaim} is true in this case, but can fail if $\chr \fie$ divides $30$. Claim \eqref{unitalizationssclaim} fails because the characterization of semisimplicity in terms of the nondegeneracy of the form $\tr L(\dum \hmlt \dum)$ fails. 
\end{remark}

\begin{remark}
If $\chr \fie = 0$, claims \ref{unitaljordanclaim} and \ref{unitalpaclaim} of Lemma \ref{jordanminimalsectlemma} are redundant.
A Jordan algebra is power associative. 
Conversely, by the main theorem of \cite{Kokoris}, combined with results from \cite{Albert-powerassociative}, a simple commutative power associative algebra over a field of characteristic zero is Jordan. 
\end{remark}

A Jordan algebra is \emph{Euclidean} if it is real and admits an invariant positive definite symmetric bilinear form. 
The rank of an element $x$ in a real or complex Jordan algebra with unit $e$ is the smallest integer $k$ such that $\{e, x, x^{2}, \dots, x^{k}\}$ is linearly dependent (this uses that a Jordan algebra is power associative), and the rank $r$ of the Jordan algebra is the maximum of the ranks of its elements. In a unital Jordan algebra, a set of primitive orthogonal idempotents is \emph{complete} if it sums to the unit.
For a simple Euclidean Jordan algebra, the rank $r$ is the number of elements in a complete system of primitive orthogonal idempotents $\{e_{1}, \dots, e_{r}\}$, and the \emph{Peirce invariant} $d$ is the common dimension of the image of $L_{\mlt}(e_{i})L_{\mlt}(e_{j})$ for $i \neq j$ (equivalently, it is the dimension of the intersection of the $1/2$-eigenspaces of any two orthogonal primitive idempotents; see \cite[chapter $V$]{Faraut-Koranyi}). It is related to the rank and the dimension $N$ by the equivalent formulas (see \cite[p. $71$]{Faraut-Koranyi}) 
\begin{align}
&N = r + \tfrac{dr(r-1)}{2},& &d = \tfrac{2(N-r)}{r(r-1)}.
\end{align}
Every simple Euclidean Jordan algebra over the real numbers having rank at least two is isomorphic to one of the Jordan algebras listed in table \ref{jordantable}. The algebra $(\rea^{n}\oplus \rea, \star)$ is the unitalization of the trivial algebra $(\rea^{n}, \mlt)$ with respect to the standard Euclidean metric. It is often called a \emph{spin factor}.

A unital composition algebra is also called a \emph{Hurwitz algebra}.
Let $\mat(n, \fie)$ denote the $n \times n$ matrices over the real Hurwitz algebra $\fie \in \{\rea, \com, \quat, \cayley\}$. The subspace of $n \times n$ Hermitian matrices over $\fie$ is the subspace $\herm(n, \fie) = \{x \in \mat(n, \fie): \bar{x}^{t} = x\}$ comprising the fixed points of the conjugate transpose $x \to \bar{x}^{t}$, where $x \to \bar{x}$ is the involution of $\fie$ determined by its structure as a composition algebra and fixing the real subfield (so is trivial when $\fie = \rea$). When $\fie = \cayley$ there is always supposed $n = 3$. The space $\herm(n, \fie)$ is a Jordan algebra when equipped with the symmetrized matrix product $x \star y = \tfrac{1}{2}(xy + yx)$, where juxtaposition, as in the expression $xy$, indicates the matrix product in $\herm(n, \rea)$. 

The following facts are used in what follows. 
\begin{enumerate}
\item Because a Jordan algebra is power associative, the trace $\tr x$ and determinant $\det x$ of an element $x \in \balg$ can be defined as the coefficients of the subleading term and the zeroth order term in the minimal polynomial of $x$. The trace and determinant of a real Jordan algebra $(\balg, \star)$ are invariant under $\Aut(\balg, \star)$. See \cite[Proposition II.$4.2$]{Faraut-Koranyi}.
\item For an element $x$ of a rank $r$ Euclidean Jordan algebra $(\balg, \star)$, there is a complete system of orthogonal idempotents, $\{e_{1}, \dots, e_{r}\}$, such that $x \in \spn\{e_{1}, \dots, e_{r}\}$. See \cite[Theorem III.$1.1$]{Faraut-Koranyi}.
\item Given two complete systems of orthogonal idempotents, $\{e_{1}, \dots, e_{r}\}$ and $\{f_{1}, \dots, f_{r}\}$, in a simple Euclidean Jordan algebra $(\balg, \star)$ of rank $r$, there is $g \in \Aut(\balg, \star)$ such that $ge_{i} = f_{i}$ for $1 \leq i \leq r$. See \cite[Theorem IV.$2.5$]{Faraut-Koranyi}.
\item (Principal axis theorem) Every element of $\herm(n, \fie)$ (where $n = 3$ if $\fie = \cayley$) is equivalent via an element of $\Aut(\herm(n, \fie), \star)$ to a diagonal matrix. For $\fie \in \{\rea, \com, \quat\}$ this is well known. For $\fie = \cayley$ this is \cite[Theorem $5.1$]{Freudenthal}; see also \cite[Theorem V.$2.5$]{Faraut-Koranyi}.
\end{enumerate}

\begin{table}[!htbp]
\caption{Simple Euclidean Jordan algebras}\label{jordantable}
\def\arraystretch{1.5}
\begin{tabular}{|l|c|c|c|c|}
\hline
Description & Notation & Dimension & Rank & $d$ \\
\hline\hline
Euclidean spin factor&$\rea^{n} \oplus \rea$ & $n+1$ & $2$ & $n-1$\\\hline
Real symmetric matrices & $\herm(n, \rea)$ & $n(n+1)/2$ & $n$ & $1$ \\\hline
Complex Hermitian matrices &$\herm(n, \com)$ & $n^{2}$ & $n$ & $2$  \\\hline
Quaternionic Hermitian matrices & $\herm(n, \quat)$ &$2n^{2} - n$& $n$ & $4$\\\hline
Octonionic Hermitian matrices & $\herm(3, \cayley)$ & $27$ & $3$ & $8$ \\\hline
\end{tabular}
\end{table}

If a Jordan algebra $(\balg, \star)$ is simple, then any invariant metric is a multiple of the invariant metric $g(x, y) = \tr L_{\star}(x \star y)$. 
Thus the deunitalization $(\alg, \mlt, h)$ of a simple unital Euclidean Jordan algebra $(\balg, \star)$ means unambiguously its deunitalization with respect to $g$. 
For a simple Euclidean Jordan algebra $(\balg, \star)$ of rank $r \geq 2$ and dimension $N$, the unit satisfies $\tr L_{\star}(e) = N$ and the deunitalization of $(\balg, \star)$ is defined using $\hat{h}(x, y) = N^{-1}\tr L_{\star}(x\star y)= r^{-1}\tr(x\star y)$. The deunitalization of $\herm(n, \fie)$ is the subspace $\herm_{0}(n, \fie) = \{x \in \herm(n, \fie): \tr x = 0\}$ equipped with the product $\mlt$, invariant metric $h$, and cubic polynomial $P$ given by
\begin{align}\label{jrank}
\begin{split}
x\mlt y &= x \star y - n^{-1}\tr(x\star y)I = \tfrac{1}{2}\left(xy + yx - n^{-1}\tr(xy + yx)I\right),\\ 
h(x, y) &= \tfrac{1}{n}\tr(x \star y) = \tfrac{1}{2n}\tr(xy + yx)= \tfrac{1}{n}\re \tr (xy),\\
6P(x) &= \tfrac{1}{2n}\tr((x^{2})x + x(x^{2})) = \tfrac{1}{n}\re \tr((x^{2})x), 
\end{split}
\end{align}
where juxtaposition, as in $xy$ or $x(x^{2})$, indicates the matrix product in $\herm(n, \rea)$, and $I$ is the identity matrix in $\herm(n, \rea)$. It is necessary to write $x(x^{2}) + (x^{2})x$ and not $2x^{3}$ because the matrix multiplication in $\herm(3, \cayley)$ is not associative, although it is true that $x(x^{2}) - (x^{2})x$ is a multiple of the identity by an element of $\cayley$ \cite[Lemma V$.2.4$]{Faraut-Koranyi}. Note also that $\re \tr(xy) = \re \tr(yx)$ and $\re \tr((xy)z) = \re \tr(x (yz))$ are true for $n\times n$ matrices over $\fie$ \cite[Proposition V.$2.1$]{Faraut-Koranyi}.

\begin{remark}
In \cite[section $3$]{Okubo}, the simple Euclidean Jordan algebras are constructed as the unitalizations of the nonunital algebras $\herm_{0}(n, \fie)$ equipped with the multiplication $\mlt$ and metric $h$ defined in \eqref{jrank} (see Remark $3.3$ and equations $3.27$a and $3.27$b in Remark $3.4$ of \cite{Okubo}).
\end{remark}

\begin{remark}
The deunitalizations $\herm_{0}(n, \fie)$ are studied in terms quite similar to those here in \cite{Tkachev-jordan} and in \cite[section $6$]{Nadirashvili-Tkachev-Vladuts} where they are called \emph{special REC algebras}. In particular, the computation of the Einstein constants is made, albeit not in this language, in \cite[equation $(6.4.15)$]{Nadirashvili-Tkachev-Vladuts}, and this justifies the attribution of Theorem \ref{jordandeunitalizationtheorem} to these authors.
\end{remark}

\begin{theorem}[{\cite[Section $6$]{Nadirashvili-Tkachev-Vladuts}}]\label{jordandeunitalizationtheorem}
The deunitalization $(\alg, \mlt, h)$ of a simple Euclidean Jordan algebra $(\balg, \star)$ of rank $r \geq 2$ and dimension $N$ is an exact Killing invariant commutative algebra with Killing form $\tau_{\mlt}$ satisfying $\tau_{\mlt} = \ka h$ where $h$ is as in \eqref{jrank} and $\ka = (r - 2)(1 + \tfrac{r(N-r)}{2r(r-1)})$. In particular, if $r > 2$, then $(\herm_{0}(r, \fie), \star, h)$ is an exact Killing metrized commutative algebra with $h = \ka\tau_{\mlt}$.
\end{theorem}

\begin{proof}
Let $(\balg, \star)$ be a simple Euclidean Jordan algebra of rank $r \geq 2$ and dimension $N$. It is claimed that
\begin{align}\label{fk0}
\begin{split}
\tau_{\star}(x, y) &  = (1 +\tfrac{d(r-2)}{4})\tr(x\star y) + \tfrac{d}{4}\tr(x)\tr(y)\\
& = \tfrac{r}{N}\left(1 + \tfrac{(r-2)(N-r)}{2r(r-1)}\right)\tr L_{\star}(x\star y) + \tfrac{r(N-r)}{2N^{2}(r-1)}\tr L_{\star}(x)\tr L_{\star}(y).
\end{split}
\end{align}
The first equality of \eqref{fk0} follows from \cite[Lemma VI.$1.1$]{Faraut-Koranyi}, proved using the Peirce decomposition. By \cite[Proposition III.$4.2$]{Faraut-Koranyi}, $\tr L_{\star}(x) = \tfrac{N}{r}\tr x$, where $N$ and $r$ are the dimension and rank of $(\balg, \star)$. It follows that $\tr L_{\star}(x \star y)  = \tfrac{N}{r}\tr(x\star y)$. This yields the second equality of \eqref{fk0}. Noting $\tau_{\star}(x, e) = \tr L_{\star}(x)$ and rewriting \eqref{fk0} yields
\begin{align}
\tau_{\star}(x, y) & = r(1 +\tfrac{d(r-2)}{4})\hat{h}(x, y) + \tfrac{r^{2}d}{4}\hat{h}(x, e)\hat{h}(y, e).
\end{align}
By Lemma \ref{griesseinsteinlemma} with $A = \tfrac{r^{2}d}{4}$ and $B = r(1 + \tfrac{d(r-2)}{4})$, the deunitalization $(\alg, \mlt, h)$ is exact and Einstein with Einstein constant $B\hat{h}(e, e) - 2 = r(1 + \tfrac{d(r-2)}{4}) - 2 = (r - 2)(1 + \tfrac{dr}{4})$.
\end{proof}
The deunitalizations corresponding to the Jordan algebras of table \ref{jordantable}, and their Einstein constants, are listed in table \ref{jordantable2}. Note that the table includes some entries for which the Einstein constant is equal to $0$; in these cases, which the rank $2$ cases, the Killing form $\tau_{\mlt}$ vanishes.

\begin{table}[!htbp]
\caption{Einstein constants of deunitalizations of the simple Euclidean Jordan algebras}\label{jordantable2}
\def\arraystretch{1.35}
\begin{tabular}{|l|l|c|c|}
\hline
Description & Notation & Dimension & Einstein constant $\ka$\\
\hline\hline
Trivial algebra &$\rea^{n} $ & $n$ & $0$\\\hline
Traceless symmetric $\rea$-matrices & $\herm_{0}(n, \rea)$ & $\tfrac{(n+2)(n-1)}{2}$& $\tfrac{(n-2)(n+4)}{4}$\\\hline
Traceless Hermitian $\com$-matrices &$\herm_{0}(n, \com)$ & $n^{2} - 1$ &$\tfrac{n^{2} - 4}{2}$ \\\hline
Traceless Hermitian $\quat$-matrices & $\herm_{0}(n, \quat)$ &$(2n + 1)(n-1)$& $(n-2)(n+1) $\\\hline
Traceless Hermitian $\cayley$-matrices & $\herm_{0}(3, \cayley)$ & $26$ & $7$\\\hline
\end{tabular}
\end{table}

\begin{lemma}\label{herm0cartanlemma}
Suppose $\fie$ is a real Hurwitz algebra and $n \geq 3$, with equality if $\fie = \cayley$.
The diagonal subalgebra $(\balg, \mlt, h) \subset (\herm_{0}(n, \fie), \mlt, h)$ is isometrically isomorphic to $(\ealg^{n-1}(\rea), (n-2)\tau)$.
\end{lemma}

\begin{proof}
Let $e_{ii}$ be the $n\times n$ matrix with a $1$ in the $ii$ position and all other entries $0$. The elements
\begin{align}\label{gaidef}
\ga(i) = \tfrac{n}{n-2}\left(e_{ii} - \tfrac{1}{n}I\right) \in \balg
\end{align}
satisfy $\ga(i) \mlt \ga(i) = \ga(i)$ and $(2-n)\ga(i)\mlt \ga(j) = \ga(i) + \ga(j)$ for $i \neq j$. By Corollary \ref{ealgmodelcorollary}, this suffices to show that $(\balg, \mlt)$ is isomorphic to $\ealg^{n-1}(\rea)$. Straightforward computations show $h(\ga(i), \ga(i)) = \tfrac{n}{(n-2)^{2}}$ and $h(\ga(i), \ga(j)) = -\tfrac{1}{(n-2)^{2}}$, so that, by \eqref{tauperm}, $(\balg, \mlt, h)$ is isomorphic as a metrized commutative algebra to $(\ealg^{n-1}(\rea), (n-2)\tau)$.
\end{proof}

\begin{lemma}\label{herm0simplelemma}
The deunitalization of a simple Euclidean Jordan algebra of rank $r \geq 3$ is simple. 
\end{lemma}

\begin{proof}
Such an algebra has the form $(\herm_{0}(n, \fie), \mlt)$ where $\fie$ is a real Hurwitz algebra, $n \geq 3$, and $n = 3$ if $\fie = \cayley$.
Let $J \subset (\herm_{0}(n, \fie), \mlt)$ be a nontrivial ideal. By the principal axis theorem, a nonzero $x \in J$ is equivalent to a diagonal element of $\herm_{0}(n, \fie)$ via an element of $\Aut(\herm_{0}(n, \fie), \mlt)$. Since the image of $J$ under an automorphism of $(\herm_{0}(n, \fie), \mlt)$ is again an ideal, without loss of generality it can be supposed that $J$ contains some element of the diagonal subalgebra $(\balg, \mlt) \subset (\herm_{0}(n, \fie), \mlt)$. Then $J \cap \balg$ is a nontrivial ideal in $(\balg, \cdmlt)$. By Lemma \ref{herm0cartanlemma}, $(\balg, \mlt)$ is isomorphic to $\ealg^{n-1}(\rea)$, which is simple by Corollary \ref{ealgmodelcorollary}, so $J \cap \balg = \balg$. That is, $J$ contains $\balg$. Together with $\balg$, $\herm_{0}(n, \fie)$ is spanned as a vector space by elements of the form $ze_{ij} + \bar{z}e_{ji}$  where $e_{ij}$ is the elementary matrix with a $1$ in the $ij$ entry and $0$ in every other entry, $z \in \fie$, and $i \neq j$. The matrix $D_{ij} = e_{ii} + e_{jj} - 2e_{kk}$ is in $\balg$, and so $ ze_{ij} + \bar{z}e_{ji} = (ze_{ij} + \bar{z}e_{ji})\mlt D_{ij} \in \balg$. This proves $\herm_{0}(n, \fie) = J$.
\end{proof}

For $I = \{i_{1}, \dots, i_{r}\} \subset \{1, \dots, n\}$ let $\ga(I) = \ga(i_{1}) + \dots + \ga(i_{r})$ and write $|I| = r$. 
Note that $\ga(1) + \ga(2) + \dots + \ga(n-1) + \ga(n) = 0$. As a consequence of the proof of Lemma \ref{herm0cartanlemma}, the products $\ga(I)\mlt\ga(J)$ are given by the formulas \eqref{eIeJ} (specialized with $\al = -1/(n-1)$). If $1 \leq |I| \leq n-1$ and $2|I| \neq n$ then $E(I) = \tfrac{2(2-n)}{n - 2|I|}\ga(I)$ is an idempotent in $(\herm_{0}(n, \fie), \mlt)$, while if $2|I| = n$ then $N(I) = \ga(I)$ satisfies $N(I)\mlt N(I) = 0$. Explicitly, 
\begin{align}
E(I) & = \left(\tfrac{n-|I|}{n-2|I|}\sum_{i \in I}e_{ii} - \tfrac{|I|}{n-2|I|}\sum_{i \notin I}e_{ii} \right)
\end{align}

\begin{lemma}\label{herm0idempotentlemma}
Suppose $\fie$ is a real Hurwitz algebra and $n \geq 3$, with $n = 3$ if $\fie = \cayley$.
\begin{enumerate}
\item\label{herm0idem1} If $X \in \szero(\herm_{0}(n, \fie), \mlt)$, then $X$ is in the $\Aut(\herm_{0}(n, \fie), \mlt)$ orbit of a matrix of the form $c\ga(I)$ with $0 \neq c \in \fie$ and $|I| = n/2$. In particular, if $n$ is odd then $\szero(\herm_{0}(n, \fie), \mlt) = \{0\}$.
\item\label{herm0idem2} $\idem(\herm_{0}(n, \fie), \mlt)$ is the union of the $\Aut(\herm_{0}(n, \fie), \mlt)$ orbits of the matrices $E(I)$ where $I \subset \{1, \dots, n\}$ satisfies $1 \leq |I| \leq n-1$ and $2|I| \neq n$. Moreover, $E(I)$ and $E(J)$ are in the same $\Aut(\herm_{0}(n, \fie), \mlt)$ orbit if and only if $|I| = |J|$ or $|I| = n - |J|$. 
\item\label{herm0idem3} The idempotent $E(I)$ satisfies $|E(I)|^{2}_{h} = \tfrac{|I|(n-|I|)}{(n-2|I|)^{2}}$, and so $|E(I)|^{2}_{h} < |E(J)|^{2}_{h}$ if $||I| - n/2| > ||J| - n/2|$. Consequently the idempotents in $(\herm_{0}(n, \fie), \mlt, h)$ having minimal norm are the elements in the $\Aut(\herm_{0}(n, \fie), \mlt)$ orbits of $E(i)$, $1 \leq i \leq n$.
\item\label{herm0idem4} A minimal idempotent in $(\herm_{0}(n, \fie), \mlt, h)$ has a positive eigenvalue on its orthogonal complement.
\end{enumerate}
\end{lemma}
\begin{proof}
Every element of $\herm_{0}(n, \fie)$ is $\Aut(\herm_{0}(n, \fie), \mlt)$ equivalent to a diagonal matrix. By Lemma \ref{herm0cartanlemma}, the subalgebra $(\balg, \mlt, h) \subset (\herm_{0}(n, \fie), \mlt, h)$ of diagonal matrices is isomorphic as a metrized commutative algebra to $(\ealg^{n-1}(\rea), (n-2)\tau)$. The claims \eqref{herm0idem1}-\eqref{herm0idem3} follow from the description of the idempotents and square-zero elements of $\ealg^{n-1}(\rea)$ given in Lemma \ref{haradalemma}. By \eqref{herm0idem2} and \eqref{herm0idem3}, a minimal norm idempotent is in the $\Aut(\herm_{0}(n, \fie), \mlt)$ orbit of $E(n)$. Since $\Aut(\herm_{0}(n, \fie), \mlt)$ acts transitively on this orbit, to show \eqref{herm0idem4} it suffices to show that $\spec(E(n))$ contains a positive number. For $x_{i} \in \rea$ define $x = \sum_{i = 1}^{n-1}x_{i}(e_{in} + e_{ni})$. Then $h(x, E(n)) = 0$ and $E(n)\mlt x = (1/2)x$, which proves the claim.
\end{proof}

\begin{remark}\label{nodoremark}
As an element of the diagonal subalgebra of $(\herm_{0}(n, \fie), \mlt, h)$, the orthogonal spectrum of $E(i)$ is contained in $(-\infty, 0])$, but \eqref{herm0idem4} shows that with respect to the full algebra $(\herm_{0}(n, \fie), \mlt, h)$ the orthogonal spectrum contains $1/2$.
\end{remark}

There is an alternative description of $(\herm_{0}(n, \com), \mlt)$.
Let $SU(n)$ be the special unitary group and let $\su(n)$ be its Lie algebra, regarded as the space of $n\times n$ skew-Hermitian matrices $\su(n) = \{x\in \mat(n, \com): x = -\bar{x}^{t}\}$. The Lie Killing form is $B_{\su(n)}(x, y) = 2n\tr xy$. Take $h(x, y) = -2n^{-2}B_{\su(n)}(x, y) = - n^{-1}\tr(xy)$ as the metric.
Define a commutative, nonassociative multiplication $\cdmlt$ on $\su(n)$ by
\begin{align}\label{sunmlt}
x \cdmlt y = \tfrac{\j}{2}\left(xy + yx - \tfrac{2}{n}\tr(xy)I\right),
\end{align}
where $I$ is the $n \times n$ identity matrix. The map $\Psi:\su(n) \to \herm_{0}(n, \com)$ defined by $\Psi(x) = \j x$ is an isometric isomorphism from $(\su(n), \cdmlt, h)$ to $(\herm_{0}(n, \com), \mlt, h)$.

The multiplication \eqref{sunmlt} appears in \cite{Laquer}, where it is used in the construction of a nontrivial family of bi-invariant affine connections on $SU(n)$. See also \cite{Benito-Draper-Elduque}.

The adjoint action $\Ad$ of $SU(n)$ on $\su(n)$, $\Ad(g) X = gXg^{-1}$ preserves $h$ and $\cdmlt$, so $\Ad$ induces an injective homomorphism $SU(n) \to \Aut(\su(n), \cdmlt, h)$. However $\Ad(SU(n))$ is contained properly in $\Aut(\su(n), \cdmlt, h)$, for the outer Lie algebra antiautomorphism $\Theta(x) =  -\bar{x}$ of $(\su(n), [\dum, \dum])$ is an $h$-isometric automorphism of $(\su(n), \cdmlt)$. Lemma \ref{sunautomorphismlemma} shows that $\Ad(SU(n))$ and $\Theta$ generate $\Aut(\su(n), \mlt)$.

\begin{lemma}\label{sunautomorphismlemma}
The adjoint action of $SU(n)$ on $(\su(n), [\dum, \dum])$ and the outer automorphism $\Theta(X) = -\bar{X}$ generate the group $\Aut(\su(n), \cdmlt)$, which preserves the invariant metric $h$ on $(\su(n), \cdmlt)$.
\end{lemma}
\begin{proof}
This follows from Lemma \ref{unitalizationisomorphismlemma}, relating the automorphism groups of $(\su(n), \mlt, h)$ and its unitalization and the fact (see \cite{Vinberg-automorphisms}) that the automorphism group of $(\herm(n, \com), \star)$ is generated by the conjugation action of $U(n)$ and the outer automorphism $x \to \bar{x}$. 
\end{proof}

\begin{remark}
The $\Ad(SU(n))$-invariance of $\cdmlt$ is equivalent to the statement that $\ad(X)$ is a derivation of $\cdmlt$ for $X \in \su(n)$. That is,
\begin{align}\label{adder}
[X, Y\cdmlt Z] = [X, Y]\cdmlt Z + Y\cdmlt[X, Z].
\end{align}
A useful consequence of \eqref{adder} is that $[X, X\cdmlt X] = 0$. By \eqref{adder}, the data $(\su(n), \cdmlt, [\dum, \dum])$ constitutes what might be called a nonassociative Poisson algebra. 
The associator of $(\su(n), \cdmlt)$ is
\begin{align}\label{sunass}
\begin{split}
[Y, Z, X]_{\cdmlt} &= \tfrac{1}{4}[Z, [X, Y]] + \tfrac{1}{n}(\tr(XY)Z - \tr(YZ)X)\\
&= \tfrac{1}{4}[Z, [X, Y]] + h(X, Y)Z - h(Y, Z)X.
\end{split}\end{align}
From \eqref{adder} and \eqref{sunass} it follows that $[X\cdmlt X, Y, X] = h(X, Y)X\cdmlt X - h(X \cdmlt X, Y)X$, so, by Lemma \ref{notjordanlemma}, the $h$-unitalization $(\su(n)_{g}, \hmlt, h)$ is a Jordan algebra. In \cite[section $3$]{Okubo}, the simple Euclidean Jordan algebras are constructed in essentially this manner as the unitalizations of the nonunital algebras $\herm_{0}(n, \fie)$ equipped with the multiplication $\mlt$ and metric $h$ defined in \eqref{jrank} (see in particular Remark $3.3$ and equations $3.27$a and $3.27$b in Remark $3.4$ of \cite{Okubo}).

The product $\cdmlt$ on $\su(3)$ appears implicitly in \cite[Equation $(5.2)$]{Gell-Mann} where a constant multiple of it is denoted $\{X, Y\}$ and it is written in terms of an orthogonal basis, regarded as a generalization of the usual Pauli basis in $\su(2)$. For $\su(n)$ the identities \eqref{adder} and \eqref{sunass} appear in the physics literature, for example \cite[page $79$]{Macfarlane-Sudbery-Weisz} and \cite[page $3$]{Macfarlane-Pfeiffer}, although their form is obfuscated because they are written in terms of an orthonormal basis. 
\end{remark}

The decomposition of $(\su(n), \cdmlt)$ into the $\pm 1$ eigenspaces $\su(n)_{\pm}$ (the subspaces comprising the imaginary skew-Hermitian and the real antisymmetric matrices) of the outer automorphism $\Theta$ is an $h$-orthogonal decomposition that makes $(\su(n), \cdmlt)$ a $\integer/2\integer$-graded algebra.
The outer automorphism $\Theta$ corresponds to the outer autormorphism $\theta$ of $(\herm_{0}(n, \com), \mlt)$ given by $\theta(x) = \bar{x}$, and the $+1$ eigenspace of $\theta$ is isomorphic to $(\herm_{0}(n, \rea), \mlt)$). In this setting, Lemma \ref{herm0cartanlemma} admits the following reformulation.

\begin{lemma}\label{sucartanlemma}
A Cartan subalgebra $\balg$ of the Lie algebra $(\su(n), [\dum, \dum])$ is a subalgebra of $(\su(n), \cdmlt)$ isomorphic to $\ealg^{n-1}(\rea)$. Moreover the Weyl group of $\su(n)$ acts as the automorphism group of $(\balg, \cdmlt)$. 
\end{lemma}
\begin{proof}
Since any two Cartan subalgebras are unitarily conjugate and the multiplication $\cdmlt$ is unitarily invariant, it suffices to prove the claims for the subspace $\balg$ of $(\su(n), \cdmlt)$ comprising diagonal matrices. It is apparent from \eqref{sunmlt} that $\balg$ is a subalgebra of $(\su(n), \cdmlt)$ on which the Weyl group acts by automorphisms. Since $\balg$ is an abelian Lie subalgebra of $\su(n)$, it follows from \eqref{sunass} that $(\balg, \cdmlt)$ is conformally associative. The proof of the final claim can be based on this observation and the uniqueness statement in Theorem \ref{confassclassificationtheorem}. Alternatively, it follows from Lemma \ref{herm0cartanlemma}.
\end{proof}

Corollary \ref{jordandeunitalizationsectcorollary} bounds the sectional nonassociativity of $(\herm_{0}(n, \fie), \mlt, h)$. It follows from Lemma \ref{hermsectlemma}, that bounds the sectional nonassociativity of $(\herm(n, \fie), \star, \hat{h})$. The proof uses the extension to real Hurwitz algebras of the Chern-do Carmo-Kobayashi-Böttcher-Wenzel inequality reported as Theorem \ref{bwtheorem} in the appendix.

\begin{lemma}\label{hermsectlemma}
Let $\fie$ be a real Hurwitz algebra. 
Equip $(\herm(n, \fie), \balg)$ with the metric $\hat{h}(x, y) = N^{-1}\tr L_{\star}(x\star y)= n^{-1}\tr(x\star y)$ where $N = \dim \herm(n, \fie)$ and $n$ is the rank of $\herm(n, \fie)$. For linearly independent $x, y \in \herm(n, \fie)$, the sectional nonassociativity of $(\herm(n, \fie), \star, \hat{h})$ satisfies
\begin{align}\label{hermsect}
0 \leq \sect(x, y) = \sect_{\hat{h}, \star}(x, y) \leq \tfrac{n}{2},
\end{align}
for all linearly independent $x, y \in \herm(n, \fie)$.
\begin{enumerate}
\item If $\fie$ is associative, equality holds in the lower bound of \eqref{hermsect} if and only if $x$ and $y$ commute with respect to the matrix product; in particular, if $x$ and $x \star x$ are linearly independent, then $\sect(x, x\star x) = 0$. 
\item Equality holds in the upper bound of \eqref{hermsect} if and only if $x$ and $y$ are simultaneously equivalent under $\Aut(\herm(n, \fie), \star)$ to scalar multiples of $e_{11} - e_{nn}$ and $e_{1n} + e_{n1}$.
\end{enumerate}
\end{lemma}

\begin{proof}
Let $[x, y] = xy - yx$ for $x, y \in \mat(n, \fie)$. The associators $[x, y, z]_{\star} = - [z, y, x]_{\star}$ and $[x, y, z]_{\cdot}$ of $\star$ and the matrix product $\cdot$ satisfy 
\begin{align}\label{jordanassociators}
\begin{split}
4[x, y, z]_{\star}& = [y, [x, z]] + [x, y, z]_{\cdot}  - [z, y, x]_{\cdot} + [y, x, z]_{\cdot}  - [y, z, x]_{\cdot} + [x, z, y]_{\cdot}- [z, x, y]_{\cdot} ,\\
& = 2[x, y, z]_{\cdot} - 2[z, y, x]_{\cdot} + [x, [y, z]] + [z, [x, y]],
\end{split}
\end{align}
where the second equality follows from the identity $\cycle \left([x, y, z] - [y, x, z]\right) = \cycle [[x, y], z]$ valid in any algebra \cite[Equation $2.19$]{Okubo-octonion}.
When $\fie \in \{\rea, \com, \quat\}$, the algebra $(\mat(n, \fie), \cdot)$ is associative, so \eqref{jordanassociators} becomes simply $4[x, y, z]_{\star} = [y, [x, z]]$. 
Taking $y = x$ in \eqref{jordanassociators} shows that 
\begin{align}\label{condensedassociator}
4[x, x, z]_{\star} = [x, [x, z]] + 2[x, x, z]_{\cdot} - 2[z, x, x]_{\cdot}.
\end{align} 
If $\fie$ is associative, then \eqref{condensedassociator} yields $4[x, x, z]_{\star} = [x, [x, z]]$.
Hence
\begin{align}\label{hermsect0}
\begin{split}
\hat{h}(x\star x, z \star z) & - \hat{h}(x\star z, z \star x)  = \hat{h}([x, x, z]_{\star}, z) = \tfrac{1}{4}\hat{h}([x, [x, z]], z) = \tfrac{1}{4n}\tr([x, [x, z]]\star y) \\
&= \tfrac{1}{8n}\tr\left([x, [x, z]]z + z[x, [x, z]]\right) = -\tfrac{1}{4n}\tr [x, z][x, z]= \tfrac{1}{4n}\tr \overline{[x, z]}^{t}[x, z].
\end{split}
\end{align}
When $\fie$ is associative, an automorphism of $(\herm(n, \fie), \star)$ preserves the ordinary matrix product and commutator bracket on $\mat(n, \fie)$, but this might not be true when $\fie = \cayley$, so in this case a an alternative argument is necessary (what follow works for associative $\fie$ too).
If $\fie = \cayley$, by the principal axis theorem there are $g \in \Aut(\herm(n, \fie), \star))$ and diagonal $\la \in \herm(n, \fie)$ such that $x = g\la$. Let $y = g^{-1}z$. Since $\la$ is Hermitian its elements are real, so $L_{\cdot}(\la \cdot \la) = L_{\cdot}(\la)^{2}$ and $R_{\cdot}(\la \cdot \la) = R_{\cdot}(\la)^{2}$, and hence $[\la, \la, y]_{\cdot} =  [y, \la, \la]_{\cdot} = 0$. In \eqref{condensedassociator} this yields
\begin{align}
 4[x, x, z]_{\star} = 4g[\la, \la,y]_{\star} = g[\la, [\la, y]] .
\end{align} 
As in \eqref{hermsect0} this yields
\begin{align}\label{hermsect0b}
\begin{split}
\hat{h}&(x\star x, z \star z) - \hat{h}(x\star z, z \star x)  = \hat{h}([x, x, z]_{\star}, z) = \tfrac{1}{4}\hat{h}(g[\la,  [\la, y]], gy)  = \tfrac{1}{4}\hat{h}([\la,  [\la, y]], y)\\
&= \tfrac{1}{4n}\tr([\la, [\la, y]]\star y) = \tfrac{1}{8n}\tr\left([\la, [\la, y]]y + y[\la, [\la, y]]\right) = -\tfrac{1}{4n}\tr [\la, y][\la, y]= \tfrac{1}{4n}\tr \overline{[\la, y]}^{t}[\la, y],
\end{split}
\end{align}
where, in the case $\fie = \cayley$, the penultimate equality follows from the invariance $\tr([x, y]z) + \tr(y[x, z]) = 0$ for $x, y, z \in\herm(n, \fie)$ proved in \cite[section $4.4$]{Freudenthal}. 
By \eqref{hermsect0} and \eqref{hermsect0b}, $\sect(x, y) \geq 0$. If $\fie$ is associative, by \eqref{hermsect0}, equality holds if and only if $[x, y] = 0$; in particular, $[x\star x, x] = [x, x, x]_{\star} = 0$ because $\star$ is commutative, so if $x$ and $x \star x$ are linearly independent, then $\sect(x, x \star x) = 0$.

Since $\tr\bar{x}^{t}x$ equals the Frobenius norm on $\herm(n, \fie)$, by Lemma \ref{cdklemma}, if $\fie$ is associative
\begin{align}
\hat{h}(x\star x, z\star z) - \hat{h}(x\star z, z \star x) = \tfrac{1}{4n}\tr \overline{[x, z]}^{t}[x, z] \leq \tfrac{n}{2}\left(|x|^{2}_{\hat{h}}|z|^{2}_{\hat{h}} - \hat{h}(x, z)^{2}\right),
\end{align}
which shows the upper bound in \eqref{hermsect}. The characterization of the equality case in Lemma \ref{cdklemma} yields the characterization of the equality case in the upper bound of \eqref{hermsect}.
If $\fie = \cayley$, then by Lemma \ref{cdklemma},
\begin{align}
\tfrac{1}{4n}\tr \overline{[\la, y]}^{t}[\la, y] \leq \tfrac{n}{2}\left(|\la|^{2}_{\hat{h}}|y|^{2}_{\hat{h}} - \hat{h}(\la, y)^{2}\right) =\tfrac{n}{2}\left(|x|^{2}_{\hat{h}}|z|^{2}_{\hat{h}} - \hat{h}(x, z)^{2}\right) ,
\end{align}
(where $n  = 3$) and in \eqref{hermsect0b} this shows the upper bound in \eqref{hermsect}. The characterization of the equality case in Lemma \ref{cdklemma} again yields the characterization of the equality case in the upper bound of \eqref{hermsect}.
\end{proof}

\begin{remark}
In \cite[Table $5.1$]{Liu-curvatureestimates}, X. Liu gives sharp upper bounds on the sectional curvatures of irreducible Riemannian symmetric spaces. It would be interesting to know if the upper bound of \eqref{hermsect} can be deduced from Liu's result for the relevant symmetric spaces (types AI, AII, and EIV). In this regard \cite{Benito-Draper-Elduque} is relevant.
\end{remark}

\begin{lemma}\label{nodominantlemma}
If a Euclidean metrized commutative algebra $(\alg,\mlt, h)$ has nonnegative sectional nonassociativity, then it is not exact and $\specp(e) \subset [0, 1/2]$ for any minimal idempotent $e \in \midem(\alg, \mlt, h)$.
\end{lemma}
\begin{proof}
By \eqref{minimalidempotentlemma} of Lemma \ref{criticalpointlemma} there exists $0 \neq e \in \idem(\alg, \mlt)$. Were $(\alg, \mlt)$ exact, then, because $\tr L_{\mlt}(e) = 0$ and $1$ is an eigenvalue of $L_{\mlt}(e)$, there would be a negative eigenvalue of $L_{\mlt}(e)$. Let $z \in \alg$ be $h$-orthogonal to $e \in \midem(\alg, \mlt, h)$ and suppose $L_{\mlt}(e)z = \la z$. Because the sectional nonassociativity is nonnegative, $0 \leq h(e\mlt e, z \mlt z) - |e\mlt z|^{2}_{h}= \la(1 - \la)|z|^{2}_{h}$, so $\la \in [0, 1/2]$ (the upper bound by Lemma \ref{criticalpointlemma}). 
\end{proof}

\begin{example}
Together Lemmas \ref{hermsectlemma} and \ref{nodominantlemma} show that $\min\specp(e) \geq 0$ for $e \in \midem (\herm(n, \fie), \star, h)$. 
\end{example}

\begin{lemma}\label{jordanminimalsectlemma}
Over a base field of characteristic zero, if the unitalization of the metrized commutative algebra $(\alg, \mlt, h)$ is power associative, then, for any $0 \neq x \in \alg$, either $x$ generates a one-dimensional subalgebra of $(\alg, \mlt)$ or $\sect(x\mlt x, x) = -1$.
\end{lemma}
\begin{proof}
By Lemma \ref{notjordanlemma}, if the unitalization of $(\alg, \mlt, h)$ is power associative, then $[x\mlt x, x, x] = |x|^{2}_{h}x\mlt x - h(x\mlt x, x)x$ for all $x \in \alg$. The quantity $|x|^{2}_{h}x\mlt x - h(x\mlt x, x)x$ vanishes if and only if $x\mlt x$ is a multiple of $x$, in which case  $|x|^{2}_{h}x\mlt x = h(x\mlt x, x)x$ and $x$ generates a one-dimensional subalgebra of $(\alg, \mlt)$. Otherwise, pairing this quantity with $x\mlt x$ and comparing with \eqref{sectnadefined} yields 
\begin{align}
\begin{split}
-\sect(x, x\mlt x) & = \tfrac{h([x\mlt x, x, x], x\mlt x)}{|x|^{2}_{h}|x\mlt x|^{2}_{h} - h(x\mlt x, x)^{2}} = 1,
\end{split}
\end{align}
which completes the proof.
\end{proof}

\begin{corollary}
\label{jordandeunitalizationsectcorollary}
Let $\fie$ be a real Hurwitz algebra. The sectional nonassociativity of the deunitalization $(\herm_{0}(n, \fie), \mlt, h)$ of a simple Euclidean Jordan algebra of rank $n \geq 2$ ($n = 3$ if $\fie = \cayley$) satisfies
\begin{align}\label{hermsectb}
-1 \leq \sect(x, y) \leq \tfrac{n-2}{2}.
\end{align}
and both bounds are attained. If $\fie$ is associative, the associator is given by 
\begin{align}\label{herm0ass}
[x, y, z]_{\mlt} & = \tfrac{1}{4}[y, [x, z]] - h(x, y)z + h(y, z)x,
\end{align}
for $x, y, z \in \herm_{0}(n, \fie)$.
\end{corollary}

\begin{proof}
The bounds \eqref{hermsectb} are immediate from Lemma \ref{hermsectlemma} and \eqref{secasecb}. By Lemma \ref{jordanminimalsectlemma}, if $x \in \herm_{0}(n, \fie)$ is not a multiple of an idempotent or a square-zero element, then $\sect(x \mlt x, x) = -1$, showing that the lower bound in \eqref{hermsectb} is attained. That the upper bound in \eqref{hermsectb} is attained for $x = e_{1n} + e_{n1}$ and $y = e_{11} - e_{nn}$ follows from computations using $x\mlt x = e_{11} + e_{nn} - \tfrac{2}{n}I = y \mlt y$ and $x \mlt y = 0$. The identity \eqref{herm0ass} follows from \eqref{unitalassoc} and \eqref{jordanassociators} when $\fie$ is associative. 
\end{proof}

\section{The conformal extension of a Killing metrized commutative algebra}
An $n$-dimensional Killing metrized exact commutative $\fie$-algebra gives rise to one of $(n+1)$-dimensions when $\chr\fie = 0$ and $\fie$ is algebraically closed or $\fie = \rea$. The restriction on the field is because square-roots are taken in the construction. This construction is an algebraic version of \cite[Lemma $6.9$]{Fox-cubicpoly} and \cite[Lemma $7.15$]{Fox-ahs}, whose purely algebraic content is somewhat obscured because they are stated in terms of the associated cubic polynomials.

\begin{lemma}\label{confextensionlemma}
Let $\fie$ be either an algebraically closed field of characteristic zero or $\rea$.
Let $(\alg, \mlt)$ be an $n$-dimensional Killing metrized exact commutative $\fie$-algebra. Equip the vector space $\lalg = \alg \oplus \fie$ with the commutative multiplication $\lmlt$ defined by
\begin{align}\label{confmult}
\begin{split}
\bar{x}\lmlt\bar{y} &= \tfrac{1}{\sqrt{(n+1)n}}\left( \sqrt{(n+2)(n-1)} x \mlt y - sx - ry , nrs - \tau_{\mlt}(x, y)\right).
\end{split}
\end{align}
where $\bar{x} = (x, r)$, $\bar{y} = (y, s)$.
Then:
\begin{enumerate}
\item The multiplication $\lmlt$ is exact.
\item The Killing form $\tau_{\lmlt}$ is nondegenerate and invariant and satisfies
\begin{align}\label{confmulttau}
\tau_{\lmlt}(\bar{x}, \bar{y}) = \tau_{\mlt}(x, y) + rs.
\end{align}
\item The conformal nonassociativity tensors $\om_{\lmlt}$ and $\om_{\mlt}$ are related by
\begin{align}\label{lconfass}
\tfrac{n}{n+2}\om_{\lmlt}(\bar{x}_{1}, \bar{x}_{2}, \bar{x}_{3}, \bar{x}_{4}) =  \tfrac{n-1}{n+1}\om_{\mlt}(x_{1}, x_{2}, x_{3}, x_{4}),
\end{align}
where $x_{i} = \pi(\bar{x}_{i})$ is image of $\bar{x}_{i} \in \lalg$ under the canonical projection $\pi:\lalg \to \alg$.
In particular, $(\lalg, \lmlt, \tau_{\lmlt})$ is conformally associative if and only if $(\alg, \mlt, \tau_{\mlt})$ is conformally associative.
\item\label{confassdominant} The element $e = (0, \sqrt{\tfrac{n+1}{n}})$ is an idempotent of $(\lalg, \lmlt)$ satisfying $\tau_{\lmlt}(e, e) = \tfrac{n+1}{n}$ and such that $L_{\lmlt}(e)$ acts on $\alg \oplus \{0\}$ as multiplication by $-\tfrac{1}{n}$, so $e$ is an absolutely primitive, and hence primitive, idempotent.
\end{enumerate}
\end{lemma}

\begin{proof}
The claims all follow from straightforward computations that are summarized here for convenience and in order to indicate where the hypotheses are used. Suppose $\lmlt$ is given by 
\begin{align}\label{confmulttemp}
\bar{x} \lmlt \bar{y} = \bn{}\left(x\mlt y - \cn{}(ry + sx), \cn{}(nrs - \tau_{\mlt}(x, y))\right),
\end{align}
where $\bn{}, \cn{} \in \fiet$ are parameters to be chosen. The operator $L_{\lmlt}(\bar{x})$ corresponding with \eqref{confmulttemp} can be represented matricially as
\begin{align}
L_{\lmlt}(\bar{x}) = \bn{}\begin{pmatrix}L_{\mlt}(x) - cr\Id & -cx\\ -c\tau_{\mlt}(x, \dum) & cnr\end{pmatrix},
\end{align}
from which it follows that $\tr L_{\lmlt}(\bar{x}) = \bn{}\tr L_{\mlt}(x) = 0$, so that $\lmlt$ is exact, and 
\begin{align}\label{confmulttau0}
\begin{split}
\bn{}^{-2}\tau_{\lmlt}(\bar{x}, \bar{y})& = (1 + 2\cn{}^{2})\tau_{\mlt}(x, y) + \cn{}^{2}n(n+1)rs.
\end{split}
\end{align}
Let $\bar{z} = (z, t) \in \lalg$. From \eqref{confmulttau0} there follows
\begin{align}\label{confmultinv}
\begin{split}
\bn{}^{-3}&\left(\tau_{\lmlt}(\bar{x},\lmlt \bar{y}, \bar{z})  - \tau_{\lmlt}(\bar{x}, \bar{y} \lmlt \bar{z})\right)\\
&= (1 + 2\cn{}^{2})\left(\tau_{\mlt}(x\mlt y, z) - \tau_{\mlt}(x, y \mlt z)\right) +  \cn{}\left(\cn{}^{2}(n+2)(n-1) - 1\right)\left(r\tau_{\mlt}(y, z)- t\tau_{\mlt}(x, y) \right)\\
&=   \cn{}\left(\cn{}^{2}(n+2)(n-1) - 1\right)\left(r\tau_{\mlt}(y, z)- t\tau_{\mlt}(x, y) \right).
\end{split}
\end{align}
To have \eqref{confmulttau} it is necessary that $\bn{}^{-2} = 1 + 2\cn{n}^{2} = \cn{n}^{2}n(n+1)$, or, equivalently, $\cn{}^{2}(n+2)(n-1) = 1$. 
Similarly, by \eqref{confmultinv}, in order that $\tau_{\lmlt}$ be invariant, it is necessary that $\cn{}^{2}(n+2)(n-1) = 1$. These equations have the four solutions $\cn{} = \pm \tfrac{1}{\sqrt{(n+2)(n-1)}}$ and $\bn{} = \pm \sqrt{\tfrac{(n+2)(n-1)}{(n+1)n}}$. (The hypotheses on $\fie$ guarantee these equations have solutions in $\fie$.)
However, the different choices of signs yield isomorphic algebras, so no generality is lost by choosing for $\cn{}$ the solution $\cn{n} = ((n+2)(n-1))^{-1/2}$ and for $\bn{}$ the solution $\bn{n} = \sqrt{\tfrac{(n+2)(n-1)}{(n+1)n}}$. In this case it follows from \eqref{confmultinv} that $\tau_{\lmlt}$ is invariant.
Further calculation using $\bn{n}^{-2} = 1 + 2\cn{n}^{2} = \cn{n}^{2}n(n+1)$ yields
\begin{align}\label{lmltassoc}
\begin{split}
[\bar{x}, \bar{y}, \bar{z}]_{\lmlt}  &= \left(\bn{n}^{2}[x, y, z]_{\mlt} + \tfrac{1}{n(n+1)}\left(\tau_{\mlt}(x, y)z - \tau_{\mlt}(y, z)x\right)+ \tfrac{1}{n}(st x - rs z) ,\right.\\
&\qquad \left. \tfrac{1}{n}\left(r\tau_{\mlt}(y, z) - t\tau_{\mlt}(x, y)\right)\right).
\end{split}
\end{align}
Let $\bar{w} = (w, u) \in \lalg$. Combining \eqref{lmltassoc} with \eqref{confmulttau} yields
\begin{align}
\begin{split}
\tau_{\lmlt}([\bar{x}, \bar{y}, \bar{z}]_{\lmlt}, \bar{w}) & = \bn{n}^{2}\tau_{\mlt}([x, y, z]_{\mlt}, w) + \tfrac{1}{n(n+1)}\left(\tau_{\mlt}(x, y)\tau_{\mlt}(z, w) - \tau_{\mlt}(y, z)\tau_{\mlt}(x, w)\right) \\
&\qquad+ \tfrac{1}{n}\left( st \tau_{\mlt}(x, w) - rs\tau_{\mlt}(z, w) + ru\tau_{\mlt}(y, z) - tu \tau_{\mlt}(x, y)\right), 
\end{split}\\
\begin{split}
\tau_{\lmlt}(\bar{x}, \bar{y})\tau_{\lmlt}(\bar{z}, \bar{w}) & - \tau_{\lmlt}(\bar{y}, \bar{z})\tau_{\lmlt}(\bar{z}, \bar{w})= \tau_{\mlt}(x, y)\tau_{\mlt}(z, w) - \tau_{\mlt}(y, z)\tau_{\mlt}(x, w) \\
&\qquad - st \tau_{\mlt}(x, w) + rs\tau_{\mlt}(z, w) - ru\tau_{\mlt}(y, z) + tu \tau_{\mlt}(x, y).
\end{split}
\end{align}
Since $\ricc_{\lmlt} = -\tau_{\lmlt}$ and $\ricc_{\mlt} = -\tau_{\mlt}$, there follows
\begin{align}
\begin{split}
\om_{\lmlt}&(\bar{z}, \bar{x}, \bar{y}, \bar{w})  = \tau_{\lmlt}([\bar{x}, \bar{y}, \bar{z}]_{\lmlt}, \bar{w}) \\
&\qquad + \tfrac{1}{n-1}\left(\ricc_{\lmlt}(\bar{y}, \bar{z})\tau_{\lmlt}(\bar{z}, \bar{w})  -\ricc_{\lmlt}(\bar{x}, \bar{y})\tau_{\lmlt}(\bar{z}, \bar{w}) +  \ricc_{\lmlt}(\bar{x}, \bar{w})\tau_{\lmlt}(\bar{y}, \bar{z})  -\ricc_{\lmlt}(\bar{z}, \bar{w})\tau_{\lmlt}(\bar{x}, \bar{y}) \right) 
\\& \qquad - \tfrac{n+1}{n(n-1)}\left( \tau_{\lmlt}(\bar{x}, \bar{y})\tau_{\lmlt}(\bar{z}, \bar{w})  - \tau_{\lmlt}(\bar{y}, \bar{z})\tau_{\lmlt}(\bar{z}, \bar{w})\right)\\
& =  \tau_{\lmlt}([\bar{x}, \bar{y}, \bar{z}]_{\lmlt}, \bar{w}) + \tfrac{1}{n}\left( \tau_{\lmlt}(\bar{x}, \bar{y})\tau_{\lmlt}(\bar{z}, \bar{w})  - \tau_{\lmlt}(\bar{y}, \bar{z})\tau_{\lmlt}(\bar{z}, \bar{w})\right)\\
& =\bn{n}^{2}\left(\tau_{\mlt}([x, y, z]_{\mlt}, w) + \tfrac{1}{n-1}(\tau_{\mlt}(x, y)\tau_{\mlt}(z, w) - \tau_{\mlt}(y, z)\tau_{\mlt}(x, w) \right)\\
& = \tfrac{(n+2)(n-1)}{(n+1)n}\om_{\mlt}(z, x, y, w).
\end{split}
\end{align}
This proves \eqref{lconfass}. The quantitative assertions of claim \eqref{confassdominant} regarding $e = (0, \sqrt{\tfrac{n+1}{n}})$ follow from straightforward computations. 
Since $e$ is absolutely primitive, it is primitive by Lemma \ref{absolutelyprimitivelemma}.
\end{proof}

\begin{definition}
For a field $\fie$ that is algebraically closed of characteristic zero or equal to $\rea$, define the \emph{conformal extension} of an $n$-dimensional exact commutative $\fie$-algebra $(\alg, \mlt)$ to be the exact commutative algebra $(\lalg, \lmlt)$ with multiplication $\lmlt$ as in \eqref{confmult}.
\end{definition}

\begin{remark}
For any $\be \in \fiet$, the rescaled multiplication 
\begin{align}\label{confmultbe}
\begin{split}
\hat{x}\bmlt\hat{y} &= \tfrac{\be}{\sqrt{(n+1)n}}\left( \sqrt{(n+2)(n-1)} x \mlt y - sx - ry , nrs - \tau_{\mlt}(x, y)\right).
\end{split}
\end{align}
satisfies
\begin{align}\label{confmulttaube}
&\tau_{\bmlt}(\bar{x}, \bar{y}) = \be^{2}\left(\tau_{\mlt}(x, y) + rs\right),&&
\om_{\bmlt}(\bar{x}_{1}, \bar{x}_{2}, \bar{x}_{3}, \bar{x}_{4}) = \be^{4} \tfrac{(n+2)(n-1)}{(n+1)n}\om_{\mlt}(x_{1}, x_{2}, x_{3}, x_{4}),
\end{align}
and the element $e = (0, \be^{-1}\sqrt{\tfrac{n+1}{n}})$ is an idempotent satisfying $\tau_{\bmlt}(e, e) = \tfrac{n+1}{n}$ and such that $L_{\bmlt}(e)$ acts on $\alg \oplus \{0\}$ as multiplication by $-\tfrac{1}{n}$.
The choice $\be = \sqrt{\frac{n(n+1)}{(n-1)(n+2)}}$ minimizes the occurrence of square-roots in \eqref{confmultbe} (there appear only $\sqrt{n+2}$ and $\sqrt{n-1}$, but not $\sqrt{n}$ or $\sqrt{n+1}$) and leads to the most symmetric scaling in \eqref{confmulttaube}, for it yields 
$\be^{4}\frac{(n-1)(n+2)}{n(n+1)} = \be^{2}$ so that
\begin{align}\label{confmultnm}
\begin{split}
\hat{x}\nmlt\hat{y} &= \left( x \mlt y - \cn{n}\left(sx + ry\right) , \cn{n}\left(nrs - \tau_{\mlt}(x, y)\right)\right),
\end{split}
\end{align}
where $\cn{n} = ((n+2)(n-1))^{-1/2}$, and
\begin{align}\label{confmulttaunm}
&\tfrac{n+2}{n+1}\tau_{\nmlt}(\bar{x}, \bar{y}) = \tfrac{n}{n-1}\left(\tau_{\mlt}(x, y) + rs\right),&&
\tfrac{n+2}{n+1}\om_{\nmlt}(\bar{x}_{1}, \bar{x}_{2}, \bar{x}_{3}, \bar{x}_{4}) = \tfrac{n}{n-1}\om_{\mlt}(x_{1}, x_{2}, x_{3}, x_{4}).
\end{align}
Over an algebraically closed field with $\chr \fie = 0$ or over $\rea$ there seems to be no reason to prefer any particular choice of $\bmlt$, but for more general fields, different choices might lead to nonisomorphic extensions. Here, any such choice is called a \emph{conformal extension} (the normalization is that in \eqref{confmult} unless otherwise indicated).
\end{remark}

\begin{lemma}
Let $\fie$ be a field that is algebraically closed of characteristic zero or $\rea$. If $(\lalg, \lmlt)$ is a Killing metrized exact commutative $\fie$-algebra of dimenion $n+1$ and $\bar{e} \in \idem(\lalg, \lmlt)$ satisfies $\tau_{\lmlt}(\bar{e}, \bar{e}) = \tfrac{n+1}{n}$ and the restriction of $L_{\lmlt}(\bar{e})$ to the $\tau_{\lmlt}$-orthogonal complement $\balg$ of $\bar{e}$ is multipication by $-1/n$, then the multiplication $\tilde{\mlt}$ on $\balg$ defined by 
\begin{align}
\bar{x} \tilde{\mlt} \bar{y} = \sqrt{\tfrac{n(n+1)}{(n+2)(n-1)}}\left(\bar{x}\lmlt \bar{y} + \tfrac{1}{n+1}\tau_{\lmlt}(\bar{x}, \bar{y})\bar{e}\right)
\end{align}
for $\bar{x}, \bar{y} \in \balg$,  makes $(\balg, \tilde{\mlt})$ a Killing metrized exact commutative $\fie$-algebra, and the conformal extension of $(\balg, \tilde{\mlt})$ is isomorphic to $(\lalg, \lmlt)$ via the map $(\bar{x}, r) \in \balg \oplus \fie \to \bar{x} + r\bar{e} \in \lalg$. 
\end{lemma}
\begin{proof}
This follows from computations based on the proof of Lemma \ref{confextensionlemma}.
\end{proof}

\begin{lemma}
Let $\fie$ be a field that is algebraically closed of characteristic zero or $\rea$.
\begin{enumerate}
\item\label{cea1} If $\Psi:(\alg_{1}, \mlt_{1}) \to (\alg_{2}, \mlt_{2})$ is an algebra homomorphism between Killing metrized exact commutative $\fie$-algebras, then $
\bar{\Psi}:\alg_{1}\oplus \fie \to \alg_{2}\oplus \fie$ defined by $\bar{\Psi}(x, r) = (\Psi(x), r)$ is an algebra homomorphism between their conformal extensions. 
\item\label{cea2} For a Killing metrized exact commutative $\fie$-algebra $(\alg, \mlt)$, the map $\Psi \to \bar{\Psi}$ is an injective group homomorphism $\Aut(\alg, \mlt) \to \Aut(\lalg, \lmlt)$.
\end{enumerate}
\end{lemma}

\begin{proof}
Because $\Psi$ is an algebra homomorphism, it preserves Killing forms. A computation using this observation shows that $\bar{\Psi}$ is an algebra homomorphim. Claim \eqref{cea2} follows from \eqref{cea1}.
\end{proof}

\begin{corollary}
For a field $\fie$ that is algebraically closed of characteristic zero or equal to $\rea$, the conformal extension of $\ealg^{n}(\fie)$ is isomorphic to $\ealg^{n+1}(\fie)$. 
\end{corollary}

\begin{proof}
By Lemma \ref{confextensionlemma}, the conformal extension of $\ealg^{n}(\fie)$ is a conformally associative Killing metrized exact commutative $\fie$-algebra with Killing form that is positive definite when $\fie = \rea$. By Theorem \ref{confassclassificationtheorem}, such an algebra is isomorphic to $\ealg^{n+1}(\fie)$.
\end{proof}

\begin{remark}
By Lemma \ref{confextensionlemma}, the conformal extension of a Killing metrized exact commutative algebra that is not conformally associative is not conformally associative, so any such example yields examples in all higher dimensions. The construction of such algebras in \cite{Fox-ahs} was the author's original motivation for this entire project.
\end{remark}

\begin{lemma}\label{confextensionsectlemma}
For the conformal extension $(\lalg = \alg \oplus \rea, \lmlt)$ of an $n$-dimensional Killing metrized exact commutative $\rea$-algebra $(\alg, \mlt)$,
\begin{align}\label{confextsect}
\begin{split}
\isect_{(\lalg, \lmlt)}(\bar{x}, \bar{y})
& =  \left(\tfrac{(n+2)(n-1)\isect_{(\alg,\mlt)}(x, y) + 1}{(n+1)n}\right)\tfrac{|x|^{2}|y|^{2} - \lb x , y \ra^{2}}{|x|^{2}|y|^{2} - \lb x , y \ra^{2} + |sx - ry|^{2}} - \tfrac{1}{n}\tfrac{|sx - ry|^{2}}{|x|^{2}|y|^{2} - \lb x , y \ra^{2} + |sx - ry|^{2}},
\end{split}
\end{align}
where $\bar{x} = (x, r), \bar{y} = (y, s) \in \lalg$, $|\dum|^{2} = \tau_{\mlt}(\dum, \dum)$, and $\lb x , y \ra = \tau_{\mlt}(x, y)$. 
Consequently:
\begin{enumerate}
\item\label{confsect1} $\tfrac{n}{n+2}\left(\isect_{(\lalg, \lmlt)}(\hat{x}, \hat{y}) + \tfrac{1}{n}\right)$ lies between $\tfrac{n-1}{n+1}\left(\isect_{(\alg,\mlt)}(x, y) + \tfrac{1}{n-1}\right)$ and $0$.
\item\label{confsect2} $(\lalg, \lmlt, \tau_{\lmlt})$ has nonpositive sectional nonassociativity if and only if $(\alg, \mlt, \tau_{\mlt})$ has sectional nonassociativity bounded above by $-\tfrac{1}{(n+2)(n-1)}$.
\item\label{confsect3} $(\lalg, \lmlt, \tau_{\lmlt})$ has sectional nonassociativity bounded above (respectively, below) by $-\tfrac{1}{n}$ if and only if $(\alg, \mlt, \tau_{\mlt})$ has sectional nonassociativity bounded above (respectively, below) by $-\tfrac{1}{n-1}$.
\item\label{confsect4} $(\lalg, \lmlt, \tau_{\lmlt})$ has constant sectional nonassociativity if and only if $(\alg, \mlt, \tau_{\mlt})$ has constant sectional nonassociativity. In this case $(\lalg, \lmlt, \tau_{\lmlt})$ has constant sectional nonassociativity $-\tfrac{1}{n}$ and $(\alg, \mlt, \tau_{\mlt})$ has constant sectional nonassociativity $-\tfrac{1}{n-1}$.
\end{enumerate}
\end{lemma}
\begin{proof}
The identity \eqref{confextsect} follows from the definition of the sectional nonassociativity, \eqref{confmult}, and \eqref{confmulttau}. It exhibits $\isect_{(\lalg, \lmlt)}(\bar{x}, \bar{y})$ as a convex combination of $\tfrac{(n+2)(n-1)\isect_{(\alg,\mlt)}(x, y) + 1}{(n+1)n}$ and $-\tfrac{1}{n}$, and so $\isect_{(\lalg, \lmlt)}(\bar{x}, \bar{y})$ must lie between $\tfrac{(n+2)(n-1)\isect_{(\alg,\mlt)}(x, y) + 1}{(n+1)n}$ and $-\tfrac{1}{n}$. Rewriting this relation yields \eqref{confsect1} and claims \eqref{confsect2} and \eqref{confsect3} follow directly from \eqref{confsect1}.
If $(\lalg, \lmlt, \tau_{\lmlt})$ has constant sectional nonassociativity $\bar{c}$, then \eqref{confextsect} equals $\bar{c}$ for all linearly independent choices of $\bar{x} = (x, r)$ and $\bar{y} = (y, s)$. Choosing $r = 0$ and $y = 0$ in \eqref{confextsect} yields $\bar{c} \isect_{(\lalg, \lmlt)}((x, 0), (0, s)) = -\tfrac{1}{n}$, while choosing $r = 0 = s$ in \eqref{confextsect} yields $-\tfrac{1}{n} = \bar{c} = \tfrac{(n+2)(n-1)\isect_{(\alg,\mlt)}(x, y) + 1}{(n+1)n}$ which has the solution $\isect_{(\alg,\mlt)} = -\tfrac{1}{n-1}$. Suppose $(\alg, \mlt, \tau_{\mlt})$ has constant sectional nonassociativity. Then it is conformally associative, so by Lemma \ref{confextensionlemma}, its conformal extension is conformally associative; since it is also Killing metrized, it has constant sectional nonassociativity by Lemma \ref{beltramilemma}, and so the preceding applies. This shows \eqref{confsect4}.
\end{proof}

\begin{remark}
Define the \emph{modified (intrinsic) sectional nonassociativity} of $\spn\{x, y\} \subset (\alg, \mlt)$ by 
\begin{align}
\msect_{\mlt}(x, y) = \tfrac{n-1}{n+1}\left(\isect_{\mlt}(x, y) + \tfrac{1}{n-1}\right).
\end{align}
By Lemma \ref{confextensionlemma}, conformal extension preserves the modified sectional nonassociativity's sign and decreases its absolute value, and an $n$-dimensional Euclidean Killing metrized exact commutative algebra has vanishing modified sectional nonassociativity if and only if it is isomorphic to $\ealg^{n}(\rea)$.
\end{remark}

\begin{lemma}\label{taumaxlemma}
Let $(\alg, \mlt, h)$ be a Euclidean metrized exact commutative algebra of dimension $n \geq 2$.
If $e \in \idem(\alg, \mlt)$, then $\tau_{\mlt}(e, e) \geq n/(n-1)$, with equality if and only if $\specp(e)$ comprises $-1/(n-1)$ with multiplicity $n-1$.
\end{lemma}

\begin{proof}
Because $L_{\mlt}(e)$ is $h$-self-adjoint, it is diagonalizable with eigenvalues $1$ and $\la_{1},\dots, \la_{n-1} \in \rea$. Because $\tr L_{\mlt}(e) = 0$ and $L_{\mlt}(e)$ preserves $\eperp$, the $\la_{1}, \dots, \la_{n-1}$ equal the eigenvalues of the restriction of $L_{\mlt}(e)$ to $\eperp$ and sum to $-1$. By the Cauchy-Schwarz inequality, $\tau_{\mlt}(e, e) = \tr L_{\mlt}(e)^{2} = 1 + \sum_{i = 1}^{n-1}\la_{i}^{2} \geq n/(n-1)$, with equality if and only if $\la_{i} = -1/(n-1)$ for all $1 \leq i \leq n-1$.
\end{proof}

\begin{lemma}\label{conformalextensionidempotentlemma}
Let $(\alg, \mlt)$ be an $n$-dimensional Euclidean Killing metrized exact commutative algebra and let $(\lalg, \lmlt)$ be its conformal extension normalized with $\cn{n} = \left((n+2)(n-1)\right)^{-1/2}$ as in \eqref{confmultnm}. For $x \in \alg$, let $|x|^{2} = \tau_{\mlt}(x, x)$, and for $\bar{x} \in \lalg$, let $|\bar{x}|^{2} = \tau_{\lmlt}(\bar{x}, \bar{x})$.
\begin{enumerate}
\item For $x \in \alg$ define
\begin{align}\label{sxpm}
\begin{split}
s_{\pm} &= s(x)_{\pm}  = \left(\tfrac{-1 - 4\cn{n}^{2}|x|^{2} \pm \sqrt{1 + 4(n+2)\cn{n}^{2}|x|^{2}}}{4\cn{n}^{2}|x|^{2}}\right)=  \tfrac{n - 4\cn{n}^{2}|x|^{2}}{1 + 4\cn{n}^{2}|x|^{2} \pm \sqrt{1 + 4(n+2)\cn{n}^{2}|x|^{2}}}.
\end{split}
\end{align}
When $e \in \idem(\alg, \mlt)$ and $s_{\pm} = s(e)_{\pm} \neq 0$, the element 
\begin{align}\label{epmdefined}
\bar{e}_{\pm} = s_{\pm}^{-1}\bar{e}_{\pm} = s^{-1}_{\pm}((1 + s_{\pm})e, \tfrac{1}{2\cn{n}}) \in \lalg
\end{align} 
is idempotent and satisfies
\begin{align}
|\bar{e}_{\pm}|^{2} = \phi(s(e)_{\pm}) = \tfrac{n(n+1)(n+1 - 2s_{\pm})}{4s_{\pm}^{2}},
\end{align}
where $\phi(s) = \tfrac{n(n+1)(n+1 -2s)}{4s^{2}}$.

\item\label{confextidempotents} A nonzero idempotent of $(\lalg, \lmlt)$ has one of the following forms.
\begin{enumerate}
\item\label{confextidem1} If $z \in \szero(\alg, \mlt)$, then 
\begin{align}\label{cei1}
\bar{z}_{\pm} = \left(\pm\tfrac{(n+2)\sqrt{n-1}}{2}\tfrac{z}{|z|}, - \tfrac{\sqrt{(n+2)(n-1)}}{2}\right) = \left(\pm\tfrac{\sqrt{n+2}}{2\cn{n}}\tfrac{z}{|z|}, - \tfrac{1}{2\cn{n}}\right) \in \lalg
\end{align}
are idempotent, and $|\bar{z}_{\pm}|^{2} = \phi(-1)= \tfrac{(n+3)(n+1)n}{4} = \tfrac{n+1}{4\cn{n+1}^{2}}$.
\item\label{confextidem2} If $e \in \idem(\alg, \mlt)$ satisfies $|e|^{2} = \tfrac{(n+2)n(n-1)}{4} = \tfrac{n}{4\cn{n}^2}$ then
\begin{align}\label{cei2}
\bar{e}_{-} = \left( \tfrac{n+2}{2(n+1)}e, -\tfrac{n\sqrt{(n+2)(n-1)}}{4(n+1)}\right) =\left( \tfrac{n+2}{2(n+1)}e, -\tfrac{n}{4(n+1)\cn{n}}\right)  \in \lalg
\end{align}
is idempotent, and $|\bar{e}_{-}|^{2} =\phi(-\tfrac{2(n+1)}{n}) = \tfrac{(n+4)n^{2}}{16}$.
\item\label{confextidem} If $e \in \idem(\alg, \mlt)$ and $4\cn{n}^{2}|e|^{2} \neq n$, then 
\begin{align}\label{cei3}
\begin{split}
\bar{e}_{\pm} &=  s^{-1}_{\pm}((1 + s_{\pm})e, \tfrac{1}{2\cn{n}})\\
& = \tfrac{n+2}{n+1 \mp \sqrt{1 + 4(n+2)\cn{n}^{2}|e|^{2}}}\left(e, \tfrac{1}{2\cn{n}}\left(\tfrac{1\pm \sqrt{1 + 4(n+2)\cn{n}^{2}|e|^{2}}}{n+2} \right)\right)  \in \lalg 
\end{split}
\end{align}
are idempotent, and $|\bar{e}_{\pm}|^{2} = \phi(s(e)_{\pm})  = \tfrac{n(n+1)(n+1 - 2s_{\pm})}{4s_{\pm}^{2}}$. 
\end{enumerate}
Which of the forms \eqref{confextidem1}-\eqref{confextidem} the idempotent of $(\lalg, \lmlt)$ has is determined by its $\tau_{\lmlt}$ norm. Precisely, for $e \in \idem(\alg, \mlt)$, there hold
\begin{align}\label{spmbounds}
& s(e)_{-} \in [-\tfrac{n(n+1)}{2}, -1),& &  s(e)_{+}  \in (-1, 0) \cup (0, \tfrac{n}{2}),
\end{align}
and $\phi$ is strictly monotone increasing on $(-\infty, 0)$ and strictly monotone decreasing on $(0, \tfrac{n}{2})$ and there hold
\begin{align}\label{fspmbounds}
\begin{split}
 \tfrac{n+1}{n} &\leq \phi(s(e)_{-}) < \tfrac{n(n+1)(n+3)}{4},\\
 \phi(s(e)_{+}) &> \tfrac{n+1}{n} \quad \text{if}\,\, s(e)_{+} \in (0, \tfrac{n}{2}),\\
 \phi(s(e)_{+}) &\geq \tfrac{n(n+1)(n+3)}{4} \quad \text{if}\,\, s(e)_{+} \in (-1, 0).
\end{split}
\end{align}
Equality holds in either of the nonstrict inequalities of \eqref{fspmbounds} (equivalently in the nonstrict inequalities of \eqref{spmbounds}) if and only if $|e|^{2} = \tfrac{n}{n-1}$, in which case $e$ is a minimal idempotent.

\item\label{confextsz} A  $\bar{w} \in \lalg$ is square-zero if and only if it is a constant multiple of an element of the form
\begin{align}
\left(e, \tfrac{\sqrt{(n+2)(n-1)}}{2}\right) = \left(e, \tfrac{1}{2\cn{n}}\right) \in \lalg,
\end{align}
where $e \in \idem(\alg, \mlt)$ satisfies $|e|^{2} = \tfrac{(n+2)n(n-1)}{4} = \tfrac{n}{4\cn{n}^{2}}$. 
\item\label{ceispec1} If $e \in \idem(\alg, \mlt)$, then
\begin{align}\label{specepm}
\specp(\bar{e}_{\pm}) = \{\tfrac{n+2}{2s(e)_{\pm}}\} \cup \{\tfrac{2(s(e)_{\pm} + 1)\la - 1}{2s(e)_{\pm}}: \la \in \specp(e)\}.
\end{align}
\end{enumerate}
\end{lemma}

\begin{proof}
Every nonzero element in $\lalg$ can be written in the form $(x, r)$ for some $x \in \alg$ and $r \in \rea$. Since $(x, r)\lmlt (x, r) -(x, r) = (x\mlt x - (1 + 2\cn{n}r)x, \cn{n}(nr^{2} - |x|^{2}) - r)$, if $r = 0$, then $(x, r)$ is not idempotent, so it can be assumed from the beginning that $r \neq 0$.
If $1 + 2\cn{n}r \neq 0$, then $r = \tfrac{1}{2\cn{n}s}$ for some $s \neq -1$ and there is $e \in \alg$ so that $x = (1 + s^{-1})e$. Consequently every element $(x, r) \in \lalg$ with $r \neq 0$ can be written either in the form $\bar{e}_{s} = ((1+ s^{-1})e, \tfrac{1}{2\cn{n}s}) = s^{-1}((1 + s)e, \tfrac{1}{2\cn{n}})$ for some $-1 \neq s \in \rea$ and some $e \in \alg$ or in the form $\tilde{x} = (x, -\tfrac{1}{2\cn{n}})$ for some $x \in \alg$. 
From
\begin{align}
\tilde{x}\lmlt \tilde{x} - \tilde{x} = (x\mlt x, \tfrac{n+2}{4\cn{n}} - \cn{n}|x|^{2}),
\end{align}
it follows that $\tilde{x}$ is idempotent if and only if $x\mlt x= 0$ and $|x|^{2} = \tfrac{n+2}{4\cn{n}^{2}}$. In this case $\tilde{x} = (\pm\frac{\sqrt{n+2}}{2\cn{n}} \tfrac{x}{|x|}, -\tfrac{1}{2\cn{n}})$ has the form \eqref{cei1}, and, by \eqref{confmulttau}, $|\tilde{x}|^{2} = n(n+1)\cn{n}^{2}(|x|^{2} + \tfrac{1}{4\cn{n}^{2}}) = n(n+1)(n+3)/4$.
From 
\begin{align}
\bar{e}_{s}\lmlt \bar{e}_{s} - \bar{e}_{s} = s^{-2}\left((s+1)^{2}(e\mlt e - e), \tfrac{1}{4\cn{n}}\left(n - 2s - 4\cn{n}^{2}(s + 1)^{2}|e|^{2} \right)\right),
\end{align}
it follows that $\bar{e}_{s}$ is idempotent if and only if $e \in \idem(\alg, \mlt)$ and 
\begin{align}\label{cei2b}
0 & = n - 2s - 4\cn{n}^{2}(s + 1)^{2}|e|^{2} = n+1-2(s+1)- 4\cn{n}^{2}(s + 1)^{2}|e|^{2}.
\end{align}
The equation \eqref{cei2b} has the distinct real roots
\begin{align}
s(e)_{\pm} + 1 = \tfrac{1}{4\cn{n}^{2}|e|^{2}}\left( - 1 \pm \sqrt{ 1+ 4(n+2)\cn{n}^{2}|e|^{2}}\right),
\end{align}
so that
\begin{align}\label{sepm}
\begin{split}
s_{\pm} &= s(e)_{\pm}  = \left(\tfrac{-1 - 4\cn{n}^{2}|e|^{2} \pm \sqrt{1 + 4(n+2)\cn{n}^{2}|e|^{2}}}{4\cn{n}^{2}|e|^{2}}\right)=  \tfrac{n - 4\cn{n}^{2}|e|^{2}}{1 + 4\cn{n}^{2}|e|^{2} \pm \sqrt{1 + 4(n+2)\cn{n}^{2}|e|^{2}}}.
\end{split}
\end{align}
If $n = 4\cn{n}^{2}|e|^{2}$, then $s(e)_{+} = 0$, so if $\bar{e}_{s}$ is to be idempotent, this solution cannot occur. The other solution, $s(e)_{-}$ yields an idempotent $\bar{e}_{s_{-}}= \bar{e}_{-}$ of the form \eqref{cei2}.
If $n \neq 4\cn{n}^{2}|e|^{2}$, then \eqref{cei2b} yields the two distinct nonzero real solutions \eqref{sepm}, and these yield the idempotents $\bar{e}_{s_{\pm}} = \bar{e}_{\pm}$ of \eqref{cei3}. 

By \eqref{confmulttau} and the fact that $s_{\pm}$ solves \eqref{cei2b},
\begin{align}
\begin{split}
|\bar{e}_{\pm}|^{2} &= n(n+1)\cn{n}^{2}\left(\left(1 +\tfrac{1}{s_{\pm}}\right)^{2}|e|^{2} + \tfrac{1}{4\cn{n}^{2}s_{\pm}^{2}}\right)  = \tfrac{n(n+1)(n+1 - 2s_{\pm})}{4s_{\pm}^{2}} = \phi(s_{\pm}).
\end{split}
\end{align}
The function $\phi(s) = \tfrac{n(n+1)(n+1 - 2s)}{4s^{2}}$ satisfies $\phi^{\prime}(s) = -\tfrac{n(n+1)(n+1 -s)}{2s^{3}}$, so is monotone decreasing on $(0, n+1)$ and monotone increasing on $(-\infty, 0)$.
Because $\sqrt{1 + x} < 1 + \tfrac{x}{2}$ if $x > 0$,
\begin{align}
-1 \leq s(e)_{+} = \tfrac{-1 - 4\cn{n}^{2}|e|^{2} + \sqrt{1 + 4(n+2)\cn{n}^{2}|e|^{2}}}{4\cn{n}^{2}|e|^{2}} < \tfrac{n}{2}.
\end{align}
Since $\phi$ is strictly monotone decreasing on $(0, n+1)$,
\begin{align}
\phi(s(e)_{+}) > \phi(\tfrac{n}{2}) = \tfrac{n+1}{n},
\end{align}
if $s(e)_{+} \geq 0$, and, since $\phi$ is monotone increasing on $(-\infty, 0)$,
\begin{align}
\phi(s(e)_{+}) \geq \phi(-1) = \tfrac{n(n+1)(n+3)}{4},
\end{align}
if $s(e)_{+} \in (-1, 0)$. 
The function $g(x) = \tfrac{1 + x + \sqrt{1 + (n+2)x}}{x}$ is monotone decreasing on $(0, \infty)$. By Lemma \ref{taumaxlemma}, $|e|^{2} \geq \tfrac{n}{n-1}$, so 
\begin{align}
-1 > s(e)_{-} = \tfrac{-1 - 4\cn{n}^{2}|e|^{2} - \sqrt{1 + 4(n+2)\cn{n}^{2}|e|^{2}}}{4\cn{n}^{2}|e|^{2}} = -g(4\cn{n}^{2}|e|^{2}) \geq -g(\tfrac{4n}{(n+2)(n-1)^{2}})  = - \tfrac{n(n+1)}{2}. 
\end{align}
Since $\phi$ is monotone increasing on $(-\infty, 0)$ this implies
\begin{align}
\tfrac{n(n+1)(n+3)}{4} = \phi(-1) > \phi(s(e)_{-})  \geq \phi(-\tfrac{n(n+1)}{2}) = \tfrac{n+1}{n}.
\end{align}
An element $\bar{w} = (w, r) \in \lalg$ is square-zero if and only if $w \mlt w = 2\cn{n}r w$ and $nr^{2} = |w|^{2}$. Because $nr^{2} = |w|^{2}$, if $\bar{w} \neq 0$, then $r \neq 0$ and $w \neq 0$. In this case $\tfrac{1}{2\cn{n}r}w$ is idempotent in $(\alg, \mlt)$. Since if $\bar{w}$ is square-zero, any multiple $\la \bar{w}$ is also square-zero, it can be supposed without loss of generality that $2\cn{n} r = 1$ and $w$ is idempotent. In this case $|w|^{2} = nr^{2} = 4^{-1}n\cn{n}^{-2} = (n+2)n(n-1)/4$. This proves \eqref{confextsz}. 

If $e \in \idem(\alg, \mlt)$ and $\la \in \specp(e)$ then there is $u \in \alg$ such that $\tau_{\mlt}(e, u) =0$ and $e \mlt u = \la u$. A straightforward computation shows $\bar{e}_{\pm}\lmlt(u, 0) = (\tfrac{2(s_{\pm} + 1)\la - 1}{2s_{\pm}}u, 0)$ and $\tau_{\lmlt}(\bar{e}_{\pm}, (u, 0)) = 0$, so that $\tfrac{2(s_{\pm} + 1)\la - 1}{2s_{\pm}} \in \specp(\bar{e}_{\pm})$. The elements of the form $(u, 0)$ with $\tau_{\mlt}(e, u) =0$ span a codimension one subspace of the $\tau_{\lmlt}$-orthogonal of $\bar{e}_{\pm}$ in $\lalg$. The vector $(e, -2\cn{n}(s(e)_{\pm} + 1)|e|^{2})$ is $\tau_{\lmlt}$-orthogonal to $\bar{e}_{\pm}$ and straightforward computations show that it is an eigenvector of $L_{\lmlt}(\bar{e}_{\pm})$ with eigenvalue $\tfrac{n+1}{2s(e)_{\pm}}$. This proves \eqref{specepm} and confirms \eqref{ceispec1}.
\end{proof}

\begin{lemma}\label{doubleconformallemma}
A Euclidean Killing metrized exact commutative algebra contains a nonzero square-zero element if and only if the conformal extension of its conformal extension contains a nonzero square-zero element.
\end{lemma}

\begin{proof}
By \eqref{confextidem1} of Lemma \ref{conformalextensionidempotentlemma}, if an $n$-dimensional algebra $(\alg, \mlt)$ contains a nontrivial square-zero element, then the $(n+1)$-dimensional algebra $(\lalg, \lmlt)$ contains an idempotent with squared-norm $\tfrac{n+1}{4\cn{n+1}^{2}}$. By \eqref{confextidempotents}  of Lemma \ref{conformalextensionidempotentlemma}, the converse is true as well, for if $(\lalg, \lmlt)$ contains an idempotent with squared-norm $\tfrac{n+1}{4\cn{n+1}^{2}}$, then this idempotent must have one of the forms listed in \eqref{confextidempotents} of Lemma \ref{conformalextensionidempotentlemma}, and the condition on the norm forces it to have the form \eqref{cei1} of \eqref{confextidem1} of Lemma \ref{conformalextensionidempotentlemma}.
By \eqref{confextsz} of Lemma \ref{conformalextensionidempotentlemma}, the $(n+2)$-dimensional double conformal extension $(\overline{\bar{\alg}}, \overline{\bar{\mlt}})$ contains a square-zero element if and only if the $(n+1)$-dimensional algebra $(\lalg, \lmlt)$ contains an idempotent with squared-norm $\tfrac{n+1}{4\cn{n+1}^{2}}$. 
\end{proof}

\begin{lemma}\label{confextminimalslemma} 
Let $(\alg, \mlt)$ be a Euclidean Killing metrized exact commutative algebra with conformal extension $(\lalg, \lmlt, \tau_{\lmlt})$. For $\bar{e} \in \idem(\lalg, \lmlt, \tau_{\lmlt})$, the following are equivalent:
\begin{enumerate}
\item\label{confextminimals1} $\bar{e}$ is minimal.
\item\label{confextminimals3} $\tau_{\lmlt}(\bar{e}, \bar{e}) = \tfrac{n+1}{n}$.
\item\label{confextminimals2} $\bar{e}$ is minimal and $\specp(\bar{e}) = \{-1/n\}$.
\item\label{confextminimals4} Either $\bar{e}$ is the canonical minimal idempotent $\bar{w} = (0, n^{-1}\cn{n}^{-1})$ or it has the form $\bar{e}_{-}$ for a minimal idempotent $e \in \midem(\alg, \mlt, \tau_{\mlt})$ (where $\bar{e}_{-}$ is as in \eqref{epmdefined}) satisfying $\tau_{\mlt}(e, e) = \tfrac{n}{n-1}$.
\end{enumerate}
\end{lemma}

\begin{proof}
That $\bar{w} = (0, n^{-1}\cn{n}^{-1})$ is minimal with $\tau_{\lmlt}(\bar{w}, \bar{w}) = \tfrac{n+1}{n}$ follows from \eqref{confassdominant} of Lemma \ref{confextensionlemma}.
Consequently any other minimal idempotent in $(\lalg, \lmlt, \tau_{\lmlt})$ satisfies $\tau_{\lmlt}(\bar{e}, \bar{e}) = \tfrac{n+1}{n}$. By Lemma \ref{taumaxlemma}, for such an idempotent $\bar{e}$, $\specp(\bar{e}) = \{-1/n\}$. This proves the equivalence of claims \eqref{confextminimals1} and \eqref{confextminimals3}. 
By \eqref{confextidempotents} of Lemma \ref{conformalextensionidempotentlemma} such an idempotent has the form $\bar{e}_{\pm}$ for some $e \in \idem(\alg, \mlt)$ and moreover there holds $\tfrac{n(n+1)(n+1 - 2s(e)_{\pm})}{4s(e)_{\pm}^{2}} = \tfrac{n+1}{n}$. This forces $s(e)_{\pm} \in \{-\tfrac{n(n+1)}{2}, \tfrac{n}{2}\}$. By \eqref{spmbounds} and \eqref{fspmbounds} this can happen only if there is $e \in \idem(\alg, \mlt)$ such that $s(e)_{-} = -\tfrac{n(n+1)}{2}$, in which case a simple computation shows that $\tau_{\mlt}(e, e) = \tfrac{n}{n-1}$. 
\end{proof}

\begin{corollary}\label{confextsaturatedcorollary}
The conformal extension of a Euclidean Killing metrized exact commutative algebra $(\alg, \mlt)$ is spanned by minimal idempotents if and only if $\alg$ is spanned by minimal idempotents.
\end{corollary}
\begin{proof}
If $\midem(\alg, \mlt, \tau_{\mlt})$ spans $\alg$, by Lemma \ref{confextminimalslemma}, the set $\{(0, -\tfrac{1}{n\cn{n}})\} \cup\{\bar{e}_{-}: e \in \midem(\alg, \mlt, \tau_{\mlt})\}$ is contained in $\midem(\lalg, \lmlt, \tau_{\lmlt}) $ and spans $(\lalg, \lmlt, \tau_{\lmlt})$. Conversely, if $(\lalg, \lmlt, \tau_{\lmlt})$ is spanned by minimal idempotents, by Lemma \ref{confextminimalslemma} these include $\bar{w}$ and at least $n$ more linearly independent minimal idempotents of the form $\bar{e}_{i\,-}$ for $e_{i} \in \midem(\alg, \mlt, \tau_{\mlt})$ satisfying $\tau_{\mlt}(e_{i}, e_{i}) = n/n-1$. The $e_{1}, \dots, e_{n}$ span $\alg$.
\end{proof}

\section{Triple algebras and Nahm algebras}\label{triplesection}
This section describes an iterative construction yielding simple Killing metrized exact commutative algebras of arbitrarily high dimension. Given a commutative algebra $(\alg, \mlt)$, the construction produces a commutative algebra structure on the \emph{triple} $\trip(\alg) = \alg \oplus \alg \oplus \alg$. This iterative construction preserves both the Killing metrized condition and simplicity, so yields simple Killing metrized exact commutative algebras of arbitrarily high dimension. A preliminary version of this material appeared in \cite[Section $7.2$]{Fox-ahs}.

Lemma \ref{tripletensorlemma} shows that the triple algebra is simply the tensor product with $\ealg^{3}$. Among tensor products with $\ealg^{n}$ what is special about the exponent $n = 3$ is that the polarization $Q(x, y, z) = P(x + y + z) - P(x + y) - P(y + z) - P(z + x) + P(x) + P(y) + P(z)$ of the cubic polynomial $P(x)$ of a metrized commutative algebra $(\alg, \mlt, h)$ can be viewed as a cubic polynomial on $\alg \oplus \alg \oplus \alg$, so defines on $\alg \oplus \alg \oplus \alg$ (equipped with the direct sum metric) a structure of a metrized commutative algebra. See \eqref{ptrip} of Lemma \ref{triplealgebrastructurelemma}.

\begin{definition}
The \emph{triple} $(\trip(\alg), \tmlt)$ of the algebra $(\alg, \mlt)$ is the exact commutative algebra defined as follows. Let $\trip(\alg) = \alg \oplus \alg \oplus \alg$,  let $\pi_{i}:\trip(\alg) \to \alg$ be projection onto the $i$th component, and write $x_{i} = \pi_{i}(x)$. Write $\nu_{i}:\alg \to \trip(\alg)$ for the inclusion as the $i$th component. The notations $x = \nu_{1}(x) + \nu_{2}(x) + \nu_{3}(x)$ and $x = (x_{1}, x_{2}, x_{3})$ are synonymous, and which is used depends on context.
Define a commutative multiplication $\tmlt$ on $\trip(\alg)$ by
\begin{align}\label{tripmlt}
\begin{split}
x \tmlt y & = \tfrac{1}{2}\left(\nu_{1}(x_{2} \mlt y_{3} + y_{2}\mlt x_{3}) + \nu_{2}(x_{3}\mlt y_{1} + y_{3} \mlt x_{1})  + \nu_{3}({1}\mlt y_{2} + y_{1}\mlt x_{2}) \right)\\
&= \tfrac{1}{2}\left(x_{2} \mlt y_{3} + y_{2}\mlt x_{3}, x_{3}\mlt y_{1} + y_{3} \mlt x_{1}, x_{1}\mlt y_{2} + y_{1}\mlt x_{2}\right).
\end{split}
\end{align}
Alternatively, the left multiplication $L_{\tmlt}(x)$ of $\trip(\alg)$ has the block matrix form
\begin{align}\label{generaltriple}
L_{\tmlt}(x) = \tfrac{1}{2}\begin{pmatrix} 0 & R_{\mlt}(x_{3}) & L_{\mlt}(x_{2}) \\ L_{\mlt}(x_{3}) & 0 & R_{\mlt}(x_{1})\\ R_{\mlt}(x_{2}) & L_{\mlt}(x_{1}) & 0\end{pmatrix},
\end{align}
where $L_{\mlt}, R_{\mlt}:\alg \to \eno(\alg)$ are the left and right regular representations of $(\alg, \mlt)$.
\end{definition}

One reason for choosing the factor of $1/2$ in \eqref{tripmlt} is that when $(\alg, \mlt)$ is commutative the injective diagonal map $\diag:\alg \to \trip(\alg)$ defined by $\diag(x) = \nu_{1}(x) + \nu_{2}(x) + \nu_{3}(x)$ is an algebra homomorphism. 

\begin{example}\label{rtripleexample}
Regard $\fie$ as a Killing metrized commutative algebra. The Killing form of $\fie$ is $\tau_{\rea}(x, x) = x^{2}$ and the associated cubic polynomial is $\tfrac{1}{6}x^{3}$. The triple $\trip(\fie)$ has multiplication table \eqref{n3alg} with $c = 1/2$, so by Example \ref{3dexample} is isomorphic to $\ealg^{3}(\fie)$. This provides motivation for Lemma \ref{tripletensorlemma}.
The trace-form $\tau_{\trip(\fie)}$ equals $\tfrac{1}{2}(x_{1}y_{1} + x_{2}y_{2} + x_{3}y_{3})$ for $x, y \in \fie^{3}$, in accord with \eqref{tautrip} below, and the cubic polynomial of $(\trip(\fie), \mlt, \tau_{\trip(\fie)})$ is $\tfrac{1}{4}x_{1}x_{2}x_{3}$, in accord with \eqref{ptrip} below. 
\end{example}

The Lie algebra $\so(3)$ is defined over a ground field $\fie$ of characteristic not equal to $2$ as a $3$-dimensional $\fie$-vector space $\ste$ equipped with a basis $\{e_{1}, e_{2}, e_{3}\}$ and the Lie bracket satisfying $[e_{1}, e_{2}] = e_{3}$ and its cyclic permutations. With respect to its Killing form $B$, the basis $\{e_{1}, e_{2}, e_{3}\}$ is orthogonal and satisfies $B(e_{i}, e_{i}) = -2$ for $1 \leq i \leq 3$. The special orthogonal group $SO(3)$ is defined to be the group of invertible linear automorphisms preserving $B$ and fixing $e_{1}\wedge e_{2} \wedge e_{3} \in \ext^{3}\ste$.

\begin{lemma}\label{tripletensorlemma}
\noindent
\begin{enumerate}
\item If $(\g, [\dum, \dum])$ is an anticommutative algebra, then $(\trip(\g), \tmlt)$ is isomorphic to $\so(3)\tensor_{\fie}\g$.
\item If $(\alg, \mlt)$ is a commutative algebra, then $(\trip(\alg), \tmlt)$ is isomorphic to $\ealg^{3}\tensor_{\fie}\alg$.
\end{enumerate}
\end{lemma}

\begin{proof}
When $(\g, [\dum,\dum])$ is anticommutative, so that $\ad = L_{[\dum, \dum]} = -R_{[\dum, \dum]}$, \eqref{generaltriple} becomes
\begin{align}\label{nahmmlt}
L_{\tmlt}(x) = \tfrac{1}{2}\begin{pmatrix} 0 & -\ad(x_{3}) & \ad(x_{2}) \\ \ad(x_{3}) & 0 & -\ad(x_{1})\\ -\ad(x_{2}) & \ad(x_{1}) & 0\end{pmatrix},
\end{align}
so that $\trip(\g)$ carries the commutative multiplication
\begin{align}\label{nahmmult}
x \tmlt y = \tfrac{1}{2}\left([x_{2}, y_{3}] + [y_{2}, x_{3}], [x_{3}, y_{1}] + [y_{3}, x_{1}], [x_{1}, y_{2}] + [y_{1}, x_{2}] \right).
\end{align}
There is a basis $\{e_{1}, e_{2}, e_{3}\}$ of $\so(3)$ that satisfies $[e_{1}, e_{2}] = e_{3}$ and its cyclic permutations. The map $\Psi:\so(3)\tensor_{\fie}\g \to \trip(\g)$ defined by $\Psi((a_{1}e_{1} + a_{2}e_{2} + a_{3}e_{3}) \tensor x) = 2(a_{1}x, a_{2}x, a_{3}x)$ is the desired isomorphism.
When $(\alg, \mlt)$ is commutative, \eqref{generaltriple} becomes
\begin{align}\label{lefttriple}
L_{\tmlt}(x) = \tfrac{1}{2}\begin{pmatrix} 0 & L_{\mlt}(x_{3}) & L_{\mlt}(x_{2}) \\ L_{\mlt}(x_{3}) & 0 & L_{\mlt}(x_{1})\\ L_{\mlt}(x_{2}) & L_{\mlt}(x_{1}) & 0\end{pmatrix}.
\end{align}
By Example \ref{3dexample} there is a basis $\{e_{1}, e_{2}, e_{3}\}$ of $\ealg^{3}(\fie)$ that satisfies $e_{1}\mlt e_{1} = 0$, $e_{1}\mlt e_{2} = e_{3}$ and their cyclic permutations. The map $\Psi:\ealg^{3}\tensor_{\fie}\alg \to \trip(\alg)$ defined by $\Psi((a_{1}e_{1} + a_{2}e_{2} + a_{3}e_{3}) \tensor x) = 2(a_{1}x, a_{2}x, a_{3}x)$ is the desired isomorphism.
\end{proof}

When $(\g, [\dum, \dum])$ is a Lie algebra, the algebra $(\nahm(\g) = \trip(\g), \tmlt)$ is the \emph{Nahm algebra} of $(\g, [\dum, \dum])$ defined by M. Kinyon and A. Sagle in \cite{Kinyon-Sagle}. This observation provides additional motivation for studying triple algebras. As will become apparent, some of the results about Nahm algebras shown in \cite{Kinyon-Sagle} carry over to the triple algebras of commutative algebras, although the precise claims and their proofs usually require modification. On the other hand, that the Nahm algebra of a Lie algebra and the triple algebra of a commutative algebra are specializations of a single construction is somewhat misleading, as the following alternative construction of the triple algebra of a commutative algebra suggests. 

Let $(\alg^{4} = \alg\oplus \alg \oplus \alg \oplus \alg, \mlt)$ be the direct sum of $4$ copies of the algebra $(\alg, \mlt)$.
The symmetric group $S_{4}$ acts freely on $\alg^{4}$ by permuting the components, and this action is by automorphisms of $(\alg^{4}, \mlt)$. Let $\ste = \{z \in \alg^{4}: z_{0} + z_{1} + z_{2} + z_{3} = 0\}$ where $z_{i}$ is the projection of $z \in \alg^{4}$ on the component $i$, and the components are indexed by $\{0, 1, 2, 3\}$ and let $\diag(\alg) = \{(z, z, z, z) \in \alg^{4}: z \in \alg\} \subset \alg^{4}$ be the $S_{4}$-invariant diagonal subspace. The projection $\Pi \in \eno(\alg^{4})$ onto $\ste$ along $\diag(\alg)$ is $S_{4}$-equivariant. The linear map $\phi:\trip(\alg) \to \alg^{4}$ defined by 
\begin{align}\label{s4phi}
\phi(x) = \tfrac{1}{2}(x_{1} + x_{2} +x_{3}, x_{1} - x_{2} - x_{3}, -x_{1} + x_{2} - x_{3}, - x_{1} - x_{2} + x_{3})
\end{align}
is injective, with image equal to $\ste$ (this step requires that $\chr \fie \neq 2$). A computation shows that 
\begin{align}\label{tripprojected0}
\phi(x \tmlt y) - \Pi(\phi(x) \mlt \phi(y))  & = \phi\left(\tfrac{1}{2}\left(x_{3}\mlt y_{2} - y_{2}\mlt x_{3}, x_{1} \mlt y_{3} - y_{3}\mlt x_{1}, x_{2} \mlt y_{1} - y_{1}\mlt x_{2} \right)\right).
\end{align}
In particular, if $(\alg, \mlt)$ is commutative, then
\begin{align}\label{tripprojected}
\Pi(\phi(x) \mlt \phi(y)) = \phi(x \tmlt y)
\end{align}
for $x, y \in \trip(\alg)$. For commutative $(\alg, \mlt)$, \eqref{tripprojected} could be taken as the definition of $\tmlt$.

Lemma \ref{tripleautomorphismlemma}, that shows that there are actions of $SO(3)$ and the symmetric group $S_{4}$ by outer automorphisms of $(\trip(\alg), \tmlt)$, when $(\alg, \mlt)$ is, respectively, Lie or commutative. 

The action of $S_{4}$ on $\trip(\alg)$ is induced from the action of $S_{4}$ on $\alg^{4}$ via $\phi$, by $\phi(\si \cdot x) = \si \cdot \phi(x)$ for $x \in \trip(\alg)$ and $\si \in S_{4}$. Consider the group $\Ga$ acting on $\trip(\alg)$ by permutation of the summands (so as $S_{3})$ and changes of signs of an even number (so $0$ or $2$) of components. Let $(ij)$ be the permutation of $i, j \in \{1, 2, 3\}$ and let $(\pm, \pm, \pm)$ indicate the sign change on the components of $\trip(\alg)$. It follows from \eqref{s4phi} that $(ij)$ corresponds to the permutation $(ij)$ of $\alg\oplus \alg \oplus \alg \oplus \alg$, while $(+, -, -)$, $(-, + , -)$, and $(-, -, +)$ correspond respectively to the permutations $(01)(23)$, $(02)(13)$ and $(03)(12)$ of $\{0, 1, 2, 3\}$. Since the product of $(01)(23)$ with $(23)$ is $(01)$ and similarly cycling the the indices, the image of $\Ga$ contains all the transpositions in $S_{4}$, so generates $S_{4}$. As $\Ga$ has the same order as $S_{4}$, it is isomorphic to $S_{4}$, and moreover, the preceding shows that its action is the action induced via $\phi$ defined previously. More explicitly, if $(ij)$ denotes a transposition, and $x = (x_{1}, x_{2}, x_{3}) \in \trip(\alg)$,
\begin{align}\label{s4action}
\begin{aligned}
&(01)\cdot x = (x_{1}, -x_{3}, -x_{2}),& & (02)\cdot x = (-x_{3}, x_{2}, -x_{1}), & & (03)\cdot x = (-x_{2}, -x_{1}, x_{3}),\\
&(12)\cdot x = (x_{2}, x_{1}, x_{3}),& & (13)\cdot x = (x_{3}, x_{2}, x_{1}), & & (23)\cdot x = (x_{1}, x_{3}, x_{2}).
\end{aligned}
\end{align}

\begin{lemma}\label{tripleautomorphismlemma}
Let $\chr \fie = 0$ and let $(\alg, \mlt)$ and $(\balg, \mlt)$ be $\fie$-algebras.
\begin{enumerate}
\item\label{tripauto} If $\Psi \in \hom(\alg, \balg)$ is an algebra homomorphism then $\diag(\Psi)(x) = \sum_{ i = 1}^{3}\nu_{i}(\Psi(x_{i}))$ is an algebra homomorphism from $(\trip(\alg), \tmlt)$ to $(\trip(\balg), \tmlt)$. In particular, the map $\diag: \Aut(\alg, \mlt) \to\Aut(\trip(\alg), \tmlt)$ is an injective group homomorphism.
\item\label{nahmauto} If $(\alg, \mlt) = (\g, [\dum, \dum])$ is a Lie 
algebra, the action of $SO(3)$ on $(\nahm(\g), \tmlt)$ defined by $\nu_{i}(r \cdot x) = R_{i}\,^{j}\nu_{j}(x)$ where $r = R_{i}\,^{j} \in SO(3)$ and $x \in \nahm(\g)$ is an action by automorphisms.
\item\label{s4auto} If $(\alg, \mlt)$ is commutative, the action of $S_{4}$ on $(\trip(\alg), \tmlt)$ defined by $\phi(\si \cdot x) = \si \cdot \phi(x)$ for $x \in \trip(\alg)$ and $\si \in S_{4}$, where $\phi$ is as in \eqref{s4phi}, is an action by automorphisms. 
\end{enumerate}
\end{lemma}

\begin{proof}
The only part of \eqref{tripauto} that requires comment is the injectivity of $\diag$. If $x \in \alg$, then, by definition, $\diag(\Psi)(\diag(x)) = \diag(\Psi(x))$. Hence if $\diag(\Psi)$ is the identity, then $\diag(x) = \diag(\Psi(x))$ for all $x \in \alg$, and this shows $\Psi$ is the identity. Claim \eqref{nahmauto} is \cite[Corollary $9.1$]{Kinyon-Sagle}. Claim \eqref{s4auto} follows from \eqref{tripprojected}, the definition of the $S_{4}$ action on $\trip(\alg)$, that $S_{4}$ acts on $(\alg^{4}, \mlt)$ by automorphisms, and the $S_{4}$-equivariance of the projection $\Pi$. Explicitly, for $\si \in S_{4}$ and $x, y, \in \trip(\alg)$,
\begin{align}
\begin{split}
\phi(\si \cdot x \tmlt \si \cdot y) &= \Pi(\phi(\si \cdot x) \mlt \phi(\si \cdot y)) = \Pi(\si \cdot \phi(x) \mlt \si \cdot \phi(y))  = \Pi(\si\cdot(\phi(x) \mlt \phi(y))\\
& = \si \cdot \Pi(\phi(x) \mlt \phi(y) = \si \cdot \phi(x \tmlt y) = \phi(\si \cdot (x\tmlt y)),
\end{split}
\end{align}
and by the injectivity of $\phi$ this shows $\si \cdot x \tmlt \si \cdot y = \si \cdot (x \tmlt y)$.
\end{proof}

\begin{remark}
Lemma \ref{tripleautomorphismlemma} should be compared with the discussion in \cite[Chapter $5.2$]{Ivanov}.
\end{remark}

\begin{remark}
It follows from Lemma \ref{tripleautomorphismlemma} that a derivation $D$ of $(\alg, \mlt)$ extends to a derivation $\trip(D)$ of $(\trip(\alg), \tmlt)$ defined by $\nu_{i}(\trip(D)x) = Dx_{i}$ for $x \in \trip(\alg)$. If $(\g, [\dum, \dum])$ is a Lie algebra, $\nahm(\g)$ is also a Lie algebra, with the componentwise Lie bracket. For $z \in \g$ the derivation $\ad(z)$ of $\g$ extends to a derivation $\trip(\ad(z))$ of $\trip(\alg)$, that is given by the componentwise Lie bracket with the diagonal image $\diag(z)$, so that $\trip(\ad_{\g}(z)) = \ad_{\trip(\g)}(\diag(z))$. Equivalently,
\begin{align}
[\diag(z), x\tmlt y] = [\diag(z), x]\tmlt y + x \tmlt [\diag(z), y]
\end{align}
for $x, y \in \nahm(\g)$ and $z \in \g$. 
\end{remark}

For a commutative algebra $(\alg, \mlt)$ and a Lie algebra $(\g, [\dum, \dum])$, there follow from \eqref{lefttriple} and \eqref{nahmmlt}
\begin{align}\label{tautrip}
&2\tau_{\tmlt}(x, y) = \sum_{i = 1}^{3}\tau_{\mlt}(x_{i}, y_{i}), & &2\tau_{\nahm(\g)}(x, y) = -\sum_{i = 1}^{3}B_{g}(x_{i}, y_{i}).
\end{align}
The second equality of \eqref{tautrip} is \cite[Theorem $7.2$]{Kinyon-Sagle}. From \eqref{tautrip} and the invariance of the Killing form it follows that $\tau_{\nahm(\g)}$ is invariant \cite[Theorem $7.4$]{Kinyon-Sagle}. By \cite[Theorem $5.1$]{Kinyon-Sagle}, $\nahm(\g)$ is simple if and only if $\g$ is simple, and by \cite[Theorem $5.3$]{Kinyon-Sagle}, $\nahm(\g)$ is a direct sum of simple ideals if and only if $\g$ is semisimple. From these and the preceding remarks there follows:
\begin{theorem}[{\cite[Corollary $7.7$]{Kinyon-Sagle}}]
For a real Lie algebra $\g$ the following are equivalent.
\begin{enumerate}
\item $\g$ is semisimple.
\item The Killing form $B_{\g}$ is nondegenerate. 
\item $\nahm(\g)$ is a direct sum of simple ideals.
\item $\tau_{\nahm(\g)}$ is nondegenerate. 
\end{enumerate}
\end{theorem}

\begin{corollary}\label{nahmcorollary}
If $\g$ is a semisimple (simple) real Lie algebra, then $(\nahm(\g), \mprod, \tau)$ is a semisimple (simple) Killing metrized exact commutative algebra. 
\end{corollary}

The formally analogous results for the triple algebra of a commutative algebra $(\alg, \mlt)$ are also valid, with the changes that the invariance of the trace form of $(\alg, \mlt)$ is no longer automatic, and has to be introduced as a hypothesis, and that the proofs follow different lines. 

\begin{lemma}\label{triplealgebrastructurelemma}
Let $\chr \fie = 0$ and let $(\trip(\alg), \tmlt)$ be the triple of the commutative $\fie$-algebra $(\alg, \mlt)$.
\begin{enumerate}
\item\label{ttrip2} $\tau_{\tmlt}$ is invariant if and only if $\tau_{\mlt}$ is invariant. In this case, the cubic polynomials $P_{\tmlt}$ and $P_{\mlt}$ of $(\trip(\alg), \tmlt, \tau_{\tmlt})$ and $(\alg, \mlt, \tau_{\mlt})$ are related by 
\begin{align}\label{ptrip}
\begin{aligned}
4P_{\tmlt}(x)  = \tau_{\mlt}(x_{1}\mlt x_{2}, x_{3}) &= P_{\mlt}(x_{1} + x_{2} + x_{3}) - P_{\mlt}(x_{1} + x_{2}) - P_{\mlt}(x_{1} + x_{3}) - P_{\mlt}(x_{2} + x_{3}) \\&\quad + P_{\mlt}(x_{1}) + P_{\mlt}(x_{2}) + P_{\mlt}(x_{3}).
\end{aligned}
\end{align}
In particular $2P_{\tmlt}(\diag(x)) = 3P_{\mlt}(x)$ for $x \in \alg$.
\item\label{ttrip1} $\tau_{\tmlt}$ is nondegenerate if and only if $\tau_{\mlt}$ is nondegenerate.
\item\label{ttrip1b} If $\fie = \rea$, $\tau_{\tmlt}$ is positive definite if and only if $\tau_{\mlt}$ is positive definite.
\end{enumerate}
\end{lemma}

\begin{proof}
That the invariance of $\tau_{\mlt}$ implies the invariance of $\tau_{\tmlt}$ is immediate from \eqref{tripmlt} and \eqref{tautrip}. By \eqref{tautrip}, $2\tau_{\tmlt}(\Delta(x), \Delta(y)) = 3\tau_{\mlt}(x, y)$. The invariance of $\tau_{\tmlt}$ implies that of $\tau_{\mlt}$ follows from this and the fact that $\diag$ is a homomorphism. By \eqref{tautrip} and the invariance of $\tau_{\mlt}$,
\begin{align}\label{ptrip0}
\begin{split}
12P_{\tmlt}(x) &= 2\tau_{\tmlt}(x\tmlt x, x) =  \tau_{\mlt}(x_{2}\mlt x_{3}, x_{1}) +  \tau_{\mlt}(x_{3}\mlt x_{1}, x_{2}) + \tau_{\mlt}(x_{1}\mlt x_{2}, x_{3}) \\
& = 3\tau_{\mlt}(x_{1}\mlt x_{2}, x_{3}).
\end{split}
\end{align}
The second equality of \eqref{ptrip} is true in general. 
This shows \eqref{ttrip2}. 
If $x \in \trip(\alg)$ satisfies $\tau_{\tmlt}(x, y) = 0$ for all $y \in \trip(\alg)$, then, by \eqref{tautrip}, $0 = 2\tau_{\tmlt}(x, \sum_{i = 1}^{3}\al_{i}\nu_{i}(v)) = \tau_{\mlt}(\sum_{i = 1}^{3}\al_{i}x_{i}, v) = 0$ for all $v \in \alg$ and $\al_{i} \in \fie$. If $\tau_{\mlt}$ is nondegenerate, this implies all linear combinations of $x_{1}$, $x_{2}$, and $x_{3}$ vanish, so $x_{1} = x_{2} = x_{3} = 0$, which shows $x = 0$. If $x \in \alg$ is such that $\tau_{\mlt}(x, y) = 0$ for all $y \in \alg$, then $\tau_{\tmlt}(\diag(x), y) = \tau_{\mlt}(x, y_{1}) +  \tau_{\mlt}(x, y_{2}) +  \tau_{\mlt}(x, y_{3})  = 0$ for all $y \in \trip(\alg)$. If $\tau_{\tmlt}$ is nondegenerate, this implies $\diag(x) = 0$, so $x = 0$, because $\diag$ is injective. This proves \eqref{ttrip1}. 
Suppose $\fie = \rea$. From $2\tau_{\tmlt}(\nu_{1}(x), \nu_{1}(x)) = \tau_{\mlt}(x, x)$ for $x \in \alg$, it follows that $\tau_{\mlt}$ is positive definite if $\tau_{\tmlt}$ is positive definite, while from $2\tau_{\tmlt}(x, x) = \sum_{i = 1}^{3}\tau_{\mlt}(x_{i}, x_{i})$ for $x \in \trip(\alg)$ there follows the converse. This shows \eqref{ttrip1b}. 
\end{proof}

\begin{theorem}\label{triplealgebratheorem}
Let $\chr \fie = 0$. The triple algebra $(\trip(\alg), \tmlt, \tau_{\tmlt})$ of a commutative $\fie$-algebra $(\alg, \mlt)$ is Killing metrized if and only if $(\alg, \mlt)$ is Killing metrized.
\end{theorem}
\begin{proof}
This follows from \eqref{ttrip2} and \eqref{ttrip1} of Lemma \ref{triplealgebrastructurelemma}.
\end{proof}

\begin{lemma}\label{tripideallemma}
Let $(\trip(\alg), \tmlt)$ be the triple of the $\fie$-algebra $(\alg, \mlt)$ where $\chr \fie \neq 2$.
Let $J$ be an ideal in $\trip(\alg)$ and define $J_{i} = \pi_{i}(J) \subset \alg$. Then $J_{1}\cap J_{2} \cap J_{3}$ and $J_{1} + J_{2} + J_{3}$ are ideals in $(\alg, \mlt)$.
\end{lemma}

\begin{proof}
Let $a \in \alg$ and $x \in J$. By definition $2\nu_{1}(x_{2}\mlt a) + \nu_{2}(a\mlt x_{1}) = \nu_{3}(a)\tmlt x \in J$, so $x_{2}\mlt a = 2\pi_{1}(\nu_{3}(a)\tmlt x) \in J_{1}$ and $a \mlt x_{1} = 2\pi_{2}(\nu_{3}(a)\tmlt x) \in J_{2}$. These relations remain valid under cyclic permutations of the indices. This shows $L_{\mlt}(\alg)J_{1}\subset J_{2}$, $L_{\mlt}(\alg)J_{2}\subset J_{3}$, $L_{\mlt}(\alg)J_{3}\subset J_{1}$, $R_{\mlt}(\alg)J_{2} \subset J_{1}$, $R_{\mlt}(\alg)J_{3} \subset J_{2}$, and $R_{\mlt}(\alg)J_{1} \subset J_{3}$. Hence $L_{\mlt}(\alg)(J_{1}\cap J_{2} \cap J_{3}) \subset J_{1}\cap J_{2}\cap J_{3}$, $R_{\mlt}(\alg)(J_{1}\cap J_{2} \cap J_{3}) \subset J_{1}\cap J_{2}\cap J_{3}$, $L_{\mlt}(\alg)(J_{1} +  J_{2} + J_{3}) \subset J_{1}+  J_{2} +  J_{3}$, and $R_{\mlt}(\alg)(J_{1} +  J_{2} + J_{3}) \subset J_{1}+  J_{2} +  J_{3}$, showing that $J_{1}\cap J_{2}\cap J_{3}$ and $J_{1} + J_{2} + J_{3}$ are two-sided ideals of $(\alg, \mlt)$. 
\end{proof}

The proof of Theorem \ref{tripsimpletheorem} is motivated by the proof of \cite[Theorem $5.1$]{Kinyon-Sagle}, but even in the antisymmetric case the argument is slightly stronger, as it makes no use of the Jacobi identity.
\begin{theorem}\label{tripsimpletheorem}
If an $\fie$-algebra $(\alg, \mlt)$, where $\chr \fie \neq 2$, is either commutative or anticommutative, then $(\trip(\alg), \tmlt)$ is simple if and only if $(\alg, \mlt)$ is simple.
\end{theorem}

\begin{proof}
Let $J$ be an ideal in $\trip(\alg)$ and define $J_{i} = \pi_{i}(J) \subset \alg$. Since $R_{\mlt} = \pm L_{\mlt}$, the proof of Lemma \ref{tripideallemma} shows $L_{\mlt}(\alg) J_{1} \subset J_{2}\cap J_{3}$ and the similar inclusions obtained by permuting the indices cyclically. There follows $\alg \mlt (\alg \mlt J_{1}) \subset \alg \mlt (J_{2} \cap J_{3}) \subset (J_{1}\cap J_{3}) \cap (J_{1} \cap J_{2}) =  J_{1} \cap J_{2} \cap J_{3}$, and similarly with $J_{2}$ or $J_{3}$ in place of $J_{1}$. If $\alg$ is simple then, by Lemma \ref{tripideallemma}, either $J_{1} \cap J_{2} \cap J_{3} = \{0\}$ or $J_{1} \cap J_{2} \cap J_{3} = \alg$. In the latter case, $J = \trip(\alg)$, so suppose $J_{1} \cap J_{2} \cap J_{3} = \{0\}$. Then if $x, y \in \alg$ and $a\in J_{i}$ there holds $x\mlt(y\mlt a) = 0$. Since $\alg$ is simple, $\alg\mlt \alg = \alg$. The set of $z \in \alg$ such that $z \mlt u = 0$ for all $u \in \alg$ is an ideal. Were it equal to $\alg$, then the multiplication in $\alg$ would be trivial, so this set is $\{0\}$. Hence $x \mlt(y\mlt a) = 0$ for all $x, y \in \alg$ implies $y \mlt a = 0$ for all $y \in \alg$, which in turn implies $a = 0$. Hence $J_{1} = J_{2} = J_{3} = \{0\}$, and so $J = \{0\}$. This shows that if $\alg$ is simple, then $\trip(\alg)$ is simple. Now suppose $\trip(\alg)$ is simple and let $J$ be an ideal in $\alg$. Then $ \nu_{1}(J) + \nu_{2}(J) + \nu_{3}(J) = \{(x, y, z) \in \trip(\alg): x, y, z \in J\}$ is an ideal in $\trip(\alg)$. If it equals $\trip(\alg)$ then $J = \alg$, while if it equals $\{0\}$ then $J = \{0\}$, so $\alg$ is simple. 
\end{proof}

\begin{remark}
Because, by Example \eqref{rtripleexample}, $\trip(\fie)$ is isomorphic to $\ealg^{3}(\fie)$, Theorems \ref{triplealgebratheorem} and \ref{tripsimpletheorem} give an alternative proof that $\ealg^{3}(\fie)$ is a simple Killing metrized exact commutative algebra.
\end{remark}

\begin{example}\label{permanenttripleexample}
The triple algebra $\trip(\ealg^{3}(\fie))$ is a simple $9$-dimensional Killing metrized exact commutative algebra. Let $\ga_{0}, \ga_{1}, \ga_{2}, \ga_{3}$ be the generators of $\ealg^{3}(\fie)$ as in \eqref{ealgrelations}. The vectors $\{e_{1} = -(\ga_{2} + \ga_{3}), e_{2} = -(\ga_{1} + \ga_{3}), e_{3} = -(\ga_{1} + \ga_{2})\}$ are $\tau_{\mlt}$-orthogonal, satisfy $\tau_{\mlt}(e_{i}, e_{i}) = 2$, and satisfy \eqref{n3alg} with $c = 1$.
By \eqref{tautrip}, $\{f_{j + 3(i - 1)} = \nu_{i}(e_{j}): 1 \leq i, j \leq 3\}$ is a $\tau_{\tmlt}$-orthonormal basis of $\trip(\ealg^{3}(\fie))$. It follows from \eqref{ptrip} that, for $x = \sum_{i = 1}^{9}x_{i}f_{i}$, 
\begin{align}
\begin{aligned}
x \tmlt x &= (x_{5}x_{9} + x_{6}x_{8})f_{1} + (x_{6}x_{7} + x_{4}x_{9})f_{2} + (x_{4}x_{8} + x_{5}x_{7})f_{3}\\
& + (x_{8}x_{3} + x_{9}x_{2})f_{4} + (x_{9}x_{1} + x_{7}x_{3})f_{5} + (x_{7}x_{2} + x_{8}x_{1})f_{6}\\
&+  (x_{2}x_{6} + x_{3}x_{5})f_{7} + (x_{3}x_{4} + x_{1}x_{6})f_{8} + (x_{1}x_{5} + x_{2}x_{4})f_{9},
\end{aligned}
\end{align}
and the cubic polynomial of $(\trip(\ealg^{3}(\fie)), \tmlt, \tau_{\tmlt})$ is
\begin{align}
\label{permpoly}
2P_{\tmlt}(x)  & = \begin{aligned}
& x_{1}x_{5}x_{9} + x_{2}x_{6}x_{7} + x_{3}x_{4}x_{8} \\
 &+ x_{1}x_{6}x_{8} + x_{2}x_{4}x_{9} + x_{3}x_{5}x_{7}
\end{aligned}= \perm \begin{pmatrix} x_{1} & x_{2} & x_{3} \\ x_{4} & x_{5} & x_{6}\\ x_{7} & x_{8} & x_{9}\end{pmatrix},
\end{align}
which is the permanent of the indicated $3 \times 3$ matrix.
\end{example}

\begin{example}\label{so3nahmexample}
Identifying $x \in \fie^{3}$ with $\hat{x} \in \so(3)$ defined by $\hat{x}y = x \times y$ gives an isomorphism $(\rea^{3}, \times) \to (\so(3), [\dum, \dum])$. Let $e_{1}, e_{2}, e_{3}$ be the standard basis of $\so(3)$ such that $e_{1} \times e_{2} = e_{3}$, $e_{2} \times e_{3} = e_{1}$ and $e_{3} \times e_{1} = e_{2}$. The basis is orthogonal with respect to the Killing form $B_{\so(3)}(x, y) = \tr \hat{x}\hat{y}$ and satisfies $B_{\so(3)}(e_{i}, e_{i}) = -2$. By \eqref{tautrip}, $\{f_{j + 3(i - 1)} = \nu_{i}(e_{j}): 1 \leq i, j \leq 3\}$ is a $\tau_{\tmlt}$-orthonormal basis of $\nahm(\so(3))$. It follows from \eqref{ptrip} that, for $x = \sum_{i = 1}^{9}x_{i}f_{i}$, 
\begin{align}
\begin{aligned}
x \tmlt x &= (x_{5}x_{9} - x_{6}x_{8})f_{1} + (x_{6}x_{7} - x_{4}x_{9})f_{2} + (x_{4}x_{8} - x_{5}x_{7})f_{3}\\
& + (x_{8}x_{3} - x_{9}x_{2})f_{4} + (x_{9}x_{1} - x_{7}x_{3})f_{5} + (x_{7}x_{2} - x_{8}x_{1})f_{6}\\
&+  (x_{2}x_{6} - x_{3}x_{5})f_{7} + (x_{3}x_{4} - x_{1}x_{6})f_{8} + (x_{1}x_{5} - x_{2}x_{4})f_{9},
\end{aligned}
\end{align} 
and 
\begin{align}
\label{nahmpolyso3}
2P(x)  &  = \begin{aligned}
&x_{1}x_{5}x_{9} + x_{2}x_{6}x_{7} + x_{3}x_{4}x_{8} \\
- &x_{1}x_{6}x_{8} - x_{2}x_{4}x_{9} - x_{3}x_{5}x_{7}
\end{aligned}= \det \begin{pmatrix} x_{1} & x_{2} & x_{3} \\ x_{4} & x_{5} & x_{6}\\ x_{7} & x_{8} & x_{9}\end{pmatrix}
\end{align}
is the cubic polynomial $P(x) = \tfrac{1}{6}\tau_{\tmlt}(x\tmlt x, x)$ corresponding to $\tmlt$.
\end{example}

\begin{lemma}[M. Kinyon - A. Sagle; \cite{Kinyon-Sagle}]\label{kslemma}
For a finite-dimensional real Lie algebra $\g$ a nonzero element $x \in \nahm(\g)$ is idempotent if and only if the components $x_{i} = \pi_{i}(x) \in \g$ of $x$ generate a subalgebra of $\g$ isomorphic to $\so(3)$.
\end{lemma}
\begin{proof}
That $x$ be idempotent implies $x_{1} = [x_{2}, x_{3}]$, $x_{2} = [x_{3}, x_{1}]$, and $x_{3} = [x_{1}, x_{2}]$. As is shown at the end of section $4$ of \cite{Kinyon-Sagle} (and is easy to check), these relations cannot be satisfied unless $x_{1} \wedge x_{2}\wedge x_{3} \neq 0$, in which case the $x_{i}$ generate a subalgebra of $\g$ isomorphic to $\so(3)$.
\end{proof}

\begin{example}\label{203nahmidempotentsexample}
Let $e_{1}, e_{2}, e_{3}$ be the standard basis of $\so(3)$ such that $e_{1} \times e_{2} = e_{3}$, $e_{2} \times e_{3} = e_{1}$ and $e_{3} \times e_{1} = e_{2}$. For an orientation preserving rotation $g$ of $\rea^{3}$ let $g_{i} = ge_{i}$ and note that $g_{1}\times g_{2} = g_{3}$, etc. Then $X_{g} = \nu_{1}(g_{1}) + \nu_{2}(g_{2}) + \nu_{3}(g_{3}) \in \nahm(\g)$ is idempotent. Let $X_{0}$ be $X_{g}$ for $g$ the identity transformation. It is straightforward to check that
\begin{align}
-2X_{0} \tmlt y = (-y_{5} - y_{9} , y_{4} , y_{7} , y_{2} , -y_{1} - y_{9} , y_{8} , y_{3} , y_{6} , -y_{1} - y_{5}).
\end{align}
The orthocomplement of $X_{0}$ is given by $y_{1} + y_{5} + y_{9} = 0$. On this orthocomplement, the eigenvalues of $L_{\tmlt}(X_{0})$ are $1/2$ with $3$-dimensional eigenspace spanned by the eigenvectors $\nu_{1}(e_{2}) - \nu_{2}(e_{1})$, $\nu_{1}(e_{3}) - \nu_{3}(e_{1})$, and $\nu_{2}(e_{3}) - \nu_{3}(e_{2})$, and $-1/2$, with $5$-dimensional eigenspace spanned by the eigenvectors $\nu_{1}(e_{1}) - \nu_{3}(e_{3})$, $\nu_{2}(e_{2}) - \nu_{3}(e_{3})$, $\nu_{1}(e_{2}) + \nu_{2}(e_{1})$, $\nu_{1}(e_{3}) + \nu_{3}(e_{1})$, and $\nu_{2}(e_{3}) + \nu_{3}(e_{2})$.

For $\g = \so(3)$, any triple $\{x_{1}, x_{2}, x_{3}\}$ of elements as in the proof of Lemma \ref{kslemma} is the image of the standard triple $\{e_{1}, e_{2}, e_{3}\}$ under an element of $SO(3)$. It follows that every idempotent of $\nahm(\so(3))$ has the form $X_{g}$. Since all the idempotents have the same norm, and by \eqref{minimalidempotentlemma} of Lemma \ref{criticalpointlemma} there is necessarily a minimal idempotent, all the $X_{g}$ are minimal idempotents. However as the eigenvalues of $L_{\tmlt}(X_{g})$ are the same as those of $L_{\tmlt}(X_{0})$, and among them $1/2$ occurs with multiplicity $3$, no $X_{g}$ has orthogonal spectrum contained in $(-\infty, 0]$, and $\nahm(\so(3))$ has no minimal idempotent with orthogonal spectrum contained in $(-\infty, 0]$.

Let $f_{i} = 2\nu_{i}(e_{i})$ where $\{e_{1}, e_{2}, e_{3}\}$ is the standard orthonormal basis of $\so(3)$. The $f_{i}$ satisfy the relations \eqref{n3alg}  (with $c = 1$) defining the algebra $\ealg^{3}$, showing that $\nahm(\so(3))$ contains $\ealg^{3}$ as a subalgebra. This is observed, albeit not quite in these terms, in \cite[example $5.2$]{Kinyon-Sagle}.
\end{example}

\begin{theorem}\label{compactnahmtheorem}
For a compact simple real Lie algebra, $\g$, the Nahm algebra $(\nahm(\g), \mlt, \tau)$ is a simple Killing metrized exact commutative algebra that has sectional nonassociativities of both signs. In particular it is not isomorphic to a direct sum of real simplicial algebras. 
\end{theorem}

\begin{proof}
Let $(\g, [\dum, \dum])$ be a real Lie algebra. The subspace $\diag(\g) = \{x \in \nahm(\g): \nu_{1}(x) = \nu_{2}(x) = \nu_{3}(x)\}$ is a $\mlt$-trivial subalgebra. Its $\tau_{\mlt}$-orthocomplement $\diag(\g)^{\perp}$ comprises those $ x\in \nahm(\g)$ for which $\nu_{1}(x) + \nu_{2}(x) + \nu_{3}(x) = 0$. These satisfy $\diag(\g)\mlt \diag(\g) = \{0\}$, $\diag(\g) \mlt \diag(\g)^{\perp} \subset \diag(\g)^{\perp}$, and $\diag(\g)^{\perp}\mlt \diag(\g)^{\perp} \subset \diag(\g)$. Since $(\diag(\g), \mlt)$ is an associative subalgebra of $\nahm(\g)$, $\isect_{\mlt}(x, y) = 0$ for $x, y \in \diag(\g)$. Straightforward calculations using \eqref{tautrip} and the invariance of $B_{\g}$ yield
\begin{align}\label{nahmsect}
\begin{split}
&\tau_{\mlt}(x \mlt x, y \mlt y) - \tau_{\mlt}(x\mlt y, x \mlt y) \\
&= \tfrac{1}{2}\cycle_{\{1, 2, 3\}}
\left(- B_{\g}([x_{1}, x_{2}], [y_{1}, y_{2}]) + \tfrac{1}{4}B_{\g}([x_{1}, y_{2}] + [y_{1}, x_{2}], [x_{1}, y_{2}] + [y_{1}, x_{2}])\right)\\
& =  \tfrac{1}{2}\cycle_{\{1, 2, 3\}}\left(- B_{\g}([x_{1}, y_{1}], [x_{2}, y_{2}]) + \tfrac{1}{4}B_{\g}([x_{1}, y_{2}] - [y_{1}, x_{2}], [x_{1}, y_{2}] - [y_{1}, x_{2}])\right),
\end{split}
\end{align}
in which $\cycle_{\{1, 2, 3\}}$ indicates the sum over the cyclic permutations of the indicated indices. Suppose $\g$ is simple and compact, so that $B_{\g}$ is negative definite. If $x \in \diag(\g)$, then the first term on the second line of \eqref{nahmsect} vanishes and so, as the second term on the second line of \eqref{nahmsect} is nonpositive, it follows that $\isect_{\mlt}(x, y) \leq 0$ whenever $x \in \diag(\g)$.
Since $\g$ is simple there are $a$ and $b$ in $\g$ such that $[a, b] \neq 0$. Let $x = \nu_{1}(a) + \nu_{2}(b)$, $y = \nu_{1}(b) - \nu_{2}(a)$, and $z = \nu_{1}(a) - \nu_{2}(b)$. Since $\nu_{i}(a) \mlt \nu_{i}(a) = 0$ for $i = 1, 2, 3$, and $2\nu_{1}(a) \mlt \nu_{2}(b) = \nu_{3}([a, b])$ there hold $x \mlt y = 0$, $x \mlt z = 0$, and $x \mlt x = y \mlt y = - z \mlt z$. Hence $\tau_{\mlt}(x \mlt x, y \mlt y) - \tau_{\mlt}(x\mlt y, x \mlt y) = -\tfrac{1}{2}B_{\g}([a, b], [a, b]) > 0$, and $\tau_{\mlt}(x \mlt x, z \mlt z) - \tau_{\mlt}(x \mlt z, x \mlt z) = \tfrac{1}{2}B_{\g}([a, b], [a, b]) < 0$, the inequalities because $\g$ is compact. This shows that $\isect_{\mlt}(x, y) > 0$ and $\isect_{\mlt}(x, z) < 0$, so $(\nahm(\g), \mlt, \tau)$ has sectional nonassociativities of both signs. By Lemma \ref{sectdirectsumlemma}, a direct sum of real simplicial algebras has nonpositive sectional nonassociativity, so $(\nahm(\g), \mlt, \tau)$ is not isomorphic to any such direct sum.
\end{proof}

What follows shows how conclusions like those of Theorem \ref{compactnahmtheorem} can be obtained for triple algebras.
Let $(\trip(\alg), \tmlt)$ be the triple algebra of the algebra $(\alg, \mlt)$. 
For $i \in \{0, 1, 2, 3\}$ define linear maps $\Ga_{i}:\alg \to \trip(\alg)$ by 
\begin{align}
&\Ga_{0}(x)  = \diag(x) ,&& \Ga_{1}(x) = (x, -x, -x),&&\Ga_{2}(x)  = (-x, x, -x), &&\Ga_{3}(x) = (-x, -x, x),
\end{align}
and note $\Ga_{0}(x) + \Ga_{1}(x) + \Ga_{2}(x) + \Ga_{3}(x) = 0$. From \eqref{s4action} it follows that, for $i \in \{1, 2, 3\}$, $\Ga_{i}$ is the composition of an element of $S_{4}$ with $\Ga_{0}$. For example $\Ga_{1} = (01)\cdot \Ga_{0}$. Using that $S_{4}$ acts by algebra automorphisms, it is straightforward to check the relations
\begin{align}\label{gammagammageneral}
\begin{split}
&\Ga_{i}(x) \tmlt \Ga_{i}(y) = \tfrac{1}{2}\Ga_{i}(x \mlt y + y \mlt x),\\
&\Ga_{i}(x)\tmlt \Ga_{j}(y)  + \Ga_{j}(x)\tmlt \Ga_{i}(y) = -\tfrac{1}{2}\left(\Ga_{i}(x\mlt y + y \mlt x) + \Ga_{j}(x \mlt y + y \mlt x)\right),\\
&\Ga_{i}(x)\tmlt \Ga_{j}(y)  - \Ga_{j}(x)\tmlt \Ga_{i}(y) = \tfrac{\sign(ijkl)}{2}\left(\Ga_{k}(x\mlt y - y \mlt x) - \Ga_{l}(x \mlt y - y \mlt x)\right)
\end{split}
\end{align}
for $x, y \in \alg$, where $\{i, j, k, l\}$ is a permutation of $\{0, 1, 2, 3\}$ and $\sign(ijkl)$ is $\pm 1$ as it is an even or odd permutation.
The first identity of \eqref{gammagammageneral} shows that $\Ga_{i}(\alg)$ is a subalgebra of $(\trip(\alg), \tmlt)$ that is trivial if $\mlt$ is antisymmetric. 

Let $\nabla(\alg) = \{x \in \trip(\alg): x_{1} + x_{2} + x_{3} = 0\}$. 
For $i, j \in \{0, 1,2, 3\}$ define linear maps $\nabla^{ij}:\alg \to \nabla(\alg) \subset \trip(\alg)$ by $\nabla^{ij}= \tfrac{1}{2}(\Ga_{i} - \Ga_{j})  = -\nabla^{ji}$. For example $\nabla^{23}(x) = \tfrac{1}{2}(\Ga_{2}(x) - \Ga_{3}(x)) = (0, x, -x) $. Any two of $\nabla^{12}(x)$, $\nabla^{23}(x)$, $\nabla^{31}(x)$ constitute a basis of $\nabla(\alg)$. 
There hold
\begin{align}\label{gammanablageneral}
\Ga_{0}(x) \tmlt \nabla^{12}(y) = \tfrac{1}{2}\left(\nabla^{23}(x \mlt y) + \nabla^{31}(y \mlt x)\right),
\end{align}
and the identities obtained from it by cyclically permuting the indices $\{1, 2, 3\}$. This shows that $\diag(\alg)\tmlt \nabla(\alg) \subset \nabla(\alg)$.  
If $\tau_{\tmlt}$ is nondegenerate then the orthogonal complement $\diag(\alg)^{\perp}$ of $\diag(\alg)$ equals $\nabla(\alg)$ and $\diag(\alg)\mlt \diag(\alg)^{\perp} \subset \diag(\alg)^{\perp}$. 

Now suppose $(\alg, \mlt)$ is commutative. In this case, the relations \eqref{gammagammageneral} and \eqref{gammanablageneral} simplify to
\begin{align}\label{gammagamma}
&\begin{aligned}
&\Ga_{i}(x) \tmlt \Ga_{i}(y) = \Ga_{i}(x \mlt y),&
&\Ga_{i}(x)\tmlt \Ga_{j}(y) = -\tfrac{1}{2}\left(\Ga_{i}(x\mlt y) + \Ga_{j}(x \mlt y)\right),
\end{aligned}\\
\label{gammanabla}
&\Ga_{i}(x)\tmlt \nabla^{jk}(y) = -\tfrac{1}{2}\nabla^{jk}(x\mlt y),&
\end{align}
for distinct $i,j, k \in \{0, 1,2, 3\}$.
The first equation of \eqref{gammagamma} shows that $\Ga_{i}:\alg \to \trip(\alg)$ is an injective algebra homomorphism for all $i \in \{0, 1, 2, 3\}$.

Define $\nabla_{i}:\alg \to \nabla(\alg) \subset \trip(\alg)$ by $\nabla_{i} = \diag - 3\nu_{i} = \nabla^{ji} + \nabla^{ki}$ where $\{i, j, k\} = \{1, 2, 3\}$.
For example, $\nabla_{1}(x) = (-2x, x, x)$. Note that $\nabla_{i}(\alg)\cap \nabla_{j}(\alg) = \{0\}$ if $i \neq j$ and $\nabla_{1}(x) + \nabla_{2}(x) + \nabla_{3}(x) = 0$. Straightforward computations show that, for $x, y \in \alg$ and $a, b, c, d \in \fie$,
\begin{align}\label{nablanabla}
\begin{split}
(a\nabla^{13}(x) + b\nabla^{23}(x))& \tmlt (c\nabla^{13}(y) + d\nabla^{23}(y)) \\
& = - bd \nabla^{13}(x\mlt y) - ac \nabla^{23}(x \mlt y) -\tfrac{1}{2}(ad + bc)\Ga_{0}(x \mlt y),\\
\nabla^{12}(x)&\tmlt(a\nabla^{13}(y) + b\nabla^{23}(y)) = \tfrac{a+b}{2}\nabla^{12}(x\mlt y) + \tfrac{b-a}{2}\nu_{3}(x\mlt y).
\end{split}
\end{align}
Specializing \eqref{nablanabla} in different ways yields
\begin{align}\label{nablanabla2}
\begin{split}
&\nabla_{i}(x\mlt y) + \nabla_{i}(x)\tmlt \nabla_{i}(y) = -\diag(x \mlt y),\\
&\nabla_{i}(x \mlt y) + \nabla_{j}(x) \tmlt \nabla_{k}(y) = \tfrac{1}{2}\diag(x \mlt y),\\
&\diag(x) \tmlt \nabla_{i}(y) = -\tfrac{1}{2}\nabla_{i}(x \mlt y),
\end{split}
\end{align}
where $\{i, j, k\}$ is any permutation of $\{1, 2, 3\}$. 

From \eqref{tautrip} there follow the relations
\begin{align}
\label{tautripgamma}
&\tau_{\tmlt}(\Ga_{i}(x), \Ga_{i}(y)) = \tfrac{3}{2}\tau_{\mlt}(x, y), && \tau_{\tmlt}(\Gamma_{i}(x), \Gamma_{j}(y)) = -\tfrac{1}{2}\tau_{\mlt}(x, y),& \\
\label{tautripnabla}
&\tau_{\tmlt}(\nabla_{i}(x), \nabla_{i}(y)) = 3\tau_{\mlt}(x, y), && \tau_{\tmlt}(\nabla_{i}(x), \nabla_{j}(y)) = -\tfrac{3}{2}\tau_{\mlt}(x, y),&\\
\label{tautripdiag}
& \tau_{\tmlt}(\nabla^{ij}(x), \nabla^{ij}(y)) = \tau_{\mlt}(x, y),&&\tau_{\tmlt}(\nabla^{ij}(x), \nabla^{ik}(y)) = \tfrac{1}{2}\tau_{\mlt}(x, y),&
\end{align} 
for $x, y \in \alg$, where distinct indices take on distinct values in the admissible ranges $\{0, 1,2 ,3\}$ (in \eqref{tautripgamma}) or $\{1, 2, 3\}$ (in \eqref{tautripnabla} and \eqref{tautripdiag}).

\begin{theorem}\label{ealgtripletheorem}
Over a field $\fie$ of characteristic $0$, the triple algebra $\trip(\ealg^{n}(\fie)) \simeq \ealg^{3}(\fie)\tensor \ealg^{n}(\fie)$ is a simple Killing metrized exact commutative algebra that is not isomorphic to $\ealg^{3n}(\fie)$.
\end{theorem}

\begin{proof}
That $\trip(\ealg^{n}(\fie))$ is a simple Killing metrized exact commutative algebra follows from Theorems \ref{tripsimpletheorem} and \ref{triplealgebratheorem}.  The spectrum of a minimal idempotent $e \in \midem(\ealg^{n}(\fie))$ contains only $-1/(n-1)$ with multiplicity $n-1$. Let $y \in \ealg^{n}(\fie)$ be an eigenvector with eigenvalue $-1/(n-1)$. By \eqref{tautripgamma} and \eqref{nablanabla2},
\begin{align}
\begin{split}
[\diag(e), \diag(e), \nabla_{i}(y)]_{\tmlt} & - \tfrac{1}{3n-1}\tau_{\tmlt}(\diag(e), \diag(e))\nabla_{i}(y) + \tfrac{1}{3n-1}\tau_{\tmlt}(\diag(e), \nabla_{i}(y))\diag(e)\\
& = \diag(e)\tmlt \nabla_{i}(y) + \tfrac{1}{2}\diag(e)\tmlt\nabla_{i}(e\mlt y) - \tfrac{3}{2(3n-1)}\tau_{\mlt}(e, e)\nabla_{i}(y) \\
& = \left(\tfrac{1}{2(n-1)} + \tfrac{1}{4(n-1)^{2}} - \tfrac{3n}{2(3n-1)(n-1)}\right)\nabla_{i}(y) =\tfrac{1}{4(3n-1)(n-1)}\nabla_{i}(y) .
\end{split}
\end{align}
This shows $\trip(\ealg^{n}(\fie))$ is not intrinsically projectively associative, so is not isomorphic to $\ealg^{3n}(\fie)$.
\end{proof}

Let $(\alg, \mlt)$ be a commutative algebra.
For $x \in \alg$ define a $3$-dimensional subspace 
\begin{align}
\ealg(x) = \spn \{\Ga_{i}(x): 0 \leq i \leq 3\}
\end{align}
and a two-dimensional subspace 
\begin{align}
\begin{split}
\zalg(x) &= \spn\{\nabla^{ij}(x): i, j \in \{1, 2, 3\}\}  = \spn\{\nabla_{i}(x): i \in \{1, 2, 3\}\} \\
&=  \{(ax, bx, -(a + b)x) \in \trip(\alg): a, b \in \fie\}. 
\end{split}
\end{align}
Suppose $w \in \zalg(x) \cap \zalg(y)$. Then there are $a, b, c, d \in \fie$ such that $(ax, bx, -(a+b)x) = w = (cy, dy, -(c+d)y)$, and this can be so only if $x$ and $y$ are linearly dependent. Hence, if $x, y \in \alg$ are linearly independent, then $\zalg(x) \cap \zalg(y) = \{0\}$. 

\begin{lemma}\label{tripidemlemma}
Let $(\trip(\alg), \tmlt)$ be the triple of the commutative $\fie$-algebra $(\alg, \mlt)$, where $\chr \fie = 0$. 
\begin{enumerate}
\item If $e \in \idem(\alg, \mlt)$, then $\Ga_{i}(e) \in \idem(\trip(\alg), \tmlt)$ for $i \in \{0, 1, 2, 3\}$ and:
\begin{enumerate}
\item The subspace $\ealg(e)$ of $(\trip(\alg), \tmlt)$ is a subalgebra isomorphic to $\ealg^{3}(\fie)$.
\item For $i \in \{0, 1, 2, 3\}$, $-1/2$ is an eigenvalue of $L_{\tmlt}(\Ga_{i}(e))$ with multiplicity at least $2$.
\end{enumerate}
\item If, for $x \in \alg$, $\la$ is an eigenvalue of $L_{\mlt}(x)$ with multiplicity $m$, then $-\la/2$ is an eigenvalue of $L_{\tmlt}(\diag(x))$ with multiplicity at least $2m$.
\item If $z \in \szero(\alg, \mlt)$, then $\nabla_{i}(z) \in \szero(\trip(\alg), \tmlt)$ for $1 \leq i \leq 3$ and the subspace $\ealg(z)$ is a trivial subalgebra of $(\trip(\alg), \tmlt)$.
\end{enumerate}
\end{lemma}

\begin{proof}
That, for $e \in \idem(\alg, \mlt)$, $\ealg(e)$ is a subalgebra of $(\trip(\alg), \tmlt)$ and, for $z \in \szero(\alg, \mlt)$, $\ealg(z)$ is a trivial subalgebra of $(\trip(\alg), \tmlt)$, are immediate from \eqref{gammagamma}. From \eqref{gammagamma} it follows that $\Ga_{i}(e)$, $i \in \{0, 1, 2, 3\}$, satisfy the relations \eqref{ealgrelations}, so $\ealg(e)$ is isomorphic to $\ealg^{3}(\fie)$. By \eqref{gammanabla}, $\Ga_{i}(e)\tmlt \nabla^{jk}(e) = -\tfrac{1}{2}\nabla^{jk}(e)$ for distinct $i, j, k \in \{0, 1, 2, 3\}$. This shows the two-dimensional subspace $\zalg(e)$ is contained in the $-1/2$-eigenspace of $L_{\tmlt}(\Ga_{i}(e))$, so the multiplicity of $-1/2$ as an eigenvalue is at least $2$.
If $L_{\mlt}(x)y = \la y$, then, by \eqref{gammanabla}, $L_{\tmlt}(\diag(x))\nabla^{ij}(y) = -(\la/2)\nabla^{ij}(y)$, so $\zalg(y)$ is contained in the $-1/2$-eigenspace of $L_{\tmlt}(\Delta(x))$. Since the subspaces $\zalg(y)$ corresponding to linearly independent elements of the $\la$-eigenspace of $L_{\mlt}(x)$ are pairwise transverse, if $\la$ is an eigenvalue of multiplicity $m$ of $L_{\mlt}(x)$, then $-\la/2$ is an eigenvalue of $L_{\tmlt}(\diag(x))$ having multiplicity at least $2m$.
\end{proof}

\begin{lemma}\label{tripsectlemma}
Let $(\trip(\alg), \tmlt)$ be the triple of a Killing metrized commutative $\fie$-algebra $(\alg, \mlt)$, where $\chr \fie = 0$. The intrinsic sectional nonassociativities of $(\trip(\alg), \mlt, \tau_{\tmlt})$ and $(\alg, \mlt, \tau_{\mlt})$ are related by
\begin{align}
\label{secttripgaigai}
\isect_{\tmlt}(\Ga_{i}(x), \Ga_{i}(y)) &= \tfrac{2}{3}\isect_{\mlt}(x, y), && i \in \{0, 1,2, 3\},\\
\label{secttripnainai}
\isect_{\tmlt}(\nabla_{i}(x), \nabla_{i}(y)) &= \tfrac{1}{2}\isect_{\mlt}(x, y),&& i \in \{1,2, 3\},\\
\label{secttripgaigaj}
\isect_{\tmlt}(\Ga_{i}(x), \Ga_{j}(y)) &= -2 \tfrac{\tau_{\mlt}(x\mlt x, y \mlt y) + |x \mlt y|_{\tau_{\mlt}}^{2}}{9|x|^{2}_{\tau_{\mlt}}|y|^{2}_{\tau_{\mlt}} -\tau_{\mlt}(x, y)^{2}}, && i \neq j \in \{0, 1,2, 3\},\\
\label{secttripnainaj}
\isect_{\tmlt}(\nabla_{i}(x), \nabla_{j}(y)) &= -\tfrac{3}{2}\tfrac{|x \mlt y|_{\tau_{\mlt}}^{2}}{4|x|^{2}_{\tau_{\mlt}}|y|^{2}_{\tau_{\mlt}} - \tau_{\mlt}(x, y)^{2}}, && i \neq j \in \{1,2, 3\},\\
\label{secttripdiagnai}
\isect_{\tmlt}(\diag(x), \nabla_{i}(y)) &= - \tfrac{1}{6}\tfrac{2\tau_{\mlt}(x\mlt x, y \mlt y) + |x \mlt y|_{\tau_{\mlt}}^{2}}{|x|^{2}_{\tau_{\mlt}}|y|^{2}_{\tau_{\mlt}}}, && i  \in \{1,2, 3\},
\end{align}
for $x, y \in \alg$ for which these expressions are defined. In particular:
\begin{enumerate}
\item\label{secttripeij} If $e \in \idem(\alg, \mlt)$ is $\tau_{\mlt}$-anisotropic, then $\isect_{\tmlt}(\nabla_{i}(e), \nabla_{j}(e)) = -(1/2)\tau_{\mlt}(e, e)^{-1}$ and $\isect_{\tmlt}(\Ga_{i}(e), \Ga_{j}(e)) = -(3/4)\tau_{\mlt}(e, e)^{-1}$.
\item\label{secttripenainaj} If $x \in \alg$ is $\tau_{\mlt}$-anisotropic and $y \in \alg$ is an eigenvector of $L_{\mlt}(x)$ that is $\tau_{\mlt}$-orthogonal to $x$ and having eigenvalue $\la$, then $\isect_{\tmlt}(\nabla_{i}(x), \nabla_{j}(y)) = - (3\la^{2}/8)\tau_{\mlt}(x, x)^{-1}$.
\item\label{secttripegaigaj} If $e \in \idem(\alg, \mlt)$ is $\tau_{\mlt}$-anisotropic and $y \in \alg$ is an eigenvector of $L_{\mlt}(e)$ that is $\tau_{\mlt}$-orthogonal to $e$ and having eigenvalue $\la$, then $\isect_{\tmlt}(\Ga_{i}(e), \Ga_{j}(y)) = - (2\la(1+\la)/9)\tau_{\mlt}(e, e)^{-1} = (2(1 + \la)/(9(1 - \la)))\isect_{\mlt}(e, y)$.\item\label{secttripzgaigaj} If $z \in \szero(\alg, \mlt)$ is $\tau_{\mlt}$-anisotropic and $y \in \alg$ is an eigenvector of $L_{\mlt}(e)$ that is $\tau_{\mlt}$-orthogonal to $e$ and having eigenvalue $\la$, then $\isect_{\tmlt}(\Ga_{i}(z), \Ga_{j}(y)) = - (2\la^{2}/9)\tau_{\mlt}(z, z)^{-1} = (2\la/(9(1 - \la)))\isect_{\mlt}(z, y)$.
\item\label{secttripidem} If $e \in \idem(\alg, \mlt)$ and $y$ is a vector $\tau_{\mlt}$-orthogonal to $e$ and contained in the $\la$-eigenspace of $L_{\mlt}(e)$, then
$\isect_{\tmlt}(\diag(e), w) = -(\la(2 + \la)/6)\tau_{\mlt}(e, e)^{-1} = ((\la + 2)/(6(\la - 1))\isect_{\mlt}(e, y)$, for any $w \in \zalg(y)$, where the second equality is valid provided $\la \neq 1$.
\item\label{secttripsz} If $z \in \szero(\alg,\mlt)$ and $y$ is an eigenvector of $L_{\mlt}(z)$ that is $\tau_{\mlt}$-orthogonal to $z$ and having eigenvalue $\la$, then
$\isect_{\tmlt}(\diag(z), w) = -(\la^{2}/6)\tau_{\mlt}(z, z)^{-1}  = (1/6)\isect_{\mlt}(z, y)$, for any $w \in \zalg(y)$.
\end{enumerate}
\end{lemma}
\begin{proof}
The identities \eqref{secttripgaigai} and \eqref{secttripgaigaj} follow from \eqref{tautrip} and \eqref{tautripgamma}.
The identity \eqref{secttripnainai} follows from \eqref{tautrip} and \eqref{tautripnabla}. Straightforward computations using \eqref{tautrip}, \eqref{nablanabla}, and \eqref{tautripdiag} yield \eqref{secttripnainaj} and \eqref{secttripdiagnai}. 
Taking $x = y = e$ in \eqref{secttripgaigaj} and \eqref{secttripnainaj} yields claim \eqref{secttripeij}.
Specializing \eqref{secttripnainaj} yields \eqref{secttripenainaj}.
Specializing \eqref{secttripgaigaj} yields \eqref{secttripegaigaj} and \eqref{secttripzgaigaj}.
Specializing \eqref{secttripgaigai}-\eqref{secttripdiagnai} yields \eqref{secttripidem} and \eqref{secttripsz}.
By the proof of Lemma \ref{tripidemlemma}, if $L_{\mlt}(x)y = \la y$ then any $w \in \zalg(y)$ satisfies $L_{\tmlt}(\diag(x))w = -(\la/2)w$. The identities \eqref{secttripidem} and \eqref{secttripsz} follow from \eqref{sectnadefined} using $w \in \zalg(y) \subset \diag(\alg)^{\perp}$.
\end{proof}

\begin{corollary}
\noindent
\begin{enumerate}
\item\label{tripnotnonneg} The triple algebra of a Euclidean Killing metrized commutative algebra does not have nonnegative intrinsic sectional nonassociativity.
\item\label{tripnotnonpos} The triple algebra of a Euclidean Killing metrized commutative algebra with nonpositive intrinsic sectional nonassociativity does not have nonpositive intrinsic sectional nonassociativity.
\end{enumerate}
\end{corollary}

\begin{proof}
A Euclidean Killing metrized commutative algebra contains an idempotent $e$, so, by \eqref{secttripeij} of Lemma \ref{tripsectlemma}, $\isect_{\tmlt}(\nabla_{i}(e), \nabla_{j}(e) < 0$. This shows claim \eqref{tripnotnonneg}. By Lemma \ref{criticalpointlemma}, a Euclidean Killing metrized commutative algebra $(\alg, \mlt)$ contains a minimal idempotent $e$. If $L_{\mlt}(e)z = \la z$ and $\tau_{\mlt}(e, z) =0$, then nonpositive sectional nonassociativity implies $0 \geq \tau_{\mlt}(e\mlt e, z \mlt z) -\tau_{\mlt}(e\mlt z, e\mlt z) = \la(1-\la)\tau_{\mlt}(z, z)$, so $\la \leq 0$ and $\specp(e) \subset (-\infty, 0]$. If $(\alg, \mlt)$ is moreover exact, then $L_{\mlt}(e)$ must have a negative eigenvalue $\la$. By claim \eqref{secttripidem} of Lemma \ref{triplealgebrastructurelemma}, if $y \in \alg$ is an eigenvector of $L_{\mlt}(e)$ with eigenvalue $\la$, then $\sect(\diag(e), \nabla_{i}(y)) > 0$. This shows claim \eqref{tripnotnonpos}.
\end{proof}

\section{Concluding remarks and speculations}\label{conclusionsection}
This section records some remarks about topics not treated here and indicates directions for further study.

\subsection{}
There are several crude dichotomies that help to organize Killing metrized exact commutative algebras. There is the distinction between algebras whose automorphism groups are Lie groups and those whose automorphism groups are discrete. There is the distinction between nonnegative and nonpositive sectional nonassociativity (Norton inequality and reverse Norton inequality), or, more generally, lower bounds and upper bounds on sectional nonassociativity. In very imprecise terms these distinctions overlap in that Lie groups of automorphisms are associated with lower bounds on sectional nonassociativity (nonnegativity) while upper bounds on sectional nonassociativity (nonpositivity) tend to force discreteness of automorphism groups. Substance is given to such claims by the example of formally real simple Euclidean Jordan algebras and their deunitalizations (Lemma \ref{hermsectlemma} and Corollary \ref{jordandeunitalizationsectcorollary}) and results such as Lemma \ref{finiteautolemma}. From this point of view that Griess algebras whose automorphism groups are finite satisfy the Norton inequality (nonnegativity) could appear somewhat anomalous, but could be explained via a more refined study of the relation between sectional nonassociativity and the spectra of the multiplication endomorphisms of idempotents. The machinery developed by Tkachev in \cite{Tkachev-universality} facilitates precise formulations, and these ideas will be pursued in joint work with Tkachev.

\subsection{}
It is implicit in the work of R. Hamilton and G. Huisken on the Ricci flow that a certain commutative multiplication on the space of metric curvature tensors discovered by Hamilton makes this space into a metrized commutative $\rea$-algebra. This is explained in \cite{Fox-curvtensor}, where it is shown that the subspaces of Weyl and Kähler Weyl curvature tensors constitute subalgebras that are Killing metrized and exact. This an interesting family of examples. Among the examples considered in the present paper they most closely resemble the deunitalizations of simple Euclidean Jordan algebras in that their automorphism groups are Lie groups and consequently their idempotents constitute submanifolds (these are studied in detail in \cite{Fox-curvtensor}). Several interesting questions arise:
\begin{enumerate}
\item Obtain estimates of the sectional nonassociativities of the algebras of Weyl curvature tensors. The author expects there is a lower bound similar to that in Corollary \ref{jordandeunitalizationsectcorollary}. 
\item As vector spaces the spaces of Weyl and Kähler Weyl curvature tensors are naturally identified with the second Cartan powers of the Lie algebras $\so(n)$ and $\su(n)$. There should be a general construction of a commutative product on the second Cartan power of a compact simple real Lie algebra $\g$ that makes it a Killing metrized exact commutative algebra and that recovers the Hamilton product when $\g$ is $\so(n)$ o $\su(n)$. 
\item The geometric motivations lead to considering these algebras over $\rea$. It should be interesting to study them over more general fields, starting with $\com$.
\end{enumerate}

\subsection{}
In \cite{Fox-cubicpoly} it was shown how to associate with a tight frame in a Euclidean vector space a harmonic cubic polynomial. This polynomial determines an exact metrized commutative algebra. For a balanced tight two-distance unit norm frame, the author has shown that the associated exact metrized commutative algebra is Killing metrized. A tight two-distance frame determines a strongly regular graph. In fact these combinatorial objects are essentially equivalent; \cite[Theorem $1.2$]{Barg-Galzyrin-Okoudjou-Yu} characterizes nonequiangular tight two-distance frames in terms of strongly regular graphs and \cite{Waldron} charaterizes equiangular tight frames in similar terms. A strongly regular graph in turn determines an association scheme \cite{Bannai-Ito, Cameron-Goethals-Seidel-stronglyregular}. When it satisfies some hypothesis, a joint eigenspace of the Bose-Mesner algebra of an association scheme carries a structure of a commutative nonassociative algebra, called a \emph{Norton algebra}, introduced in \cite[Section $6$]{Cameron-Goethals-Seidel}, \cite{Cameron-Goethals-Seidel-stronglyregular}. For a balanced tight two-distance unit frame, the algebra determined by the cubic polynomial of the frame as in \cite{Fox-cubicpoly} can be shown to be isomorphic to the Norton algebra of the associated association scheme (the author did not understand this when he wrote \cite{Fox-cubicpoly}). The conclusion that these Norton algebras are Killing metrized is new, and the notion of sectional nonassociativity offers interesting possibilities for studying these Norton algebras. These examples are particularly interesting because explicit computations are viable because of the detailed combinatorial information available regarding strongly regular graphs. However, developing the background necessary for their presentation to a readership not specializing in the study of the relevant combinatorial objects requires considerable space, and so these results will be treated in a companion paper.

A general question suggested by these examples is for what metrized commutative exact algebras does some subset of $\idem(\alg, \mlt)$ have some particular combinatorial structure. For example, when does some subset of $\idem(\alg, \mlt)$ constitute a two-distance tight frame? When it does, to what extent does such combinatorial structure on its idempotents determine the algebra up to isomorphism?

\subsection{}
It would be interesting to adapt the notions considered here to work over fields of positive characteristic, in particular over finite fields. Interesting examples are known. The Harada-Norton group is the sporadic simple group $HN$ of order $2^{14}\cdot 3^{6}\cdot 5^{6}\cdot 7\cdot 11 \cdot 19$. In \cite{Ryba}, A. Ryba constructed a $133$-dimensional metrized commutative algebra defined over the finite field $\mathbb{F}_{5}$ and it follows from Theorem $1$, Lemma $9.1(ii)$, and Theorems $5$ and $6$ of \cite{Ryba} that this algebra is Killing metrized with automorphism group equal to $HN$. That the invariant bilinear form equals twice the Killing form is claim $(ii)$ of Lemma $9.1$, while the invariance is Theorem $6$. 

\appendix
\section{The Chern-do Carmo-Kobayashi inequality over real Hurwitz algebras}\label{bwappendix}
For a real Hurwitz algebra $\fie$, let $\mat(n, \fie)$ be the space of $n \times n$ matrices over $\fie$. Let $\herm(n, \fie) = \{X \in \mat(n, \fie): \bar{X}^{t} = X\}$ where the conjugate transpose $X \to \bar{X}^{t}$ is induced by the involution of $\fie$ fixing the real subfield. The Frobenius bilinear form $f(X, Y) = \re \tr \bar{X}^{t}Y = \tfrac{1}{2}\tr(\bar{X}^{t}Y + \bar{Y}^{t}X)$ is symmetric \cite[Proposition V.$2.1$]{Faraut-Koranyi} and positive definite (this is true for $\fie = \cayley$ even if $n > 3$).

The Böttcher-Wenzel inequality, first proved for real symmetric matrices by Chern-do Carmo-Kobayashi \cite{Chern-Docarmo-Kobayashi}, is an optimal bound on the Frobenius norm of the commutator of two complex matrices in terms of the product of the Frobenius norms of the matrices, that is stronger than the bound obtained via a naive application of the Cauchy-Schwarz inequality. This appendix recalls these inequalities, shows that the Chern-do Carmo-Kobayashi inequality holds for $\herm(n, \quat)$ and partially for $\herm(3, \cayley)$, and makes some remarks about the proofs of both. 
These refinements are used in the main text to obtain bounds on sectional nonassociativities. Theorem \ref{bwtheorem} and the variant Lemma \ref{cdklemma}, valid over a real Hurwitz algebra, are used in the proof of Lemma \ref{hermsectlemma} to estimate the sectional nonassociativity of the deunitalization of a simple Euclidean Jordan algebra. In a sense, this amounts to reinterpreting the  Böttcher-Wenzel and Chern-do Carmo-Kobayashi inequalities as upper bounds on sectional nonassociativity. 

\begin{lemma}\label{bwreductionlemma}
Let $\fie$ be a real Hurwitz algebra. For $n \geq 1$, the number
\begin{align}\label{bwdefined}
\bw = \bw(\mat(n, \fie)) = \sup\left\{\tfrac{|[X, Y]|^{2}_{f}}{|X|_{f}^{2}|Y|_{f}^{2}}: X, Y \in \mat(n, \fie) \setminus\{0\}\right\}
\end{align}
is finite and
\begin{align}\label{bwpre}
|[X, Y]|^{2}_{f} \leq \bw(|X|_{f}^{2}|Y|_{f}^{2} - f(X, Y)^{2}) \leq \bw|X|_{f}^{2}|Y|_{f}^{2},
\end{align}
for all $X, Y \in \mat(n, \fie)$. If $|[X, Y]|^{2}_{f} \leq \bw(|X|_{f}^{2}|Y|_{f}^{2} - f(X, Y)^{2})$ then $P = |Y|_{f}X - |X|_{f}Y$ and $Q = |Y|_{f}X + |X|_{f}Y$ satisfy $|[P, Q]|^{2}_{f} = \bw|P|_{f}^{2}|Q|_{f}^{2}$.
\end{lemma}
\begin{proof}
That $\bw(\mat(n, \fie))$ is finite follows from the Cauchy-Schwarz inequality. The following argument from the proof of \cite[Theorem $2.2$]{Wu-Liu} shows that the apparently weaker inequality $|[X, Y]|^{2}_{f} \leq \bw(\mat(n, \fie))|X|_{f}^{2}|Y|_{f}^{2}$ implies the first inequality of \eqref{bwpre}. If either $X$ or $Y$ equals $0$ there is nothing to show, so suppose $X \neq 0$ and $Y \neq 0$. Take $P = |Y|_{f}X - |X|_{f}Y$ and $Q = |Y|_{f}X + |X|_{f}Y$. Then $[P, Q] = 2|X|_{f}|Y|_{f}[X, Y]$, and straightforward computation shows
\begin{align}\label{bwreduction}
\begin{split}
|[X, Y]|^{2}_{f} = \tfrac{1}{4|X|^{2}_{f}|Y|^{2}_{f}}|[P, Q]|^{2} \leq \tfrac{1}{2|X|^{2}_{f}|Y|^{2}_{f}}|P|^{2}_{f}|Q|^{2}_{f} = 2(|X|_{f}^{2}|Y|_{f}^{2} - f(X, Y)^{2}).
\end{split}
\end{align}
If there is equality in the first inequality of \eqref{bwpre}, then \eqref{bwreduction} shows that $|[P, Q]|^{2}_{f} = \bw|P|_{f}^{2}|Q|_{f}^{2}$.
\end{proof}

\begin{theorem}[\cite{Chern-Docarmo-Kobayashi}, \cite{Bottcher-Wenzel}]\label{bwtheorem}
There holds $\bw(\mat(n, \com)) = 2$. That is for $X, Y \in \mat(n, \com)$, 
\begin{align}\label{bw}
|[X, Y]|^{2}_{f} \leq 2(|X|_{f}^{2}|Y|_{f}^{2} - f(X, Y)^{2}) \leq 2|X|_{f}^{2}|Y|_{f}^{2},
\end{align}
for the inner product $f(X, Y) = \re \tr \bar{X}^{t}Y$.
\begin{enumerate}
\item\label{submax0} If either $X, Y \in \herm(n, \com)$ or $X, Y \in \su(n)$, then $|[X, Y]|^{2}_{f} = 2|X|_{f}^{2}|Y|_{f}^{2}$ if and only if $X$ and $Y$ are simultaneously unitarily conjugate to scalar multiples of $e_{11} - e_{nn}$ and $e_{1n} + e_{n1}$.
\item\label{submax} If either $X, Y \in \herm(n, \com)$ or $X, Y \in \su(n)$, then $|[X, Y]|^{2}_{f} = 2(|X|_{f}^{2}|Y|_{f}^{2} - f(X, Y)^{2})$ if and only if $X$ and $Y$ either are linearly dependent over $\com$ or are simultaneously unitarily conjugate to scalar multiples of $e_{11} - e_{nn}$ and $e_{1n} + e_{n1}$.
\end{enumerate}
\end{theorem}

\begin{proof}
The inequality \eqref{bw} as such was first proved for real symmetric matrices, with a characterization of the case of equality $|[X, Y]|^{2}_{f} = 2|X|_{f}^{2}|Y|_{f}^{2}$, as \cite[Lemma $1$]{Chern-Docarmo-Kobayashi}. It was proved for $\mat(n, \com)$ in full generality by Z. Lu in \cite{Lu-normalscalar} and as stated in the proof of \cite[Theorem $3.1$]{Bottcher-Wenzel}. A nice exposition of the proof of \eqref{bw}, following \cite{Audenaert}, can be found in \cite[chapter $9$]{Zhan}. Its history is somewhat complicated; see \cite{Audenaert, Bottcher-Wenzel-howbig, Bottcher-Wenzel, Cheng-Vong-Wenzel, Lu-normalscalar, Lu-remarks,Lu-Wenzel,  Zhan} among many references. 
Characterizations of equality in \eqref{bw} are given in various references, for instance \cite{Bottcher-Wenzel}. Here these characterizations are needed when $X, Y \in \herm(n, \com)$.
Since multiplication by $\j$ is an isometric real linear isomorphism from $\su(n)$ to $\herm(n, \com)$, the validity of either of \eqref{submax0} or \eqref{submax} for $X$ and $Y$ both Hermitian or both skew-Hermitian implies its validity for the other class. The characterization of equality in \eqref{submax0} for Hermitian matrices is proved as it is in \cite{Chern-Docarmo-Kobayashi} for $\herm(n, \rea)$; see the proof of Lemma \ref{cdklemma} below. Suppose there holds $|[X, Y]|^{2}_{f} = 2(|X|_{f}^{2}|Y|_{f}^{2} - f(X, Y)^{2})$ in \eqref{bw} for linearly independent $X, Y \in \herm(n, \com)$. Then \eqref{bwreduction} shows that the Hermitian matrices $P$ and $Q$ satisfy $|[P, Q]|^{2} = 2|P|^{2}_{f}|Q|^{2}_{f}$, so, by \eqref{submax0}, are simultaneously unitarily conjugate to (nonzero) scalar multiples of $e_{11} - e_{nn}$ and $e_{1n} + e_{n1}$. Since this implies that $|P|^{2}_{f}= |Q|^{2}_{f}$, from $X = \tfrac{P+Q}{2|X|}$ and $Y = \tfrac{Q - P}{2|Y|}$ it follows that $f(X, Y) = 0$, so, by \eqref{bwreduction}, there holds $|[X, Y]|^{2}_{f} = 2|X|_{f}^{2}|Y|_{f}^{2}$. By \eqref{submax0}, $X$ and $Y$ are simultaneously unitarily conjugate to scalar multiples of $e_{11} - e_{nn}$ and $e_{1n} + e_{n1}$. This proves \eqref{submax}. 
\end{proof}

The inequality \eqref{bw} is not valid over $\quat$ and $\cayley$. Taking $X = \j I$ and $Y = \mathsf{j}I \in \mat(n, \quat)$ shows $\bw(\mat(n, \fie)) \geq 4$ when $\fie \in \{\quat, \cayley\}$. In \cite[Theorem $3.1$]{Ge-Li-Zhou} it is shown that $\bw(\mat(n, \quat)) = 4$ and the equality case is characterized. Nonetheless, the proof of \eqref{bw} given in \cite{Chern-Docarmo-Kobayashi} for $\herm(n, \rea)$ adapts for $\herm(n, \fie)$ over any real Hurwitz algebra $\fie \in \{\rea, \com, \quat, \cayley\}$. Over $\rea$ or $\com$, because a Hermitian matrix is unitarily diagonalizable, it suffices to prove the claim in the case one of the matrices is diagonal. This proof requires the prinicipal axis theorem, which is valid over $\quat$ and $\cayley$, and one has to check that the noncommutativity of $\quat$ and $\cayley$ does not effect the rest of the argument. This works over $\quat$, but the reduction via diagonalization faces an obstacle over $\cayley$, namely that, for $X, Y \in \herm(n, \fie)$, the inequality \eqref{bw} is not manifestly invariant with respect to the action of $\Aut(\herm(n, \fie), \star)$. Over associative $\fie$, $\Aut(\herm(n, \fie))$ preserves the commutator of the matrix product on $\mat(n, \fie)$ restricted to $\herm(n, \fie)$, but when $\fie = \cayley$, in which case $\Aut(\herm(3, \cayley))$ is the compact form of the simple real Lie group of type $F_{4}$ \cite{Schafer}, it is not clear to the author whether this is true. However the partial result, that \eqref{bw} holds when $X$ is diagonal, is true, and this suffices for the application to sectional nonassociativity in the proof of Lemma \ref{hermsectlemma}.

\begin{lemma}\label{cdklemma}
Let $\fie$ be a real Hurwitz algebra. 
Let $f(X, Y) = \re \tr \bar{X}^{t}Y$ on $\mat(n, \fie)$.
\begin{enumerate}
\item If $\fie$ is associative, for all $X, Y \in \herm(n, \fie)$ there holds 
\begin{align}\label{cdk}
|[X, Y]|^{2}_{f} \leq 2(|X|_{f}^{2}|Y|_{f}^{2} - f(X, Y)^{2})\leq 2|X|_{f}^{2}|Y|_{f}^{2}.
\end{align}
\item If $\fie = \cayley$ and $n = 3$, \eqref{cdk} holds for all $Y \in \herm(3, \cayley)$ and all \emph{diagonal} $X \in \herm(3, \cayley)$.
\item\label{maximalcommutator} The equality $|[X, Y]|^{2}_{f} = 2|X|_{f}^{2}|Y|_{f}^{2}$ holds in \eqref{cdk} (where, if $\fie = \cayley$ and $n = 3$, $X$ is assumed diagonal) if and only if $X$ and $Y$ are simultaneously equivalent via an automorphism of the Jordan algebra $(\herm(n, \fie), \star)$ to real multiples of $e_{11} - e_{nn}$ and $e_{1n} + e_{n1}$.
\item\label{submaximalcommutator} The equality $|[X, Y]|^{2}_{f} = 2(|X|_{f}^{2}|Y|_{f}^{2} - f(X, Y)^{2})$ holds in \eqref{cdk} (where, if $\fie = \cayley$ and $n = 3$, $X$ is assumed diagonal) if and only if $X$ and $Y$ are either linearly dependent over $\fie$ or are simultaneously equivalent via an automorphism of the Jordan algebra $(\herm(n, \fie), \star)$ to real multiples of $e_{11} - e_{nn}$ and $e_{1n} + e_{n1}$. 
\end{enumerate}
\end{lemma}
\begin{proof}
Suppose $X \in \herm(n, \fie)$ is diagonal, where $n = 3$ if $\fie = \cayley$. Since $X$ is Hermitian, its entries are real. Let $x_{i} = X_{ii}$. Then $[X, Y]_{ij} = x_{i}Y_{ij} - Y_{ij}x_{j} = (x_{i} - x_{j})Y_{ij}$, the last equality because $x_{j}$ is real, so in the center of $\fie$. Now the proof goes through as in \cite{Chern-Docarmo-Kobayashi}:
\begin{align}\label{cdk0}
\begin{split}
|[X, Y]|^{2}_{f} &= \sum_{i, j = 1}^{n}(x_{i} -x_{j})^{2}|Y_{ij}|^{2}  \leq 2\sum_{i, j = 1}^{n}(x_{i}^{2} + x_{j}^{2})|Y_{ij}|^{2} \leq 2\sum_{k = 1}^{n}x_{k}^{2}\sum_{i, j = 1}^{n}|Y_{ij}|^{2} = 2|X|^{2}_{f}|Y|^{2}_{f}.
\end{split}
\end{align}
This proves $|[X, Y]|^{2}_{f} \leq 2|X|^{2}_{f}|Y|^{2}_{f}$ when $X$ is diagonal. By the principal axis theorem, any element of $(\herm(n, \fie), \star)$ is equivalent via an automorphism of $(\herm(n, \fie), \star)$ to a diagonal matrix. When $\fie = \cayley$ (so $n = 3$) this is due to \cite[Theorem $5.1$]{Freudenthal} (see also \cite[Theorem V.$2.5$]{Faraut-Koranyi})). Since an automorphism of $(\herm(n, \fie), \star)$ preserves the trace, it is isometric, and, when $\fie$ is associative it preserves the the commutator of the matrix product on $\mat(n, \fie)$ restricted to $\herm(n, \fie)$, when $\fie$ is associative to prove $|[X, Y]|^{2}_{f} \leq 2|X|^{2}_{f}|Y|^{2}_{f}$ in general it suffices to prove it when $X$ is diagonal. By Lemma \ref{bwreductionlemma}, in either case $|[X, Y]|^{2}_{f} \leq 2|X|^{2}_{f}|Y|^{2}_{f}$ implies \eqref{cdk}.

The characterization of the equality case \eqref{maximalcommutator} goes through as in the proof of \cite[Lemma $1$]{Chern-Docarmo-Kobayashi}. Precisely, if there is equality in \eqref{cdk0}, then 
\begin{align}
0 = 4\sum_{i = 1}^{n}x_{i}^{2}|Y_{ii}|^{2} + 2\sum_{i\neq j}(x_{i} + x_{j}))^{2}|Y_{ij}|^{2} =  2\sum_{k = 1}^{n}x_{k}^{2}\sum_{i = 1}^{n}|Y_{ii}|^{2} + 2\sum_{i\neq j}(x_{i} + x_{j}))^{2}|Y_{ij}|^{2}.
\end{align}
If $X$ is nonzero, this implies $Y_{ii} = 0$ for $1 \leq i \leq n$, $x_{i} + x_{j} = 0$ when $Y_{ij} \neq 0$ for $i \neq j$, and that at exactly two of $x_{1}, \dots, x_{n}$ are nonzero. Suppose $i \neq j$ are the indices such that $x_{i} = -x_{j} \neq 0$. It follows that $Y_{kl} = 0$ if $\{k, l\} \neq \{i, j\}$. Hence $X$ and $Y$ are scalar multiples of $e_{ii} - e_{jj}$ and $e_{ij} + e_{ji}$, respectively. Since $\aut(\herm(n, \fie))$ contains a subgroup acting as permutations on the diagonal subalgebra of $\herm(n, \fie)$, it can be assumed that $i = 1$ and $j = n$. The proof of the equality case \eqref{submaximalcommutator} follows from \eqref{maximalcommutator} exactly as the in the proof of \eqref{submax} of Theorem \ref{bwtheorem}.
\end{proof}
 
\begin{remark}
Observe that although the same difficulty related to diagonalization occurs in the proof of Lemma \ref{hermsectlemma} as occurs in the proof of Lemma \ref{cdklemma}, the end result in Lemma \ref{hermsectlemma} is stated in terms manifestly invariant with respect to $\Aut(\herm(n, \fie))$, and as a result takes the same form whether or not $\fie$ is associative. This observation suggests that the upper bound on sectional nonassociativity, which is closely related to the commutator bound, is the more natural bound to consider, at least from the algebraic point of view, although it could also simply reflect a technical deficiency in the proof of Lemma \ref{cdklemma}.
\end{remark}

\begin{remark}
It would be useful to know whether \eqref{cdk} is true for $\herm(3, \cayley)$. This seems likely. More, generally, it would be interesting to know whether \eqref{cdk} is true for $\herm(n, \cayley)$ for $n > 3$. This would follow from a principal axis theorem and a characterization of the automorphism group, but it appears that neither is known. See \cite{Zohrabi-Zumanovich} for what is known in this regard.
\end{remark}

\begin{remark}
Although this does not seem to be realized widely, for skew-Hermitian matrices the inequality \eqref{bw} is closely related to Vinberg's results on invariant norms on compact simple Lie algebras in \cite{Vinberg-invariantnorms} applied in the special case of $\su(n)$. Precisely, Vinberg shows that, for an invariant norm, $||\dum||$ on a compact simple Lie algebra $\g$ the quantity $\theta(x) = \sup_{0 \neq y \in \g}\tfrac{||[x, y]||}{||y||}$ does not depend on the choice of invariant norm, and equals the spectral norm of $\ad_{\g}(x)$. Since by definition $\theta([x, y]) \leq \theta(x)\theta(y)$, taking $\g$ to be a compact simple Lie algebra of matrices, e.g. skew-symmetric or skew-Hermitian matrices, and taking $||\dum ||$ to be the Frobenius norm there results $|[X, Y]|_{f} \leq \theta(X)|Y|_{f}$. It suffices then to check that the spectral norm of $\ad_{\g}(X)$ is no greater than $\sqrt{2}|X|_{f}$. In the case $\g = \su(n)$ this can be shown as follows. Since any skew-Hermitian matrix is unitarily conjugate to a diagonal matrix and the spectral norms of the adjoint representations of unitarily conjugate matrices are the same, it suffices to consider the case of diagonal $X \in \su(n)$. In this case the nonzero eigenvalues of $\ad_{\su(n)}(X)$ have the form $\la_{i} - \la_{j}$ where $\la_{1}, \dots, \la_{n}$ are the diagonal elements of $X$. Since $|\la_{i} - \la_{j}|^{2} \leq 2|\la|^{2} \leq 2|X|^{2}_{f}$, where $\la$ is the element of $X$ with the greatest modulus, the spectral norm of $\ad_{\su(n)}(X)$ is no greater than $\sqrt{2}|X|_{f}$. In the case of $\g = \so(n)$, something equivalent is shown in \cite{Bloch-Iserles}. 
\end{remark}

\bibliographystyle{amsplain}

\begin{thebibliography}{10}

\bibitem{Albert-radical}
A.~A. Albert, \emph{The radical of a non-associative algebra}, Bull. Amer.
  Math. Soc. \textbf{48} (1942), 891--897.

\bibitem{Albert-powerassociative}
\bysame, \emph{Power-associative rings}, Trans. Amer. Math. Soc. \textbf{64}
  (1948), 552--593.

\bibitem{Ash-Mumford-Rapoport-Tai}
A.~Ash, D.~Mumford, M.~Rapoport, and Y.-S. Tai, \emph{Smooth compactifications
  of locally symmetric varieties}, second ed., Cambridge Mathematical Library,
  Cambridge University Press, Cambridge, 2010.

\bibitem{Audenaert}
K.~M.~R. Audenaert, \emph{Variance bounds, with an application to norm bounds
  for commutators}, Linear Algebra Appl. \textbf{432} (2010), no.~5,
  1126--1143.

\bibitem{Bannai-Ito}
E.~Bannai and T.~Ito, \emph{Algebraic combinatorics. {I}}, The
  Benjamin/Cummings Publishing Co., Inc., Menlo Park, CA, 1984, Association
  schemes.

\bibitem{Barg-Galzyrin-Okoudjou-Yu}
A.~Barg, A.~Glazyrin, K.~A. Okoudjou, and W.-H. Yu, \emph{Finite two-distance
  tight frames}, Linear Algebra Appl. \textbf{475} (2015), 163--175.

\bibitem{Beilinson-Drinfeld-chiral}
A.~Beilinson and V.~Drinfeld, \emph{Chiral algebras}, American Mathematical
  Society Colloquium Publications, vol.~51, American Mathematical Society,
  Providence, RI, 2004.

\bibitem{Benito-Draper-Elduque}
P.~Benito, C.~Draper, and A.~Elduque, \emph{On some algebras related to simple
  {L}ie triple systems}, J. Algebra \textbf{219} (1999), no.~1, 234--254.

\bibitem{Bloch-Iserles}
A.~M. Bloch and A.~Iserles, \emph{Commutators of skew-symmetric matrices},
  Internat. J. Bifur. Chaos Appl. Sci. Engrg. \textbf{15} (2005), no.~3,
  793--801.

\bibitem{Bordemann}
M.~Bordemann, \emph{Nondegenerate invariant bilinear forms on nonassociative
  algebras}, Acta Math. Univ. Comenian. (N.S.) \textbf{66} (1997), no.~2,
  151--201.

\bibitem{Bottcher-Wenzel-howbig}
A.~B{\"o}ttcher and D.~Wenzel, \emph{How big can the commutator of two matrices
  be and how big is it typically?}, Linear Algebra Appl. \textbf{403} (2005),
  216--228.

\bibitem{Bottcher-Wenzel}
\bysame, \emph{The {F}robenius norm and the commutator}, Linear Algebra Appl.
  \textbf{429} (2008), no.~8-9, 1864--1885.

\bibitem{Bourbaki-algebre4}
N.~Bourbaki, \emph{\'{E}l\'{e}ments de math\'{e}matique}, Masson, Paris, 1981,
  Alg\`ebre. Chapitres 4 \`a 7.

\bibitem{Cameron-Goethals-Seidel}
P.~J. Cameron, J.-M. Goethals, and J.~J. Seidel, \emph{The {K}rein condition,
  spherical designs, {N}orton algebras and permutation groups}, Nederl. Akad.
  Wetensch. Indag. Math. \textbf{40} (1978), no.~2, 196--206.

\bibitem{Cameron-Goethals-Seidel-stronglyregular}
\bysame, \emph{Strongly regular graphs having strongly regular
  subconstituents}, J. Algebra \textbf{55} (1978), no.~2, 257--280.

\bibitem{Cheng-Vong-Wenzel}
C.-M. Cheng, S.-W. Vong, and D.~Wenzel, \emph{Commutators with maximal
  {F}robenius norm}, Linear Algebra Appl. \textbf{432} (2010), no.~1, 292--306.

\bibitem{Chern-Docarmo-Kobayashi}
S.~S. Chern, M.~do~Carmo, and S.~Kobayashi, \emph{Minimal submanifolds of a
  sphere with second fundamental form of constant length}, Functional
  {A}nalysis and {R}elated {F}ields (F.~E. Browder, ed.), Springer, New York,
  1970, pp.~59--75.

\bibitem{Chern-Griffiths}
S.~S. Chern and Phillip Griffiths, \emph{Abel's theorem and webs}, Jahresber.
  Deutsch. Math.-Verein. \textbf{80} (1978), no.~1-2, 13--110.

\bibitem{Conway-monster}
J.~H. Conway, \emph{A simple construction for the {F}ischer-{G}riess monster
  group}, Invent. Math. \textbf{79} (1985), no.~3, 513--540.

\bibitem{DeSole-Kac}
A.~De~Sole and V.~G. Kac, \emph{Finite vs affine {$W$}-algebras}, Jpn. J. Math.
  \textbf{1} (2006), no.~1, 137--261.

\bibitem{Dieudonne-killing}
J.~Dieudonn\'{e}, \emph{On semi-simple {L}ie algebras}, Proc. Amer. Math. Soc.
  \textbf{4} (1953), 931--932.

\bibitem{Dong-Griess}
C.~Dong and R.~L. Griess, Jr., \emph{Rank one lattice type vertex operator
  algebras and their automorphism groups}, J. Algebra \textbf{208} (1998),
  no.~1, 262--275.

\bibitem{Elashvili-Jibladze-Kac-semisimple}
A.~G. Elashvili, M.~Jibladze, and V.~G. Kac, \emph{{Semisimple cyclic elements
  in semisimple Lie algebras}},
  \href{https://arxiv.org/abs/1907.09170}{arXiv:1907.09170}.

\bibitem{Faraut-Koranyi}
J.~Faraut and A.~Kor{\'a}nyi, \emph{Analysis on symmetric cones}, Oxford
  Mathematical Monographs, The Clarendon Press Oxford University Press, New
  York, 1994, Oxford Science Publications.

\bibitem{Fox-curvtensor}
D.~J.~F. Fox, \emph{The commutative nonassociative algebra of metric curvature
  tensors}, \href{https://arxiv.org/abs/1901.04012}{arXiv:1901.04012}.

\bibitem{Fox-ahs}
\bysame, \emph{Geometric structures modeled on affine hypersurfaces and
  generalizations of the {E}instein {W}eyl and affine hypersphere equations},
  \href{http://arxiv.org/abs/0909.1897}{arXiv:0909.1897}.

\bibitem{Fox-cubicpoly}
\bysame, \emph{Harmonic cubic homogeneous polynomials such that the
  norm-squared of the {H}essian is a multiple of the euclidean quadratic form},
  \href{https://arxiv.org/abs/1905.00071}{arXiv:1905.00071}.

\bibitem{Fox-crm}
\bysame, \emph{Geometric structures modeled on affine hypersurfaces and
  generalizations of the {E}instein-{W}eyl and affine sphere equations},
  Extended abstracts {F}all 2013---geometrical analysis, type theory, homotopy
  theory and univalent foundations, Trends Math. Res. Perspect. CRM Barc.,
  vol.~3, Birkh\"{a}user/Springer, Cham, 2015, pp.~15--19.

\bibitem{Fox-frames}
\bysame, \emph{Tight two-distance unimodular frames and {N}orton algebras}, in
  preparation.

\bibitem{Frenkel-Ben-zvi}
E.~Frenkel and D.~Ben-Zvi, \emph{Vertex algebras and algebraic curves},
  Mathematical Surveys and Monographs, vol.~88, American Mathematical Society,
  Providence, RI, 2001.

\bibitem{Frenkel-Lepowsky-Meurman}
I.~Frenkel, J.~Lepowsky, and A.~Meurman, \emph{Vertex operator algebras and the
  {M}onster}, Pure and Applied Mathematics, vol. 134, Academic Press Inc.,
  Boston, MA, 1988.

\bibitem{Freudenthal}
H.~Freudenthal, \emph{Oktaven, {A}usnahmegruppen und {O}ktavengeometrie}, Geom.
  Dedicata \textbf{19} (1985), no.~1, 7--63.

\bibitem{Ge-Li-Zhou}
J.~Ge, F.~Li, and Y.~Zhou, \emph{Some generalizations of the ddvv-type
  inequalities}, \href{https://arxiv.org/abs/math/1807.07307}{
  arXiv:1807.07307}.

\bibitem{Gebert}
R.~W. Gebert, \emph{Introduction to vertex algebras, {B}orcherds algebras and
  the {M}onster {L}ie algebra}, Internat. J. Modern Phys. A \textbf{8} (1993),
  no.~31, 5441--5503.

\bibitem{Gell-Mann}
M.~Gell-Mann, \emph{Symmetries of baryons and mesons}, Phys. Rev. (2)
  \textbf{125} (1962), 1067--1084.

\bibitem{Griess-friendlygiant}
R.~L. Griess, Jr., \emph{The friendly giant}, Invent. Math. \textbf{69} (1982),
  no.~1, 1--102.

\bibitem{Griess-monster}
\bysame, \emph{The {M}onster and its nonassociative algebra}, Finite
  groups---coming of age ({M}ontreal, {Q}ue., 1982), Contemp. Math., vol.~45,
  Amer. Math. Soc., Providence, RI, 1985, pp.~121--157.

\bibitem{Griess-Meierfrankenfeld-Segev}
R.~L. Griess, Jr., U.~Meierfrankenfeld, and Y.~Segev, \emph{A uniqueness proof
  for the {M}onster}, Ann. of Math. (2) \textbf{130} (1989), no.~3, 567--602.

\bibitem{Hall-transpositionalgebras}
J.~I. Hall, \emph{Transposition algebras}, Bull. Inst. Math. Acad. Sin. (N.S.)
  \textbf{14} (2019), no.~2, 155--187.

\bibitem{Hall-Rehren-Shpectorov-primitive}
J.~I. Hall, F.~Rehren, and S.~Shpectorov, \emph{Primitive axial algebras of
  {J}ordan type}, J. Algebra \textbf{437} (2015), 79--115.

\bibitem{Hall-Segev-Shpectorov}
J.~I. Hall, Y.~Segev, and S.~Shpectorov, \emph{Miyamoto involutions in axial
  algebras of {J}ordan type half}, Israel J. Math. \textbf{223} (2018), no.~1,
  261--308.

\bibitem{Hall-Segev-Shpectorov-primitiveaxial}
\bysame, \emph{On primitive axial algebras of {J}ordan type}, Bull. Inst. Math.
  Acad. Sin. (N.S.) \textbf{13} (2018), no.~4, 397--409.

\bibitem{Harada}
K.~Harada, \emph{On a commutative nonassociative algebra associated with a
  multiply transitive group}, J. Fac. Sci. Univ. Tokyo Sect. IA Math.
  \textbf{28} (1981), no.~3, 843--849 (1982).

\bibitem{Ivanov}
A.~A. Ivanov, \emph{The {M}onster group and {M}ajorana involutions}, Cambridge
  Tracts in Mathematics, vol. 176, Cambridge University Press, Cambridge, 2009.

\bibitem{Jacobson}
N.~Jacobson, \emph{Lie algebras}, Interscience Tracts in Pure and Applied
  Mathematics, No. 10, Interscience Publishers, New York-London, 1962.

\bibitem{Jacobson-jordan}
\bysame, \emph{Structure and representations of {J}ordan algebras}, American
  Mathematical Society Colloquium Publications, Vol. XXXIX, American
  Mathematical Society, Providence, R.I., 1968.

\bibitem{Kac-vertexbook}
V.~Kac, \emph{Vertex algebras for beginners}, second ed., University Lecture
  Series, vol.~10, American Mathematical Society, Providence, RI, 1998.

\bibitem{Kinyon-Sagle}
M.~K. Kinyon and A.~A. Sagle, \emph{Nahm algebras}, J. Algebra \textbf{247}
  (2002), no.~2, 269--294.

\bibitem{Knus-Merkurjev-Rost-Tignol}
M.-A. Knus, A.~Merkurjev, M.~Rost, and J.-P. Tignol, \emph{The book of
  involutions}, American Mathematical Society Colloquium Publications, vol.~44,
  American Mathematical Society, Providence, RI, 1998.

\bibitem{Kokoris}
L.~A. Kokoris, \emph{Simple power-associative algebras of degree two}, Ann. of
  Math. (2) \textbf{64} (1956), 544--550.

\bibitem{Krasnov-Tkachev}
Y.~Krasnov and V.~G. Tkachev, \emph{{Variety of idempotents in nonassociative
  algebras}}, \href{http://arxiv.org/abs/1801.00617}{arXiv:1801.00617}.

\bibitem{Krasnov-Tkachev-idempotentgeometry}
\bysame, \emph{Idempotent geometry in generic algebras}, Adv. Appl. Clifford
  Algebr. \textbf{28} (2018), no.~5, Art. 84, 14.

\bibitem{Laquer}
H.~T. Laquer, \emph{Invariant affine connections on {L}ie groups}, Trans. Amer.
  Math. Soc. \textbf{331} (1992), no.~2, 541--551.

\bibitem{Lepowsky-Li}
J.~Lepowsky and H.~Li, \emph{Introduction to vertex operator algebras and their
  representations}, Progress in Mathematics, vol. 227, Birkh\"auser Boston,
  Inc., Boston, MA, 2004.

\bibitem{Liu-curvatureestimates}
X.~Liu, \emph{Curvature estimates for irreducible symmetric spaces}, Chinese
  Ann. Math. Ser. B \textbf{27} (2006), no.~3, 287--302.

\bibitem{Lu-normalscalar}
Z.~Lu, \emph{Normal scalar curvature conjecture and its applications}, J.
  Funct. Anal. \textbf{261} (2011), no.~5, 1284--1308.

\bibitem{Lu-remarks}
\bysame, \emph{Remarks on the {B}\"ottcher-{W}enzel inequality}, Linear Algebra
  Appl. \textbf{436} (2012), no.~7, 2531--2535.

\bibitem{Lu-Wenzel}
Z.~Lu and D.~Wenzel, \emph{The normal {R}icci curvature inequality}, Recent
  advances in the geometry of submanifolds---dedicated to the memory of
  {F}ranki {D}illen (1963--2013), Contemp. Math., vol. 674, Amer. Math. Soc.,
  Providence, RI, 2016, pp.~99--110.

\bibitem{Macfarlane-Pfeiffer}
A.~J. Macfarlane and H.~Pfeiffer, \emph{On characteristic equations, trace
  identities and {C}asimir operators of simple {L}ie algebras}, J. Math. Phys.
  \textbf{41} (2000), no.~5, 3192--3225.

\bibitem{Macfarlane-Sudbery-Weisz}
A.~J. Macfarlane, A.~Sudbery, and P.~H. Weisz, \emph{On {G}ell-{M}ann's
  {$\lambda$}-matrices, {$d$}- and {$f$}-tensors, octets, and parametrizations
  of {$SU(3)$}}, Comm. Math. Phys. \textbf{11} (1968), no.~1, 77--90.

\bibitem{Mason-fivepieces}
G.~{Mason}, \emph{{Five not-so-easy pieces:open problems about vertex rings}},
  \href{https://arxiv.org/abs/1812.06206}{arXiv:1812.06206}.

\bibitem{Matsuo}
A.~Matsuo, \emph{Norton's trace formulae for the {G}riess algebra of a vertex
  operator algebra with larger symmetry}, Comm. Math. Phys. \textbf{224}
  (2001), no.~3, 565--591.

\bibitem{Matsuo-Nagatomo}
A.~Matsuo and K.~Nagatomo, \emph{Axioms for a vertex algebra and the locality
  of quantum fields}, MSJ Memoirs, vol.~4, Mathematical Society of Japan,
  Tokyo, 1999.

\bibitem{Meyer-Neutsch}
W.~Meyer and W.~Neutsch, \emph{Associative subalgebras of the {G}riess
  algebra}, J. Algebra \textbf{158} (1993), no.~1, 1--17.

\bibitem{Miyamoto-griessalgebras}
Masahiko Miyamoto, \emph{Griess algebras and conformal vectors in vertex
  operator algebras}, J. Algebra \textbf{179} (1996), no.~2, 523--548.

\bibitem{Nadirashvili-Tkachev-Vladuts}
N.~Nadirashvili, V.~Tkachev, and S.~Vl\u{a}du\c{t}, \emph{Nonlinear elliptic
  equations and nonassociative algebras}, Mathematical Surveys and Monographs,
  vol. 200, American Mathematical Society, Providence, RI, 2014.

\bibitem{Okubo}
S.~Okubo, \emph{Construction of nonassociative algebras from representation
  modules of simple {L}ie algebras}, Algebras Groups Geom. \textbf{3} (1986),
  no.~1, 60--127.

\bibitem{Okubo-octonion}
\bysame, \emph{Introduction to octonion and other non-associative algebras in
  physics}, Montroll Memorial Lecture Series in Mathematical Physics, vol.~2,
  Cambridge University Press, Cambridge, 1995.

\bibitem{Poonen}
B.~Poonen, \emph{Why all rings should have a 1}, Math. Mag. \textbf{92} (2019),
  no.~1, 58--62.

\bibitem{Popov}
V.~L. Popov, \emph{An analogue of {M}. {A}rtin's conjecture on invariants for
  nonassociative algebras}, Lie groups and {L}ie algebras: {E}. {B}. {D}ynkin's
  {S}eminar, Amer. Math. Soc. Transl. Ser. 2, vol. 169, Amer. Math. Soc.,
  Providence, RI, 1995, pp.~121--143.

\bibitem{Ryba}
A.~J.~E. Ryba, \emph{A natural invariant algebra for the {H}arada-{N}orton
  group}, Math. Proc. Cambridge Philos. Soc. \textbf{119} (1996), no.~4,
  597--614.

\bibitem{Schafer}
R.~D. Schafer, \emph{Structure and representation of nonassociative algebras},
  Bull. Amer. Math. Soc. \textbf{61} (1955), 469--484.

\bibitem{Schafer-book}
\bysame, \emph{An introduction to nonassociative algebras}, Pure and Applied
  Mathematics, Vol. 22, Academic Press, New York, 1966.

\bibitem{Smith}
S.~D. Smith, \emph{Nonassociative commutative algebras for triple covers of
  {$3$}-transposition groups}, Michigan Math. J. \textbf{24} (1977), no.~3,
  273--287.

\bibitem{Springer-Veldkamp}
T.~A. Springer and F.~D. Veldkamp, \emph{Octonions, {J}ordan algebras and
  exceptional groups}, Springer Monographs in Mathematics, Springer-Verlag,
  Berlin, 2000.

\bibitem{Suzuki-automorphismgroups}
H.~Suzuki, \emph{Automorphism groups of multilinear maps}, Osaka J. Math.
  \textbf{20} (1983), no.~3, 659--673.

\bibitem{Thompson}
J.~G. Thompson, \emph{Uniqueness of the {F}ischer-{G}riess monster}, Bull.
  London Math. Soc. \textbf{11} (1979), no.~3, 340--346.

\bibitem{Tits-friendly}
J.~Tits, \emph{On {R}. {G}riess' ``friendly giant''}, Invent. Math. \textbf{78}
  (1984), no.~3, 491--499.

\bibitem{Tits-monstre}
\bysame, \emph{Le {M}onstre (d'apr\`es {R}. {G}riess, {B}. {F}ischer et al.)},
  Ast\'erisque (1985), no.~121-122, 105--122, Seminar Bourbaki, Vol. 1983/84.

\bibitem{Tkachev-universality}
V.~G. Tkachev, \emph{The universality of one half in commutative nonassociative
  algebras with identities},
  \href{http://arxiv.org/abs/1808.03808}{arXiv:1808.03808}.

\bibitem{Tkachev-jordan}
\bysame, \emph{A {J}ordan algebra approach to the cubic eiconal equation}, J.
  Algebra \textbf{419} (2014), 34--51.

\bibitem{Tkachev-correction}
\bysame, \emph{A correction of the decomposability result in a paper by
  {M}eyer-{N}eutsch}, J. Algebra \textbf{504} (2018), 432--439.

\bibitem{Tkachev-extremal}
\bysame, \emph{On an extremal property of {J}ordan algebras of {C}lifford
  type}, Comm. Algebra \textbf{47} (2019), no.~3, 1057--1066.

\bibitem{Tkachev-summary}
\bysame, \emph{{V.M. Miklyukov: from dimension 8 to nonassociative algebras}},
  Mathematical Physics and Computer Simulation \textbf{22} (2019), no.~2,
  33--50.

\bibitem{Vinberg-automorphisms}
{\`E}.~B. Vinberg, \emph{Structure of the group of automorphisms of a
  homogeneous convex cone}, Trudy Moskov. Mat. Ob\v s\v c. \textbf{13} (1965),
  56--83.

\bibitem{Vinberg-invariantnorms}
\bysame, \emph{Invariant norms in compact {L}ie algebras}, Funkcional. Anal. i
  Prilo\v zen \textbf{2} (1968), no.~2, 89--90.

\bibitem{Waldron}
S.~Waldron, \emph{On the construction of equiangular frames from graphs},
  Linear Algebra Appl. \textbf{431} (2009), no.~11, 2228--2242.

\bibitem{Waldron-frames}
S.~F.~D. Waldron, \emph{An introduction to finite tight frames}, Applied and
  Numerical Harmonic Analysis, Birkh\"{a}user/Springer, New York, 2018.

\bibitem{Wu-Liu}
Y.-D. Wu and X.-Q. Liu, \emph{A short note on the {F}robenius norm of the
  commutator}, Math. Notes \textbf{87} (2010), no.~6, 903--907.

\bibitem{Zhan}
X.~Zhan, \emph{Matrix theory}, Graduate Studies in Mathematics, vol. 147,
  American Mathematical Society, Providence, RI, 2013.

\bibitem{Zhu-modularinvariance}
Y.~Zhu, \emph{Modular invariance of characters of vertex operator algebras}, J.
  Amer. Math. Soc. \textbf{9} (1996), no.~1, 237--302.

\bibitem{Zohrabi-Zumanovich}
A.~Zohrabi and P.~Zusmanovich, \emph{On {H}ermitian and skew-{H}ermitian matrix
  algebras over octonions}, \href{https://arxiv.org/abs/math/1912.03370}{
  arXiv:1912.03370}.

\end{thebibliography}
\def\polhk#1{\setbox0=\hbox{#1}{\ooalign{\hidewidth
  \lower1.5ex\hbox{`}\hidewidth\crcr\unhbox0}}} \def\cprime{$'$}
  \def\cprime{$'$} \def\cprime{$'$}
  \def\polhk#1{\setbox0=\hbox{#1}{\ooalign{\hidewidth
  \lower1.5ex\hbox{`}\hidewidth\crcr\unhbox0}}} \def\cprime{$'$}
  \def\cprime{$'$} \def\cprime{$'$} \def\cprime{$'$} \def\cprime{$'$}
  \def\polhk#1{\setbox0=\hbox{#1}{\ooalign{\hidewidth
  \lower1.5ex\hbox{`}\hidewidth\crcr\unhbox0}}} \def\cprime{$'$}
  \def\Dbar{\leavevmode\lower.6ex\hbox to 0pt{\hskip-.23ex \accent"16\hss}D}
  \def\cprime{$'$} \def\cprime{$'$} \def\cprime{$'$} \def\cprime{$'$}
  \def\cprime{$'$} \def\cprime{$'$} \def\cprime{$'$} \def\cprime{$'$}
  \def\cprime{$'$} \def\cprime{$'$} \def\cprime{$'$} \def\cprime{$'$}
  \def\dbar{\leavevmode\hbox to 0pt{\hskip.2ex \accent"16\hss}d}
  \def\cprime{$'$} \def\cprime{$'$} \def\cprime{$'$} \def\cprime{$'$}
  \def\cprime{$'$} \def\cprime{$'$} \def\cprime{$'$} \def\cprime{$'$}
  \def\cprime{$'$} \def\cprime{$'$} \def\cprime{$'$} \def\cprime{$'$}
  \def\cprime{$'$} \def\cprime{$'$} \def\cprime{$'$} \def\cprime{$'$}
  \def\cprime{$'$} \def\cprime{$'$} \def\cprime{$'$} \def\cprime{$'$}
  \def\cprime{$'$} \def\cprime{$'$} \def\cprime{$'$} \def\cprime{$'$}
  \def\cprime{$'$} \def\cprime{$'$} \def\cprime{$'$} \def\cprime{$'$}
  \def\cprime{$'$} \def\cprime{$'$} \def\cprime{$'$} \def\cprime{$'$}
  \def\cprime{$'$} \def\cprime{$'$} \def\cprime{$'$}
\providecommand{\bysame}{\leavevmode\hbox to3em{\hrulefill}\thinspace}
\providecommand{\MR}{\relax\ifhmode\unskip\space\fi MR }
\providecommand{\MRhref}[2]{%
  \href{http://www.ams.org/mathscinet-getitem?mr=#1}{#2}
}
\providecommand{\href}[2]{#2}

\end{document}